\documentclass[10pt]{article}

\usepackage{amsthm}
\usepackage{amsmath}
\usepackage{amsxtra}
\usepackage{amssymb}
\usepackage{enumitem}
\usepackage{accents}
\usepackage{mathrsfs}
\usepackage{turnstile}
\usepackage{thmtools}
\usepackage{makeidx}
\usepackage{url}
\usepackage[all]{xy}
\usepackage{caption}
\usepackage{graphics}
\usepackage{graphicx}
\usepackage{tikz}
\usetikzlibrary{matrix}
\usetikzlibrary{calc,decorations.pathmorphing,shapes}
\usepackage[colorlinks=true,linkcolor=blue]{hyperref}
\usepackage{latexsym}
\usepackage{verbatim}
\usepackage{changepage}
\usepackage{makeidx}
\usepackage{marvosym}
\usepackage{geometry}
\title{Varsovian models $\om$}
\author{Farmer Schlutzenberg}

\declaretheoremstyle[bodyfont=\sl]{slanted}

\declaretheorem[name=Definition,style=definition,qed=$\dashv$,numberwithin=section]{definition}
\declaretheorem[name=Definition,style=definition,numbered=no,qed=$\dashv$]{definition*}
\declaretheorem[name=Definition,style=definition,qed=$\dashv$,sibling=definition]{dfn}
\declaretheorem[name=Definition,style=definition,numbered=no,qed=$\dashv$]{dfn*}

\declaretheorem[name=Theorem,style=slanted,sibling=dfn]{tm}
\declaretheorem[name=Theorem,style=slanted,sibling=dfn]{tm*}

\declaretheorem[name=Theorem,style=slanted,numbered=no]{theorem*}

\declaretheorem[name=Lemma,style=slanted,sibling=dfn]{lem}

\declaretheorem[name=Corollary,style=slanted,numbered=no]{corollary*}
\declaretheorem[name=Corollary,style=slanted,sibling=dfn]{cor}
\declaretheorem[name=Corollary,style=slanted,numbered=no]{cor*}
\declaretheorem[name=Remark,style=definition,sibling=dfn]{rem}

\declaretheoremstyle[headfont=\scshape]{claimstyle}
\declaretheorem[name=Claim,style=claimstyle]{clm}

\declaretheorem[name=Case,style=claimstyle]{case}
\declaretheorem[name=Subcase,style=claimstyle]{scase}

\declaretheorem[name=Claim,style=claimstyle,numbered=no]{clm*}

\declaretheorem[name=Subclaim,style=claimstyle,numbered=no]{sclm*}

\declaretheorem[name=Subsubclaim,style=claimstyle,numbered=no]{ssclm*}

\makeindex
	{
	\end{minipage}
	\vspace*{\stretch{3}}
	\clearpage
}

\def\undertilde#1{{\baselineskip=0pt\vtop
		{\hbox{$#1$}\hbox{$\scriptscriptstyle\sim$}}}{}}

\newcommand{\down}{{\downarrow}}

\newcommand{\ph}{\mathfrak{P}}

\newcommand{\swsw}{\mathrm{swsw}}
\newcommand{\stk}{\mathrm{stk}}

\newcommand{\PP}{\mathbb{P}}

\newcommand{\M}{\mathcal{M}}
\newcommand{\N}{\mathcal{N}}
\newcommand{\J}{\mathcal{J}}

\newcommand{\Def}{\mathrm{Def}}

\renewcommand{\undertilde}{\undertilde}

\newcommand{\wW}{\mathscr{W}}
\newcommand{\ds}{\mathrm{ds}}

\newcommand{\rest}{{\upharpoonright}}

\newcommand{\forces}{\Vdash}

\renewcommand{\models}{\vDash}

\newcommand{\dom}{{\rm dom}}

\newcommand{\lh}{{\rm lh}}

\newcommand{\crit}{{\rm crit }}

\newcommand{\mM}{\mathscr{M}}
\newcommand{\nN}{\mathscr{N}}
\newcommand{\vV}{\mathscr{V}}
\newcommand{\HOD}{\mathrm{HOD}}
\newcommand{\sub}{\subseteq}
\newcommand{\om}{\omega}
\newcommand{\Coll}{\mathrm{Coll}}
\newcommand{\Mswsw}{M_{\mathrm{swsw}}}
\newcommand{\OR}{\mathrm{OR}}
\newcommand{\es}{\mathbb{E}}
\newcommand{\sss}{\mathrm{sh}}

\newcommand{\Ll}{\mathcal{L}}

\newcommand{\Tt}{\mathcal{T}}
\newcommand{\Uu}{\mathcal{U}}
\newcommand{\Vv}{\mathcal{V}}
\newcommand{\Ww}{\mathcal{W}}
\newcommand{\Xx}{\mathcal{X}}
\newcommand{\Yy}{\mathcal{Y}}

\newcommand{\conc}{\ \widehat{\ }\ }
\newcommand{\Ult}{\mathrm{Ult}}

\newcommand{\pins}{\triangleleft}
\newcommand{\sats}{\models}
\newcommand{\id}{\mathrm{id}}
\newcommand{\ins}{\trianglelefteq}
\newcommand{\dirlim}{\mathrm{dirlim}}
\newcommand{\Hull}{\mathrm{Hull}}
\newcommand{\inter}{\cap}
\newcommand{\com}{\circ}

\newcommand{\BB}{\mathbb{B}}
\newcommand{\cHull}{\mathrm{cHull}}

\newcommand{\OD}{\mathrm{OD}}
\newcommand{\rg}{\mathrm{rg}}

\newcommand{\dropset}{\mathscr{D}}
\newcommand{\pred}{\mathrm{pred}}

\newcommand{\tu}{\textup}

\newcommand{\passive}{{\mathrm{pv}}}
\newcommand{\pow}{\mathcal{P}}
\newcommand{\pvec}{\vec{p}}
\newcommand{\sn}{\mathrm{sn}}

\newcommand{\cut}{\backslash}
\newcommand{\rSigma}{\mathrm{r}\Sigma}
\newcommand{\rDelta}{\mathrm{r}\Delta}
\newcommand{\Th}{\mathrm{Th}}

\newcommand{\ZFC}{\mathrm{ZFC}}

\renewcommand{\P}{\mathcal{P}}
\newcommand{\CC}{\mathbb{C}}
\newcommand{\Vop}{\mathrm{Vop}}

\newcommand{\swom}{\sw\om}
\newcommand{\exit}{\mathrm{exit}}
\newcommand{\Mswom}{M_{\swom}}
\newcommand{\rank}{\mathrm{rank}}
\newcommand{\all}{\forall}
\newcommand{\sw}{\mathrm{sw}}

\begin{document}
\maketitle

\begin{abstract}
For $n<\om$, let $M_{\sw n}$ be the 
 minimal iterable proper class mouse $M$
 such that $M\sats$ ``there are ordinals $\delta_0<\kappa_0<\ldots<\delta_{n-1}<\kappa_{n-1}$ such that
 each $\delta_i$ is a Woodin cardinal and each $\kappa_i$ is a strong cardinal'', and $M_{\sw\om}$ likewise,
 but with $M\sats$ ``there is an ordinal $\lambda$ which is a limit of Woodins and strongs''.
 Under appropriate large cardinal hypotheses, Sargsyan and Schindler  \cite{vm1} introduced and analysed the Varsovian model of $M_{\sw1}$,
 and Sargsyan, Schindler and the author \cite{vm2_v2} introduced and anlaysed the Varsovian model of $M_{\sw2}$.
We extend this to $M_{\sw\om}$,
assuming $(\dagger)$ that $*$-translation integrates routinely with the P-constructions of this paper;
the write-up of this is  to appear. We show in particular (assuming $(\dagger)$)
 that $M_{\sw\om}$ has a proper class inner model $\vV_\om$
 which is a fully iterable strategy mouse with $\om$ Woodin cardinals, closed under its own strategy, and that the universe of $\vV_\om$
 is the eventual generic HOD, and the mantle, of $M$. We also show (assuming $(\dagger)$) that the core model $K^M$ of $M$
  (which can be defined in a natural manner)
  is an iterate  of $M$, is an inner model of $\vV_\om$,
 and  is fully iterable in $M$ and in $\vV_\om$. 
\end{abstract}

\tableofcontents

\section{Introduction}

Work in ZFC. Let $M_{\swom}$ be the minimal fully iterable proper class mouse such that \begin{equation}\label{eqn:sw_omega_property}\begin{split}M\sats\text{``}&\text{There is
a  cardinal }\lambda\text{ which is both }\\
&\text{a limit of Woodin cardinals and a limit of strong cardinals''.}\end{split}\end{equation}
In this paper we analyze  $M_{\swom}$ in the style of the papers \cite{vm0}, \cite{vm1} and \cite{vm2_v2}.
Let us first be precise about the setup.
For a premouse $N$, $F^N$ denotes the active extender of $N$.

\begin{dfn}\label{dfn:M_swom}
 For $n<\om$, $M_{\sw n}^\#$, assuming it exists, denotes the least $(0,\om_1+1)$-iterable active premouse $N$
 such that $N|\crit(F^N)\sats$ ``There are ordinals
 \[\delta_0<\kappa_0<\ldots<\delta_{n-1}<\kappa_{n-1} \]
 such that each $\delta_i$ is Woodin and each $\kappa_i$ is strong''. If $M_{\sw n}^\#$ exists,
 then $M_{\sw n}$
 is the proper class premouse left behind by iterating $F^{M_{\sw n}^\#}$ out of the universe.
 
 Analogously, $M_{\sw\om}^\#$ is the least $(0,\om_1+1)$-iterable active premouse $N$
 such that $N|\crit(F^N)$ satisfies the statement in line (\ref{eqn:sw_omega_property}),
 and $M_{\sw\om}$ is the result of iterating $F^{M_{\swom}^\#}$ out of the universe.
\end{dfn}

In \cite{vm1}, Sargsyan and Schindler analyzed the model $M_{\sw1}$, and in its sequel \cite{vm2_v2}, Sargsyan, Schindler and the author analzyed $M_{\sw2}$, showing for example in \cite{vm2_v2} that
$M_{\sw2}$ contains an inner model $\vV_2$,
which is a self-iterable strategy mouse satisfying ``there are exactly 2 Woodin cardinals'',
and which is also iterable in $M_{\sw 2}$.
It is also shown there that the universe of $\vV_2$
is $\HOD^{M_{\sw2}[G]}$ for $G$ being $(M_{\sw2},\Coll(\om,\eta))$-generic,
and is also the mantle of all set generic extensions of $M_{\sw2}$. In \cite{vm1}, similar results were established, with ``$1$ (Woodin)'' replacing ``$2$ (Woodins)'', although the kind of analysis and presentation of $\vV_2$ used in \cite{vm2_v2}
was somewhat different to the analogue in \cite{vm1}.

In \cite{vm1}, some further results were also established, which were not fully echoed in \cite{vm2_v2}. That is, in \cite{vm1}, a certain premouse $N_1=\M_{\infty}^{M_{\sw1}}$ was identified (the direct limit of the direct limit system considered there), and  it was shown that $N_1$ is an iterate of $M_{\sw1}$,
 $N_1$ is fully iterable in $M$,
and  $N_1$ is the core model $K^{M_{\sw1}}$ of $M_{\sw1}$ (with the core model defined in a natural manner for that context). But the full analogue of these results were missing from \cite{vm2_v2}. Here, we will establish the analogue of the results of \cite{vm1} in this regard.

 \begin{rem}\label{rem:*-trans}
Some of the proofs in this paper (including those of the main results),  like some in \cite{vm2_v2},
are not as of yet complete, because they rely on 
an integration of $*$-translation \cite{closson}
with the kinds of P-constructions we use here.
The integration requires some work to write up in detail, 
but is expected to be straightforward.
This work will appear in \cite{*-trans_add}.
\end{rem}

In sumary we prove, modulo the remaining details just mentioned above, the following theorem:

\begin{tm}\label{tm:main}Assume ZFC 
 + ``$M_{\sw\om}^\#$ exists and is $(\omega,\OR)$-iterable''. 
Assume that the appropriate adaptation of $*$-translation goes through routinely (see Remark \ref{rem:*-trans}).

Let $M=M_{\swom}$.
Then  $M$ has an inner model $\vV_\om$,
and $\vV_\om$ has an inner model $K$, such that:
\begin{enumerate}
 \item $\vV_\om$ is a fully iterable strategy mouse with $\om$ Woodin cardinals that is closed under its iteration strategy,
 \item $\vV_\om$ is lightface definable over $M$.
 \item Let $U$ be the universe of $\vV_\om$.
 Then:
 \begin{enumerate}
 \item $U$  is a ground of $M$,
 \item\label{item:main_3b} $U$ 
is the mantle of $M$ and of all set-generic extensions of $M$,
\item\label{item:main_3c} $U=\HOD^{M[G]}$
whenever $G\sub\Coll(\om,\eta)$ is $M$-generic and $\eta\geq\lambda$,
\end{enumerate}
\item $K$ is a non-dropping iterate of $M$,
via the unique $(0,\OR)$-strategy for $M$,
\item $K$ is fully iterable in $\vV_\om$ and in $M$ and in $V$,
\item\label{item:main_6} $K$ is the core model of $\vV_\om$ and of $M$.
\end{enumerate}
 \end{tm}
 
 What is meant by \emph{core model} above
 is clarified at the start of \S\ref{sec:main_proof}.

We now discuss somewhat the setup for the argument.
Recall that in \cite{vm2_v2},
starting with $\Vv_0=M_{\sw2}$,
the model $M_\infty\sub\vV_0$
is first constructed; this is the direct limit
of all correct iterates of $\vV_0$
via trees $\Tt$ such that $\Tt$ is based on $\vV_0|\delta_0^{\vV_0}$
and $\Tt$ has successor length $\eta+1$ and $b^\Tt$ does not drop, and $\Tt\rest\eta\in \vV_0|\kappa_0$,
where $\kappa_0$ is the least strong of $\vV_0$.
The model $\vV_1$ is then constructed,
with $M_\infty\sub\vV_1\sub\vV_0$.
Here $\vV_1$ has universe $M_\infty[*]$,
where $*$ is the $*$-map associated
to the direct limit system leading to $M_\infty$,
but $\vV_1$ is given a presentation as a fine structural presentation strategy mouse,
closed under strategy for trees based on $M_\infty|\delta_0^{M_\infty}$;
also, letting $\delta=\delta_0^{M_\infty}$,
we have that $\vV_0|\delta=M_\infty|\delta$
and $V_{\delta}^{\vV_0}=V_\delta^{M_\infty}$
is the universe of $\vV_0|\delta$,
and $\delta$ is Woodin in $\vV_0$,
although $V_{\delta+1}^{M_\infty}\subsetneq V_{\delta+1}^{\vV_0}$.
A second direct limit system
is then considered,
this for iterating $\vV_1$,
with trees based on $\vV_1|\delta_1^{\vV_1}$
(and $\delta_1^{\vV_1}=\delta_1^{M_\infty}=\delta_1^M$ is Woodin in $\vV_1$),
and, analogously to before, whose maximal initial segment $\Tt\rest\eta\in \vV_1|\kappa_1$,
where $\kappa_1=\kappa_1^{M_\infty}=\kappa_1^M$ is the least strong cardinal of $M$ such that $\delta_1^M<\kappa_1$ (and this is the least strong cardinal of $\vV_1$ overall).
This produces a direct limit model $N_\infty$,
with $N_\infty\sub\vV_1$,
and then the model $\vV_2$
is constructed, with $N\sub\vV_1\sub\vV_1$,
and $\vV_2$ is a fine structural strategy
mouse with 2 Woodins, closed under its own strategy, and it is fully iterable in $V$ and in $M$ and in itself. Its universe is the mantle and eventual generic HOD of $M$.

In this paper we will proceed analogously,
but now we will produce a sequence $\left<\vV_n\right>_{n\leq\om}$,
with $\vV_0=M=M_{\sw\om}$
and $\vV_{n+1}\sub\vV_n$ and $\vV_\om\sub\vV_n$ for all $n$,
and also an $n$th level $M_\infty^{\vV_n}\sub\vV_n$,
which is an iterate of $\vV_n$,
the direct limit all correct iterates of $\vV_n$
via trees based on $\vV_n|\delta_n^{\vV_n}$
whose maximal initial segment $\Tt\rest\eta\in\vV_n|\kappa_n^M$.
And then $\vV_{n+1}$ will have
universe that of $M_\infty^{\vV_n}[*]$,
where $*$ is the $*$-map associated
to this direct limit,
but $\vV_{n+1}$ will be presented
as a fine structural strategy mouse,
closed under strategy for trees based on $\vV_{n+1}|\delta_n^{\vV_n}$.
We will also define a closely related sequence, $\left<\wW_n\right>_{n\leq\om}$,
with $\wW_0=\vV_0=M$ and  $\wW_i\sub\wW_j$ for all $i\leq j\leq\om$, and $\wW_{n+1}$ related to $\wW_n$ in the same manner that $\vV_{n+1}$ relates to $M_\infty^{\vV_n}$.

Aside from verifying that the finite stages
$\left<\vV_n\right>_{n<\om}$ and $\left<\wW_n\right>_{n<\om}$ 
work out fine (there are some minor further subtleties to handle here, which were not dealt with in \cite{vm2_v2}),
the next most obvious issue is how to
identify a/the limit model
$\vV_\om$. The key here will
be to see that non-dropping iteration maps $j$ on $\vV_{n+1}$ act on the inner model $M_\infty^{\vV_n}\sub\vV_{n+1}$
also as iteration maps on that model.
Therefore we can consider the eventual
iterate $\vV_{n\om}$ of $\vV_n$.
That is, writing $j_{n,n+1}:\vV_n\to(M_\infty)^{\vV_n}$ for the iteration map, then $j_{n+1,n+2}(M_\infty^{\vV_n})$ is in fact a correct iterate of $(M_\infty)^{\vV_n}$,
and $j_{n+1,n+2}\rest (M_\infty)^{\vV_n}$ is the corresponding iteration map;
likewise for $j_{n+2,n+3}(j_{n+1,n+2}(M_\infty^{\vV_n}))$, etc.
Thus, writing $\vV_{n\om}$ for the resulting
iterate of $\vV_n$ under these maps
and $j_{n\om}:\vV_n\to\vV_{n\om}$
the iteration map (so $j_{n+1,\om}\com j_{n,n+1}=j_{n\om}$), then $\vV_{n\om}\sub\vV_{n+1,\om}$
(and this pair of models relate
much like $(M_\infty)^{\vV_n}$ relates to $\vV_{n+1}$), and we can consider the model $L[\left<\vV_{n\om}\right>_{n<\om}]$.
This will be the universe of the $\om$th Varsovian model $\vV_\om$,
and we will present $\vV_\om$
as a fine structural strategy mouse,
with $\om$ Woodins, closed under its own strategy.

The reader we have in mind is familiar with \cite{vm2_v2}. Much material there works here with essentially no or very little change, so there is much from there that we will not bother rewriting here. However, we execute significant portions
somewhat differently, so a rather large fraction
is redone here. Thus, if a reader is interested in reading this paper, but is not yet familiar with \cite{vm2_v2}, it probably makes most sense to proceed ahead, and refer to \cite{vm2_v2} as needed.

***Some of the notation/terminology used in the earlier parts of the paper are not used later. Primarily,
the notation/terminology of \S\ref{sec:vV_1,wW_1}
is based more on that of \cite{vm2_v2},
but parts of this are done somewhat differently from \S\ref{sec:vV_n+1+to_vV_n+2} onward. This might be a little annoying, but hopefully clear enough. 
For example, the terminology \emph{Vsp (Varsovian strategy premouse)} is used in \S\ref{sec:vV_1,wW_1}, but this is dropped,
and replaced with the related (but rather uninspired) terminology \emph{pVl (provisionary Varsovian-like)} in \S\ref{sec:vV_n+1+to_vV_n+2}. These issues should be improved in a later draft of the paper,
by updating \S\ref{sec:vV_1,wW_1} to match better with the later notions.

\subsection{Acknowledgments}

I would like to thank Ralf Schindler for many discussions on Varsovian models and his suggestion for simplifying a certain forcing construction, as described at the start of \S\ref{sec:con_vV_om}.
I also thank the organizers of the Berkeley conference in inner model theory, 2019,
and the European set theory society conference, Vienna, 2019, 
for providing the opportunity to   talk about the work in this paper.
I also thank Sandra M\"uller
and TU Vienna
for their hospitality during my stay
in Vienna in November 2022,
and providing further opportunity
to present work from the paper.

The work was supported by the Deutsche Forschungsgemeinschaft (DFG, German Research Foundation) under Germany's Excellence Strategy EXC 2044-390685587, Mathematics M\"unster: Dynamhcs-Geometry-Structure.
Teilweise gef\"ordert durch die Deutsche Forschungsgemeinschaft (DFG) -- Projektnummer 445387776.
Partly supported by the Deutsche Forschungsgemeinschaft (DFG, German Research Foundation) -- project number 445387776.

\subsection{Terminology and notation}

The premice in this paper are use $\lambda$-indexing.  
We formally use Mitchell-Steel fine structure,
as modified in \cite[\S5]{V=HODX_pub} (though the choice of fine structure  probably does not play a significant role). We mostly follow the terminology and notation
of \cite{fullnorm_v3}, \cite{iter_for_stacks},
and \cite{vm2_v2}, so we don't bother to repeat the summaries in those papers. However, we collect below a few notational components which either differ from or are not covered by the sources just mentioned (but this list might not be complete):
\begin{enumerate}[label=--]
 \item 
One difference between this paper and \cite{vm2_v2} is that here, if $M$ is a premouse or related structure then $M|\alpha$
denotes the initial segment of $M$ of ordinal height $\alpha$, including extender indexed at $\alpha$, whereas $M||\alpha=(M|\alpha)^{\passive}$,
where $P^\passive$ is the ``passivization'' of a premouse $P$, i.e. the passive premouse with the same height as has $P$, and which agrees with $P$ strictly below its height.
\item One difference between here and \cite{iter_for_stacks}
is that here, $\lh(E)$ denotes the index of a premouse $E$ in $\es_+^M$ (this is written differently in \cite{iter_for_stacks}).

\item If $\Tt$ is an iteration tree on a premouse,
then $\exit^\Tt_\alpha$ denotes $M^\Tt_\alpha|\lh(E^\Tt_\alpha)$.

\item  If $M$ is a proper class premouse,
then a \emph{non-dropping iterate} of $M$
is a premouse $N$ of form $M^\Tt_\infty$,
where $\Tt$ is an iteration tree on $M$
of successor length such that $[0,\infty]^{\Tt}$ does not drop, and if $\Tt$ is proper class
then $\OR^N=\OR$ (that is, we do not allow $N$ to have ``iterated past the ordinals'').

\item If $M$ is a premouse, and $\eta<\OR^M$
and $X\sub\eta$,
then $\eta$ is an \emph{$X$-cutpoint} of $M$ iff
for every $E\in\es_+^M$,
if $\crit(E)<\eta<\lh(E)$ then $\crit(E)\in X$.
And $\eta$ is a \emph{strong $X$-cutpoint}
iff for every $E\in\es_+^M$,
if $\crit(E)\leq\eta<\lh(E)$ then $\crit(E)\in X$.

\item Let $M$ be a premouse and $\Sigma$ an iteration strategy for $M$.
Then a \emph{$\Sigma$-iterate}
is just a premouse of form $M^\Tt_\infty$
for some successor length tree $\Tt$ via $\Sigma$ (hence with $M^\Tt_0=M$). For $\eta\leq\OR^M$,
 $\Sigma_{\eta}$
denotes the restriction of $\Sigma$ to trees based on $M|\eta$.

\item If $N$ is an iterate of $M$ (via some context determining strategy $\Sigma$)
then $i_{MN}:M\to N$ often denotes the iteration map.
\end{enumerate}
\section{Properties of $\Mswom$ and $\Sigma_{\Mswom}$}

\begin{dfn}Recall we defined $M_{\sw\om}^\#$ and $M_{\sw\om}$
in Definition \ref{dfn:M_swom}. We write $M=M_{\swom}$ and $M^{\#}=M_{\swom}^{\#}$.  We assume 
	throughout that $M^\#$ exists and is $(\om,\OR)$-iterable.

	Easily $\rho_1^{M^{\#}}=\om$
and $p_1^{M^{\#}}=\emptyset$ and  $M$ is sound, so
a (and the unique) $(\om,\OR)$-strategy for $M^{\#}$ is equivalent to a (and the) $(0,\OR)$-strategy for $M^{\#}$, and the latter we denote $\Sigma_{M^{\#}}$.
Now $\Sigma_{M^{\#}}$  extends
canonically to an optimal\footnote{The prefix ``optimal-'' just means that the stack is formed according to the iteration game for forming stacks in which player 1 is not allowed to drop artificially between rounds.}-$(0,\OR,\OR)$-strategy $\Sigma_{M^{\#}}^{\stk}$, via normalization as in
see \cite{fullnorm_v3} (cf.~\cite[Fact 3.4, {$\Sigma1$}]{vm2_v2}). Let $\Sigma=\Sigma_M$ be the $(0,\OR)$-strategy for $M$ induced by $\Sigma_{M^{\#}}$.
Then this is also the unique $(0,\OR)$-strategy
for $M$, and extends  canonically to an optimal-$(0,\OR,\OR)$-strategy for $M$,
which we denote $\Gamma=\Sigma_{M}^{\stk}$.
\end{dfn}

\begin{dfn} Let $\psi_{\swom}$ be the statement in the passive premouse language asserting ``There is an ordinal $\lambda$ which is a limit of strong cardinals
	and a limit of Woodin cardinals, as witnessed by $\es$''.
\end{dfn}

\begin{rem}  \cite[Remark 3.2--Definition 3.7 (inclusive)]{vm2_v2} goes through
	unmodified for $M,\Sigma,\Gamma$ as defined above (regarding the definitions there,  we make the analogous definitions here). We replace \cite[Definition 3.8]{vm2_v2} with the definition immediately below, and given that change,
	\cite[Fact 3.9]{vm2_v2}
goes through after replacing both instances of ``$\kappa_1^N$'' there with ``$\lambda^N$'',
and \cite[Lemma 3.10]{vm2_v2} goes through unmodified, with essentially the same proof. (Let us ignore \cite[Lemma 3.11]{vm2_v2} at least for the moment; it was not used directly in \cite{vm2_v2}.)\end{rem}

\begin{dfn}\label{dfn:Mswom-like}
	A premouse $N$ is \emph{$\Mswom$-like}
	iff (i) $N$ is proper class, (ii) $N\sats\psi_{\swom}$, (iii) there is no active $P\pins N$
	such that $P|\crit(F^P)\sats\psi_{\swom}$, and (iv) $N\sats T$, where $T$ is a certain finite sub-theory of 
	the theory of $M$. We leave it to the
	reader to supplement $T$ with whatever statements are needed along the way
	to make arguments work.
	For an $\Mswom$-like  model $N$, write
\begin{enumerate}[label=--]
\item $\delta_0^N=$ the least Woodin  of $N$,
\item $\kappa_0^N=$ the least strong of $N$ (note $\kappa_0^N>\delta_0^N$),
\item $\delta_{n+1}^N=$ the least Woodin $\delta$ of $N$ such that $\delta>\kappa_n^N$,
\item $\kappa_{n+1}^N=$ the least strong $\kappa$ of $N$ such that $
\kappa>\kappa_n^N$ (note $\kappa_{n+1}^N>\delta_{n+1}^N$),
\item $\lambda=\sup_{n<\om}\kappa_n^N=\sup_{n<\om}\delta_n^N$.
\end{enumerate}
(So $\kappa_n^N$ is the $n$th strong cardinal of $N$, but note each strong
is a limit of Woodins.) We also write $\kappa_n^{+N}=(\kappa_n^N)^{+N}$. 

	If $N=M$, we may supress the superscript $N$, and so $\delta_0 =
	\delta_0^{M}$, etc.\footnote{We also write $\eta^{+N}$ for the cardinal
		successor of $\eta$ as computed in $N$. So there is an ambiguity
		in the notation $\kappa_n^{+N}$, since $\kappa_n$ denotes $\kappa_n^M$.
		But whenever we write $\kappa_n^{+N}$, we mean $(\kappa_n^N)^{+N}$,
		not $(\kappa_n)^{+N}=(\kappa_n^M)^{+N}$ (though often these will in fact coincide).}
\end{dfn}

\section{The first Varsovian models $\vV_1$ and $\wW_1$}\label{sec:vV_1,wW_1}

In this section we define and anlayse $\vV_1$
and its iteration strategy $\Sigma_{\vV_1}$. Much of this is a direct adaptation of material in \cite{vm2_v2}, and many of the actual details are implemented just as there, and so we don't bother to repeat those. But we will  describe
a key new fact, Lemma \ref{lem:Psi_vV_1_is_conservative},
which is essential
to our later definition and analysis of $\vV_\om$,
and also make some modifications to the setup in \cite{vm2_v2}.
 In the end we will modify the setup in otherw ways in the general construction of $\vV_{n+2}$ from $\vV_{n+1}$ in \S\ref{sec:vV_n+1+to_vV_n+2},
and it would be somewhat more optimal to have the setup in the present section agree with the later one, but hopefully this won't be too bothersome. (***Improve this.)

Recall that \cite[\S2]{vm2_v2} is set up in general, without any particular base model $M$.
Now \cite[\S\S4.1--4.6,4.8--4.11***]{vm2_v2}  go through as there (still with the help
of \cite{*-trans_add} for $*$-translation), changing $M=\Mswsw$ to $M=\Mswom$, and \emph{$\Mswsw$-like} to \emph{$\Mswom$-like}.

We add the following corollary of \cite[Lemma 4.57***]{vm2_v2}:

\begin{lem}
	Let $N$ be any non-dropping $\Sigma$-iterate of $M$
	and $N'$ any non-dropping $\Sigma_{N}$-iterate of $N$
	such that $N|\kappa_0^N\pins N'$.
	Then there are $\Sigma$-iterates $\bar{\M_\infty^N}$
	and $\bar{\M_\infty^{N'}}$ of $M$
	which are $\delta_0^{\bar{\M_\infty^N}}$-sound
	and $\delta_0^{\bar{\M_\infty^{N'}}}$-sound
	respectively,
	with $\M_\infty^N|\delta_0^{\M_\infty^N}=\bar{\M_\infty^N}|\delta_0^{\bar{\M_\infty^N}}$
	and $\M_\infty^{N'}|\delta_0^{\M_\infty^{N'}}=
	\bar{\M_\infty^{N'}}|\delta_0^{\bar{\M_\infty^{N'}}}$,
	and $\bar{\M_\infty^{N'}}$ is a $\Sigma_{\bar{\M_\infty^N}}$-iterate
	of $\bar{\M_\infty^N}$,
	and $i_{NN'}\rest(\M_\infty^N|\delta_0^{\M^N_\infty})$
	agrees with the iteration map $i_{\bar{\M_\infty^N}\bar{\M_\infty^{N'}}}$.
	\end{lem}
\begin{proof}
	Applying \cite[Lemma 4.57***]{vm2_v2} to $\bar{N}=$ the $\kappa_0^N$-hull
	of $N$, we have that $\M_\infty^{\bar{N}}$ is a $\delta_0^{\M_\infty^{\bar{N}}}$-sound
		$\Sigma$-iterate of $M$. Setting $\bar{\M_\infty^N}=\M_\infty^{\bar{N}}$,
		it agrees with $\M_\infty^N$ appropriately.
		Let $\bar{N'}$ and $\bar{\M_\infty^{N'}}=\M_\infty^{\bar{N'}}$
		be likewise. Note that normalizing, we get a normal tree leading from $\bar{N}$
		to $N'$, and factoring this, a normal tree leading from $\bar{N}$
		to $\bar{N'}$ (via the correct strategies). But then
		by \cite[Lemma 4.57***]{vm2_v2},
		$i_{\bar{N}\bar{N'}}\rest\M_\infty^{\bar{N}}:\M_\infty^{\bar{N}}\to\M_\infty^{\bar{N'}}$ is just the iteration map $\M_\infty^{\bar{N}}\to\M_\infty^{\bar{N'}}$,
		which gives what we need.	
	\end{proof}

Note that the lemma applies in particular when $N'=\Ult(N,E)$ for some $E\in\es^N$
which is $N$-total with $\crit(E)=\kappa_0^N$.

We next follow \cite[\S12]{vm2_v2} through its Lemma 4.75 (inclusive).
At this point, we give a neater definition for $\Psi^{\sn}_{\vV_1}$ than that
used in \cite{vm2_v2}, following instead the form of \cite[Definition 5.67***]{vm2_v2}
(note we are adapting all notation and terminology directly from \cite{vm2_v2},
so, for example, $M$ now denotes $\Mswom$, $\vV_1$ denotes $\vV_1^M$, etc).

\begin{dfn}\label{dfn:1-Vsp}
 Recall the notion of \emph{Vsp} from \cite[Definition 4.51***]{vm2_v2}.
 We have now adapted this to the present notion of \emph{Vsp} (without explicit mention).
 We also call this \emph{1-Vsp}
(later we will use instead \emph{$n$-pVl},
 a related notion).
Given a 1-Vsp,
 an extender in $\es^N$  is \emph{$0$-long}  if it is long (and hence its measure space is $\delta_0^N$,
including that $\delta_0^N$ exists in the first place).
 Recall that by definition, every initial segment of an 1-Vsp is also an 1-Vsp. So regular premice are also 1-Vsps.
A $1$-Vsp is \emph{proper}
if there is a $0$-long extender in $\es_+^N$.

Suppose $N$ is a proper $1$-Vsp.
Then $\gamma^N=\gamma_0^N$ |denotes the least $\gamma$ such that $F^{N|\gamma}$ is $0$-long,
$e^N=e_0^N$ denotes $F^{N|\gamma_0^N}$,
and $\Delta^{N}$ denotes $\delta_{0}^{N}$.
If $k\leq\om$ is the soundness
degree of $N$ (if $N$ is proper class, it doesn't matter what $k$ is), then we write $N\down 0$
for the structure $P$ such that $P|\gamma^N=N||\gamma^N$, and letting $U=\Ult_k(N,e^N)$,
then $\OR^P=\OR^U$
and $\es^P\rest(\gamma^N,\OR^U]$ is given by translating $\es^U\rest(\gamma^N,\OR^U]$;
that is, for each $\nu\in(\gamma^N,\OR^U]$,
\begin{enumerate}[label=--]
\item $P|\nu$ is active iff $U|\nu$ is active, and
\item if $U|\nu$ is active then $F^{P|\nu}\rest\OR=F^{U|\nu}\rest\OR$
(and $F^{P|\nu}$ is short),
\end{enumerate}
We demand moreover of proper class
proper $1$-Vsps $N$ (as part of the definition)
 that $U=\vV^P=\vV_{1}^P$. We also write $\vV_{1\down 0}$ for $\vV_1\down 0=\mM_{\infty0}$.
\end{dfn}

Recall that $\vV_{1\down 0}=\mM_{\infty0}$ is
	 a non-dropping $\Sigma_{\vV_0}$-iterate of 
	 $\vV_0$, via a (short-normal) tree based on $\delta_{0}^{\vV_0}$ (and recall here that $\vV_0=M$).

\begin{dfn}
 Let $\vV$ be a proper class proper 1-Vsp.
 Let $\Sigma$ be a short-normal $(0,\OR)$-strategy for $\vV$.
 
 We say that $\Sigma$ has \emph{$\Delta$-mhc}
 iff $\Sigma$ has mhc (minimal hull condensation; cf.~\cite{fullnorm_v3})
 for all trees based on $\vV|\delta_0^{\vV}$. 

 If $\Sigma$ has $\Delta$-mhc
 and $\wW$ is a non-dropping $\Sigma$-iterate of $\wW$ via a tree based on $\vV|\delta_0^{\vV}$, then $\Sigma_{\wW}$ denotes the $(0,\OR)$-strategy for $\wW$
 induced by $\Sigma$ via normalization (see \cite{fullnorm_v3}).\end{dfn}
 
 \begin{dfn}Let $\vV$ be a proper class proper 1-Vsp.
 Let $\Sigma$ be a short-normal $(0,\OR)$-strategy for $\vV$.
 
 We say that $\Sigma$ is \emph{self-coherent}
 if $\Sigma$  has $\Delta$-mhc and for every
 tree $\Tt$ via $\Sigma$ (so $\Tt$ is short-normal) of length $\alpha+1$, if $E\in\es_+(M^\Tt_\alpha)$ is long,
 then letting $\beta\leq_\Tt\alpha$ be least such that $[0,\beta]_\Tt\cap\dropset^\Tt=\emptyset$ and either $\beta=\alpha$ or $\delta_0^{M^\Tt_\beta}<\lh(E^\Tt_\beta)$,
 then $E^\Tt_\alpha=j\rest (M^\Tt_\beta|\delta_0^{M^\Tt_\beta})$ for
 some $\Sigma_{M^\Tt_\beta}$-iteration map $j$ (via a tree  based on $M^\Tt_\beta|\delta_0^{M^\Tt_\beta}$,
 with non-dropping main branch).
\end{dfn}
\begin{dfn}\label{dfn:self-consistent}
Let $\vV$ be a proper class proper 1-Vsp.
 Let $\Sigma$ be a $(0,\OR)$-strategy for $\vV$.
 (Note trees via $\Sigma$ are allowed to use long extenders.)
Let $\Sigma_{\sn}$ be the restriction of $\Sigma$ to short-normal trees.
 We say that $\Sigma$ is \emph{self-consistent}
 iff  $\Sigma_{\sn}$ has $\Delta$-mhc
 and is self-coherent, and for every tree $\Tt$ via $\Sigma$, for all $\eta<\lh(\Tt)$
 such that either:
 \begin{enumerate}[label=--]
  \item  $\eta=\beta+1$ for some $\beta$ such  that $E^\Tt_\beta$ is long,\footnote{Note that then $\delta_0^{M^\Tt_{\beta+1}}=\lh(E^\Tt_\beta)<\lh(E^\Tt_{\beta+1})$,  so $\Tt$ can only act further below $\delta_0^{M^\Tt_{\beta+1}}$
  by using another long extender at some stage $\geq\beta+1$.} or
  \item $\eta$ is a limit of ordinals $\beta<\eta$ such that $E^\Tt_\beta$ is long,\footnote{Note that in this case, $\delta(\Tt\rest\eta)$ is the least measurable of $M^\Tt_\eta$, and $\delta(\Tt\rest\eta)<\delta_0^{M^\Tt_\eta}$, so after stage $\eta$, $\Tt$ can act further below $\delta_0^{M^\Tt_\eta}$ even without using long extenders.}
  \end{enumerate}
  for every $\xi\in(\eta,\lh(\Tt))$ such that $\Tt\rest(\eta,\xi)$ uses no  long extenders, $\Tt\rest(\eta,\xi)$ is formed according to $(\Sigma_{\sn})_{M^\Tt_\eta}$. 
\end{dfn}

The following lemma is easy to see:
\begin{lem}\label{lem:unique_ext_to_(0,OR)-strat}
 Let $\vV$ be a proper class proper 1-Vsp. Let $\Sigma$ be a short-normal $(0,\OR)$-strategy for $\vV$ which has $\Delta$-mhc and is self-coherent. Then there is a unique extension of $\Sigma$ to a self-consistent $(0,\OR)$-strategy.
\end{lem}

\begin{dfn} 
A short-normal tree $\Tt$ on a proper $1$-Vsp
is called \emph{$\gamma$-stable} iff whenever $\alpha+1<\lh(\Tt)$
and $[0,\alpha]_\Tt$ does not drop,
then $\lh(E^\Tt_\alpha)\notin(\Delta^{M^\Tt_\alpha},\gamma^{M^\Tt_\alpha})$.
	\end{dfn}

\begin{dfn} Let $\vV$ be a proper class proper $1$-Vsp, $\Sigma$ a partial short-normal iteration strategy for $\vV\down 0$, and $\Psi$ a partial short-normal strategy for $\vV$. We say that
	$\Psi$ is \emph{$\Sigma$-$\Delta$-conservative}\footnote{Note that the ``$\Delta$'' here is just notation,
	whereas $\Sigma$ is a strategy. Likewise the ``$\gamma$'' in ``$\Sigma$-$\gamma$-conservative''.} (or just \emph{$\Delta$-conservative}, if $\Sigma$ is understood) iff:
	\begin{enumerate}
	\item $\Sigma,\Psi$ agree over their restrictions to trees based on $\vV|\Delta^{\vV}$, and
	\item  if $\Tt,\Uu$ are corresponding such trees,
	of successor length $\beta+1$, and $b^\Tt$ does not drop (so neither does $b^\Uu$)
	then $M^{\Tt}_\beta=M^{\Uu}_\beta\down 0=i^\Uu_{0\beta}(\vV\down 0)$,
	and $i^\Tt_{\alpha\beta}\sub i^\Uu_{\alpha\beta}$ for all $\alpha\leq^{\Tt}\beta$.\footnote{***In fact we could assert here that the natural factor maps $\pi_\alpha:M^\Tt_\alpha\to M^\Uu_\alpha\down 0$
	and $\pi_\beta:M^\Tt_\beta\to M^\Uu_\beta\down 0$ are the identity maps. I'm not sure whether this follows. It would probably be more natural to assert that here.}
	\end{enumerate}
We say that $\Psi$ is \emph{$\Sigma$-$\gamma$-conservative}
(or just \emph{$\gamma$-conservative}) iff
$\Psi$ is $\Delta$-conservative  and whenever $\Tt,\Uu$ are as above,
then
 	$\Sigma_{M^\Tt_\infty},\Psi_{M^\Uu_\infty}$ agree over their restrictions
 to above-$\Delta^{M^\Uu_\infty}$ trees based on $M^\Tt_\infty|\kappa_0^{+M^\Tt_\infty}$ and $M^\Uu_\infty|\gamma^{M^\Uu_\infty}$ respectively.
 (Note here that because $M^\Uu_\infty||\gamma^{M^\Uu_\infty}=M^\Tt_\infty|\kappa_0^{+M^\Tt_\infty}$,
 and $\rho_1(M^\Uu_\infty|\gamma^{M^\Uu_\infty})=\delta_0^{M^\Uu_\infty}$,
 given agreement in tree order between correspoding trees $\Tt',\Uu'$,
 there is consequentially  agreement between models, degrees and embeddings,
 except that if $0<\alpha$ and  $[0,\alpha]_{\Tt'}$ does not drop,
 then $[0,\alpha]_{\Uu'}$ does drop, with $\beta+1\in\mathscr{D}^{\Tt'}$
 and $M^{*\Uu}_{\beta+1}=M^\Uu_\infty|\gamma^{M^\Uu_\infty}$,
 and $\deg^{\Uu'}_{\beta+1}=0$,
 where $0=\pred^{\Uu'}(\beta+1)<^{\Uu'}\beta+1\leq^{\Uu'}\alpha$,
 but $(\beta+1,\alpha]^{\Uu'}$ does not drop,
 and $(M^{\Uu'}_\alpha)^{\passive}=M^{\Tt'}_\alpha|\kappa_0^{+M^{\Tt'}_\alpha}$,
 and $i^{\Uu'}_{\gamma\alpha}\sub i^{\Tt'}_{\gamma\alpha}$ for all $0<^{\Tt'}\gamma\leq^{\Tt'}\alpha$,
 and $i^{*\Uu'}_{\beta+1,\gamma}\sub i^{\Tt'}_{0\alpha}$.)\footnote{***Again we should probably
 say that the factor maps are the identity.}
 
 We say $\Psi$ is \emph{$\Sigma$-conservative} (or just \emph{conservative}) iff $\Psi$ is $\gamma$-conservative
and for all $\gamma$-stable trees $\Uu$ via $\Psi$
 	there is a tree $\Tt$ via $\Sigma$ such that:
 	\begin{enumerate}
 		\item $\{\lh(E^\Tt_\alpha)\bigm|\alpha+1<\lh(\Tt)\}=\{\lh(E^\Uu_\alpha)\bigm|\alpha+1<\lh(\Uu)\}$,\footnote{The fact that we want to consider trees
 			$\Tt,\Uu$ with this property is clear in the 
 			case that $\Uu$ is based on $\vV|\gamma_n^{\vV}$,
 			but less so otherwise. For trees $\Uu$, for example,
 			with $\lh(E^\Uu_0)>\gamma_n^{\vV}$,
 			$\Tt$ will not be \emph{defined} by simply using
 			extenders with the same indices; this will instead be a consequence
 			of how $\Tt$ is defined. But for the present it is more convenient
 		to simply define conservativity by demanding the extender indices match.}
 		\item $\Tt,\Uu$ have the same tree order, drop and degree structure,
 		\item if $\alpha\leq^\Tt\beta$ and $[0,\beta]_\Tt$ does not drop
 		then $M^\Tt_\beta=M^\Uu_\beta\down 0=i^\Uu_{0\beta}(\vV\down 0)$
 		and $i^\Tt_{\alpha\beta}\sub i^\Uu_{\alpha\beta}$.\footnote{***And again the factor maps.}\qedhere
 		\end{enumerate}
	\end{dfn}

\begin{dfn}
$\Psi_{\vV_{1},\vV_{1}^-}$ and 
$\Psi_{\vV_{1},\gamma}$  are the
putative partial short-normal strategies
for $\vV_{1}$, for trees based on $\vV_{1}^-=\vV_1|\delta_0^{\vV_1}$ and $\vV_{1}|\gamma^{\vV_{n+1}}$ respectively, induced by $\Sigma_{\vV_{1\down 0}}$. (So $\Psi_{\vV_{1},\vV_{1}^-}\sub\Psi_{\vV_{1},\gamma}$.)\footnote{Note that
``$\gamma$''
in ``$\Psi_{\vV_{1},\gamma}$''
is just notation.}
	\end{dfn}
\begin{lem}\label{lem:Psi_vV_1,gamma_conservative}
$\Psi_{\vV_{1},\gamma}$ is $\Sigma_{\vV_{1\down 0}}$-$\gamma$-conservative.
	\end{lem}
\begin{proof}
$\Sigma_{\vV_{1\down 0}}$-$\Delta$-conservativity is by \cite[Lemma 4.68]{vm2_v2}.
This immediately yields $\Sigma_{\vV_{1\down 0}}$-$\gamma$-conservativity, by its definition.
	\end{proof}
\begin{dfn}\label{dfn:tree-triple}
	 Let $\vV$ be a proper class
proper	$1$-Vsp. Let $\N=\vV\down 0$.
	Let $e=e_0^{\vV}=F^{\vV|\gamma^{\vV}}$.
Let $i_e^{\vV}:\vV\to\vV_{1}^\N=\Ult(\vV,e)$ be the ultrapower map.

	A \emph{tree-triple on $\vV$} is a triple $(\Tt,\Tt',\Tt'')$
	of trees $\Tt,\Tt',\Tt''$ such that:
	\begin{enumerate}[label=--]
		\item $\Tt$ is an above-$\gamma^{\vV}$  tree on $\vV$,
		\item $\Tt'$ is the minimal\footnote{Cf.\cite[Remark 4.80***]{vm2_v2},
		which discusses the minimal $j$-pullback;
		the minimal $j$-copying construction  is performed in this manner.} $i_e^{\vV}$-copy of $\Uu$ to a tree on $\vV_{1}^{\N}$ (so $\Tt'$ is above $\gamma^{\vV_{1}^{\N}}$, hence translatable),
		\item $\Tt''$ is the translation of $\Tt'$ to a tree on $\N$ (so $\Tt''$ is above $\kappa_{0}^{+\N}=\gamma^{\vV}$,
		in fact above $\gamma^{\vV_{1}^{\N}}$).
	\end{enumerate}	
Note here that $\Tt,\Tt',\Tt''$ have the same tree order, drop and degree structure.
\end{dfn}

\begin{dfn}\label{dfn:Psi^sn_vV_n+1}
Suppose that  $\Psi_{\vV_{1},\vV_{1}^-}$
 is $\Sigma_{\P}$-$\Delta$-conservative,
 where $\P=\vV_{1\down 0}$.
 Let $\vV,\N,e$ be as in Definition \ref{dfn:tree-triple}, and suppose also that $\vV$ is a  non-dropping $\Psi_{\vV_{1},\vV_{1}^-}$-iterate of $\vV_{1}$, via tree $\Uu_0$. So by $\Delta$-conservativity, $\N$ is the corresponding $\Sigma_{\P}$-iterate of $\P$,
 where $\P=\vV_{1\down 0}$. 
 We define an above-$\gamma^{\vV}$ strategy $\Psi=\Psi_{\vV,>\gamma}$\footnote{Note that  ``$\gamma$'' in ``$\Psi_{\vV,>\gamma}$''
 is just notation.}
 for $\vV$.
	Let $\Lambda$ be the
	above-$\gamma^{\vV_{1}^\N}$ strategy for $\vV_{1}^\N$  determined by translating  above-$\kappa_0^{+\N}$ trees on $\N$ via $\Sigma_{\N}$. Then for
	above-$\gamma^{\vV}$ trees $\Uu$ on $\vV$, $\Uu_0\conc\Uu$ is via $\Psi$ iff $\Uu$ is via the minimal $i_e^{\vV}$-pullback of $\Gamma$;
	that is, iff there are $\Uu',\Uu''$ such that $(\Uu,\Uu',\Uu'')$ is a tree-triple on $\vV$
	and $\Uu''$ is via $\Sigma_\N$.

	We define $\Psi^{\sn}_{\vV_{1}}$
	as the unique (putative) short-normal strategy $\Psi$ for $\vV_{1}$ such that
	$\Psi_{\vV_{1},\gamma}\sub\Psi$ and
	$\Psi_{\vV,>\gamma}\sub\Psi$
	 for each $\vV$ as above.
	\end{dfn}

	\begin{dfn}
	 We say that a short-normal tree $\Tt$
	 on a proper $1$-Vsp $\vV$
	 is \emph{$\gamma$-stable}
	 iff $\Tt=\Tt_0\conc\Tt_1$
	 where $\Tt_0$ is based on $\vV|\Delta^{\vV}$, and if $\Tt_1\neq\emptyset$
	 then $\Tt_0$ has successor length,
	 $b^{\Tt_0}$ does not drop,
	 and $\Tt_1$ is above $\gamma^{M^{\Tt_0}_\infty}$. In this situation we also say that $\Tt_0,\Tt_1$ are the \emph{lower} and \emph{upper} components of $\Tt$ respectively.
	\end{dfn}

The following lemma is the analogue of \cite[Lemma 4.83***]{vm2_v2}, but because we have defined the iteration strategy differently, we give a different proof:
   	\begin{lem}$\Psi^{\sn}_{\vV_1}$ has minimal hull condensation (mhc).\footnote{Note that we get full mhc, not just $\Delta$-mhc.}\end{lem}
   
   \begin{proof}Argue like in the proof of \cite[Lemma 5.71***]{vm2_v2}
   	(the proof of \cite[Lemma 4.83***]{vm2_v2} is different, because of the
   	change in the definition of $\Psi^{\sn}_{\vV_1}$).\end{proof}
\begin{lem} $\Psi^{\sn}_{\vV_{1}}$ is a self-coherent  short-normal iteration strategy for $\vV_1$ (cf.~\cite[Definition 4.73***]{vm2_v2}).
	\end{lem}
\begin{proof}
It is clear that $\Psi^{\mathrm{sn}}_{\vV_{1}}$ is a short-normal iteration strategy;
we just need to see it is self-coherent.
	The proof of this is similar to \cite[Lemma 5.68]{vm2_v2},\footnote{Cf.~the notion \emph{good} in  \cite{vm2_v2} (the term \emph{self-coherent} is not used there).}
	but we give it here for convenience. The fact that $\Psi_{\vV_1,\gamma}$
	is self-coherent is by (the proof of) \cite[Lemma 4.75(8)]{vm2_v2}.
	
	So let $\Uu_0\conc\Uu$ be via $\Psi_{\vV_1,>\gamma}$, as in Definition \ref{dfn:Psi^sn_vV_n+1},
	and adopt the notation of Definitions \ref{dfn:tree-triple} and \ref{dfn:Psi^sn_vV_n+1},
	giving us $\vV=M^{\Uu_0}_\infty$,
	$\N=\vV\down 0$, $e$, $i^{\vV}_e:\vV\to\vV_1^{\N}=\Ult(\vV,e)$,
	and $\Uu',\Uu''$ such that $(\Uu,\Uu',\Uu'')$
	is a tree-triple.
Let $\wW\ins M^{\Uu}_\infty$
	be active with a long extender,
	with $\gamma^{\vV}<\OR^{\wW}$.
	We want to see that $\Ult(\N,F^{\wW})$
	is a $\Sigma_\N$-iterate of $\N$ and $F^{\wW}$ is derived
	from the $\Sigma_\N$-iteration map.
		
   	Let $\wW'=\Ult_0(\wW,e)$ and  $k:\wW\to\wW'$ be the ultrapower map.
   	Then $\wW'\ins M^{\Uu'}_\infty$
   	is active with a long extender. Let $\wW''\ins M^{\Uu''}_\infty$
   	with $\OR^{\wW''}=\OR^{\wW'}$. So $\crit(F^{\wW''})=\kappa_0^{\N}$ and $F^{\wW''}$ induces $F^{\wW'}$.
   	
   	By \cite[Lemma ***4.57]{vm2_v2}, $\M_\infty^{\Ult(\N,F^{\wW''})}$
   	is a $\Sigma_{\M_\infty^\N}$-iterate of $\M_\infty^{\N}$ and $i^{\N}_{F^{\wW''}}\rest\M_\infty^\N$
   	is the iteration map. Also $\M_\infty^{\N}=\Ult(\N,e)$
   	and $e$ is derived from the $\Sigma_{\N}$-iteration map.
   	
   	Let $\Vv$ be the tree leading from $\N|\delta_0^{\N}$ to $\Ult(\N|\delta_0^{\N},F^{\wW})$. 
   	Using $k:\wW\to\wW'$,
   	note that $\Vv$ is via $\Sigma_{\N,\sss}$,
   	since the Q-structures of $\Vv$ do not overlap their local Woodin,
   	and $k$ embeds them into correct Q-structures.
   	We can apply the branch condensation lemma \cite[***Lemma 3.10]{vm2_v2}
   	to $(\Vv,b)$ where $b$ is determined by $F^{\wW}$,
   	together with the tree $\Vv'$ leading from $\N$ (via $\M_\infty^{\N}$) to $\M_\infty^{\Ult(\N,F^{\wW''})}$ and the branch determined by
   	$F^{\wW'}\com e$, together with the embedding $k$, to see that $F^{\wW}$ is correct.
   		\end{proof}
   	
   	\begin{rem}\label{rem:internal_e_Ult}
   		Continue with notation as above, still assuming that $\wW$ is active,
   		but allow $F^{\wW}$ to be either short or long. Since $e=e_0^{\vV}\in\wW$, $\Ult(\wW,e)$ is internal
   		to $\wW$, and since the largest cardinal $\iota$ of $\wW$ is inaccessible in $\wW$
   		and $\wW\sats\ZFC^-$, it follows that $\Ult(\wW^{\passive},e)\sub\wW$
   		and $k(\iota)=\iota$.
   		
   		Now suppose further that $F^{\wW}$ is long.
   		We also have $\delta_0^{\M_\infty^{\N}}=\gamma^{\vV}$ and $\M_\infty^\N|\delta_0^{\M_\infty^\N}\in\wW|\iota$
   		and $\wW'|\iota\sub\wW|\iota$,
   		and therefore all the iterates of $\M_\infty^{\N}$ in the direct limit system leading to
   		$\M_\infty^{\wW'}$ (by the latter direct limit, we mean
   		the limit of the covering system
   		defined locally over $\wW'$,
   		using iterates $P$ with $P|\delta_0^P\in\wW'|\iota$) are also iterates of $\N$ in the direct limit system leading to $\M_\infty^{\wW}$. As in \cite[\S8***]{fullnorm_v3},
   		it follows that $\M_\infty^{\wW}$ is an iterate of $\M_\infty^{\wW'}$.
   		Let $i:\M_\infty^{\wW'}\to \M_\infty^{\wW}$ be the iteration map.
   		Therefore we get
   		\[ i\com F^{\wW'}\com e_0^{\vV} =F^{\wW}.\]
   		On the other hand, $k\rest\M_\infty^{\wW}:\M_\infty^{\wW}\to\M_\infty^{\wW'}$
   		is elementary, and
   		\[ k\com F^{\wW}=F^{\wW'}\com e_0^{\vV}.\]
   	    It easily follows that \[ i\rest\rg(F^{\wW'}\com e_0^{\vV}``\delta_0^\N)=\id \]
   		and
   		\[ k\rest\rg(F^{\wW}``\delta_0^{\N})=\id, \]
   		and so
   		\[ F^{\wW}\rest\delta_0^{\N}=F^{\wW'}\com e_0^{\vV}\rest\delta_0^{\N}.\]
   		However, the two iterates $\Ult(\N,F^{\wW})$ and $\Ult(\N,F^{\wW'}\com e_0^{\vV})$
   		are not equal, and so $F^{\wW}\neq F^{\wW'}\com e_0^{\vV}$ (even though these extenders agree on the ordinals). For if those iterates were the same,
   		then so would the extenders be, and so would $\Ult(\vV,F^{\wW})$ and
   		$\Ult(\vV,F^{\wW'}\com e_0^{\vV})$. But $\wW^\passive$ is generic
   		over $\Ult(\vV,F^{\wW})$, and $e_0^{\vV}\in\wW^\passive$.
   		But $k$ is the $e_0^{\vV}$-ultrapower map of $\wW^{\passive}$, so this extension also adds $k$ and hence also $k'=k\rest\Ult(\vV|\delta_0^{\vV},F^{\wW})$. Since $k\com F^{\wW}=F^{\wW'}\com e_0^{\vV}$, note we get $\Ult(\Ult(\vV,F^{\wW}),k')=\Ult(\vV,F^{\wW'}\com e_0^{\vV})=\Ult(\vV,F^{\wW})$. But then in a generic extension
   		of $\Ult(\vV,F^{\wW})$ there is an elementary $\ell:\Ult(\vV,F^{\wW})\to\Ult(\vV,F^{\wW})$, which
   		contradicts \cite[Corollary 9]{gen_kunen_incon}.
   		\end{rem}
   	
   	The remainder of \cite[\S4.12]{vm2_v2} is now as there.
   	This completes the adaptation of the material in \cite[\S4]{vm2_v2}.
   	We now prove the first instance of a key new fact, one which is essential to our analysis of $\Mswom$.

   	\begin{lem}\label{lem:Psi_vV_1_is_conservative} $\Psi^{\sn}_{\vV_1}$ is $\Sigma_{\vV_{10}}$-conservative.   		
   	\end{lem}
   	
   	\begin{proof}
We already dealt with
   		$\Psi_{\vV_1,\gamma}$
   		in Lemma \ref{lem:Psi_vV_1,gamma_conservative}.
   		So let $\Uu=\Uu_0\conc\Uu$
   		be $\gamma$-stable short-normal on $\vV_1$, via $\Psi_{\vV_1}^{\sn}$,
   		with lower and upper components
   		$\Uu_0,\Uu$, where $\Uu\neq\emptyset$,
   		and let $\Uu',\Uu''$ be such that $(\Uu,\Uu',\Uu'')$
is the corresponding tree-triple.
   Recall that	$\Uu,\Uu',\Uu''$ each
   have the same tree order,  drop and degree stucture.
   		Adopt the other notation as in Definitions \ref{dfn:tree-triple} and \ref{dfn:Psi^sn_vV_n+1}.   	
So $\vV=M^{\Uu_0}_\infty$.
   		   		Let $\Uu_0\down 0$ be the corresponding tree on $M_\infty=\vV_{1\down 0}$
   		and
$N=M^{\Tt_0\down 0}_\infty=\vV\down 0$.
   
 \begin{clm} We have:
   		\begin{enumerate}
	\item\label{item:models_correspondence}
   			$M^{\Uu'}_\alpha=\vV_1(M^{\Uu''}_\alpha)$ and
   			$M^{\Uu''}_\alpha=^* M^{\Uu'}_\alpha[N|\kappa_0^{N}]$
   			for each $\alpha<\lh(\Tt_1)$,\footnote{Here if $[0,\alpha]_{\Tt_1}$
   				drops then we define $\vV_1^{M^{\Uu''}_\alpha}$
   				by the P-construction of $M^{\Uu''}_\alpha$
   				above $\vV'|\gamma_0^{\vV'}$.}
   			\item\label{item:models_correspondence_2} $M^{\Uu'}_\alpha=\Ult_{m_\alpha}(M^{\Uu}_\alpha,e)$ for each $\alpha<\lh(\Uu)$, where  $m_\alpha=\deg^{\Uu}(\alpha)$.
   			\item\label{item:lh(E)_matches} $\lh(E^{\Uu''}_\alpha)=\lh(E^{\Uu'}_\alpha)=\lh(E^{\Uu}_\alpha)$ for $\alpha+1<\lh(\Tt)$.
   			\item\label{item:vV_1_extenders_to_M_infty}
   			$E^{\Uu''}_\alpha\rest\OR=E^{\Uu'}_\alpha\rest\OR=E^{\Uu}_\alpha\rest\OR$
   			for each $\alpha+1<\lh(\Uu)$.
   			\item\label{item:action_of_vV_1_extenders_to_M_infty} Let $\alpha\leq_{\Uu}\beta<\lh(\Uu)$ be such that $[0,\beta]_{\Uu}$ does not drop.
   			Then:
   			\begin{enumerate}
   				\item $M^{\Uu''}_\beta=M^{\Uu}_\beta\down 0=i^{\Uu}_{0\beta}(\vV\down 0)$, and
   				\item\label{item:action_of_vV_1_extenders_on_M_infty} 
   				$i^{\Uu''}_{\alpha\beta}=i^{\Uu}_{\alpha\beta}\rest(M^{\Uu}_\alpha\down 0)$.
   			\end{enumerate}
   		\end{enumerate}
 \end{clm}
 \begin{proof}
   		Parts \ref{item:models_correspondence}
   		and \ref{item:models_correspondence_2}
   		are basically standard (a small part of normalization calculations
   		are involved in part \ref{item:models_correspondence_2}).
   		
   		Part \ref{item:lh(E)_matches}: We have $\lh(E^{\Uu''}_\alpha)=\lh(E^{\Uu'}_\alpha)$
   		as usual; and $\lh(E^{\Uu'}_\alpha)=\lh(E^{\Uu}_\alpha)$
   		by part of the discussion in Remark \ref{rem:internal_e_Ult}.
   		
   		Part \ref{item:vV_1_extenders_to_M_infty}:
   		We have $E^{\Uu'}_\alpha\rest\OR=E^{\Uu''}_\alpha\rest\OR$ as usual.
   		Let $A=\exit^{\Uu}_\alpha$ and
   		\[ A'=\exit^{\Uu'}_\alpha=\Ult_0(A,e), \]
   		\[ i_e=i^A_e:A\to A' \]
   		the ultrapower map, $F=F^A=E^{\Uu}_\alpha$,
   		and $F'=F^{A'}=E^{\Uu'}_\alpha$.
   		Write $i_F=i^A_F$.
   		As in Remark  \ref{rem:internal_e_Ult}, since $e\in A$, etc, note that $A'\sub A$,
   		\[ \kappa=\crit(F)=i_e(\kappa)=\crit(F'),\]
   		\[ \kappa^{+A}=i_e(\kappa^{+A})=\kappa^{+{A'}},\]
   		\[ \iota=\lambda(F)=i_e(\iota)=\lambda(F').\]
   		Let $\beta$ be such that
   		\[ \kappa\leq\beta<(\kappa^+)^A 
   		\text{ and }\rho_\om^B=\kappa\text{ where }B=A|\beta.\]
   		Since $F$ sends $B$ to $i_F(B)$,
   		$F'$ sends $i_e(B)$ to $i_e(i_F(B))$; that is, $i_{F'}(i_e(B))=i_e(i_F(B))$.
   		And since $i_e(B)$ also
   		projects to $\kappa$,
   		the pair $(i_e(B),i_{F'}(i_e(B)))$
   		determines $i_{F'}\rest i_e(B)$.
   		But also, $i_F(i_e(B))=i_e(i_F(B))$
   		(because $e\in A|\kappa$, so $i_F(e)=e$, so since $e$ sends
   		$B$ to $i_e(B)$,
   		$i_F(e)=e$ sends $i_F(B)$ to $i_F(i_e(B))$; that is,
   		$i_e(i_F(B))=i_F(i_e(B))$.
   		So we have 
   		\[ i_F(i_e(B))=i_e(i_F(B))=i_{F'}(i_e(B)).\]
   		But from the pair $(i_e(B),i_F(i_e(B)))$,
   		we can read off $i_F\rest i_e(B)$.
It follows that $i_F\rest i_e(B)=i_{F'}\rest i_e(B)$. But $B$ was arbitrary, so $F\rest\OR=F'\rest\OR$.
   		
   		Part \ref{item:action_of_vV_1_extenders_to_M_infty}:
   		The proof is by induction. Suppose it holds at $\beta$
   		and we want it at $\gamma+1$, where $\pred^{\Uu}(\gamma+1)=\beta$
   		and $[0,\gamma+1]_{\Uu}$ does not drop.
   		By induction we also have
   		\[ \Ult(M^{\Uu}_\beta,e)=M^{\Uu'}_\beta=\vV_1(M^{\Uu''}_\beta)\text{ and }M^{\Uu''}_\beta=^*M^{\Uu'}_\beta[N|\kappa_0^N]. \]
   		Let $P=M^{\Uu}_\beta$, $P'=M^{\Uu'}_\beta$.
   		Then $P'=\Ult(P,e)$ is a class of $P$;
   		let $i^P_e:P\to P'$ be the ultrapower map.
   		Let $A,A',F,F'$, etc, be as above. So letting
   		\[ \kappa=\crit(F)=\crit(F'),\]
   		we have
   		\[ A|\kappa^{+A}=P|\kappa^{+P}, \]
   		\[ i^A_e\rest(A|\kappa^{+A})=i^P_e\rest(P|\kappa^{+P}).\]
   		We saw above that $F\sub F'$.
   		Note that it suffices to see that
   		\begin{equation}\label{eqn:inner_and_outer_ults_match} \Ult(P',F')=i^P_F(P')\text{ and }i^{P'}_{F'}\sub i^P_F.\end{equation}
   		
   		The proof of this is very analagous
   		to  \cite[Lemma 4.5]{steel_dmt}
   		and the well-known fact mentioned in the paragraph preceding it
   		(that if $\mu,\nu$ are normal measures on $\gamma<\mu$ respectively,
   		$V'=\Ult(V,\mu)$, and $j:V\to V'$ is the ultrapower map,
   		then $j(\nu)=\nu\inter V'$ and $\Ult(V',j(\nu))=i^V_\nu(V')$ and $i^{V'}_{j(\nu)}\sub i^V_\nu$).
   		We will make use of the fact,  much as in
   		\cite[Lemma 9.1]{fsit}, \cite[Lemma 2.1.4]{coveringuptowoodin}
   		and \cite[Lemma 2.11]{extmax}, 
   		that $F'$ above is generated by $i^A_e``\iota$ where $\iota=\lambda(F)=\lambda(F')$;
   		that is, for every $\alpha<\iota$ there is $a\in[\iota]^{<\om}$
   		and a function $g\in A'$ such that
   		\[ \alpha=[i^A_e(a),g]^{A'}_{F'}.\]
   		The argument is really just a direct transcription of those mentioned above
   		into our context. We give the details, though, for convenience to the reader. We will define an isomorphism
   		\[ \pi:\Ult(P',F')\to i^P_F(P')\]
   		such that
   		$\pi\com i^{P'}_{F'}=i^P_F\rest P'$;
   		clearly this gives line (\ref{eqn:inner_and_outer_ults_match}),
   		completing the proof.
   		We define $\pi$ by sending
   		\[ x=[i^A_e(a),g]^{P'}_{F'} \]
   		(where $a\in[\lambda]^{<\om}$ and
   		$g\in P'$
   		and $g:[\kappa]^{|a|}\to P'$) to
   		\[ \pi(x)=[a,g\com i^P_e]^{P}_F,\]
   		noting  that $g\com i^P_e\in P$.
   		
   		We first verify that $\pi$ is a well-defined function.
   		So let $(a_0,g_0)$ and $(a_1,g_1)$ be as above,
   		with 
   		\[ [i^A_e(a_0),g_0]^{P'}_{F'}=[i^A_e(a_1),g_1]^{P'}_{F'};\]
   		we must see that
   		\[ [a_0,g_0\com i^P_e]^{P}_F=[a_1,g_1\com i^P_e]^P_F.\]
   		We may easily assume that $a_0=a=a_1$. Let
   		\[ X=\{u\in[\kappa]^{|a|}\mid g_0(i^P_e(u))=g_1(i^P_e(u))\}; \]
   		we want to see that $X\in F_a$. Suppose not, so
   		\[ Z=\kappa^{|a|}\cut X\in F_a.\]
   		We have
   		\[ X'=\{u\in[\kappa]^{|a|}\mid g_0(u)=g_1(u)\}\in F'_{i^A_e(a)}.\]
   		And since $\rank(e)<\kappa$
   		we can find $Y\in F_a$
   		with $i^A_e(Y)\sub X'$. So 
   		$Y\inter Z\in F_a$. Fix $u\in Y\inter Z$.
   		Since $u\in Z$, we have
   		\begin{equation}\label{eqn:images_not_equal} g_0(i^P_e(u))\neq g_1(i^P_e(u)), \end{equation}
   		but since $u\in Y$, we have
   		$u'=i^P_e(u)\in i^P_e(Y)=i^A_e(Y)\sub X'$,
   		so $g_0(u')=g_1(u')$, 
   		contradicting line (\ref{eqn:images_not_equal}).
   		
   		Similarly, $\pi$ is injective and $\in$-preserving.
   		
   		We now verify that $\pi$ is surjective.
   		Let $a,f$ be given with $a\in[\lambda]^{<\om}$
   		and $f\in P$ and $f:[\kappa]^{|a|}\to P'$; we will find some
   		$g\in P'$ 
   		with
   		\[ \{u\in[\kappa]^{|a|}\mid g(i^P_e(u))=f(u)\}\in F_a.\]

   		Work in $P$. Fix a sequence
   		$\left<(b_u,f_u)\right>_{u\in[\kappa]^{|a|}}$ such that
   		\[ f(u)=[b_u,f_u]^P_e\]
   		where $b_u\in[\delta_0^{P'}]^{<\om}$
   		and
   		$f_u:[\delta_0^P]^{<\om}\to P$
   		for each $u$.
   		 
   		Define
   		\[ h:[\delta_0^P]^{<\om}\to P,\]
   		\[ h(c):\kappa^{|a|}\to P\text{ is the map where }h(c)(u)=f_u(c).\]
   		As $\delta_0^{P'}<\kappa$,
   		we may fix  $X\in E_a$ and $b\in[\delta_0^{P'}]^{<\om}$
   		such that $b_u=b=b_{u'}$ for all $u,u'\in X$.
   		
   		Now let $g=[h,b]^P_e=i^P_e(h)(b)$.
   		Fix $u\in X$; it suffices to see that
   		\[ g(i^P_e(u))=f(u).\] But
   		\[ h(c)(u)=f_u(c)\text{ for each }c\in[\delta_0^P]^{<\om},\]
   		so by elementarity of $i^P_e$,
   		\[ i^P_e(h)(c)(i^P_e(u))=i^P_e(f_u)(c)\text{ for each }c\in[\delta_0^{P'}]^{<\om},\]
   		but applying this to $c=b=b_u$, we get
   		$g(i^P_e(u))=f(u)$, as desired. This completes the proof of the claim.\end{proof}
   		This completes the proof of the lemma.
   	\end{proof}

       As in \cite{vm2_v2}, we also have:
   \begin{lem}\label{lem:vV_1[g]_computes_Sigma_vV,vV^-}
   	Let $\vV$ be a non-dropping $\Sigma_{\vV_1}$-iterate of $\vV_1$.
   	Let $\PP\in\vV$, $\PP\subseteq\lambda$, where $\lambda\geq\delta_0^{\vV}$.
   	Let $x=\vV|\lambda^{+\vV}$.
   	Let $G$ be $(\vV,\PP)$-generic (with $G$  appearing in a generic extension of $V$).
   	Then:
   	\begin{enumerate}
   		\item $\vV$ is closed under $\Sigma_{\vV,\vV^-}$, and $\Sigma_{\vV,\vV^-}\rest\vV$
   		is lightface definable over $\vV$, uniformly in $\vV$,
   		\item  $\vV[G]$
   		is closed under $\Sigma_{\vV,\vV^-}$,
   		and $\Sigma_{\vV,\vV^-}\rest\vV[G]$
   		is definable over $\vV[G]$ from $x$, uniformly in $\vV,x$.
   		\end{enumerate}
   	\end{lem}
   
      \begin{lem}
   	Let $N$ be a non-dropping $\Sigma_{M}$-iterate of $M$.
   	Let $\PP\in M$, $\PP\subseteq\lambda$.
   	Let $x=\pow(\lambda)^N$.
   	Let $G$ be $(N,\PP)$-generic (with $G$  appearing in a generic extension of $V$).
   	Let $\vV=\vV_1^N$.
   	Then:
   	\begin{enumerate}
   		\item $N$ is closed under $\Sigma_{\vV,\vV^-}$, and $\Sigma_{\vV,\vV^-}\rest N$
   		is lightface definable over $N$, uniformly in $N$,
   		\item  $N[G]$
   		is closed under $\Sigma_{\vV,\vV^-}$,
   		and $\Sigma_{\vV,\vV^-}\rest N[G]$
   		is definable over $ [G]$ from $x$, uniformly in $N,x$.
   	\end{enumerate}
   \end{lem}

      \begin{dfn}
A structure if    \emph{$\vV_1$-like}
if it is proper class, transitive,
a proper $1$-Vsp, and has satisfies various
first order properties true in $\vV_1$;
we leave it to the reader to identify which properties are needed to make things work.

Likewise for \emph{$M_{\sw\om}$-like},
etc.
   \end{dfn}

\begin{dfn}
 Let $\wW_0=\vV_0=M$.
 Let $\wW_1=\cHull^{\vV_1}(\mathscr{I}^M)=\cHull^{\vV_1}(\mathscr{I}^{\vV_1})$ (where $\mathscr{I}^N$ denotes the Silver indiscernibles of $N$). Let $i_{11}:\wW_1\to\vV_1$ denote the uncollapse map.\end{dfn}
 
 So $\wW_1$ is $\vV_1$-like.
 Let $\mathscr{I}=\mathscr{I}^M=\mathscr{I}^{\vV_1}=\mathscr{I}^{\mM_{\infty 0}}$. As in \cite{vm2_v2}, \[ \mM_{\infty 0}\cap\Hull^{\vV_1}(\mathscr{I}^M)=\Hull^{\mM_{\infty 0}}(\mathscr{I}^M)=\rg(i_{M\mM_{\infty 0}}) \]
 where $i_{M\mM_{\infty 0}}:M\to\mM_{\infty 0}$ is the iteration map.
We have $\wW_0=M\sub\wW_1$ and in fact $\wW_0=\wW_1\down 0$
(notation as in \cite{vm2_v2}). Also,
$i_{M\mM_{\infty 0}}\sub i_{11}$
and $\vV_1=\Ult(\wW_1,e_0^{\wW_1})$
and $i_{11}$ is just the ultrapower map. This allows
us to define a $(0,\OR)$-strategy $\Sigma_{\wW_1}$ for $\wW_1$ (absorbing trees on $\wW_1$ in trees on $M$) completely analogously to how $\Sigma_{\vV_1}$ was defined for $\vV_1$ (absorbing trees on $\vV_1$ in trees on $\mM_{\infty 0}$).  Lemmas \ref{lem:Psi_vV_1,gamma_conservative}--\ref{lem:vV_1[g]_computes_Sigma_vV,vV^-} go through for $\wW_1,\Sigma_{\wW_1}$ by essentially the same proofs as for $\vV_1,\Sigma_{\vV_1}$.

   \begin{dfn}
   	Let $\vV$ be $\vV_{1}$-like. An iteration tree $\Tt$ on $\vV$ is \emph{$1$-translatable} iff
   	\begin{enumerate}[label=--]
   		\item $\Tt$ is $0$-maximal, and
   		\item for all $\alpha+1<\lh(\Tt)$, if $[0,\alpha]_\Tt$ does not drop then $\gamma^{M^\Tt_\alpha}<\lh(E^\Tt_\alpha)$.\qedhere
   		\end{enumerate}
   	\end{dfn}
   	
   	Note that a $1$-translatable tree is allowed
   	to use long extenders, though not $e^{\vV}$ or its images.
   \begin{dfn}
   	Let $N$ be $\Mswom$-like and $\Gamma$ be an above-$\kappa_0^N$ strategy for $N$.
   	Then $\mathrm{trl}(\Gamma)$ denotes the strategy for $\vV_{1}^N$ for $1$-translatable trees on $\vV_{1}^N$ induced by translation to trees on $N$ via $\Sigma$.\end{dfn} 
   	 
   	 \begin{lem}
   	  Let $N$ be any non-dropping $\Sigma_M$-iterate of $M$. Then $\vV_1^N$ is a non-dropping $\Sigma_{\wW_1}$-iterate of $\wW_1$.
   	 \end{lem}
\begin{proof}
 
\end{proof}

   \begin{dfn}
   	Let $\Psi$ be a short-normal strategy for $\vV_{1}$ with mhc (minimal hull condensation; cf.~\cite{fullnorm_v3}).
   	We say that $\Psi$ is \emph{$\Sigma_M$-stable} iff whenever
   	$N$ is a non-dropping $\Sigma$-iterate of $M$
   	and $\vV=\vV_{1}^N$ is a $\Psi$-iterate of $\vV_{1}$, then
   	$\Psi_{\vV}$ agrees with $\mathrm{trl}(\Sigma_N)$ on translatable trees.
   	\end{dfn}
   	
   	\begin{dfn}
   	 Let $\Tt$ be a $0$-maximal
   	 tree on $M$.
   	 
   	 Let $\alpha_0$ be the
   	 least ordinal, if it exists, such that either $\lh(\Tt)=\alpha+1$
   	 or $[0,\alpha]^\Tt\cap\dropset^\Tt=\emptyset$ and $\delta_0^{M^\Tt_\alpha}<\lh(E^\Tt_\alpha)$.
   	 Let $\alpha_1$ be the least 
   	 ordinal, if it exists, such that either $\lh(\Tt)=\alpha+1$ or $[0,\alpha]^\Tt\cap\dropset^\Tt=\emptyset$ and $\kappa_0^{+M^\Tt_\alpha}<\lh(E^\Tt_\alpha)$.
   	 If $\alpha_0,\alpha_1$ exist, then
   	 the \emph{standard decomposition} of $\Tt$
   	 is the tuple $(\Tt\rest(\alpha_0+1),\Tt\rest[\alpha_0,\alpha_1],\Tt\rest[\alpha_1,\lh(\Tt))$. If $\alpha_0$ does not exist then the \emph{standard decomposition} is $\Tt$ itself.
   	 If $\alpha_0$ exists but not $\alpha_1$,
   	 then the \emph{standard decomposition} if $(\Tt\rest(\alpha_0+1),\Tt\rest[\alpha_0,\lh(\Tt))$.

   	 We say that $\Tt$ is \emph{1-translatable}
   	 iff either $\alpha_1$ as above does not exist, or
 for every $\alpha\geq\alpha_1$ with $\alpha+1<\lh(\Tt)$,
   	 if $[0,\alpha]^\Tt\cap\dropset^\Tt=\emptyset$ then $\kappa_0^{+M^\Tt_\alpha}<\lh(E^\Tt_\alpha)$.   	 
   	 
   	 If $\Tt$ is 1-translatable, then the \emph{1-padding}
   	 of $\Tt$ is the padded $0$-maximal tree $\Uu$ where either:
   	 \begin{enumerate}[label=--]
   	  \item $\alpha_1$ does not exist and $\Tt=\Uu$, or
   	  \item $\alpha_1$ exists,
   	  $\Tt\rest(\alpha_1+1)=\Uu\rest(\alpha_1+1)$,
   	  $E^\Uu_{\alpha_1}=\emptyset$, $\nu^\Uu_{\alpha_1}=\kappa_1^{M^\Tt_{\alpha_1}}$,\footnote{Here $\nu^\Uu_\beta$ is the exchange ordinal associated with $E^\Uu_\beta$. So in general,
   	  $\pred^\Uu(\gamma+1)$
   	  is the least $\beta$ such that $\crit(E^\Uu_\gamma)<\nu^\Uu_\beta$.} and $\Tt\rest[\alpha_1,\lh(\Tt))$ is equivalent to $\Uu\rest[\alpha_1+1,\lh(\Tt))$. (So $\Uu$ is padded only at $\alpha_1$, and after removing padding, we recover $\Tt$.)
   	 \end{enumerate}

   	 If $\Tt$ is 1-translatable
   	 and $\Uu$ its 1-padding,
   	 then the \emph{1-translation} of $\Tt$, or of $\Uu$, is the (putative)
   	 $0$-maximal tree $\Uu'$ on $\wW_1$ such that
 $\Uu,\Uu'$ have the same length, extender indices and tree structure, except that if $\alpha_1$ as above exists (for $\Tt$), then $E^\Uu_{\alpha_1}=e_0^{M^{\Uu}_{\alpha_1}}$.
   	\end{dfn}

  \begin{lem}
   Let $N$ be a non-dropping $\Sigma_M$-iterate,
   via tree $\Tt$.
   Then the 1-translation $\Uu'$ of $\Tt$ exists, is via $\Sigma_{\wW_1}$ (hence is $0$-maximal), has successor length and $M^{\Uu'}_\infty=\vV_1^N$.
\end{lem}

   \begin{proof}
   Note first that $\Tt$ is indeed 1-translatable
   and the standard decomposition of $\Tt$ has form
   $(\Tt_0,\Tt_1,\Tt_2)$ (all 3 parts are defined),
since $N$ is a non-dropping iterate.
 Suppose first that $\Tt_2$ is trivial.
   Then the conclusion of the lemma is mostly
   directly by definition of $\Sigma_{\wW_1}$.
 However, we need a small calculation at the last step to see that $M^{\Uu'}_\infty=\vV_1^N$. Note here
   that (since $\Tt_2$ is trivial) $\Tt=\Tt_0\conc\Tt_1$ is based on $M|\kappa_0^{+M}$,
   it has successor length,
   $b^\Tt$ is non-dropping,
   and $\Uu'=\Tt_0'\conc\Tt_1'\conc\left<e\right>$ where $\Tt_0'\conc\Tt_1'$ is just $\Tt$ as a tree on $\wW_1$,
   which by definition mimics $\Sigma_M$.
   By $\Delta$-conservativity
   and elementarity,
   $\vV_1^{M^{\Tt_0}_\infty}=\Ult(M^{\Tt_0'}_\infty,e_0^{M^{\Tt_0'}_\infty})$.
   And by $\gamma$-conservativity,
   \[M^{\Tt_0\conc\Tt_1}_\infty|\kappa_0^{+M^{\Tt_0\conc\Tt_1}_\infty}=M^{\Tt_0'\conc\Tt_1'}_\infty||\gamma_0^{M^{\Tt_0'\conc\Tt_1'}_\infty} \]
   and
\[   i^{\Tt_0\conc\Tt_1}_{\alpha_0\infty}\rest(M^{\Tt_0}_\infty|\kappa_0^{+M^{\Tt_0}_\infty})=i^{\Tt_0'\conc\Tt_1'}_{\alpha_0\infty}\rest(M^{\Tt_0'}_\infty|\gamma_0^{M^{\Tt_0'}_\infty}).\]
But then letting $F$ be the $(\kappa_0^{+M^{\Tt_0}_\infty},\kappa_0^{+M^{\Tt_0\conc\Tt_1}_\infty})$-extender derived from $i^{\Tt_0\conc\Tt_1}_{\alpha_0\infty}$, we have $\delta_0^{M^{\Tt_0'}_\infty}<\crit(F)$ and \[ e_0^{M^{\Tt_0'\conc\Tt_1'}_\infty}=F\com e_0^{M^{\Tt_0'}_\infty},\]
and (as in \cite{vm2_v2},
since $\vV_1^{M^{\Tt_0}_\infty}$ is a ground of $M^{\Tt_0}_\infty$
via a $\delta_0^{\vV_1^{M^{\Tt_0}_\infty}}$-cc forcing), we have
\[ \vV_1^N=\vV_1^{M^{\Tt_0\conc\Tt_1}_\infty}=i^{\Tt_0\conc\Tt_1}_{\alpha_0\infty}(\vV_1^{M^{\Tt_0}_\infty})=\Ult_0(\vV_1^{M^{\Tt_0}_\infty},F) \]
\[ =\Ult_0(\Ult_0(M^{\Tt_0'}_\infty,e_0^{M^{\Tt_0'}_\infty}),F) \]
\[ =\Ult(M^{\Tt_0'}_\infty,e_0^{M^{\Tt_0'\conc\Tt_1'}_\infty})=M^{\Uu'}_\infty,\]
as desired.

Now suppose $\Tt_2$ is non-trivial. The foregoing applies to $\Tt_0\conc\Tt_1$.
Let $\bar{N}=M^{\Tt_0\conc\Tt_1}_\infty$. Then $\bar{N}$ is $\kappa_0^{\bar{N}}$-sound,
and note that the $0$-maximal above-$\kappa_0^{\bar{N}}$-strategy for $\bar{N}$ is unique. 
So letting $\bar{\vV}=\vV_1^{\bar{N}}=M^{\Tt_0'\conc\Tt_1'}_\infty$,
it suffices to see that $\Sigma_{\wW_1}$  induces a strategy for $\bar{N}$ for $0$-maximal trees $\Xx$ such that $\kappa_0^{+\bar{N}}<\lh(E^\Xx_0)$ (hence $\Xx$ is above $\kappa_0^{\bar{N}}$,
but may use extenders $E$ with $\crit(E)=\kappa_0^{\bar{N}}$),
 via 1-translation. More precisely, that we can iterate $\bar{N}$ for such trees $\Xx$ given that $\Tt_0\conc\Tt_1\conc\Xx$ is  1-translatable, and that this extends to a full strategy for arbitrary such trees $\Xx$.
 As long as $\Tt_0\conc\Tt_1\conc\Xx$ remains 1-translatable, this is straightforward. Now suppose that $\Tt_0\conc\Tt_1\conc\Xx\rest(\alpha+1)$ is 1-translatable, $[0,\alpha]_\Xx$ does not drop, and $\lh(E^\Xx_\alpha)<\kappa_0^{+R}$ where $R=M^\Xx_\alpha$;
 then $\Xx\rest(\alpha+2)$ is non-1-translatable.
   	Note that $\alpha$ is a limit, $\lh(E^\Xx_\beta)<\kappa_0^{R}$ for all $\beta<\alpha$, $\kappa_0^R<\lh(E^\Xx_\alpha)$,
   	and since $\kappa_0^{R}$ is a cutpoint of $R$,
   	therefore $\Xx\rest[\alpha,\infty)$ is above $\kappa_0^{R}$ and is based on some $P\pins R$
   	with $\rho_\om^P=\kappa_0^{R}$. Let $\Xx'$ be the translation
   	of $\Xx\rest(\alpha+1)$ to a tree on $\vV$.
   	Then $\delta_0^{M^{\Xx'}_\alpha}=\kappa_0^{+R}$,
   	so $P$ does not translate into a segment of $R'=M^{\Xx'}_\alpha$.
   	Let $E\in\es^R$ be least such that $E$ is $R$-total
   	and $\crit(E)=\kappa_0^{R}$. Let $Q=i_E(P)$. Then
   	note that $Q$ \emph{does} translate into $Q'\pins R'$,
   	and $\rho_\om^{Q'}=i_E(\kappa)$ and $i_E(\kappa)$ is a cutpoint of $Q'$.
   	So we can iterate $P$ above $\kappa_0^{R}$ by lifting trees on it to $Q$
   	which translate to trees on $Q'$ via $\Sigma_{R'}$. (Note this is actually the unique above-$\kappa_0^{R}$ strategy for $P$.) 
   	\end{proof}

 \section{From $\wW_{n+1}$ to $\wW_{n+2}$}\label{sec:vV_n+1+to_vV_n+2}
 \subsection{The structures and strategies}
 
 We now want to construct $\wW_{n+2}$, given $\wW_{i+1}$ and $\Sigma_{\wW_{i+1}}$ for $i\leq n$
 with various properties.
We will first introduce the key structure of the model $\wW_{n+1}$, and some of the features we want the iteration strategy $\Sigma_{\wW_{n+1}}$ to have.
We will then state some inductive hypotheses in Remark \ref{rem:ind_hyp}, and then propagate them to $\wW_{n+2}$
and $\Sigma_{\wW_{n+2}}$.

\begin{dfn}\label{dfn:pVl}
Let $\Ll$ be the passive premouse language. Let $\varphi_0,\varphi_1,\varphi_2$ be formulas of $\Ll$.
Let $\vec{\varphi}=(\varphi_0,\varphi_1,\varphi_2)$.
A \emph{$0$-$\vec{\varphi}$-pVl}
is an $M$-like premouse.
Now let $n<\om$. 
 An \emph{$(n+1)$-$\varphi$-pVl ($\vec{\varphi}$-provisionary $\vV_{n+1}$-like)} is a
proper class transitive structure $\vV=(L[\es^{\vV}],\es^{\vV})$
such that there are $\delta_0<\ldots<\delta_n\in\OR$ such that:
\begin{enumerate}[label=(V\arabic*),ref=(V\arabic*)]
 \item  $\es^{\vV}$ is a sequence of partial extenders,
 \item $\delta_0,\ldots,\delta_n$ enumerate the first $(n+1)$ Woodin cardinals of $\vV$,
 \item\label{item:extenders_of_pVl} the extenders $E\in\es^{\vV}$ are either:
 \begin{enumerate}
  \item \emph{short}, in which case they satisfy the same requirements as for Jensen indexed premice with no whole proper segments, or
  \item \label{item:pVl_m-long} \emph{$m$-long} for some $m\leq n$, in which case,
  letting $\gamma\in\OR$
  be such that $E=F^{\vV|\gamma}$,
    then $\vV||\gamma=(\vV|\gamma)^{\passive}\sats\ZFC^-+$``there is a largest cardinal $\lambda$'',
    $\delta_m^{\vV}<\gamma$, $\lambda$ is inacessible in $\vV||\gamma$ and is a limit of short-extender-cutpoints $\xi$
    (that is, ordinals $\xi$ such that there is no short $F\in\es^{\vV||\gamma}$ with $\crit(F)<\xi<\lh(F)$), $E$ is a $(\delta_m^{\vV},\gamma)$-extender over $\vV|\delta_m^{\vV}$,
  and letting $\vV'$ be the structure defined over $\vV||\gamma=(\vV|\gamma)^\passive$ from the formula $\varphi_0$ (so $\vV'$ is amenable to $\vV||\gamma$) then $\vV'\sub\vV||\gamma$ is a transitive structure of height $\gamma$, $\Ult(\vV|\delta_m^{\vV},E)=\vV'$ (and $i^{\vV|\delta_m^{\vV}}_E:\vV|\delta_m^{\vV}\to\vV'$ cofinally), and $E$ is amenable to $\vV||\gamma$ (and thus coded in the usual amenable manner).
  \end{enumerate}
  \item For each $m\leq n$,
  there are unboundedly many $\gamma\in\OR$ such that $\es^{\vV}_\gamma$ is $m$-long,
  and for each $m<m'\leq n$,
  there are unboundedly many $\gamma<\delta_{m'}^{\vV}$
  such that $\es^{\vV}_\gamma$ is $m$-long.
  \item $\vV$ has infinitely many strong cardinals, in ordertype exactly $\om$, denoted $\kappa_{n+1}^{\vV}<\kappa_{n+2}^{\vV}<\ldots$, with supremum denoted $\lambda^{\vV}$ (so the least strong cardinal of $\vV$ is denoted $\kappa_{n+1}^{\vV}$).
  \item $\lambda^{\vV}$ is a limit of Woodin cardinals of $\vV$,
  \item\label{item:normalization_cond} All proper segments of $\vV$ are (fine structurally) sound and
satisfy  the condensation properties relevant to normalization.\footnote{ Cf.~\cite[\S\S4.12.3,5.6.1***]{vm2_v2}. This does not require condensation with respect to maps going ``across Woodin cardinals'',
where we cannot have condensation, given the layered nature of a pVl.}
  \item\label{item:gamma_m_def} For $m\leq n$,
  let $\gamma_m^{\vV}$
  be the least $\gamma$ such that $F^{\vV|\gamma}$ is $m$-long
  (so $\delta_m^{\vV}<\gamma_m^{\vV}$).
  Then $\rho_1^{\vV|\gamma_m^{\vV}}=\delta_m^{\vV}$ and $p_1^{\vV|\gamma_m^{\vV}}=\emptyset$
  (so $\gamma_m^{\vV}<\delta_{m}^{+\vV}$).
  Let $\kappa_m^{\vV}$ denote the largest cardinal of $\vV||\gamma_m^{\vV}$, and $e_m^{\vV}$ denote $F^{\vV|\gamma_m^{\vV}}$.
  \item\label{dfn:CC_definability}For $m\leq n$
  let $\CC_m$ be the forcing
  $\sub\delta_m^{\vV}$
  defined over $\vV|\gamma_m^{\vV}$
  with $\varphi_1$.
  Then $\vV\sats$``$\CC_m$ is $\delta_m^{\vV}$-cc'' and $\vV||\gamma_m^{\vV}=(\vV|\gamma_m^{\vV})^{\passive}$
  is generic over $\Ult(\vV,e_m^{\vV})$
  for $i^{\vV}_{e_m^{\vV}}(\CC_m^{\vV})=\CC_m^{\Ult(\vV,e_m^{\vV})}$.
  \item \label{item:i_e(vV|gamma)_def} $i^{\vV}_{e^\vV_m}(\vV|\gamma_m^{\vV})$ is defined in the codes over $\vV||\gamma_m^{\vV}=(\vV|\gamma_m^{\vV})^{\passive}$
  with the formula $\varphi_2$.
  \item \label{item:P-con_in_pVl}The preceding two items arrange that we can form a P-construction from $U=\Ult(\vV,e_n^{\vV})$ above
  $\vV||\gamma_n^{\vV}$. We do this as follows, building
  the structure $P$ such that $P|\gamma_n^{\vV}=\vV||\gamma_n^{\vV}$,  
and $\es^P\rest(\gamma_n^{\vV},\OR)$ is given by translating $\es^U\rest(\gamma_n^{\vV},\OR)$;
that is, for each $\nu\in(\gamma_n^{\vV},\OR)$,
\begin{enumerate}[label=--]
\item $P|\nu$ is active iff $U|\nu$ is active,
\item if $U|\nu$ is active and $F^{U|\nu}$
is short or $n$-long 
then $F^{P|\nu}\rest\OR=F^{U|\nu}\rest\OR$,
and $F^{P|\nu}$ is an extender
satisfying the conditions for Jensen indexing for $P||\nu$,
\item if $U|\nu$ is active with an $i$-long
extender where $i<n$ then $F^{P|\nu}=F^{U|\nu}\circ e_n^{\vV}\rest(\vV|\delta_i^{\vV})$
(note here that by the amenability
of $e_n^{\vV}$ over $\vV||\gamma_n^{\vV}$, we have
$e_n^{\vV}\rest(\vV|\delta_i^{\vV})\in\vV|\gamma_n^{\vV}$,
and so $e_n^{\vV}\rest(\vV|\delta_i^{\vV})$ is in the generic extension $U[\vV|\gamma_n^{\vV}]$).
\end{enumerate}
We write $\vV\down n=P$.
We now demand that $\vV\down n$
is an $n$-$(\varphi_0,\varphi_1,\varphi_2)$-pVl.\footnote{Therefore, $\vV\down 0$, i.e.~$(\ldots((\vV\down n)\down (n-1))\down\ldots\down 0)$, is $M$-like, which as usual we take to include some finite but not fully specified list of first order facts true in $M$, updated as needed.}

\item\label{item:Delta-conservativity} ($\Delta$-conservativity) For every short $E\in\es^{\vV|\Delta^{\vV}}$
and  $a\in[\lambda(E)]^{<\om}$
and $f\in\vV$ with 
\[ f:[\kappa]^{|a|}\to\OR,\] where $\kappa=\crit(E)$,
there is $g\in\vV\down n$
such that $f(u)=g(u)$ for $E_a$-measure one many $u$.
  \end{enumerate}
  We say that an extender in $\es^{\vV}$ is \emph{long} if it is $m$-long for some $m\leq n$.
  We write $\Delta^{\vV}=\delta_n^{\vV}$ (the largest Woodin over which there are long extenders on $\es^{\vV}$).
  
  An \emph{$(n+1)$-pVl}
  is an $(n+1)$-$\vec{\varphi}$-pVl
  for some $\vec{\varphi}$.
  
  Given an initial segment $\vV|\gamma$ of an $(n+1)$-pVl,
  or some related such structure
  such as an iterate $N$ of $\vV|\gamma$, we also apply parts of the terminology/notation above to $N$,
  such as $\gamma_m^N$, ``$m$-long'', etc. This includes $N\down n$, but here we define $U=\Ult_k(N,e_n^N)$, where $k\leq\om$ is the soundness degree of $N$, and then define $N\down n$ from $U$ level-by-level as before.
  \end{dfn}

  \begin{rem}
 As with the $1$-Vsps and $2$-Vsps
 considered in \cite{vm2_v2},
 the long extenders need not (and will not for those we consider)
 cohere $\vV||\gamma$.
 
   In the end there will be  specific formulas $\vec{\varphi}$ that we use, and we will only be interested in these $\vec{\varphi}$-pVls. This will eventually be made explicit. Until then, let us mostly suppress mention of $\vec{\varphi}$ as a parameter.\end{rem}

   \begin{lem}\label{lem:vV_n+1^vV_down_n_def}
 Adopt the context and notation of part \ref{item:P-con_in_pVl} of \ref{dfn:pVl}. Then:
 \begin{enumerate}
  \item \label{item:gamma=kappa^+}$\gamma_n^{\vV}=\kappa_n^{+\vV\down n}$ (that is,
  the cardinal successor of the least strong cardinal of $\vV\down n$),
  \item \label{item:U_def_wout_params}
$U$ is definable without parameters over $\vV\down n$.
 \end{enumerate}
   \end{lem}
\begin{proof}
 Part \ref{item:gamma=kappa^+} is straightforward.
 Part \ref{item:U_def_wout_params}:
 By part \ref{item:gamma=kappa^+} and the definitions,
 $(\vV\down n)|\gamma_n^{\vV}=\vV||\gamma_n^{\vV}$ is definable without parameters over $\vV\down n$,
 and therefore so is $A=i^{\vV}_{e^{\vV}_n}(\vV|\gamma_n^{\vV}~)$. But then starting from $A$, $U$ can be recovered over $\vV\down n$ by reversing the P-construction 
 used to define $\vV\down n$.
\end{proof}

   \begin{dfn}\label{dfn:vV_n+1^vV}Fix some usually to be suppressed $\vec{\varphi}=(\varphi_0,\varphi_1,\varphi_2)$ to play the role mentioned above. Working with a structure $N$ of the form $\vV||\gamma$ as in part \ref{item:extenders_of_pVl}(\ref{item:pVl_m-long})
   (and also somewhat more generally), 
   we write $\mM_{\infty,m}^{N}$ for the structure
   defined by $\varphi_0$ over $N$.
   Working in a structure $N$ of the form $\vV|\gamma_m^{\vV}$ (and also somewhat more generally), $\CC^N_m$ denotes the forcing defined by $\varphi_1$ over $N$.
   Working over a structure $N$
   of form $\vV||\gamma_m^{\vV}$
   (and again somewhat more generally), $(\vV_{m+1}|\gamma_m^{\vV_{m+1}})^N$ denotes the structure
   defined in the codes by $\varphi_2$ over $N$.
   
   Working in an $n$-pVl $\wW$,
   $\vV_{n+1}^{\wW}$ denotes the model defined over $\wW$ in the manner $U$ was defined over $\vV\down n$ in Lemma \ref{lem:vV_n+1^vV_down_n_def}. Iterating this, working in an $i$-pVl $\wW$ where $i<n$, $\vV_{n+1}^{\wW}$
   denotes $\vV_{n+1}^{(\vV_{i+1}^{\vV})}$,
   so if $\vV$ is an $M$-like premouse then $\vV_3^{\vV}=\vV_3^{\vV_2^{\vV_1^{M}}}$, etc. (We will also later give an equivalent, direct definition of $\vV_{n+1}^M$.)
   \end{dfn}
\begin{rem}
 Of course the definition of $\vV_{n+1}^{\wW}$
 above depends on $\vec{\varphi}$, but for the present purposes it is enough to know that the structure has certain properties and is definable by some formula.
 
 The following definition
 was foreshadowed in \cite[Remark 5.70***]{vm2_v2}.
\end{rem}

\begin{dfn}\label{dfn:0-maximal}
 Let $\vV$ be an $(n+1)$-pVl.
 An iteration tree $\Tt$ on $\vV$
 is \emph{$0$-maximal}
 if it satisfies the usual conditions for $0$-maximality for trees on fine structural premice
 (so $E^\Tt_\alpha\in\es_+^{M^\Tt_\alpha}$, etc), except for the following modifications.
 Along with the tree we associate
 a function $\ell:\lh(\Tt)\to(n+2)$.
 Let $\ell(0)=0$. Let $\ell(\alpha+1)=\ell(\alpha)$ unless $E^\Tt_\alpha$ is $m$-long,
 in which case let $\ell(\alpha+1)=m+1$. If $\eta<\lh(\Tt)$ is a limit then let $\ell(\eta)=\liminf_{\alpha<\eta}\ell(\alpha)$.
 Now we make the following modifications to the usual requirements of $0$-maximality:
 \begin{enumerate}
  \item If $\eta+1<\lh(\Tt)$
  and $\eta$ is a limit
  and $l(\eta)=m+1>0$
  and $[0,\eta)^\Tt\cap\dropset^\Tt=\emptyset$
  then $\delta_m^{M^\Tt_\eta}<\lh(E^\Tt_\eta)$.\footnote{We will have $\delta(\Tt\rest\eta)=\sup_{\alpha<\eta}\lh(E^\Tt_\alpha)=$ the least measurable of $M^\Tt_\eta$ in this situation, 
  and since the usual rules of $0$-maximality only demand that $\alpha<\beta\implies\lh(E^\Tt_\alpha)\leq\lh(E^\Tt_\beta)$,
  if we did not make this restriction then
  one would have been able to choose $E^\Tt_\eta$ with $\delta(\Tt\rest\eta)<\lh(E^\Tt_\eta)<\delta_m^{M^\Tt_\eta}$.}
\item 
Whenever $E^\Tt_\alpha$ is $m$-long, then $\pred^\Tt(\alpha+1)$ is the least $\beta\leq\alpha$ such that
\begin{enumerate}[label=--]
 \item  $[0,\beta]_\Tt$ does not drop,
\item  $\delta_m^{M^\Tt_\beta}<\lh(E^\Tt_\beta)$,
\item for no $\gamma\in[\beta,\alpha)$ is 
$E^\Tt_\gamma$ an $m'$-long extender with $m'\leq m$.
 \end{enumerate}
 Let $\beta'$ be least such that:
 \begin{enumerate}[label=--]
 \item $[0,\beta']^{\Tt}$ does not drop,
 \item $\delta_m^{M^{\Tt}_{\beta'}}<\lh(E^\Tt_{\beta'})$,
 \item for no $\gamma\in[\beta',\alpha)$ is $E^\Tt_\gamma$ a long extender,
 \end{enumerate}
so $\beta=\pred^\Tt(\alpha+1)\leq\beta'$. Then we will have $\beta\leq^{\Tt}\beta'\leq^{\Tt}\alpha$, and
\[ i^{\Tt}_{\beta\beta'}:M^\Tt_\beta\to M^{\Tt}_{\beta'},\]
\[ i^{\Tt}_{\beta'\alpha}:M^{\Tt}_{\beta'}\to M^\Tt_\alpha \]
both exist. But if $\beta\neq\beta'$ then $\crit(i^\Tt_{\beta\beta'})<\delta_m^{M^\Tt_\beta}$. Now $E^\Tt_\alpha$ is an extender with space $M^{\Tt}_{\beta'}|\delta_m^{M^\Tt_{\beta'}}$.
Let $(E^{\Tt}_\alpha)'$ be the $(\delta_m^{M^{\Tt}_\beta},\lh(E^\Tt_\alpha))$-extender
given by $E^\Tt_\alpha\com E_j$
where $j=i^\Tt_{\beta\beta'}$. Then
we define $M^{*\Tt}_{\alpha+1}=M^\Tt_\beta$
and
\[ M^\Tt_{\alpha+1}=\Ult(M^{*\Tt}_{\alpha+1},(E^\Tt_\alpha)').\]

(Note that we could have equivalently defined $\beta'=\pred^\Tt(\alpha+1)$ and
 \[ M^{*\Tt}_{\alpha+1}=\cHull^{M^\Tt_{\beta'}}(\rg(i^\Tt_{0{\beta'}})\cup\delta_m^{M^\Tt_{\beta'}}),\]
and
 \[ M^\Tt_{\alpha+1}=\Ult(M^{*\Tt}_{\alpha+1},E^\Tt_\alpha);\]
 this would yield the same model $M^\Tt_{\alpha+1}$.
 But then we need not have a total iteration map $i^\Tt_{\beta',\alpha}$, which is annoying notationally. It seemed to be simpler to adopt the convention we have, so that whenever $\beta\leq^{\Tt}\alpha$ and there are no drops in $[0,\alpha]^{\Tt}$,
 then $i^\Tt_{\beta\alpha}:M^\Tt_\beta\to M^\Tt_\alpha$ is defined (and total).)\qedhere
\end{enumerate}\end{dfn}

\begin{dfn}
Let $\vV$ be an $(n+1)$-pVl. An
iteration tree $\Tt$ on $\vV$ is \emph{short-normal} if it is $0$-maximal and uses only short extenders (so the complications in Definition \ref{dfn:0-maximal} pertaining to long extenders do not arise).
\end{dfn}

   \begin{dfn}\label{dfn:e-vec_finite}
   Let $\vV$ be an $(n+1)$-pVl.
         Recall we defined $e_m^{\vV}=F^{\vV|\gamma^{\vV}_m}$ in part \ref{item:gamma_m_def} of \ref{dfn:pVl}.
   Let $\Tt$ be the $0$-maximal tree on $\vV$
   with $\lh(\Tt)=n+2$
   where $E^\Tt_i=e_i^{M^\Tt_i}$ for $i\leq n$.
  Then $\vec{e}^{\vV}$ denotes $(E^\Tt_0,\ldots,E^\Tt_n)$.
We identify  $\vec{e}^{\vV}$ with the $(\delta_n^{\vV},\delta_n^{N_{n+1}})$-extender $E^\Tt_n\com E^\Tt_{n-1}\com\ldots\com E^\Tt_0$ derived from the ultrapower map $i^\Tt_{0\infty}$, so $M^\Tt_\infty=\Ult(\vV,\vec{e}^{\vV})$.
   \end{dfn}
	\begin{lem}\label{lem:e-vec_finite_stage} Let $\vV,\Tt$
	be as in Definition \ref{dfn:e-vec_finite}.
	Then $M^\Tt_i\down i=\vV_{i}^{\vV\down 0}$ for $i\leq n+1$.
	So $\Ult(\vV,\vec{e}^{\vV})=\vV_{n+1}^{\vV\down 0}$.
	\end{lem}
	\begin{proof}
We have \begin{equation}\label{eqn:e-vec_ults_vV_1}\vV_1^{\vV\down 0}=\Ult(\vV\down 1,e_0^{\vV})=i^{\Tt}_{01}(\vV\down 1)=M^\Tt_1\down 1\end{equation}
and 
\[ i^{\vV\down 0}_{e_0^{\vV}}=i^\Tt_{01}\rest(\vV\down 1),\]
since $e_0^{\vV}=e_0^{\vV\down 1}$ and $\vV\down 1$ is a 1-pVl and by conservativity.
	 
	 Now suppose $n>0$. Then
	 \begin{equation}\label{eqn:e-vec_ults_vV_2_to_vV_2_down_1} 
	 \vV_2^{\vV\down 1}=\Ult(\vV\down 2,e_1^{\vV})\end{equation}
	 (note here $e_1^{\vV}\neq E^\Tt_1$).
	 But equation (\ref{eqn:e-vec_ults_vV_2_to_vV_2_down_1})
	 lifts under the embedding $i_{e_0^{\vV}}^{\vV\down 2}$ (which agreed with the $e_0^{\vV}$ ultrapower maps on $\vV\down 0$ and $\vV\down 1$), so
	 \begin{equation}\label{eqn:e-vec_ults_more} \vV_2^{\Ult(\vV\down 1,e_0^{\vV})}=\Ult(\Ult(\vV\down 2,e_0^{\vV}),i_{e_0^{\vV}}^{\vV\down 2}(e_1^{\vV})). \end{equation}
	But $E^\Tt_0=e_0^{\vV}$ and $E^\Tt_1=i^{\vV\down 2}_{e_0^{\vV}}(e_1^\vV)=i^{\vV}_{e_0^{\vV}}(e_1^\vV)$,
	 so putting equations (\ref{eqn:e-vec_ults_vV_1})
	and (\ref{eqn:e-vec_ults_more}) together and using conservativity,
	we get
	\[ \vV_2^{\vV\down 0}=\vV_2^{(\vV_1^{\vV\down 0})}=\vV_2^{\Ult(\vV\down 1,e_0^{\vV})}=M^\Tt_2\down 2,\]
	as desired.
	
Iterating the process more times leads to the full lemma.
	\end{proof}
\begin{dfn}
 Let $\Sigma$ be either:
 \begin{enumerate}[label=(\roman*)]
  \item  a
 $(0,\OR)$-strategy, or
 \item\label{item:case_short-normal_strategy} a short-normal strategy
 \end{enumerate}
 for an $(n+1)$-pVl $\vV$. We say that
 $\Sigma$ has \emph{short-normal mhc (short-normal minimal hull condensation)}
 if $\Sigma$ has minimal hull condensation (mhc) with respect to all short-normal trees; i.e. whenever $\Xx$ is via $\Sigma$,
  $\Pi:\Tt\hookrightarrow_{\min}\Xx$ is a minimal tree embedding
 and $\Tt,\Xx$ are short-normal, then $\Tt$ is via $\Sigma$.
 
 We say that $\Sigma$ has \emph{$\Delta$-short-normal mhc} if $\Sigma$ has mhc with respect to all short-normal trees based on $\vV|\Delta^{\vV}$.
\end{dfn}
\begin{rem}\label{rem:Sigma^stk}
 Short-normal minimal hull condensation works in place of (standard) minimal hull condensation for normalizing
 stacks of 
 short-normal trees (see \cite{fullnorm_v3}).
This is because being short-normal/$m$-long
 is preserved by all the relevant maps. Suppose $\Sigma$ is a short-normal strategy with mhc. Then $\Sigma$ determines canonically an extension to a $(0,\OR,\OR)$-strategy for short-normal trees, which we denote $\Sigma^{\stk}$
 (and may confuse somewhat with $\Sigma$). If $\vV'$ is a non-dropping iterate of $\vV$ via $\Sigma$, then we write $\Sigma_{\vV'}$ for the short-normal $(0,\OR)$-strategy for $\vV'$ determined from $\Sigma$ in this manner.
 
 Likewise, $\Delta$-short-normal mhc works in place of mhc for normalizing stacks of short-normal stacks based on $\vV|\Delta^{\vV}$. Suppose $\Sigma$ is a short-normal strategy with $\Delta$-short-normal mhc.
 If $\vV'$ is a non-dropping $\Sigma$-iterate  of $\vV$ 
 via a tree based on $\vV|\Delta^{\vV}$, then $\Sigma_{\vV'}$ denotes the $(0,\OR)$-strategy for $\vV'$ given by normalization.
\end{rem}

\begin{dfn}
Let $\Sigma$ be a short-normal $(0,\OR)$-strategy for an $(n+1)$-pVl $\vV$, and suppose $\Sigma$ has $\Delta$-short-normal mhc. We say that $\Sigma$ is \emph{self-coherent} iff whenever $\Tt$ is a successor length tree on $\vV$ via $\Sigma$ and $E\in\es_+^{M^\Tt_\infty}$
is $m$-long, 
then  $E$ agrees with the strategy given by normalization. That is, letting $\alpha\in b^\Tt$ be least such that either \begin{enumerate}[label=--]
\item  $\alpha+1=\lh(\Tt)$,
or \item   $[0,\alpha]_\Tt\cap\dropset^\Tt=\emptyset$  and $\delta_m^{M^\Tt_\infty}<\lh(E^\Tt_\alpha)$,
\end{enumerate}
 then (noting that $\Tt\rest(\alpha+1)$ is based on $\vV|\Delta^{\vV}$) there is a successor length tree $\Uu$ via $\Sigma_{M^\Tt_\alpha}$,
 based on $M^\Tt_\alpha|\delta_m^{M^\Tt_\alpha}$,
 such that
 $b^\Uu$ does not drop,
 $i^\Uu_{b^\Uu}(\delta_m^{M^\Tt_\alpha})=\lh(E)$
 and $E$ is the extender over $M^\Tt_\alpha|\delta_m^{M^\Tt_\alpha}$ derived from $i^\Uu_{b^\Uu}$.
\end{dfn}

\begin{dfn}
Let $\vV$ be an $(n+1)$-pVl.
 Let $\Sigma$ be a  (not only short-normal) $(0,\OR)$-strategy for $\vV$.
 We define \emph{self-consistency} for $\Sigma$ just as in Definition \ref{dfn:self-consistent}
 (but using the current notion of $0$-maximality).
\end{dfn}

Generalizing Lemma \ref{lem:unique_ext_to_(0,OR)-strat}:

\begin{lem}\label{lem:unique_ext_to_(0,OR)-strat_n+1}
 Let $\vV$ be an $(n+1)$-pVl. Let $\Sigma$ be a short-normal $(0,\OR)$-strategy for $\vV$ which has $\Delta$-mhc and is self-coherent. Then there is a unique extension of $\Sigma$ to a self-consistent $(0,\OR)$-strategy.
\end{lem}

\begin{dfn}\label{dfn:short-normal_extend_to_0-max}
 Let $\Sigma$ be a short-normal
 $(0,\OR)$-strategy for an $(n+1)$-pVl $\vV$. Suppose  that $\Sigma$ has $\Delta$-mhc and  is self-coherent. Then, with a slight abuse of notation, we also write $\Sigma$
 for the unique self-consistent $(0,\OR)$-strategy for $\vV$  extending $\Sigma$.
\end{dfn}

\begin{dfn}\label{dfn:gamma_k-unstable}
 Let $\Tt$ be a short-normal
 tree on an $(n+1)$-pVl $\vV$.
 For $i\leq n$, we say that $\Tt$ is \emph{$\gamma_i$-unstable} iff there is $\alpha+1<\lh(\Tt)$ such that $[0,\alpha]^\Tt\cap\dropset^\Tt=\emptyset$ and $\delta_i^{M^\Tt_\alpha}<\lh(E^\Tt_\alpha)<\gamma_i^{M^\Tt_\alpha}$.
 We say that $\Tt$ is \emph{$\vec{\gamma}$-stable}
if there is no $i\leq n$
such that $\Tt$ is $\gamma_i$-unstable.
\end{dfn}

\begin{dfn}\label{dfn:e-vec_min_inflation_stable}
 Let $\Tt$ be a $\vec{\gamma}$-stable short-normal
 tree on an $(n+1)$-pVl $\vV$.
 Then the \emph{standard decomposition}
of $\Tt$ is the tuple $(\Tt_0,\ldots,\Tt_k)$ where:
 \begin{enumerate}
 \item $k\leq n+1$,
  \item $\Tt=\Tt_0\conc\ldots\conc\Tt_k$,
  \item $\Tt_i$ has successor length and $b^{\Tt_i}$ is non-dropping, for each $i<k$,
  \item $\Tt_0$ is based on $\vV|\delta_0^\vV$,
  \item $\Tt_{i}$ is based on $M^{\Tt_{i-1}}_\infty|\delta_i^{M^{\Tt_{i-1}}_\infty}$,
  and is above $\gamma_{i-1}^{M^{\Tt_{i-1}}_\infty}$,
  for each $i\in(0,\min(k,n)]$,
  \item if $k=n+1$
  then $\Tt_{n+1}$ is above $\gamma_n^{M^{\Tt_n}_\infty}$,
  \item if $\Tt_k$ has successor length and $b^{\Tt_k}$ is non-dropping then $k=n+1$ (we allow some of the $\Tt_i$'s to be trivial).
 \end{enumerate}

For $i\leq n+1$, the \emph{$\vec{e}$-$i$-minimal inflation of $\Tt$} (if it exists) is the  $0$-maximal tree
\[ \Tt'=\Tt'_0\conc\ldots\conc\Tt'_i\conc\left<e_i\right>\conc\Tt'_{i+1}\conc\left<e_{i+1}\right>\conc\ldots\conc\left<e_{k-1}\right>\conc\Tt_k'\]
on $\vV$ where\begin{enumerate}[label=--]
\item $\Tt'_j=\Tt_j$ for $j\leq i$,
 \item $e_j=e_j^{M^{\Tt_j'}_\infty}$ for $j\in[i,k)$,
 \item $\Tt_j'$ is the $(e_{j-1}\com\ldots\com e_i)$-minimal inflation of $\Tt_j$ for $j\in(i, k]$.
\end{enumerate}
(See \cite{fullnorm_v3}.
This means that $\Tt_j'$ has the same length and tree order as does $\Tt_j$,
and \[\exit^{\Tt_j'}_\alpha=\Ult_0(\exit^\Tt_\alpha,e_{j-1}\com\ldots\com e_i) \]
for $\alpha+1<\lh(\Tt_j)$;
it also follows that $\Tt_j'$
has the same drop and degree structure as does $\Tt_j$.
The fact
that $\exit^{\Tt'_j}_\alpha\ins M^{\Tt'_j}_\alpha$
(when defined as above)
follows from property \ref{item:normalization_cond}
of the definition of pVl \ref{dfn:pVl}
and Steel's condensation proof;
see \cite{fullnorm_v3}. So the only thing that might go wrong is the wellfoundedness of the models.)

Define the \emph{$\vec{e}$-minimal inflation of $\Tt$} as the $\vec{e}$-$0$-minimal inflation of $\Tt$.

Let $i\leq n+1$.
Suppose that the $\vec{e}$-$i$-minimal inflation $\Tt'$ (as above)
has wellfounded models.
We define a (putative) $0$-maximal
tree $\Tt\down i$ on $\vV\down i$ (recall $\vV\down i$ is an $i$-pVl).
We set $\Tt\down (n+1)=\Tt$.
If $i\leq n$ then we set \[\Tt\down i=\widetilde{\Tt_0}\conc\widetilde{\Tt_1}\conc\ldots\conc\widetilde{\Tt_i}\conc\Tt''_{i+1}\conc\Tt''_{i+2}\conc\ldots\conc\Tt''_k,\]
where:
\begin{enumerate}[label=--]
 \item  $\widetilde{\Tt_0}\conc\ldots\conc\widetilde{\Tt_i}$
 is just $\Tt_0\conc\ldots\conc\Tt_i=\Tt_0'\conc\ldots\conc\Tt_i'$ but as a tree
 on $\vV\down i$ (recalling $(\vV\down i)|\delta_i^{\vV\down i}=\vV|\delta_i^{\vV}$), and
 \item for $j\in(i,k]$,
 $\Tt''_j$ has the same length,
 tree order and extender indices
 as has $\Tt'_j$.
\end{enumerate}
(So if $i\geq k$ (which,
if $i\leq n$, implies that either $\Tt$ has limit length or $b^\Tt$ drops) then $\Tt\down i$ is just $\Tt$ but on $\vV\down i$.)
\end{dfn}

\begin{dfn}\label{dfn:e-vec_min_inflation_unstable}Let $\Tt$ be a $\gamma_k$-unstable short-normal
 tree on an $(n+1)$-pVl $\vV$,
 where $k\leq n$. Let $\alpha+1<\lh(\Tt)$ be least such that $\Tt\rest(\alpha+1)$ is $\gamma_k$-unstable. Then $[0,\alpha]^\Tt\cap\dropset^\Tt=\emptyset$, so the standard decomposition of $\Tt\rest(\alpha+1)$ has form $(\Uu_0,\ldots,\Uu_{n+1})$,
but $\Uu_i$ is trivial for $i>k$. Note that $\Tt=\Uu_0\conc\ldots\conc\Uu_k\conc\Vv$, where $\Vv$ is based on $[\delta_k^{M^{\Uu_k}_\infty},\gamma_k^{M^{\Uu_k}_\infty})$ and $\Vv$ is non-trivial.
We say that  $(\Uu_0,\ldots,\Uu_k,\Vv)$ is the \emph{standard decomposition} of $\Tt$.

For $i\leq n$,
we define the  \emph{$\vec{e}$-$i$-minimal inflation} of $\Tt$ as the $0$-maximal tree 
 \[ \Tt'=\Uu'_0\conc\ldots\conc\Uu'_i\conc\left<e_i\right>\conc\Uu_{i+1}'\conc\left<e_{i+1}\right>\conc\ldots\conc\left<e_{k-1}\right>\conc\Uu_k'\conc\Vv'\]
on $\vV$ where:
 \begin{enumerate}[label=--]
  \item $\Uu'_j=\Uu_j$ for $j\leq i$,
  \item $e_j=e_j^{M^{\Uu'_j}_\infty}$ for $j\in[i,k)$,
  \item $\Uu'_j$ is the $(e_{j-1}\com\ldots\com e_i)$-minimal inflation of $\Uu_j$ for $j\in(i,k)$, and
  \item  $\Uu_k'\conc\Vv'$
 is the $(e_{k-1}\com\ldots\com e_i)$-minimal inflation of $
 \Uu_k\conc\Vv$.
 \end{enumerate}

 Suppose that $\Tt'$ has wellfounded models.
 For $i\leq n$ we define a $0$-maximal tree $\Tt
 \down i$ on $\vV\down i$
 as follows. If $i\in(k,n+1]$,
 then $\Tt\down i$ is just $\Tt$ but as a tree on $\vV\down i$ (so $\Tt\down(n+1)=\Tt$).
 Suppose $i\leq k$. Then we set
 \[ \Tt\down i=\widetilde{\Uu_0}\conc\widetilde{\Uu_1}\conc\ldots\conc\widetilde{\Uu_i}\conc\Uu''_{i+1}\conc\ldots\conc\Uu''_k\conc\Vv'',\]
 where:
 \begin{enumerate}[label=--]
  \item  $\widetilde{\Uu_0}\conc\ldots\conc\widetilde{\Uu_i}$ is just $\Uu_0\conc\ldots\conc\Uu_i=\Uu'_0\conc\ldots\conc\Uu'_i$ but as a tree on $\vV_i$,
and
\item for $j\in(i,k]$,
 $\Uu_j''$ has the same length, tree order and extender indices as has $\Uu'_{j}$,
 and $\Vv''$
 has the same length, tree order and extender indices as has $\Vv''$.\qedhere
 \end{enumerate}
 \end{dfn}
 
 \begin{rem}
  Note that $\Vv$ (as above) is based on $M^{\Uu_k}_\infty|\gamma_k^{M^{\Uu_k}_\infty}$ (a set-sized model), and drops in model to this structure immediately, and $\Vv'$
  is analogous, whereas   $\Vv''$ (the last portion of $\Tt\down i$)  is based on $M^{\Uu''_k}_\infty|\kappa_k^{+M^{\Uu''_k}_\infty}$ (and here if $i=k$ we will  have
  $M^{\Uu_k}_\infty||\gamma_k^{M^{\Uu_k}_\infty}=M^{\Uu''_k}|\kappa_k^{+M^{\Uu''_k}_\infty}$)
  and if $i\leq k$, need not drop in model.

  Note that a short-normal
 tree on an $(n+1)$-pVl
 can only be $
 \gamma_k$-unstable for at most one value of $k$.
\end{rem}

\begin{lem}[Conservativity ($\vec{\gamma}$-stable)] \label{lem:conservativity_stable}
 Let $\vV$ be an $(n+1)$-pVl
 and $\Tt$ be short-normal and  $\vec{\gamma}$-stable on $\vV$, with standard decomposition \[ \Tt=\Tt_0\conc\ldots\conc\Tt_k.\]
Suppose that for each $i\leq n+1$, the $\vec{e}$-$i$-minimal inflation\[\Tt'=\Tt'_0\conc\ldots\conc\Tt'_i\conc\left<e_i\right>\conc\Tt'_{i+1}\conc\left<e_{i+1}\right>\conc\ldots\conc\left<e_{k-1}\right>\conc\Tt_k'\]
 of $\Tt$
 has wellfounded models. 
 Let $i\leq n+1$ and
\[\Tt\down i=\widetilde{\Tt_0}\conc\ldots\conc\widetilde{\Tt_i}\conc\Tt''_{i+1}\conc\Tt''_{i+2}\conc\ldots\conc\Tt''_k, \]
with notation as in \ref{dfn:e-vec_min_inflation_stable}
(even in case $i\geq k$,
$\Tt\down i$ has this form).
 Then:
 \begin{enumerate}
  \item We have:
  \begin{enumerate}
   \item 
$\Tt\down i$ is a $0$-maximal
 tree on $\vV\down i$ (so in particular has wellfounded models),
 with the same tree, drop and degree structure as has $\Tt$,
 \item For each $j\leq k$,
 $\Tt_j'$
 has the same length,
 same tree, drop and degree structure, and same extender indices as has $\Tt_j$;
 likewise for $\Tt_j'$ and $\widetilde{\Tt_j}$ when $j\leq i$, and for $\Tt_j$ and $\Tt''_j$ when $j\in(i,k]$.
 \end{enumerate}
 \item For each $\beta<\lh(\Tt)$ such that $[0,\beta]_{\Tt}\cap\dropset^\Tt=\emptyset$ 
 and each $\alpha\leq^{\Tt}\beta$,
 \begin{enumerate}
 \item $M^{\Tt\down i}_\alpha=M^\Tt_\alpha\down i=i^
\Tt_{0\alpha}(\vV\down i)$, 
 \item 
$i^{\Tt\down i}_{\alpha\beta}=i^\Tt_{\alpha\beta}\rest(M^\Tt_\alpha\down i)$, and
\item the natural factor map $\pi_{\alpha i}:M^{\Tt\down i}_\alpha\to M^\Tt_\alpha\down i$
is just the identity.\footnote{The natural factor maps $\pi_{\alpha i}$ are  defined via the Shift Lemma in the usual manner.}
\end{enumerate}
 \item For each $\alpha$ such that $\lh(\Tt_0\conc\ldots\conc\Tt_{i-1})\leq\alpha<\lh(\Tt)$
 and for $j\leq i$, we have
 \[ \vV_i^{M^{\Tt\down j}_\alpha}=\Ult                                               _d(M^{\Tt\down i}_\alpha,e_{i-1}\circ\ldots\circ e_j) \]
 where $d=\deg^{\Tt}(\alpha)=\deg^{\Tt
 \down i}(\alpha)=\deg^{\Tt\down j}(\alpha)$,
 \item letting $j<i\leq k$ and
 \[ \Tt_{[i,k]}=\Tt_i\conc\ldots\conc\Tt_k \]
 and $\Uu_i=\Tt_{[i,k]}\down i$
 and $\Uu_j=\Tt_{[i,k]}\down j$
 and $e=e_{i-1}\com\ldots\com e_j$
and $\Uu_i'$
 be the minimal $e$-inflation
 of $\Uu_i$, then:
  \begin{enumerate}
\item  $\Uu_i$ is $0$-maximal on $M^{\Tt_{[0,i-1]}}_\infty\down i=M^{\Tt_{[0,i-1]}\down i}_\infty$ and is above $\gamma_{i-1}^{M^{\Uu_i}_0}=\gamma^{M^{\Uu_i}_0}$,
\item  $\Uu_j$ is $0$-maximal on $M^{\Tt_{[0,i-1]}}_\infty\down j=M^{\Tt_{[0,i-1]}\down j}_\infty=(M^{\Uu_i}_0)\down j$ and is above $\kappa_{i-1}^{+M^{\Uu_j}_0}=\gamma^{M^{\Uu_i}_0}$,
\item  $\Uu_i'$ is $0$-maximal on $
\wW=\vV_i^{M^{\Uu_j}_0}$
and is above $\delta_{i-1}^{\wW}=\gamma^{M^{\Uu_i}_0}$,
and in fact above $\gamma_{i-1}^{\wW}$,
 \item $\Uu_i,\Uu_i',\Uu_j$ have the same lengths, extender indices, tree, drop and degree structure,
 \item for each $\alpha<\lh(\Uu_i)$, 
 letting
 $d=\deg^{\Uu_i}(\alpha)=\deg^{\Uu_i'}(\alpha)=\deg^{\Uu_j}(\alpha)$,  we have
\[ M^{\Uu_i'}_\alpha=\Ult_d(M^{\Uu_i}_\alpha,e)=\vV_i^{M^{\Uu_j}_\alpha} \]
(and note $\gamma_{i-1}^{M^{\Uu_i'}_\alpha}=\gamma_{i-1}^{\wW}<\OR(M^{\Uu_i'}_\alpha)=\OR(M^{\Uu_j}_\alpha)$ and $e\in M^{\Uu_i}_\alpha$),
\item for each $\alpha<\lh(\Uu_i)$,
the natural factor map $\sigma_\alpha:M^{\Uu_i'}_\alpha\to\vV_i^{M^{\Uu_j}_\alpha}$ is the identity,
\item for all $\beta<^{\Uu_i}\alpha<\lh(\Uu_i)$,
if $(\beta,\alpha]^{\Uu_i}$ does not drop then $i^{\Uu_i'}_{\beta\alpha}\sub i^{\Uu_j}_{\beta\alpha}$,
\item for each $\alpha+1<\lh(\Uu_i)$, we have:
\begin{enumerate}[label=--]
\item $\exit^{\Uu_i'}_\alpha=\Ult_0(\exit^{\Uu_i}_\alpha,e)$ (and note  that $e\in\exit^{\Uu_i}_\alpha|\crit(E^{\Uu_i}_\alpha)$),
\item $\exit^{\Uu_i'}_\alpha\sub\exit^{\Uu_i}_\alpha$ and $E^{\Uu_i'}_\alpha\sub E^{\Uu_i}_\alpha$,
\item $\exit^{\Uu_i'}_\alpha=\vV_i^{\exit^{\Uu_j}_\alpha}\sub\exit^{\Uu_j}_\alpha$,
so $E^{\Uu_i'}_\alpha\sub E^{\Uu_j}_\alpha$,
\end{enumerate}

\item for all $\alpha+1<\lh(\Uu_i)$,  letting
 $d=\deg^{\Uu_i}(\alpha)=\deg^{\Uu_i'}(\alpha)=\deg^{\Uu_j}(\alpha)$,  we have
\[ M^{*\Uu'_i}_{\alpha+1}=\Ult_d(M^{*\Uu_i}_{\alpha+1},e)=\vV_i^{M^{*\Uu_j}_{\alpha+1}} \]

\item for all $\alpha+1<\lh(\Uu_i)$, we have
$i^{*\Uu_i'}_{\alpha+1}\sub i^{*\Uu_j}_{\alpha+1}$.
 \end{enumerate}
 \end{enumerate}
\end{lem}
\begin{proof}
Given the calculations presented earlier in the paper,
 this is a straightforward induction, particularly using (i) the $\Delta$-conservation of $(n+1)$-pVls, (ii) the fact that \[ \vV_{i+1}^{\vV\down i}=\Ult(\vV\down (i+1),e_i^{\vV\down(i+1)}) \] (which gets preserved by non-dropping iteration maps),
  and (iii) the method of proof of Lemma \ref{lem:Psi_vV_1_is_conservative},
  to secure the analogous conservativity facts ``above $\Delta$''. We leave the details to the reader.
\end{proof}

There is naturally also a version for $\gamma_k$-unstable trees:

\begin{lem}[Conservativity ($\gamma_k$-unstable)]\label{lem:conservativity_unstable}
 Let $\vV$ be an $(n+1)$-pVl
 and $\Tt$ be short-normal and $\gamma_k$-unstable on $\vV$, where $k\leq n$, with standard decomposition $(\Uu_0,\ldots,\Uu_k,\Vv)$.
 Suppose that the $\vec{e}$-minimal inflation\[\Tt'=\Uu_0'\conc\left<e_0\right>\conc\ldots\conc\left<e_{k-1}\right>\conc\Uu_k'\conc\Vv'\]
 of $\Tt$
 has wellfounded models.
 Let $\lambda=\lh(\Uu_0\conc\ldots\conc\Uu_k)$.
 Let $\Vv^\dagger$ be $\Vv$
 considered as a 
 $0$-maximal tree on $M^\Tt_\lambda|\gamma_k^{M^\Tt_\lambda}$
  (since $\Vv$ is on $M^{\Tt}_\lambda$, above $\delta_k^{M^\Tt_\lambda}$,
 based on $M^\Tt_\lambda|\gamma_k^{M^\Tt_\lambda}$,
 and $\rho_1^{M^{\Tt}_\lambda|\gamma_k^{M^\Tt_\lambda}}=\delta_k^{M^\Tt_\lambda}$,
 this makes sense).
 Then:
 \begin{enumerate}
\item \label{item:Uu_0^conc^Uu_k_stable} $\Uu_0\conc\ldots\conc\Uu_k$ is $\vec{\gamma}$-stable short-normal  (on $\vV$),
\item\label{item:Tt_down_i_stable} for $i\leq k$, $\Tt\down i$ is $\vec{\gamma}$-stable short-normal (on $\vV\down i$)
\end{enumerate}
(so Lemma \ref{lem:conservativity_stable} applies to the trees mentioned in parts \ref{item:Uu_0^conc^Uu_k_stable} and \ref{item:Tt_down_i_stable}),
\begin{enumerate}[resume*]
\item  for $i\in(k,n+1]$,
  $\Tt\down i$ is a $\gamma_k$-unstable short-normal
 tree on $\vV\down i$,
 and $(\Vv\down i)^\dagger=\Vv^\dagger$ (where we (can) define $(\Vv\down i)^\dagger$ analogously to $\Vv^\dagger$), 
\item
  $[0,\lambda+
\alpha]_{\Tt}$ drops in model to some point $\leq$ some image of $\vV|\gamma_k^{\vV}$,
\item $\Vv^\dagger$ and $\Vv\down k$
have the same lengths,
extenders, tree, drop and degree structures (recall that both $\Vv^\dagger$
and $\Vv\down k$ are $0$-maximal,
so even at points where there is no drop, the trees have the same degree),
\item for each $\alpha<\lh(\Vv)$, if $[0,\alpha]^{\Vv^\dagger}\cap\dropset^{\Vv^\dagger}=\emptyset$ and $\beta\leq^{\Vv^\dagger}\alpha$ then:
\begin{enumerate}
 \item $(M^{\Vv^\dagger}_\alpha)^\passive=M^{\Vv\down k}_\alpha|\kappa_k^{+M^{\Vv\down k}_\alpha}$,
 \item $M^{\Vv^\dagger}_\alpha$
 is active with $e_k^{M^{\Vv^\dagger}_\alpha}$,
 \item  $i^{\Vv^\dagger}_{\beta\alpha}\sub i^{\Vv\down k}_{\beta\alpha}$,
 \item $\vV_{k+1}^{M^{\Vv\down k}_\alpha}|\gamma_k^{\vV_{k+1}^{M^{\Vv\down k}_\alpha}}=\Ult_0(M^{\Vv^\dagger}_\alpha,e_k^{M^{\Vv^\dagger}_\alpha})$,
\end{enumerate}
\item 
for each $\alpha<\lh(\Vv)$,
if $[0,\alpha]^{\Vv^\dagger}\cap\dropset^{\Vv^\dagger}\neq\emptyset$ and $\beta\leq^{\Vv^\dagger}\alpha$ with $(\beta,\alpha]^{\Vv^\dagger}\cap\dropset^{\Vv^\dagger}=\emptyset$ then
 $M^{\Vv^{\dagger}}_\alpha=M^{\Vv\down k}_\alpha$ and $i^{\Vv^\dagger}_{\beta\alpha}=i^{\Vv\down k}_{\beta\alpha}$,
\item for each $\alpha+1\in\dropset^{\Vv^\dagger}$, we have
 $M^{*\Vv^\dagger}_{\alpha+1}=M^{*(\Vv\down k)}_{\alpha+1}$
and $i^{*\Vv^\dagger}_{\alpha+1}=i^{*(\Vv\down k)}_{\alpha+1}$.
 
\end{enumerate}
\end{lem}

A normalization-like lemma:
\begin{lem}\label{lem:e-vec_commutativity}
 Let $\vV$ be an $(n+1)$-pVl. Let $e=e^{\vV}=e_n^{\vV}$ and $\vec{e}=\vec{e}^{\vV}$ and $\vec{e}\rest n=\vec{e}^{\vV\down 0}$. Let $(\vec{e}\rest n)^+$ be the $(\kappa_n^{+\vV\down n},\kappa_n^{+\Ult(\vV\down n,\vec{e}\rest n)})$-extender
 derived from the iteration map $k:\vV\down n\to\Ult(\vV\down n,\vec{e}\rest n)$. Then
 \[ \Ult(\vV,(\vec{e}\rest n)^+\com e)=\vV_{n+1}^{\vV\down 0}=\Ult(\vV,\vec{e}) \]
and the resulting ultrapower maps $j,j':\vV\to\vV_{n+1}^{\vV\down 0}$ agree, and $\vec{e}$ is just the $(\delta_n^\vV,\delta_n^{\vV_{n+1}^{\vV\down 0}})$-extender derived from $j$, or equivalently,  $j'$.
\end{lem}
\begin{proof}
We have $\Ult(\vV\down n,\vec{e}\rest n)=\vV_n^{\vV\down 0}$,
and $k:\vV\down n\to\vV_n^{\vV\down 0}$ is the ultrapower map.
Since $\vV_{n+1}^{\vV\down n}$ is a ground
of $\vV\down n$ for a forcing which is $\kappa_n^{+\vV\down n}$-cc in $\vV_{n+1}^{\vV\down n}$,
we have 
\[k(\vV_{n+1}^{\vV\down n})= \Ult(\vV_{n+1}^{\vV\down n},(\vec{e}\rest n)^+)=\vV_{n+1}^{\vV\down 0} \]
and letting $\bar{k}:\vV_{n+1}^{\vV\down 0}\to\vV_{n+1}^{\vV\down 0}$ be the $(\vec{e}\rest n)^+$-ultrapower map,
then $\bar{k}=k\rest\vV_{n+1}^{\vV\down n}$.
But since $\Ult(\vV,e)=\vV_{n+1}^{\vV\down n}$,
and $e$ is the $(\delta_n^{\vV},\delta_n^{\vV_{n+1}^{\vV\down n}})$-extender derived from $i^{\vV}_e$,
and $\delta_n^{\vV_{n+1}^{\vV\down n}}=\kappa_n^{+\vV\down n}$,
we therefore get that
\[ \vV_{n+1}^{\vV\down 0}=\Ult(\Ult(\vV,e),(\vec{e}\rest n)^+). \]
But we also have
\[ \vV_{n+1}^{\vV\down 0}=\Ult(\vV,\vec{e}^{\vV}).\]
Let $j'=k\com i^{\vV}_e:\vV\to\vV_{n+1}^{\vV\down 0}$
and $j=i^{\vV}_{\vec{e}^{\vV}}$ be the two ultrapower maps.
Both result from $(\delta_n^{\vV},\delta_n^{\vV_{n+1}^{\vV\down 0}})$-extenders,
so to see $j=j'$, it is enough to see that $j\rest(\vV|\delta_n^{\vV})=j'\rest(\vV|\delta_n^{\vV})$,
and since $j(\vV|\delta_n^{\vV})=j'(\vV|\delta_n^{\vV})$,
it is therefore enough to see that $j\rest\delta_n^{\vV}=j'\rest\delta_n^{\vV}$. Let $\alpha<\delta_n^{\vV}$
and let $\beta=i^{\vV}_e(\alpha)$.
Then
\[ j'(\alpha)=k(i^{\vV}_e(\alpha))=k(\beta),\]
but $k=i^{\vV\down n}_{\vec{e}\rest n}=i^{\vV\down n}_{\vec{e}^{\vV\down n}}$ and by conservativity,
$k\sub i^{\vV}_{\vec{e}\rest n}$.
We have $j=i^{\vV'}_{e'}\com i^{\vV}_{\vec{e}\rest n}$ where $\vV'=\Ult(\vV,\vec{e}\rest n)$ and $e'=e_n^{\vV'}=i^{\vV}_{\vec{e}\rest n}(e)$ (recall $e\in \vV$).
But applying $i^{\vV}_{\vec{e}\rest n}$ to the equation  $e(\alpha)=\beta$ gives
\[ e'(i^{\vV}_{\vec{e}\rest n}(\alpha))=i^{\vV}_{\vec{e}\rest n}(\beta),\]
so
\[ j(\alpha)=e'(i^{\vV}_{\vec{e}\rest n}(\alpha))=k(\beta)=j'(\alpha), \]
as desired.
\end{proof}

\begin{dfn}\label{dfn:e-vec_minimal_pullback}
 Let $\vV$ be an $(n+1)$-pVl.
Let $i\leq n$ and $\Sigma$
 be a short-normal $(0,\OR)$-strategy for $\vV\down i$.
 Let $\Sigma'$ be a short-normal $(0,\OR)$-strategy for $\vV$.
 
 We say that $\Sigma'$ is \emph{$\Sigma$-conservative}
 if for all (short-normal)
 trees $\Tt$ via $\Sigma'$,
 $\Tt\down i$ exists (equivalently, has wellfounded models) and is via $\Sigma$.
  Comparing with Lemma \ref{lem:e-vec_minimal_pullback} below,
 we also say that $\Sigma'$ is the \emph{$\vec{e}$-minimal pullback} of $\Sigma$.
\end{dfn}
\begin{lem}\label{lem:e-vec_minimal_pullback}
 Let $\vV$ be an $(n+1)$-pVl and let $\Sigma$ be a $(0,\OR)$-strategy for $\vV\down 0$. Then there is a unique $\Sigma$-conservative short-normal $(0,\OR)$-strategy $\Sigma'$ for $\vV$. Moreover,
 for each $i\in[1,n]$,
 if we let $\Sigma_{\vV\down i}$ be the unique $\Sigma$-conservative short-normal $(0,\OR)$-strategy for $\vV\down i$,
 then $\Sigma'$ is $\Sigma_{\vV\down i}$-conservative.
\end{lem}
\begin{proof}
 This follows readily from the definitions and Lemmas \ref{lem:conservativity_stable},
  \ref{lem:conservativity_unstable}, \ref{lem:e-vec_commutativity}.
\end{proof}

\begin{lem}[MHC Inheritance]\label{lem:mhc_inheritance}
 Let $\vV$ be an $(n+1)$-pVl, $N=\vV\down 0$, 
 $\Sigma_N$ be a $(0,\OR)$-strategy for $N$,
 and $\Sigma_{\vV}$ be the $\vec{e}$-minimal pullback of $\Sigma_N$. Then:
 \begin{enumerate}[label=--]
  \item 
If $\Sigma_N$ has mhc, then so does $\Sigma_{\vV}$.
\item If $\Sigma_N$ has $\delta_n$-mhc, then so does $\Sigma_{\vV}$.
\end{enumerate}
\end{lem}
\begin{proof}
Suppose $\Sigma_N$ has mhc. Let $\Sigma_{\vV\down i}$
be the unique $\Sigma_N$-conservative short-normal $(0,\OR)$-strategy for $\vV\down i$. We show 
by induction on $i\leq n+1$
that $\Sigma_{\vV\down i}$
has mhc. For $i=0$ it is trivial. Suppose we know it holds for $\Sigma_{\vV\down i}$, and $i\leq n$. Then
mhc 
for trees  on $\vV_{i+1}$,
via $\Sigma_{\vV\down(i+1)}$,  and based on $\vV|\gamma_i^{\vV}$, follows immediately
(since this just agrees with $\Sigma_{\vV\down i}$).
So consider $\Tt=\Tt_0\conc\Tt_1$ and $\Xx=\Xx_0\conc\Xx_1$ on $\vV_{i+1}$, with $\Tt_0,\Xx_0$ based on $\vV_{i+1}|\delta_i^{\vV}$,
 of successor length, with $b^{\Tt_0},b^{\Xx_0}$ non-dropping,
 and $\Tt_1,\Xx_1$ above
 $\gamma_i^{M^{\Tt_0}_\infty},\gamma_i^{M^{\Xx_0}_\infty}$
respectively.
Let $j:M^{\Tt_0}_\infty\to M^{\Tt_1}_\infty$ be the iteration map and $E$ the $(\delta_0^{M^{\Tt_0}_\infty},\delta_0^{M^{\Tt_1}_\infty})$-extender derived from $j$.
Then $j``\Xx_0$ is a minimal hull of $\Xx_1$ (by property \ref{item:normalization_cond} of $(n+1)$-pVls, $\vV\down(i+1)$ has the right condensation properties to ensure this). Let $e=e_i^{M^{\Tt_1}_\infty}$.
So $e``\Xx_1$ translates to a tree on $M^{\Tt_1}_\infty\down i$
via $\Sigma_{M^{\Tt_1}_\infty\down i}$. Now $e``(j``\Xx_0)$
is also a minimal hull of $e``\Xx_1$, so also translates to a tree via $\Sigma_{M^{\Tt_1}_\infty\down i}$. But letting $\bar{e}=e_i^{M^{\Tt_0}_\infty}$, note that $e``(j``\Xx_0)$
is just $k``(\bar{e}``\Xx_0)$, where $k:(M^{\Tt_0}_\infty\down i)\to (M^{\Tt_1}_\infty\down i)$ is the iteration map agreeing with $j$. Therefore $\bar{e}``\Xx_0$ translates to a tree via $\Sigma_{M^{\Tt_0}_\infty\down i}$, so $\Tt_0\conc\Xx_0$ is via $\Sigma_{\vV}$, as desired.

The version when $\Sigma_N$
has $\delta_n$-mhc is similar.
\end{proof}

\subsection{Inductive hypotheses}
	\begin{rem}[Inductive hypotheses at $\wW_{n+1}$]\label{rem:ind_hyp}
	For the remainder of this section, we fix an $n<\om$ and suppose we have defined
	$\wW_0=M,\wW_1,\ldots,\wW_{n+1}$,  $\vV_0=M,\vV_1,\ldots,\vV_{n+1}$, and 
 $\mM_{\infty 0},\ldots,\mM_{\infty n}$, and associated (not just short-normal) $(0,\OR)$-strategies $\Sigma_{\vV}$  for these models $\vV$, 
and let $\Sigma^{\sn}_{\vV}$ be their restrictions to short-normal trees.
Suppose that there is some $\vec{\varphi}$ such that for all $i\leq n$, the following hypotheses \ref{item:hyp_1}--\ref{item:hyp_8} hold:

\begin{figure}
\centering
\begin{tikzpicture}
 [mymatrix/.style={
    matrix of math nodes,
    row sep=0.7cm,
    column sep=0.7cm}
  ]
   \matrix(m)[mymatrix]{
\vV_3\down 0&\vV_3\down 1&\mM_{\infty2}&\vV_3  &{}&{}&{}\\
{}&\vV_2\down 0&\mM_{\infty1}   &\vV_2  &\vV_3^{\wW_1}&{}&{}\\
{}&{}&               \mM_{\infty0}&\vV_1  &\vV_2^{\wW_1}&\vV_3^{\wW_2}&{}\\
{}&{}&{}&                            \wW_0      &\wW_1        &\wW_2        &\wW_3\\  
};
\path[->,font=\scriptsize]
(m-2-2) edge node[right] {$e_2$}  (m-1-1)
(m-2-3) edge node[right] {$e_2$}  (m-1-2)
(m-2-4) edge node[right]{$e_2$}  (m-1-3)
(m-2-5) edge node[right]{$e_2$}  (m-1-4)
(m-3-3) edge node[right]{$e_1$}  (m-2-2)
(m-3-4) edge node[right]{$e_1$}  (m-2-3)
(m-3-5) edge node[right]{$e_1$}  (m-2-4)
(m-3-6) edge node[right]{$e_1$}  (m-2-5)
(m-4-4) edge node[right]{$e_0$}  (m-3-3)
(m-4-5) edge node[right]{$e_0$}  (m-3-4)
(m-4-6) edge node[right]{$e_0$}  (m-3-5)
(m-4-7) edge node[right]{$e_0$}  (m-3-6)
;
\end{tikzpicture}
\caption{\label{fgr:layout}Layout in case $n=2$. Note that $M=\vV_0=\wW_0$,
and $\vV_{i+1}=\vV_{i+1}^{\wW_0}$, and $\mM_{\infty i}=\vV_{i+1}\down i$.  If  $P$ appears directly above $Q$, then $P\subsetneq Q$, and $Q$ is a set-generic extension of $P$.
If $P$ appears directly left of $Q$ then $P\subsetneq Q$, and $Q$ is not a set-generic extension of $P$. The diagonal arrows are ultrapower maps; those in the bottom row are ultrapowers via $e_0=e_0^{\wW_3}=e_0^{\wW_2}=e_0^{\wW_1}$, those in the next row are via $e_1=e_1^{\vV_2^{\wW_1}}=e_1^{\vV_3^{\wW_2}}=i^{\wW_3}_{e_0}(e_1^{\wW_3})$; etc.} 
\end{figure}

\begin{enumerate}[label=(Hyp\arabic*)]
 \item\label{item:hyp_1} 
 $\vV_{i+1},\wW_{i+1}$ are  $\vec{\varphi}$-$(i+1)$-pVls  and $\vV_0=\wW_0=M$ (until we have defined $\wW_{n+2}$, we work  with this fixed $\vec{\varphi}$, and mostly suppress it until then),
 \item for $i\leq j\leq n+1$, we have
 \begin{enumerate}[label=--]
\item $\wW_i=\wW_j\down i$,
 \item $\vV_{i+1}=\Ult(\wW_{i+1},\vec{e}^{\wW_{i+1}})=\vV_{i+1}^{\vV_i}$,
 \item  $\mM_{\infty i}=\Ult(\Ult(\wW_i,\vec{e}^{\wW_{i+1}}\rest i),(\vec{e}^{\wW_{i+1}})_i)=\Ult(\vV_i,(\vec{e}^{\wW_{i+1}})_i)=\vV_{i+1}\down i$ (here with $\vec{e}$ defined as $e_i\com e_{i-1}\com\ldots\com e_0$,
 $\vec{e}\rest i$ denotes $e_{i-1}\com e_{i-2}\com\ldots\com e_0$, and $(\vec{e})_i$ denotes $e_i$; note that $\vec{e}^{\wW_{i+1}}\rest i=\vec{e}^{\wW_i}$, and $(\vec{e}^{\wW_{i+1}})_i\notin\wW_i$),
 \end{enumerate}
 (cf.~Figure \ref{fgr:layout}, and for the definition of $\vV_{i+1}^{\vV_i}$,
 see \ref{dfn:vV_n+1^vV}),
 \item  $\Sigma_{\wW_{i+1}}^{\sn}$ is the
 $\vec{e}$-minimal pullback of $\Sigma_M=\Sigma_{\vV_0}$
 (so by Lemma \ref{lem:mhc_inheritance}, this strategy has mhc),
 \item $\Sigma_{\wW_{i+1}}^{\sn}$ is self-coherent,
 and (using Lemma \ref{lem:unique_ext_to_(0,OR)-strat_n+1})
 $\Sigma_{\wW_{i+1}}$ is the unique extension of $\Sigma^{\sn}_{\wW_{i+1}}$ to a self-consistent $(0,\OR)$-strategy,
 \item \label{hyp:Sigma_vV_i+1}$\Sigma_{\vV_{i+1}}^{\sn}$
 is the strategy for $\vV_{i+1}$ induced from $\Sigma_{\wW_{i+1}}^{\sn}$ via normalization,
 and $\Sigma_{\vV_{i+1}}$ is the unique extension of $\Sigma^{\sn}_{\vV_{i+1}}$ to a self-consistent $(0,\OR)$-strategy,
 
 \item we have \[\begin{split} \delta_i^M=\delta_i^{\vV_i}&<\kappa_i^M=\kappa_i^{\vV_i}=\mu^{\vV_{i+1}}<\delta_0^{\vV_{i+1}}<\gamma_0^{\vV_{i+1}}<\delta_1^{\vV_{i+1}}<\gamma_1^{\vV_{i+1}}<\ldots\\\ldots&<\delta_i^{\vV_{i+1}}=\Delta^{\vV_{i+1}}=\kappa_i^{+M}=\kappa_i^{+\vV_i}<\gamma_i^{\vV_{i+1}}=\gamma^{\vV_{i+1}}<\nu\end{split}\]
 where $\mu^{\vV_{i+1}}$ is the least measurable of $\vV_{i+1}$ and $\nu$ is least such that $M|\nu$ is admissible with $\kappa_i^{+M}<\nu$ (and recall $\kappa_i^{\vV_i}$ is the least strong cardinal of $\vV_i$
 and $\kappa_0^M,\ldots,\kappa_i^M$ the least $i+1$ strong cardinals of $M$),
  \item 
  $\kappa_i^{\vV_i}$ is a cutpoint with respect to short extenders in $\es^{\vV_i}$;
  if $i>0$ then there cofinally many $h$-long extenders in $\es^{\vV_i}$,
  for each $h<i$,
   \item\label{item:hyp_8} We have:
   \begin{enumerate}
\item      	Let $N$ be a non-dropping $\Sigma_{M}$-iterate. 
   	Let $\PP\in N$, $\PP\subseteq\lambda$,
   	where $\lambda\in\OR$.
   	Let $x=\lambda^{+M}$.
   	Let $G$ be $(M,\PP)$-generic (with $G$  appearing in a generic extension of $V$).
   	Let $\vV=\vV_{n+1}^N$.
   	Then:
   	\begin{enumerate}
   		\item $N$ is closed under $\Sigma_{\vV,\Delta^{\vV}}$, and $\Sigma_{\vV,\Delta^{\vV}}\rest N$
   		is lightface definable over $N$, uniformly in $N$,
   		\item  $N[G]$
   		is closed under $\Sigma_{\vV,\Delta^{\vV}}$,
   		and $\Sigma_{\vV,\Delta^{\vV}}\rest N[G]$
   		is definable over $N[G]$ from $x$, uniformly in $N,x$.
   		\end{enumerate}
   		
\item   	Let $\wW$ be a non-dropping $\Sigma_{\wW_{i+1}}$-iterate.
   	Let $\PP\in\wW$, $\PP\subseteq\lambda$, where $\lambda\geq\delta_i^\wW$.
   	Let $x=\wW|\lambda^{+\wW}$.
   	Let $G$ be $(\wW,\PP)$-generic (with $G$  appearing in a generic extension of $V$).
   	Let $\vV=\vV_{n+1}^{\wW}$ (so if $i=n$ then $\vV=\wW$).
   	Then:
   	\begin{enumerate}
   		\item $\wW$ is closed under $\Sigma_{\vV,\Delta}$, and $\Sigma_{\vV,\Delta}\rest\wW$
   		is lightface definable over $\wW$, uniformly in $\wW$,
   		\item  $\wW[G]$
   		is closed under $\Sigma_{\vV,\Delta}$,
   		and $\Sigma_{\vV,\Delta}\rest\wW[G]$
   		is definable over $\wW[G]$ from $x$, uniformly in $\wW,x$.
   		\end{enumerate}
\end{enumerate}
\end{enumerate}
\end{rem}

\begin{rem}Some consequences:
\begin{enumerate}
 \item  $\mM_{\infty i}\subsetneq\vV_{i+1}\subsetneq\vV_i$ and $\mM_{\infty i},\vV_{i+1}$ are both definable without parameters over $\vV_i$\footnote{In the end, one will be able to extract a uniform definition for this.};
 given a model $P\equiv\vV_i$,
 write $\mM_\infty^P$ or $\mM_{\infty i}^P$ for the model defined over $P$ as $\mM_{\infty i}$
 is defined over $\vV_i$, and write $\vV^P$ or $\vV_{i+1}^P$ for the model defined over $P$ 
 as $\vV_{i+1}$ is defined over $\vV_i$. So $\mM_\infty^P$ and $\vV^P$ denote the models defined in passing to the next stage,
 irrespective of the current index. In particular,
 we have $\mM_{\infty i}=\mM_\infty^{\vV_i}=\mM_{\infty i}^{\vV_i}$ and $\vV_{i+1}=\vV^{\vV_i}=\vV_{i+1}^{\vV_i}$.
 And
 $\mM_{\infty i}\equiv\vV_i$, but $\vV_{i+1}\not\equiv\vV_i$.

 \item \label{item:vV_i+1||gamma^vV_i+1=M_infty|kappa^+}  $\mM_{\infty i}$ is a non-dropping $\Sigma_{\vV_i}$-iterate of $\vV_i$,
 via a short-normal tree $\Tt_{i,i+1}=(\Tt,b)$, where $\Tt$ has limit length and $b$ is $\Tt$-cofinal and non-dropping, with $\Tt$ based on $\vV_i|\delta_i^{\vV_i}$, $\delta_i^{\mM_{\infty i}}=\kappa_i^{+\vV_i}=\kappa_i^{+M}$, $M(\Tt)=\mM_{\infty i}|\delta_i^{\mM_{\infty i}}$
 is definable without parameters over $\vV_i|\kappa_i^{+\vV_i}$,\footnote{This corresponds to  $\varphi_0$ of Definition \ref{dfn:pVl};
 cf.~part \ref{item:extenders_of_pVl}(\ref{item:pVl_m-long}) of that definition.} and
 \begin{equation}\label{eqn:vV_i+1||gamma^vV_i+1} \vV_{i+1}||\gamma^{\vV_{i+1}}=\mM_{\infty i}|\kappa_i^{+\mM_{\infty i}}.\end{equation} 
 \item\label{item:F^vV_i+1|gamma^vV_i+1} Letting $j=i^\Tt_b:\vV_i\to \mM_{\infty i}$.
 then $j(\Tt)$ is via $\Sigma_{\mM_{\infty i}}$
 ($\Sigma_{\mM_{\infty i}}$ is defined via normalization from $\Sigma_{\wW_i}$; see Remark \ref{rem:Sigma^stk}) and letting $c=\Sigma_{\mM_{\infty i}}(j(\Tt))$, then $c$ is non-dropping
 and $\delta_i^{M^{j(\Tt)}_c}=\gamma^{\vV_{i+1}}$
 and
 \begin{equation}\label{eqn:F^vV_i+1|gamma^vV_i+1}F^{\vV_{i+1}|\gamma^{\vV_{i+1}}} \text{ is the }(\delta_i^{\mM_{\infty i}},\delta_i^{M^{j(\Tt)}_c})\text{-extender
 derived from }i^{j(\Tt)}_c:\mM_{\infty i}\to j(\mM_{\infty i}).\end{equation}
 
 \item $V_{\delta}^{\mM_{\infty i}}=V_{\delta}^{\vV_{i+1}}$
 where $\delta=\kappa_i^{+\vV_i}$, and $\delta=\delta_i^{\mM_{\infty i}}=\delta_i^{\vV_i}$ (the $i$th Woodin in both $\mM_{\infty i}$ and $\vV_{i+1}$),
 but $V_{\delta+1}^{\mM_{\infty i}}\subsetneq V_{\delta+1}^{\vV_{i+1}}$.
 \item\label{item:vV^M_infty=Ult(vV,e)}  $j(\vV_{i+1})=\vV^{\mM_{\infty i}}=\Ult(\vV_{i+1},e_i^{\vV_{i+1}})$
 where $j=i^\Tt_b$ as in part \ref{item:F^vV_i+1|gamma^vV_i+1}.\footnote{Note though that $j$
 and $i^{\vV_{i+1}}_{e_i^{\vV_{i+1}}}$ are different embeddings; we basically have ``$j(j)=e_i^{\vV_{i+1}}$'',
 and $\vV_{i+1}$ is almost ``in $\rg(j)$''.}
 \item $\vV_{i+1}|\gamma^{\vV_{i+1}}$ is definable in the codes
 over $\vV_i|\kappa_i^{+\vV_i}$.\footnote{This corresponds to $\varphi_2$ in Definition \ref{dfn:pVl};
 cf.~part \ref{item:i_e(vV|gamma)_def} of that definition.}
  \item  The poset $\CC^{\vV_{i+1}}\sub\Delta^{\vV_{i+1}}$ is definable over $\vV_{i+1}|\gamma^{\vV_{i+1}}$\footnote{This corresponds to $\varphi_1$ in Definition \ref{dfn:pVl};
  cf.~part \ref{dfn:CC_definability} of that definition.} and is $\Delta^{\vV_{i+1}}$-cc in $\vV_{i+1}$ and $\vV_i|\kappa_i^{+\vV_i}$ is $(\vV_{i+1},\CC^{\vV_{i+1}})$-generic,
 
 \item $\es^{\vV_{i+1}}\rest(\gamma^{\vV_{i+1}},\infty)$ is given by (modified) P-construction of $\vV_i$ above $\kappa_i^{+\vV_i}$
 (see part \ref{item:P-con_in_pVl} of \ref{dfn:pVl}, and recall $\vV_{i+1}|\gamma^{\vV_{i+1}}$
 is determined by equations (\ref{eqn:vV_i+1||gamma^vV_i+1})
 and (\ref{eqn:F^vV_i+1|gamma^vV_i+1})).
 That is, for $\nu>\gamma^{\vV_{i+1}}$, we have $\vV_{i+1}|\nu$ is active iff $\vV_i|\nu$ is active, and if $\vV_i|\nu$ is active with extender $E$, then:
 \begin{enumerate}
\item if $E$ is short (so $\crit(E)\geq\kappa_i^{\vV_i}$)  then $E\rest(\vV_{i+1}||\nu)\in\es^{\vV_{i+1}}$;
 here if $\crit(E)=\kappa_i^{\vV_i}$ then $E\rest(\vV_{i+1}||\nu)$
 is $i$-long in $\vV_{i+1}$, 
 and if $\crit(E)>\kappa_i^{\vV_{i}}$ then $\gamma^{\vV_{i+1}}<\crit(E)$
 and $E$ is short over $\vV_{i+1}||\nu$,
 \item
 if $E$ is $h$-long (so $h<i$) then $F^{\vV_{i+1}||\nu}\com f=E$,
 where $f$ is the $(\delta_h^{\vV_i},\delta_h^{\mM_{\infty i}})$-extender derived the embedding $j=i^\Tt_b$ of part \ref{item:F^vV_i+1|gamma^vV_i+1},
 \end{enumerate}
It follows that $\mM_{\infty i}=\vV_{i+1}\down i$;
cf.~Definition \ref{dfn:pVl} part \ref{item:i_e(vV|gamma)_def}.
For $\vV_{i+1}\down i$
is defined via the reverse P-construction of $U=\Ult(\vV_{i+1},e_i^{\vV_{i+1}})$
over $\vV_{i+1}||\gamma^{\vV_{i+1}}=\mM_{\infty i}|\kappa_i^{+\mM_{\infty i}}$
(see line (\ref{eqn:vV_i+1||gamma^vV_i+1}) in part \ref{item:vV_i+1||gamma^vV_i+1=M_infty|kappa^+}). And by part \ref{item:vV^M_infty=Ult(vV,e)}, 
$U=j(\vV_{i+1})$. And since the reverse P-construction of $\vV_{i+1}$ over $\vV_i|\kappa_i^{+\vV_i}$ gives $\vV_i$ back,
the reverse P-construction of $j(\vV_{i+1})$ over $\mM_{\infty i}|\kappa_i^{+\mM_{\infty i}}$  gives $\mM_{\infty i}$ back.
\end{enumerate}
\end{rem}

\begin{rem}
Recall by \ref{hyp:Sigma_vV_i+1}, $\Sigma_{\vV_{i+1}}^{\sn}$
is the strategy induced by $\Sigma_{\wW_{i+1}}^{\sn}$ via normalization. We now characterize this strategy another way,
not only for $\vV_{i+1}$ but for arbitrary non-dropping $\Sigma_{\wW_{i+1}}$-iterates $\vV$.
Recall here that $\vV\down 0$ is a non-dropping $\Sigma_M$-iterate $N$, and that $\Sigma_{N}$ therefore denotes the $(0,\OR)$-strategy for $N$ induced from $\Sigma_M$ via normalization.
\end{rem}

\begin{lem}\label{lem:Sigma_vV_is_e-vec_min_pullback} Let $\vV$ be a non-dropping $\Sigma_{\wW_{i+1}}$-iterate. Then $(\Sigma_{\vV})^{\sn}$ (the short-normal component of $\Sigma_{\vV}$) is the $\vec{e}$-minimal pullback
 of $\Sigma_{\vV\down 0}$.
\end{lem}
\begin{proof}
The proof is by induction on $i$.
Let
 $\Tt=\Tt_0\conc\Tt_1$ be a tree on $\Sigma_{\wW_{i+1}}$
with $\Tt_0$ based on $\wW_{i+1}|\Delta^{\wW_{i+1}}$
and of successor length, $b^{\Tt_0}$ non-dropping,
and $\Tt_1$ above $\gamma^{M^{\Tt_0}_\infty}$
(possibly $\Tt_1$ is trivial).
For trees $\Uu$ on $\vV=M^{\Tt}_\infty$ based on $\wW_{i+1}|\gamma^{\wW_{i+1}}$ it follows essentially immediately by induction. So let $\Uu=\Uu_0\conc\Uu_1$, analogous to $\Tt=\Tt_0\conc\Tt_1$, but on $\vV$,
with $\Uu$ via $\Sigma_{\vV}$, hence induced by $\Sigma_{\wW_{i+1}}$ via normalization.
We want to see that $\Uu$ is via the $\vec{e}$-minimal pullback of $\Sigma_{\vV\down 0}$,
or equivalently, by induction and Lemma \ref{lem:e-vec_minimal_pullback},
via the $e_i$-minimal pullback of $\Sigma_{\vV\down i}$.
Let $\widetilde{\Tt}=\widetilde{\Tt_0}\conc\widetilde{\Tt_1}=\Tt\down i$.

Let $\bar{\vV}=M^{\Tt_0}_\infty$
and $\bar{\Uu}_0$ be $\Uu_0$ as a tree on $\bar{\vV}$.
Then $\bar{\Uu}_0$ is via $\Sigma_{\bar{\vV}}$.
Let $\bar{\Xx}_0$ be the normalization of $(\Tt_0,\bar{\Uu}_0)$.
Let $\bar{j}:\bar{\vV}\to M^{\bar{\Uu}_0}_\infty$ be the iteration map. Then $\bar{j}``\Tt_1$ is via $\Sigma_{M^{\bar{\Uu}_0}_\infty}$. This just means that $\bar{\Xx}_0\conc\bar{j}``\Tt_1$ is via the $\vec{e}$-minimal pullback of $\Sigma_{\vV\down 0}$,
or equivalently, via the $e_i$-minimal pullback of $\Sigma_{\vV\down i}$. So letting $e=e^{M^{\bar{\Xx}_0}_\infty}$, $(\bar{j}``\Tt_1)\down i$ is the translation of $e``(\bar{j}``\Tt_1)$ (which itself is on $\vV_{i+1}^{M^{\bar{\Xx_0}\down i}_\infty}$) to a tree on $M^{\bar{\Xx_0}}_\infty$ via $\Sigma_{M^{\bar{\Xx}_0\down i}_\infty}$. 
Let $\bar{\widetilde{j}}:M^{{\Tt}_0\down i}_\infty\to M^{\bar{\Xx_0}\down i}_\infty$ be the iteration map.
Then $\bar{\widetilde{j}}``(\Tt_1\down i)$ is also via $\Sigma_{M^{\bar{\Xx_0}}_\infty}$, and it is straightforward to check that $\bar{\widetilde{j}}``(\Tt_1\down i)=(\bar{j}``\Tt_1)\down i$.

Now $\Uu_1$ is formed by normalizing $(\bar{j}``\Tt_1,\Uu_1)$,
using $\Sigma_{M^{\bar{\Xx}_0}_\infty}$ to form the normal trees $\Yy$ used in the process. This just means that  each $e``\Yy$ translates to a tree via $\Sigma_{M^{\bar{\Xx}_0\down i}_\infty}$. But then (combining the latter fact with some details of normalization) the translation of $e``\Yy$ is just the tree which arises at the corresponding stage of the normalization of the translation of $(e``(\bar{j}``\Tt_1),e``\Uu_1)$.
But the translation of $e``(\bar{j}``\Tt_1)$ is just $(\bar{j}``\Tt_1)\down i=\bar{\widetilde{j}}``(\Tt_1\down i)$,
and therefore $\Uu_0\conc(e``\Uu_1)$ translates to a tree via $\Sigma_{M^{\Tt\down i}_\infty}=\Sigma_{\vV\down i}$, as desired.
\end{proof}

\subsection{Inference of $\Sigma_{\wW_{n+1\down i}}$ from $\Sigma_{\wW_{n+1}}$}

\begin{lem}
   	Let $N$ be a non-dropping $\Sigma_{M}$-iterate of $M$.
   	Let $\PP\in M$, $\PP\subseteq\lambda$.
   	Let $x=\pow(\lambda)^N$.
   	Let $G$ be $(N,\PP)$-generic (with $G$  appearing in a generic extension of $V$).
   	Let $\vV=\vV_{i+1}^N$.
   	Then:
   	\begin{enumerate}
   		\item $N$ is closed under $\Sigma_{\vV,\vV^-}$, and $\Sigma_{\vV,\vV^-}\rest N$
   		is lightface definable over $N$, uniformly in $N$,
   		\item  $N[G]$
   		is closed under $\Sigma_{\vV,\vV^-}$,
   		and $\Sigma_{\vV,\vV^-}\rest N[G]$
   		is definable over $ [G]$ from $x$, uniformly in $N,x$.
\end{enumerate}
\end{lem}

	\begin{dfn} 
	We say that $\Sigma_{\vV_{n+1}}$ is \emph{$\Sigma_M$-stable}
	iff for every non-dropping $\Sigma_{\vV_{n+1}}$-iterate
	$\vV$
	and $\Sigma_M$-iterate $N$,
	if $\vV=\vV_{n+1}^M$,
then $\mathrm{trl}(\Sigma_N)$
agrees with $\Sigma_{\vV}$ on trees
which are translatable with respect to $\vV_{i+1}^N$),
	\end{dfn}
\begin{dfn}
Let $\Tt$ be a successor length tree on $M$ such that $b^\Tt$ exists and is non-dropping. We say that $\Tt$
 is \emph{$(\delta_n,\kappa_n)$-stable}
 iff there is no $\alpha+1\in b^\Tt$ such that
 \[ \delta_n^{M^\Tt_\beta}<\crit(E^\Tt_\alpha)\leq\kappa_n^{M^\Tt_\beta}.\]
 We say that $\Tt$ is \emph{$\leq(\delta_n,\kappa_n)$}-stable if it is $(\delta_i,\kappa_i)$-stable
 for all $i\leq n$.
 We say that $\Tt$ is \emph{everywhere $(\delta,\kappa)$-stable}
 iff it is $(\delta_n,\kappa_n)$-stable for all $n<\om$. A non-dropping iterate $N$ of $M$
 is \emph{$(\delta_n,\kappa_n)$-stable}
 iff it is an iterate via
 a $(\delta_n,\kappa_n)$-stable tree,
 and likewise for \emph{$\leq(\delta_n,\kappa_n)$-stable} and \emph{everywhere $(\delta,\kappa)$-stable}.
\end{dfn}

\begin{lem}\label{lem:everywhere_kappa^+-stable_unique_strategy}
 Let $\Tt$ be an everywhere $(\delta,\kappa)$-stable tree on $M$ via $\Sigma_M$.
 Let $N=M^\Tt_\infty$. Then $\Sigma_N$ is the unique $(0,\OR)$-strategy
 for $N$.
\end{lem}
\begin{proof}
 Suppose $\Sigma,\Gamma$ are two such strategies for $N$.
 We compare them in the usual manner.
 That is, let $\Tt$ be a limit length tree according to both $\Sigma$ and $\Gamma$, such that $\Sigma(\Tt)=b\neq c=\Gamma(\Tt)$. We can then successfully compare the phalanxes $\Phi(\Tt\conc b)$ and $\Phi(\Tt\conc c)$. This produces a common last model $Q$, with no drops on the main branches leading from $M$ to $Q$.
 
 If $Q$ is above $M^\Tt_b$ then there is $\eta<\delta(\Tt)$ such that $\Hull^Q(\mathscr{I}^Q\cup\eta)$
 is cofinal in $\delta(\Tt)$
  (where $\mathscr{I}^Q$ is the class of Silver indiscernibles of $Q$),
 and we have
 \[ \delta(\Tt)\cap\Hull^{M^\Tt_b}(\mathscr{I}^{M^\Tt_b}\cup\eta)=\delta(\Tt)\cap\Hull^Q(\mathscr{I}^Q\cup\eta). \]
 For if $\delta(\Tt)\in\rg(i_{MQ})$
 where $i_{MQ}:M\to Q$ is the iteration map then this is clear;
othewise let $\gamma\in b$
 be such that $\delta(\Tt)\in\rg(i^\Tt_{\gamma b})$
 and let $i^\Tt_{\gamma b}(\bar{\delta})=\delta(\Tt)$.
 Then note that because the tree leading from $M$ to $N$ is everywhere strong-plus-stable,
 there is $\nu<\bar{\delta}$
 such that $\bar{\delta}\sub\Hull^{M^\Tt_\gamma}(\mathscr{I}^{M^\Tt_\gamma}\cup\nu)$,
 which suffices. (This needn't hold for non-everywhere strong-plus-stable trees, since there are many overlapped Woodins, for example in the interval $[\delta_0^M,\kappa_0^M)$.)
 
 So if $Q$ is above both $M^\Tt_b$ and $M^\Tt_c$, then we reach a contradiction to the Zipper Lemma.
 
Also, letting $Q_b=Q(\Tt,b)$,
if $M^\Tt_b\sats$``$\delta(\Tt)$ is not Woodin'', and $Q_b=M^\Tt_b$ otherwise,
then if $E\in\es_+^{Q_b}$ and $\crit(E)<\delta(\Tt)<\lh(E)$,
then letting $P$ be the model
of $\Phi(\Tt\conc b)$ 
to which $E$ applies,
then $P$ has the hull property at $\kappa$, where if $P$ is proper class, then this means that $P|\kappa^{+P}\sub\cHull^{P}(\mathscr{I}^P\cup\kappa)$.
This is because, if $\kappa\notin\rg(i_{MP})$
then $P=M^\Tt_\alpha$
for some $\alpha<\lh(\Tt)$,
and $\kappa^{+P}\sub\Hull^P(\rg(i_{NP})\cup\nu)$
where $\nu=\sup_{\beta<\alpha}\nu(E^\Tt_\beta)$,
and because the tree from $M$ to $N$ is everywhere strong-plus-stable,
it follows that $P|\kappa^{+P}\sub\Hull^{P}(\rg(i_{MP}\cup\kappa)$.

But then the usual arguments show that if $Q$ is not above $M^\Tt_b$,
then also $Q$ is not above $M^\Tt_c$,
and moreover that compatible extenders are used during the comparison, a contradiction.
\end{proof}

We now want to consider converting an iteration strategy for an $(n+1)$-pVl $\vV$ into an iteration strategy for $\vV\down 0$.
For this we will translate between trees $\Uu$ on $\vV\down 0$ and trees $\widetilde{\Uu}$ on $\vV$, given
that $\Uu$ is of the right form:
we make a restriction analogous
to the restriction on the extenders
that may be used in building $0$-maximal trees on an $(n+1)$-pVl.

\begin{dfn}
 Let $N$ be an $M$-like premouse.
 An iteration tree $\Tt$ on $N$
 is \emph{$(n+1)$-translatable}
 if $\Tt$ is $0$-maximal
 and for every limit $\eta$
 with $\eta+1<\lh(\Tt)$, if
 there is $m\leq n$ such that:
 \begin{enumerate}
\item $[0,\eta)_\Tt$ does not drop,
\item $\kappa_m^{M^\Tt_\eta}=\delta(\Tt\rest\eta)$,
\item there are cofinally many $\beta+1<^\Tt\eta$ such that $\crit(E^\Tt_\beta)=\kappa_m^{M^\Tt_{\pred^\Tt(\beta+1)}}$,
 \end{enumerate}
 then $\kappa_m^{+M^\Tt_\eta}<\lh(E^\Tt_\eta)$.\end{dfn}

 \begin{dfn}
 Let $N$ be an $M$-like premouse.
 Let $\Tt$ be a padded $0$-maximal tree on $N$
 (that is, the non-padded tree equivalent to $\Tt$ is $0$-maximal).
 Given $\alpha+1<\lh(\Tt)$
 and $m<\om$,
 say that $E^\Tt_\alpha$ is \emph{$m$-shifting} (for $\Tt$)
 if $[0,\alpha+1]^{\Tt}\cap\mathscr{D}^{\Tt}=\emptyset$
 and letting $\beta=\pred^\Tt(\alpha+1)$,
 then $\crit(E^\Tt_\alpha)=\kappa_m^{M^\Tt_\beta}$.
 We say that $\Tt$ is \emph{$(n+1)$-translatable-padded} if there is a function $v:\lh(\Tt)\to n+2$ such that: 
 \begin{enumerate}
  \item $v(0)=0$,
  \item if $v(\alpha)=m\leq n$
  and $[0,\alpha]^{\Tt}\cap\dropset^\Tt=\emptyset$ 
  and $E^\Tt_\alpha\neq\emptyset$
  then:
  \begin{enumerate}
  \item $\lh(E^\Tt_\alpha)<\kappa_m^{+M^\Tt_\alpha}$,
  and 
  \item if $m>0$ then $\kappa_{m-1}^{+M^\Tt_\alpha}<\lh(E^\Tt_\alpha)$,
  \end{enumerate}
  \item if $v(\alpha)=m\leq n$
  and $[0,\alpha]^{\Tt}\cap\dropset^\Tt=\emptyset$
  and $E^\Tt_\alpha=\emptyset$
  then $v(\alpha+1)=m+1$,
    and $\nu^\Tt_\alpha=\kappa_m^{M^\Tt_\alpha}$,\footnote{Here $\nu^\Tt_\beta$ denotes the exchange ordinal associated to $\alpha$. If $E^\Tt_\beta\neq\emptyset$ then $\nu^\Tt_\beta=\lambda(E^\Tt_\beta)$. In general $\pred^\Tt(\gamma+1)$ is the least $\xi$ such that $\crit(E^\Tt_\gamma)<\nu^\Tt_\xi$.}
  \item if $\alpha+1<\lh(\Tt)$
  and $E^\Tt_\alpha\neq\emptyset$ and there is no $m\leq n$ such that $E^\Tt_\alpha$ is $m$-shifting,
   then $v(\alpha+1)=v(\alpha)$
   (and then $v(\alpha+1)=v(\beta)$
   where $\beta=\pred^\Tt(\alpha+1)$),
   \item if $\alpha+1<\lh(\Tt)$
   and $E^\Tt_\alpha\neq\emptyset$ and $E^\Tt_\alpha$ is $m$-shifting for some $m\leq n$, then $v(\alpha+1)=m+1$,
  \item if $\eta<\lh(\Tt)$ is a limit then $v(\eta)=\lim_{\gamma\in [0,\eta)^{\Tt}}v(\gamma)$.\qedhere
 \end{enumerate}
\end{dfn}

\begin{rem}
Note that if $\Tt$ is $(n+1)$-translatable-padded, as witnessed by $v$, and $\alpha\leq^{\Tt}\beta$,
then $v(\alpha)\leq v(\beta)$.

It is straightforward to see that a tree $\Tt$ on $N$ is $(n+1)$-translatable iff
there is an $(n+1)$-translatable-padded tree $\Uu$ equivalent to $\Tt$
(that is, the $0$-maximal
tree obtained by removing padding is just $\Tt$). Moreover,
if $\Tt$ is $(n+1)$-translatable
then $\Uu$ and $v$ are uniquely determined by $\Tt$.
\end{rem}

\begin{dfn}
 Let $\vV$ be an $(n+1)$-pVl.
 A \emph{$(0,n+1)$-translation pair
 on $\vV$} is a pair $(\Tt,\Tt')$
 of iteration trees such that:
 \begin{enumerate}
  \item $\Tt$ is $(n+1)$-translatable-padded on $\vV\down 0$, as witnessed by $v$,
  \item $\Tt'$ is $0$-maximal
  on $\vV$ (non-padded),
 \item $\lh(\Tt)=\lh(\Tt')$,
 \item for each $\alpha+1<\lh(\Tt)$:
\begin{enumerate}
\item if $E^\Tt_\alpha\neq\emptyset$
 then $\lh(E^{\Tt'}_\alpha)=\lh(E^\Tt_\alpha)$,
 \item if $E^\Tt_\alpha=\emptyset$
 then $\lh(E^{\Tt'}_\alpha)=\kappa_{v(\alpha)}^{+M^\Tt_\alpha}=\gamma_{v(\alpha)}^{M^{\Tt'}_\alpha}$,
 \end{enumerate}
 \item if $\eta<\lh(\Tt)$
 is a limit then there is $\zeta<\eta$ such that $[0,\eta)_\Tt\cut\zeta=[0,\eta)_{\Tt'}\cut\zeta$.
 \end{enumerate}
 
 We similarly define a tree
 $\Tt\uparrow i$ on $\vV\down i$,
 for each $i\in[0,n+1]$.
 We set $\Tt\uparrow 0=\Tt$
 and $\Tt\uparrow(n+1)=\Tt'$.
 Now let $0<i<n+1$, and define $\Tt\uparrow i$ as follows:
 \begin{enumerate}
  \item $\Tt\uparrow i$ is padded $0$-maximal on $\vV\down i$,
  \item  $\lh(\Tt\uparrow i)=\lh(\Tt)$,
  \item for each $\alpha+1<\lh(\Tt)$:
  \begin{enumerate}
  \item if $E^\Tt_\alpha\neq\emptyset$ then $E^{\Tt\uparrow i}_\alpha\neq\emptyset$ and $\lh(E^{\Tt\uparrow i}_\alpha)=\lh(E^{\Tt}_\alpha)$,
  \item if $E^\Tt_\alpha=\emptyset$
  and $v(\alpha)<i$ then $E^{\Tt\uparrow i}_\alpha\neq\emptyset$  and \[\lh(E^{\Tt\uparrow i}_\alpha)=\kappa_{v(\alpha)}^{+M^{\Tt}_\alpha}=\gamma_{v(\alpha)}^{M^{\Tt\uparrow i}_\alpha}=\gamma_{v(\alpha)}^{M^{\Tt'}_\alpha},\]
  \item if $E^\Tt_\alpha=\emptyset$ and $v(\alpha)\geq i$ then $E^{\Tt\uparrow i}_\alpha=\emptyset$
  and $\nu^{\Tt\uparrow i}_\alpha=\kappa_{v(\alpha)}^{M^{\Tt\uparrow i}_\alpha}=\kappa_{v(\alpha)}^{M^{\Tt}_\alpha}=\nu^\Tt_\alpha$,
  \end{enumerate}
   \item if $\eta<\lh(\Tt)$ is a limit
  then there is $\zeta<\eta$ such that $[0,\eta)_{\Tt}\cut\zeta=[0,\eta)_{\Tt\uparrow i}\cut\zeta$.\qedhere
 \end{enumerate}
\end{dfn}

\begin{lem}\label{lem:n+1-translation_pair}
 Let $\vV$ be an $(n+1)$-pVl.
 Let $\Tt$ on $\vV\down 0$
 be $(n+1)$-translatable-padded, as witnessed by $v$.
 Then:
 \begin{enumerate}
  \item  $\Tt\uparrow i$ exists
 for each $i\leq n+1$,
 \item \label{item:if_T_uparrow_n+1_no_drop_to_alpha}if $[0,\alpha]^{\Tt\uparrow(n+1)}\cap\dropset^{\Tt\uparrow(n+1)}=\emptyset$ then:
 \begin{enumerate}
 \item $[0,\alpha]^{\Tt\uparrow i}\cap\dropset^{\Tt\uparrow i}=\emptyset$ for all $i\leq n+1$,
 \item $\vV_{v(\alpha)}^{M^{\Tt\uparrow i}_\alpha}=M^{\Tt\uparrow v(\alpha)}_\alpha$ for each $i\leq v(\alpha)$ (this is trivial for $i=v(\alpha)$),
 \item $M^{\Tt\uparrow i}_\alpha=(M^{\Tt\uparrow j}_\alpha)\down i$ for all $i,j$ with $v(\alpha)\leq i\leq j\leq n+1$,
 \item if $\beta\leq^{\Tt}\alpha$
 and $v(\beta)=v(\alpha)$ then $i^{\Tt\uparrow i}_{\beta\alpha}$
 agrees with $i^{\Tt}_{\beta\alpha}$
 on $M^{\Tt\uparrow i}_{\beta}\cap M^\Tt_\beta$,
 
 \end{enumerate}
 \item Suppose  $[0,\alpha]^{\Tt\uparrow(n+1)}\cap\dropset^{\Tt\uparrow(n+1)}\neq\emptyset$. Let
 $\beta+1$
 be least in $\dropset^{\Tt\uparrow(n+1)}\cap[0,\alpha]^{\Tt\uparrow(n+1)}$.
 Let $\xi=\pred^{\Tt\uparrow(n+1)}(\beta+1)$.
 Then:
 \begin{enumerate}
\item  Suppose there is $m\leq n$ such that
 $\delta<\OR(M^{*\Tt\uparrow(n+1)}_{\beta+1})\leq\gamma$
 where $\delta=\delta_m^{M^{\Tt\uparrow(n+1)}_\xi}$ and $\gamma= \gamma_m^{M^{\Tt\uparrow(n+1)}_\xi}$. Then:
 \begin{enumerate}
 \item $v(\xi)=v(\beta+1)=v(\alpha)=m$.
 So by part \ref{item:if_T_uparrow_n+1_no_drop_to_alpha},
 \[ \vV_m^{M^{\Tt\uparrow i}_\xi}=M^{\Tt\uparrow m}_\xi=M^{\Tt\uparrow m'}_\xi\down m \]
 and $\gamma=\kappa_m^{+M^{\Tt\uparrow i}_\xi}=\kappa_m^{+\vV_m^{M^{\Tt\uparrow i}_\xi}}$
 for  all $i,m'$ such that $i\leq m\leq m'\leq n+1$, and
\[ M^{\Tt\uparrow i}_\xi|\gamma=M^{\Tt\uparrow(n+1)}_\xi|\gamma \]
for all $i$ such that $m<i\leq n+1$,
and
\[ \vV_m^{M^{\Tt\uparrow i}_\xi}|\gamma=M^{\Tt\uparrow m}_\xi|\gamma=M^{\Tt\uparrow(n+1)}_\xi||\gamma\]
for all $i\leq m$.
\item $[0,\alpha]^{\Tt}\cap\dropset^{\Tt}\neq\emptyset$ iff $[0,\alpha]^{\Tt\uparrow(n+1)}$
drops below the image of $\gamma$
in the interval $(\xi,\alpha]^{\Tt\uparrow(n+1)}$; that is, iff either
\begin{enumerate}[label=--]
\item $\OR(M^{*\Tt\uparrow(n+1)}_{\beta+1})<\gamma$, or
\item $(\beta+1,\alpha]^{\Tt\uparrow(n+1)}\cap\dropset^{\Tt\uparrow(n+1)}\neq\emptyset$,
\end{enumerate}
\item for each $i\leq m$, we have $[0,\alpha]^{\Tt\uparrow i}\cap\dropset^{\Tt\uparrow i}\neq\emptyset$ iff $[0,\alpha]^{\Tt}\cap\dropset^{\Tt}\neq\emptyset$
\item if $[0,\alpha]^{\Tt}\cap\dropset^{\Tt}=\emptyset$,
then for all $i,i',k,k'$ such that $i,i'\leq m< k,k'\leq n+1$ and all $\delta\in[\beta+1,\alpha]^{\Tt}$:
\begin{enumerate}
\item $\vV_m^{M^{\Tt}_\alpha}|\kappa_m^{+M^\Tt_\alpha}=\vV_m^{M^{\Tt\uparrow i}_\alpha}|\kappa_m^{+M^{\Tt\uparrow i}_\alpha}=(M^{\Tt\uparrow k}_\alpha)^\passive$,
\item \label{item:first_emb_agreement}$i^{\Tt\uparrow i}_{\delta\alpha}$
agrees with $i^{\Tt\uparrow i'}_{\delta\alpha}$ over $\vV_m^{M^{\Tt}_\delta}=\vV_m^{M^{\Tt\uparrow i}_\delta}=\vV_m^{M^{\Tt\uparrow i'}_\delta}$,
\item \label{item:second_emb_agreement}$i^{\Tt\uparrow i}_{\delta\alpha}$ agrees with $i^{\Tt\uparrow k}_{\delta\alpha}$
over $\vV_m^{M^{\Tt}_\delta}|\kappa_m^{+M^{\Tt}_\delta}=(M^{\Tt\uparrow k}_\delta)^{\passive}$,
\item $M^{\Tt\uparrow k}_\alpha=M^{\Tt\uparrow k'}_\alpha$,
\item $i^{\Tt\uparrow k}_{\delta\alpha}=i^{\Tt\uparrow k'}_{\delta\alpha}$,
\item $\vV_m^{M^{*\Tt}_{\beta+1}}
|\kappa_m^{+M^{*\Tt}_{\beta+1}}=\vV_m^{M^{*(\Tt\uparrow i)}_{\beta+1}}|\kappa_m^{+M^{*(\Tt\uparrow i)}_{\beta+1}}=(M^{*(\Tt\uparrow k)}_{\beta+1})^\passive$,
\item  the embeddings $i^{*\Tt}_{\beta+1}$, etc, agree in the  manner analogous to $i^{\Tt}_{\delta\alpha}$, etc,
as stated in parts \ref{item:first_emb_agreement}
and \ref{item:second_emb_agreement}
\item $M^{*(\Tt\uparrow k)}_{\beta+1}=M^{*(\Tt\uparrow k')}_{\beta+1}=M^{\Tt\uparrow(n+1)}_\xi|\gamma$,
\item $i^{*(\Tt\uparrow k)}_{\beta+1}=i^{*(\Tt\uparrow k')}_{\beta+1}$.
\end{enumerate}

\item if $[0,\alpha]^{\Tt}\cap\dropset^{\Tt}\neq\emptyset$,
and $\beta'+1$ is the least element of $[0,\alpha]^{\Tt}\cap\dropset^{\Tt}$,
then for all $\delta\in[\beta'+1,\alpha]^{\Tt}$, we have:
\begin{enumerate}
\item the models and embeddings of $\Tt\uparrow m$ in $[\beta'+1,\alpha]^{\Tt\uparrow m}$
are precisely  those of $\Tt\uparrow k$ in  $[\beta'+1,\alpha]^{\Tt\uparrow k}$, for all $k\in(m,n+1]$,
and likewise for $M^{*(\Tt\uparrow m)}_{\beta+1}$ and $i^{*(\Tt\uparrow m)}_{\beta+1}$, etc, and
\item $\vV_m^{M^{\Tt\uparrow i}_\alpha}=M^{\Tt\uparrow m}_\alpha$ and $i^{\Tt\uparrow i}_{\delta\alpha}\rest\vV_m^{M^{\Tt\uparrow i}_\delta}=i^{\Tt\uparrow m}_{\delta\alpha}$,
and likewise for $i^{*(\Tt\uparrow m)}_{\beta+1}$, etc.
\end{enumerate}

 \end{enumerate}
 \item If otherwise then $\beta+1=\min([0,\alpha]^{\Tt\uparrow i})$ for all $i\leq n+1$, and the models and embeddings of all trees correspond above this drop much as above (after replacing $M^{\Tt\uparrow i}_\alpha$ with $\vV_m^{M^{\Tt\uparrow i}_\alpha}$ for $i<v(\alpha)=v(\beta+1)$ like before; but $M^{\Tt\uparrow v(\alpha)}_\alpha=M^{\Tt\uparrow k}_\alpha$ for $k\geq v(\alpha)$).
 \end{enumerate}
 \end{enumerate}
\end{lem}
\begin{proof}
 This is a straightforward induction on the length of the tree and using standard calculations, together with the properties of an $(n+1)$-pVl. Let us point out the two key features of the latter which come up. Firstly, using \ref{item:Delta-conservativity} (${\Delta}$-conservativity) we get for example that if $E\in\es^{\vV}$ with $\lh(E)<\delta_i^{\vV}$ and (so $E\in\es^{\vV\down i}$ also) and $E$ is $\vV$-total then $\Ult(\vV\down i,E)=i^{\vV}_E(\vV\down i)$, and $i^{\vV\down i}_{E}=i^{\vV}_E\rest(\vV\down i)$.
 This is important in maintaining the inductive hypotheses.
Secondly,
 by Lemma \ref{lem:e-vec_finite_stage},
 we have
 \begin{equation}\label{eqn:Ult_is_next_vV}\Ult(\vV\down(i+1),e_i^{\vV\down(i+1)})=\vV_{i+1}^{\vV\down i}.\end{equation}
 Suppose $\alpha<\lh(\Tt)$
 and $[0,\alpha]^{\Tt\uparrow(n+1)}$
 is non-dropping, and $\nu(E^\Tt_\beta)\leq\delta_i^{M^{\Tt\uparrow(i+1)}_\alpha}$ for all $\beta+1<\lh(\Tt)$. Then we will have that
 $M^{\Tt\uparrow(i+1)}_\alpha=\Ult(\vV\down(i+1),E)$
 and $M^{\Tt\uparrow i}_\alpha=\Ult(\vV\down i,E)$
 where $E$ is the $(\delta_i^{\vV\down(i+1)},\delta_i^{M^{\Tt\uparrow(i+1)}_\alpha})$-extender
 derived from $i^{\Tt\uparrow(i+1)}_{0\alpha}$, or equivalently,
 derived from $i^{\Tt\uparrow i}_{0\alpha}$, and in fact
 $i^{\Tt\uparrow i}_{0\alpha}=i^{\Tt\uparrow(i+1)}_{0\alpha}\rest(\vV\down i)$
 and $i^{\Tt\uparrow(i+1)}_{0\alpha}(\vV\down i)=M^{\Tt\uparrow i}_\alpha$.
So line (\ref{eqn:Ult_is_next_vV})
is propagated by the iteration maps,
and therefore
\begin{equation}\label{eqn:Ult_is_next_vV_propagated} \Ult(M^{\Tt\uparrow (i+1)}_\alpha,e_i^{M^{\Tt\uparrow(i+1)}_\alpha})=\vV_{i+1}^{M^{\Tt\uparrow i}_\alpha}.\end{equation}

Now suppose that $E^\Tt_\alpha\neq\emptyset$ and $\delta_i^{M^{\Tt\uparrow i}_\alpha}<\lh(E^\Tt_\alpha)<\kappa_i^{+M^\Tt_\alpha}$, and note $\kappa_i^{+M^\Tt_\alpha}=\kappa_i^{+M^{\Tt\uparrow i}_\alpha}=\gamma_i^{M^{\Tt\uparrow(i+1)}_\alpha}$.
Suppose $\alpha<^\Tt\beta$
and $[0,\beta]^{\Tt}$ does not drop,
so $[0,\beta]^{\Tt\uparrow(i+1)}$
drops exactly to an image of $\vV|\gamma_i^{\vV}$. Suppose $E^\Tt_\beta=\emptyset$,
so $E^{\Tt\uparrow(i+1)}_\beta=F(M^{\Tt\uparrow(i+1)}_\beta)=e_i^{M^{\Tt\uparrow(i+1)}_\beta}$, which is an image of $e_i^{\vV}$. By induction,
we have that $j=i^{\Tt\uparrow i}_{\alpha\beta}\rest M^{\Tt\uparrow i}|\kappa_i^{+M^{\Tt\uparrow i}_\alpha}=i^{\Tt\uparrow(i+1)}_{*\xi+1,\beta}$,
where $\xi+1=\min((\alpha,\beta]^{\Tt})=\min((\alpha,\beta]^{\Tt\uparrow(i+1)}$.
And $\delta_i^{M^{\Tt\uparrow(i+1)}_\alpha}<\crit(j)$.
Also $E^{\Tt\uparrow i}_\beta=\emptyset$, so
\[ M^{\Tt\uparrow i}_{\beta+1}=M^{\Tt\uparrow i}_\beta=\Ult(M^{\Tt\uparrow i}_\alpha,E_j), \]
where $E_j$ is the full extender derived from $j$.
But 
\[ M^{\Tt\uparrow(i+1)}_{\beta+1}=\Ult(M^{\Tt\uparrow(i+1)}_\alpha,E^\Tt_\beta)=\Ult(M^{\Tt\uparrow(i+1)}_\alpha,F(M^{\Tt\uparrow(i+1)}_\beta)).\]
By a standard calculation, we get
\[ 
\vV_{i+1}^{M^{\Tt\uparrow i}_{\beta+1}}=i^{\Tt\uparrow i}_{\alpha\beta}(\vV_{i+1}^{M^{\Tt\uparrow i}_\alpha})=\Ult(\vV_{i+1}^{M^{\Tt\uparrow i}_\alpha},E_j),\] 
but then considering line (\ref{eqn:Ult_is_next_vV_propagated}), therefore
\[\vV_{i+1}^{M^{\Tt\uparrow i}_{\beta+1}}=\Ult(\Ult(M^{\Tt\uparrow(i+1)}_\alpha,e_i^{M^{\Tt\uparrow(i+1)}_\alpha}),E_j).\]
But note that the extender derived
from the two step-iteration given by first applying
$e_i^{M^{\Tt\uparrow(i+1)}_\alpha}$
and then $E_j$,
is just  $F(M^{\Tt\uparrow(i+1)}_\beta)$, and so
\[\vV_{i+1}^{M^{\Tt\uparrow i}_{\beta+1}}=\Ult(M^{\Tt\uparrow(i+1)}_\alpha,E^{\Tt\uparrow(i+1)}_\beta)=M^{\Tt\uparrow(i+1)}_{\beta+1},\]
as required by the lemma.
\end{proof}

 \begin{dfn}
  Let $\vV$ be an $(n+1)$-pVl
  and $\Sigma$ a  $(0,\OR)$-strategy
  for $\vV$.
  We define a (not just short-normal) $(0,\OR)$-strategy
  for $\Sigma\down 0$ for $\vV\down 0$ as follows.
  
  If $\Uu$ on $\vV\down 0$
  is $(n+1)$-translatable-padded,
  then $\Uu$ is according to $\Sigma\down 0$ iff there is $\Uu'$ via $\Sigma$
  such that $(\Uu,\Uu')$ is a $(0,n+1)$-translation pair.
  
  We now extend this to include non-$(n+1)$-translatable trees $\Uu$. We will reduce such $\Uu$ to $(n+1)$-translatable trees $\Uu'$. We will have $\lh(\Uu')=\lh(\Uu)$.
  Let $\Uu\rest(\eta_0+1)$ be the longest segment of $\Uu$ which is $(n+1)$-translatable.
  Then we will have $\Uu'\rest(\eta_0+1)=\Uu\rest(\eta_0+1)$ (and via $\Sigma\down 0$).
Suppose $E^\Uu_{\eta_0}$
results in a non-$(n+1)$-translatable tree $\Uu\rest(\eta_0+2)$. 
Let $N=M^{\Uu}_{\eta_0}$.
There is $m\leq n$ such that $\kappa=\kappa_m^{N}=\delta(\Uu\rest\eta_0)$
  and there are cofinally many $\beta+1<^\Tt\eta_0$ such that $\crit(E^\Tt_\beta)=\kappa_m^{M^\Tt_{\pred^\Tt(\beta+1)}}$. And $E=E^\Uu_{\eta_0}\in\es^N$ is such that $\kappa<\lh(E)<\kappa^{+N}$. 
  Let $P\pins N$
  be least such that $\lh(E)\leq\OR^P$
  and $\rho_\om^P=\kappa$.
  Let $D\in\es^{N}$
  be the total order $0$ measure on $\kappa$. Let $j_0:N\to\Ult(N,D)$
  be the ultrapower map.
  We form $\Uu\rest[\eta_0,\eta_1]$
  (for some yet to be specified $\eta_1$)
  by lifting the interval $\Uu\rest[\eta_0,\eta_1]$
  for critical points $\geq\kappa$,
 with $j_0$, to models of the $(n+1)$-translatable tree $\Uu'\rest(\eta_1+1)$.
(In more detail,
define the phalanx \[\ph=((\Phi(\Uu\rest(\eta_0+1)),{<\kappa}),P).\]
We lift
$\mathfrak{P}$ to \[\mathfrak{P}'=((\Phi(\Uu\rest(\eta_0+1)),{<j_0(\kappa)}),j_0(P)),\]
using as lifting maps $\id:M^\Uu_\alpha\to M^\Uu_\alpha$ for all $\alpha\leq\eta_0$
(including $\id:M^\Uu_{\eta_0}\to M^\Uu_{\eta_0}$, but this is only relevant to critical points ${<\kappa}$),
and $j_0\rest P:P\to j_0(P)$
(for critical points $\geq\kappa$).)
Note that we don't actually use $D$ in $\Uu$ or $\Uu'$.
Note that $\Uu\rest(\eta+1)\conc\left<j_0(E^\Uu_{\eta_0})\right>$
in particular is $(n+1)$-translatable, so we have made some progress. We continue producing $\Uu\rest(\eta_1+1)$ in this way as long as the resulting tree $\Uu'\rest(\eta_1+1)$
is $(n+1)$-translatable (and via $\Sigma\down 0$).

So this either  completes $\Uu$,
or we reach $\eta_1$ such that $\Uu\rest(\eta_1+1)$ and $\Uu'\rest(\eta_1+1)$ are as above
but $E^\Tt_{\eta_1}$ is such that if we continued as above, then $\Uu'\rest(\eta_1+2)$ would be non-$(n+1)$-translatable. Suppose the latter holds. Note here that
$\Uu\rest[\eta_0,\eta_1]$ is not above $\kappa$, so $\Uu\rest[\eta_0,\eta_1]$ uses some $G$ with $\crit(G)=\kappa_{i}^N$ for some $i<m$ (as $\kappa$ is a $\{\kappa_0^N,\ldots,\kappa_{m-1}^N\}$-cutpoint of $P$). So $\eta_0\not\leq^{\Uu}\eta_1$.
Let $\pi_{\eta_1}:M^\Uu_{\eta_1}\to M^{\Uu'}_{\eta_1}$ be the lifting map resulting from the process described so far (so $\crit(\pi_{\eta_1})=\kappa$ etc).
Sincde $\pi_{\eta_1}(E^{\Uu}_{\eta_1})$ creates the first violation of $(n+1)$-translatability for $\Uu'$,
the same situation holds regarding $\Uu'\rest(\eta_1+1)$ as it did for $\Uu\rest(\eta_0+1)$ (though letting
 $m_1\leq n$ be such that $\kappa_1'=\kappa_{m_1}^{M^{\Uu'}_{\eta_1}}=\delta(\Uu'\rest\eta_1)$,
 we need not have $m_1=m$).
 By commutativity between iteration and copy maps, etc,
 $[0,\eta_1)^{\Uu}\cap\dropset^{\Uu}=\emptyset$ and  letting $\kappa_1=\kappa_{m_1}^{M^{\Uu}_{\eta_1}}$, we have $\kappa_1=\delta(\Uu\rest\eta_1)$
 and $\pi_{\eta_1}(\kappa_1)=\kappa_1'$.
 We have $\kappa_1<\lh(E^\Uu_{\eta_1})<\kappa_1^{+M^{\Uu}_{\eta_1}}$.
Let $D_1\in\es^{M^{\Uu'}_{\eta_1}}$ be the total order $0$ measure on $\kappa_1'$. Let \[j_1':M^{\Uu'}_{\eta_1}\to\Ult(M^{\Uu'}_{\eta_1},D_1) \]
be the ultrapower map
and let $j_1=j_1'\com\pi_{\eta_1}$.
Then we continue lifting $\Uu\rest[\eta_1,\eta_2]$
to $\Uu'\rest[\eta_1,\eta_2]$,
for a yet to be determined $\eta_2$,
using the earlier copy maps (produced up until stage $\eta_1$,
for lifting critical points ${<\kappa_1}$),
and $j_1\rest P_1$ (for lifting critical points $\geq\kappa_1$),
with details generalizing those like before. 

Now let $\lambda$ be a limit, and suppose we have defined $\Uu\rest\eta$ and $\Uu'\rest\eta$ where $\eta=\sup_{\alpha<\lambda}\eta_\alpha$, and also copy maps continuing those above.  The tree structures of $\Uu\rest\eta$ and $\Uu'\rest\eta$ are identical, and so we can set $\Sigma_{\Uu\down 0}(\Uu\rest\eta)=\Sigma_{\Uu\down 0}(\Uu'\rest\eta)$. The iteration maps also commute with copy maps as usual, so we can set \[\pi_\eta:M^\Uu_\eta\to M^{\Uu'}_\eta \]
as the unique commuting map as usual,
and continue the process. (Actually, the structure of the branches at stage $\eta$ is very simple.
Note that for each $\alpha<\lambda$,
there is some $\gamma\in[\eta_\alpha,\eta_{\alpha+1})$ and some $m<m_\alpha$ such that $\crit(E^{\Uu}_{\gamma})=\kappa_{m}^{M^\Uu_{\eta_\alpha}}$. Let $m$ be least such that for cofinally many $\alpha<\lambda$, there is such a $\gamma\in[\eta_\alpha,\eta_{\alpha+1})$. This $m$ is defined in the same manner from $\Uu'\rest\sup_{\alpha<\lambda}\eta_\alpha$, by the commutativity. But then this $m$ easily determines a unique cofinal branch $b$ through $\Uu\rest\eta$, and it is also the unique cofinal branch through $\Uu'\rest\eta$, so this is the branch chosen at stage $\eta$ in both trees. Note also that $[0,\eta)^{\Uu}$ is non-dropping and $\kappa_m^{M^\Uu_\eta}=\delta(\Uu\rest\eta)$,
so possibly $\eta_\lambda=\eta$.)
The rest of the construction
proceeds in this manner.
 \end{dfn}
 The reader will now happily verify
 the following lemma, assisted by Lemma \ref{lem:n+1-translation_pair}:
\begin{lem}\label{lem:Sigma_down_0_is_strat}
   Let $\vV$ be an $(n+1)$-pVl
  and $\Sigma$ a  (not just short-normal) $(0,\OR)$-strategy
  for $\vV$. Then $\Sigma\down 0$
  is a  $(0,\OR)$-strategy
   for $\vV\down 0$.
\end{lem}

\begin{cor}\label{cor:Sigma_wW_n+1_down_0=Sigma_M}
For each $i\leq n$, we have:
\begin{enumerate}
 \item \label{item:Sigma_down_0=Sigma}
 $\Sigma_{\wW_{i+1}}\down 0=\Sigma_M$.
 \item \label{item:vV_i+1^N}For every non-dropping $\Sigma_M$-iterate $N$,
 $\vV_{i+1}^N$ is a non-dropping $\Sigma_{\wW_{i+1}}$-iterate.
 \end{enumerate}
\end{cor}
\begin{proof}Part \ref{item:Sigma_down_0=Sigma}:
Since $\Sigma_M$ is the unique  $(0,\OR)$-strategy
for $M$, this follows immediately from Lemma \ref{lem:Sigma_down_0_is_strat}.

Part \ref{item:vV_i+1^N}: This is an immediate corollary
of part \ref{item:Sigma_down_0=Sigma}
and Lemma \ref{lem:n+1-translation_pair}.
\end{proof}

\begin{lem}\label{lem:tree_on_vV_n+1_ev_sps}
 Let $\Tt$ be a successor length short-normal tree on an $(n+1)$-pVl $\vV$,
 based on $\vV|\delta_{n+1}^{\vV}$, and
 such that $b^{\Tt}$ does not drop.
 Then $\Tt\down 0$ is everywhere $(\delta,\kappa)$-stable.
 
In particular, $\Tt_{0,n+1}$ is an everywhere $(\delta,\kappa)$-stable, where $\Tt$ is the $\Sigma_{\wW_{n+1}}$-tree with last model $\vV_{n+1}$
and $\Tt_{0,n+1}=\Tt\down 0$ \tu{(}equivalently,
$\Tt_{0,n+1}$ is the $\Sigma_M$-tree with last model $\vV_{n+1\down 0}$\tu{)}.
\end{lem}

\begin{lem}
 Let $\Tt$ be a successor length short-normal tree on $\wW_{n+1}$, according to $\Sigma_{\wW_{n+1}}$,
 based on $\wW_{n+1}|\delta_{n+1}^{\wW_{n+1}}$,
 such that $b^{\Tt}$ does not drop.
 Let $\vV=M^\Tt_\infty$.
 Then $\Sigma_{\vV}$  is the unique
 self-consistent self-coherent $(0,\OR)$-strategy for $\vV$ with short-normal-mhc.
\end{lem}

Recall here that $\Sigma_{\vV}$ is defined from $\Sigma_{\wW_{n+1}}$ via normalization,
and by Lemma \ref{lem:Sigma_vV_is_e-vec_min_pullback},
its short-normal component  $(\Sigma_{\vV})^{\sn}$ is the $\vec{e}$-minimal pullback of $(\Sigma_{\vV\down 0})^{\sn}$.
\begin{proof}
We already know the strategy $\Sigma=\Sigma_{\vV}$
has these properties.
Now let $\Gamma$ be any such strategy, and suppose $\Sigma\neq\Gamma$. Then they differ in their restriction to short-normal trees, by Lemma \ref{lem:unique_ext_to_(0,OR)-strat_n+1}.
So let $\Tt$ be limit length short-normal, according to both $\Sigma,\Gamma$,
but with $\Sigma(\Tt)=b\neq c=\Gamma(\Tt)$.

Let $\Tt'\conc b'$ be the $\vec{e}$-minimal inflation of $\Tt\conc b$
and $\Tt'\conc c'$ that of $\Tt\conc c$; then $b',c'$ are $\Tt'$-cofinal,
but $b'\neq c'$. Also, $\Tt'\conc b'$ and $\Tt'\conc c'$
are according to $\Sigma,\Gamma$ respectively,
since both are self-coherent, self-consistent and have short-normal mhc (and the $\vec{e}$-extender is that derived from a short-normal tree). Then there are
$(n+1)$-translatable-padded trees $\Uu\conc b$ and $\Uu\conc c$ on $\vV\down 0$ such that
$(\Uu\conc b,\Tt'\conc b')$
and $(\Uu\conc c,\Tt'\conc c')$
are $(0,n+1)$-translation pairs.
But then by Lemma \ref{lem:Sigma_down_0_is_strat},
both $\Uu\conc b$ and $\Uu\conc c$ are according to (distinct) $(0,\OR)$-strategies for $\vV\down 0$.
But by Lemma
\ref{lem:tree_on_vV_n+1_ev_sps}, the $\Sigma_M$-tree leading from $M$ to $\vV\down 0$ is everywhere $(\delta,\kappa)$-stable, so by Lemma \ref{lem:everywhere_kappa^+-stable_unique_strategy},
there is in fact a unique $(0,\OR)$-strategy for $\vV\down 0$, a contradiction.
\end{proof}

   \subsection{$\delta_{n+1}$-short tree strategies}
   
We now compute the $\delta_{n+1}$-short tree strategies for $\vV_{n+1}$
and $\mM_{\infty n}$,
in $M$ and in other models.
We also compute correct branch models $M^\Tt_b$ for $\delta_{n+1}$-maximal
trees $\Tt$, in the right context.

\begin{dfn}
 For $k\leq\ell\leq n+1$, let $\Tt_{k\ell}$
 be the tree on $\vV_k$, leading from $\vV_k$ to $\vV_{\ell\down k}$. Let $j_{k\ell}:\vV_k\to\vV_{\ell\down k}$ be the iteration map.
\end{dfn}

\begin{dfn}
 Let $\vV$ be an $(n+1)$-pVl and $\Tt$ a short-normal tree on $\vV$. Let $\Tt_0,\Tt_1$ be such that $\Tt=\Tt_0\conc\Tt_1$, where $\Tt_0$ is based on $\vV|\delta_n^{\vV}$, and either (i) $\Tt_1=\emptyset$ (so $\Tt=\Tt_0$), or (ii) $\Tt_0$ has successor length, $b^{\Tt_0}$ does not drop, and $\Tt_1$ is above $\delta_n^{M^{\Tt_0}_\infty}$. Then we say that $\Tt_0$ is the \emph{lower component} and $\Tt_1$ the \emph{upper component} of $\Tt$.
\end{dfn}

The following definition is essentially
as in \cite[***Definitions 5.3, 5.12]{vm2_v2},
though we don't presently restrict to \emph{dsr} trees (we handle the latter a little differently):

\begin{dfn}\label{dfn:gamma-stable-P-suitable}
Let $\PP\in M$ and $g\sub\PP$ be $M$-generic. Working in $M[g]$, we say
 that a tree $\Tt$
  on $\vV_{n+1}$ is \emph{$\vec{\gamma}$-stable-$*$-suitable} if there are $E,F,U,\eta,\delta,\lambda,\Tt_0,\Tt_1$ such that:
 \begin{enumerate}
  \item\label{item:Tt_in_terms_of_Tt_0,Tt_1} $\Tt$ is short-normal $\vec{\gamma}$-stable,
   based on $\vV_{n+1}|\delta_{n+1}^{\vV_{n+1}}$,
   with lower and upper components
   $\Tt_0,\Tt_1$ respectively,
   with $\Tt_1$ of limit length,
  \item\label{item:upper_extender_F} either:
  \begin{enumerate}
   \item $F=\emptyset$ and $U=M$ and $\lambda=\kappa_{n+1}^M$, or
   \item $F\in\es^M$ is $M$-total with $\crit(F)=\kappa_{n+1}^M$
   and $U=\Ult(M,F)$ and $\lambda=\lambda(E)=\kappa_{n+1}^U$,
  \end{enumerate}
  
  \item\label{item:what_E_is_like} either:
  \begin{enumerate}
  \item $E=\emptyset$ and $\Tt_0$ is trivial, or
  \item\label{item:E_is_M-total} $E\in\es^M$ is $M$-total
  with $\crit(E)=\kappa_n^M$,
  and $\Tt_0$ is the (successor-length) tree induced by $E$,
  \end{enumerate}
  so in any case, $M^{\Tt_0}_\infty=i^M_E(\vV_{n+1})=\vV_{n+1}^{\Ult(M,E)}$,
  \item\label{item:eta,delta,lambda} $\eta$ is a strong $\{\kappa_0^M,\ldots,\kappa_n^{M}\}$-cutpoint of $U$ and a cardinal of $U$ and is the largest cardinal of $U|\delta$,
   $\Delta^{M^{\Tt_0}_\infty}<\eta<\delta=\delta(\Tt_1)<\lambda$ (and note $\Delta^{M^{\Tt_0}_\infty}=\lh(E)$ if $E\neq\emptyset$),
  and $\PP\in U|\eta$ 
  \item\label{item:U|delta_is_generic}$(U|\delta,g)=(M|\delta,g)$ is extender algebra generic over $M(\Tt_1)$
  for the extender algebra at $\delta$, and $\Tt_1$ is definable from parameters over $(U|\delta)[g]$.
 \end{enumerate}
 
(The reader should probably ignore the following paragraph until after
going through the proof of Lemma \ref{lem:*-suitable_P-con_correctness_no_short_overlaps} for the case that $\Tt\in M$.)
Working instead in $\vV_{i+1}[g]$
where $i\leq n$
and $g\sub\PP\in\vV_{i+1}$
is $\vV_{i+1}$-generic,
we say that a tree $\Tt$ on $\vV_{n+1}$
(where $\Tt\in\vV_{i+1}[g]$)
is \emph{$\vec{\gamma}$-stable-$*$-suitable}
if there are $E,F,\ldots,\Tt_1$ as as before, except that we  make the following modifications:
\begin{enumerate}[label=--]
 \item 
we replace ``$M$'' throughout
with ``$\vV_{i+1}$'',
\item 
 in clause \ref{item:what_E_is_like}(\ref{item:E_is_M-total}),
if $i=n$, we replace ``$\crit(E)=\kappa_n^M$''
with ``$E$ is $n$-long and $\gamma^{\vV_{n+1}}<\lh(E)$'', and
\item in clause \ref{item:eta,delta,lambda},
we replace ``$\eta$ is a strong $\{\kappa_0^M,\ldots,\kappa_n^M\}$-cutpoint of $U$''
with ``$\eta$ is a strong $\{\kappa_{i+1},\ldots,\kappa_n\}$-cutpoint with respect to short extenders''
(that is, if $E\in\es^{\vV_{i+1}}$
 is short
and $\crit(E)\leq\eta<\lh(E)$
then $\crit(E)\in\{\kappa_{i+1},\ldots,\kappa_n\}$; recall that (partially by notational convention) $\kappa_{j}^{\vV_{i+1}}=\kappa_j^M$ for $j\in[i+1,\om)$, and $\kappa_{i+1}^{\vV_{i+1}}=\kappa_{i+1}^M$ is the least
strong cardinal of $\vV_{i+1}$).
\end{enumerate}
For ease of reference we also record the resulting conditions in full:
 \begin{enumerate}[label=\arabic*'.,ref=\arabic*']
  \item\label{item:Tt_in_terms_of_Tt_0,Tt_1_vV} $\Tt$ is short-normal $\vec{\gamma}$-stable,
   based on $\vV_{n+1}|\delta_{n+1}^{\vV_{n+1}}$,
   with lower and upper components
   $\Tt_0,\Tt_1$ respectively,
   with $\Tt_1$ of limit length,
  \item\label{item:upper_extender_F_vV} either:
  \begin{enumerate}
   \item $F=\emptyset$ and $U=\vV_{i+1}$ and $\lambda=\kappa_{n+1}^{\vV_{i+1}}$, or
   \item $F\in\es^{\vV_{i+1}}$ is $\vV_{i+1}$-total with $\crit(F)=\kappa_{n+1}^{\vV_{i+1}}$ (hence $F$ is short)
   and $U=\Ult(\vV_{i+1},F)$ and $\lambda=\lambda(E)=\kappa_{n+1}^U$,
  \end{enumerate}
  
  \item\label{item:what_E_is_like_vV} either:
  \begin{enumerate}
  \item $E=\emptyset$ and $\Tt_0$ is trivial, or
  \item\label{item:E_is_M-total_vV} we have:
  \begin{enumerate}[label=--]
  \item
  if $i<n$ then
  $E\in\es^{\vV_{i+1}}$ is $\vV_{i+1}$-total with $\crit(E)=\kappa_n^{\vV_{i+1}}=\kappa_n^M$,
  \item if $i=n$ then $E\in\es^{\vV_{n+1}}$ 
   is $n$-long and $\gamma^{\vV_{n+1}}<\lh(E)$
   \end{enumerate}
  and $\Tt_0$ is the (successor-length) tree (on $\vV_{n+1}$)  induced by $E$,
  \end{enumerate}
  so in any case, $M^{\Tt_0}_\infty=i^{\vV_{i+1}}_E(\vV_{n+1})=\vV_{n+1}^{\Ult(\vV_{i+1},E)}$,
  \item\label{item:eta,delta,lambda_vV} $\eta$ is a strong $\{\kappa_{i+1}^{\vV_{i+1}},\ldots,\kappa_n^{\vV_{i+1}}\}$-cutpoint with respect to short extenders, a cardinal of $U$ and the largest cardinal of $U|\delta$,
   $\Delta^{M^{\Tt_0}_\infty}<\eta<\delta=\delta(\Tt_1)<\lambda$ (and note $\Delta^{M^{\Tt_0}_\infty}=\lh(E)$ if $E\neq\emptyset$),
  and $\PP\in U|\eta$ 
  \item\label{item:U|delta_is_generic_vV}$(U|\delta,g)=(\vV_{i+1}|\delta,g)$ is extender algebra generic over $M(\Tt_1)$
  for the extender algebra at $\delta$, and $\Tt_1$ is definable from parameters over $(U|\delta)[g]$.\qedhere
 \end{enumerate}
\end{dfn}

\begin{dfn}\label{dfn:gamma-unstable-P-suitable}
Let $\PP\in M$ and $g\sub\PP$ be $M$-generic.
 Working in $M[g]$, we say
 that a tree $\Tt$
  on $\vV_{n+1}$ is \emph{$\gamma_n$-unstable-$*$-suitable} if there are $E,F,U,\eta,\delta,\lambda,\Tt_0,\Tt_1$ such that:
 \begin{enumerate}
  \item \label{item:Tt_in_terms_of_Tt_0,Tt_1_unstable}\footnote{As in part \ref{item:Tt_in_terms_of_Tt_0,Tt_1} of \ref{dfn:gamma-stable-P-suitable},
  except that ``$\vec{\gamma}$-stable'' is replaced by ``$\gamma_n$-unstable''.}
   $\Tt$ is short-normal \underline{$\gamma_n$-unstable},
   based on $\vV_{n+1}|\delta_{n+1}^{\vV_{n+1}}$,
   with lower and upper components
   $\Tt_0,\Tt_1$ respectively,
   with $\Tt_1$ of limit length,
  \item \footnote{Identical to part \ref{item:upper_extender_F}
  of \ref{dfn:gamma-stable-P-suitable}.} either:
  \begin{enumerate}
   \item $F=\emptyset$ and $U=M$ and $\lambda=\kappa_{n+1}^M$, or
   \item $F\in\es^M$ is $M$-total with $\crit(F)=\kappa_{n+1}^M$
   and $U=\Ult(M,F)$ and $\lambda=\lambda(E)=\kappa_{n+1}^U$,
  \end{enumerate}
  \item\label{item:what_E_is_like_gamma-unstable} \footnote{This is quite different to part \ref{item:what_E_is_like} of \ref{dfn:gamma-stable-P-suitable}.}$\Tt_0$ is via $\Sigma_{\vV_{n+1}}$, and $\Tt_0=\Tt_{00}\conc\Tt_{01}$ where either:
\begin{enumerate}
  \item $n=0$, $E=\emptyset$, $\Tt_{00}$ is trivial (so $\Tt_0=\Tt_{01}$), or
  \item \label{item:T_0_as_concat_of_T_00,T_01}$n>0$, $\Tt_{00}$ is based on $\vV_{n+1}|\delta_{n-1}^{\vV_{n+1}}$, $b^{\Tt_{00}}$ is non-dropping, 
  $E\in\es^M$ is $M$-total
  with $\crit(E)=\kappa_{n-1}^M<\kappa_n^M<\kappa_n^{+M}<\lh(E)$, 
  $\Tt_{00}$ is  the (successor length) tree on $\vV_{n+1}$ which iterates  $\vV_{n+1}|\delta_{n-1}^{\vV_{n+1}}$ out
  to $\vV_n^{\Ult(M,E)}|\delta_{n-1}^{\vV_n^{\Ult(M,E)}}$,
  and $\Tt_{01}$ is above $\delta_{n-1}^{\vV_n^{\Ult(M,E)}}$
  (hence above $\gamma_{n-1}^{\vV_n^{\Ult(M,E)}}$),
  \end{enumerate}
  \item \footnote{Identical to part \ref{item:eta,delta,lambda}
  of \ref{dfn:gamma-stable-P-suitable}.} $\eta$ is a strong $\{\kappa_0^M,\ldots,\kappa_n^{M}\}$-cutpoint of $U$ and a cardinal of $U$ and is the largest cardinal of $U|\delta$,
   $\Delta^{M^{\Tt_0}_\infty}<\eta<\delta=\delta(\Tt_1)<\lambda$,
  and $\PP\in U|\eta$ 
\item \footnote{Identical to part \ref{item:U|delta_is_generic}
  of \ref{dfn:gamma-stable-P-suitable}.} $(U|\delta,g)=(M|\delta,g)$ is extender algebra generic over $M(\Tt_1)$
  for the extender algebra at $\delta$, and $\Tt_1$ is definable from parameters over $(U|\delta)[g]$.
 \end{enumerate}
(The reader should probably ignore the following paragraph until after
going through the proof of Lemma \ref{lem:*-suitable_P-con_correctness_no_short_overlaps} for the case that $\Tt\in M$.)
Working instead in $\vV_{i+1}[g]$
where $i\leq n$
and $g\sub\PP\in\vV_{i+1}$ is $\vV_{i+1}$-generic,
we say that a tree $\Tt$ on $\vV_{n+1}$
(so now $\Tt\in\vV_{i+1}$)
is \emph{$\gamma_n$-unstable-$*$-suitable}
if likewise, but making changes as for the $\vV_{i+1}[g]$ version of \emph{$\vec{\gamma}$-stable-$*$-suitable},
except that in clause \ref{item:what_E_is_like_gamma-unstable}(\ref{item:T_0_as_concat_of_T_00,T_01}),
if $i=n$, changing ``$\crit(E)=\kappa_{n-1}^M<\kappa_n^{+M}<\lh(E)$'' to ``$E$ is $(n-1)$-long and $\gamma^{\vV_{n+1}}<\lh(E)$'',
and if $i=n-1$, changing
``$\crit(E)=\kappa_{n-1}^M<\kappa_n^{+M}<\lh(E)$'' to ``$E$ is $(n-1)$-long and $\kappa_n^{+\vV_n}<\lh(E)$''
.

 \begin{enumerate}[label=\arabic*'.]
  \item\label{item:Tt_in_terms_of_Tt_0,Tt_1_vV_unstable}\footnote{Identical to part \ref{item:Tt_in_terms_of_Tt_0,Tt_1_unstable} of \ref{dfn:gamma-unstable-P-suitable}.} $\Tt$ is short-normal \underline{$\gamma_n$-unstable},
   based on $\vV_{n+1}|\delta_{n+1}^{\vV_{n+1}}$,
   with lower and upper components
   $\Tt_0,\Tt_1$ respectively,
   with $\Tt_1$ of limit length,
  \item\label{item:upper_extender_F_vV_unstable} \footnote{Identical to part \ref{item:upper_extender_F_vV} of \ref{dfn:gamma-stable-P-suitable}.} either:
  \begin{enumerate}
   \item $F=\emptyset$ and $U=\vV_{i+1}$ and $\lambda=\kappa_{n+1}^{\vV_{i+1}}$, or
   \item $F\in\es^{\vV_{i+1}}$ is $\vV_{i+1}$-total with $\crit(F)=\kappa_{n+1}^{\vV_{i+1}}$ (hence $F$ is short)
   and $U=\Ult(\vV_{i+1},F)$ and $\lambda=\lambda(E)=\kappa_{n+1}^U$,
  \end{enumerate}
  
  \item\label{item:what_E_is_like_vV_unstable}  $\Tt_0$ is via $\Sigma_{\vV_{n+1}}$, and $\Tt_0=\Tt_{00}\conc\Tt_{01}$ where either:
\begin{enumerate}
  \item $n=0$, $E=\emptyset$, $\Tt_{00}$ is trivial (so $\Tt_0=\Tt_{01}$), or
  \item \label{item:T_0_as_concat_of_T_00,T_01_vV_unstable}$n>0$, $\Tt_{00}$ is based on $\vV_{n+1}|\delta_{n-1}^{\vV_{n+1}}$, $b^{\Tt_{00}}$ is non-dropping, 
  $E\in\es^{\vV_{i+1}}$ is $\vV_{i+1}$-total,  either:
  \begin{enumerate}
\item $i<n-1$ and $E$ is short with \[ \crit(E)=\kappa_{n-1}^{\vV_{i+1}}=\kappa_{n-1}^M<\kappa_n^{\vV_{i+1}}<\kappa_n^{+\vV_{i+1}}<\lh(E),\]
or
\item $i=n-1$, $E$ is $(n-1)$-long
and $\kappa_n^{+\vV_n}=\kappa_n^{+M}<\lh(E)$, or
\item $i=n$, $E$ is $(n-1)$-long
and $\gamma^{\vV_{n+1}}<\lh(E)$, 
\end{enumerate}  
and   $\Tt_{00}$ is  the (successor length) tree on $\vV_{n+1}$ which iterates  $\vV_{n+1}|\delta_{n-1}^{\vV_{n+1}}$ out
  to $\vV_n^{\Ult(M,E)}|\Delta^{\vV_n^{\Ult(M,E)}}$,
  so \[ i^{\Tt_{00}}_{0\infty}\com j_{n,n+1}\rest(\vV_n|\Delta^{\vV_n})=i^{\vV_n}_E\rest(\vV_n|\Delta^{\vV_n}),\]
  \end{enumerate}
  and $\Tt_{01}$ is above $\Delta^{\vV_n^{\Ult(\vV_{i+1},E)}}$
  (hence above $\gamma^{\vV_n^{\Ult(\vV_{i+1},E)}}$),
  \item\label{item:eta,delta,lambda_vV_unstable}\footnote{Identical to \ref{item:eta,delta,lambda_vV} of \ref{dfn:gamma-stable-P-suitable}.} $\eta$ is a strong $\{\kappa_{i+1}^{\vV_{i+1}},\ldots,\kappa_n^{\vV_{i+1}}\}$-cutpoint with respect to short extenders, a cardinal of $U$ and the largest cardinal of $U|\delta$,
   $\Delta^{M^{\Tt_0}_\infty}<\eta<\delta=\delta(\Tt_1)<\lambda$,
  and $\PP\in U|\eta$ 
  \item\label{item:U|delta_is_generic_vV_unstable}\footnote{Identical to \ref{item:U|delta_is_generic_vV} of \ref{dfn:gamma-stable-P-suitable}.}$(U|\delta,g)=(\vV_{i+1}|\delta,g)$ is extender algebra generic over $M(\Tt_1)$
  for the extender algebra at $\delta$, and $\Tt_1$ is definable from parameters over $(U|\delta)[g]$.\qedhere
 \end{enumerate}
\end{dfn}

\begin{dfn}\label{dfn:P-suitable}
 Working in $\vV_i[g]$ where $i\leq n+1$ and $g$ is set-generic over $\vV_i$, we say that a tree $\Tt$ on $\vV_{n+1}$ is \emph{$*$-suitable}
 iff $\Tt$ is either $\gamma$-stable-$*$-suitable or $\gamma$-unstable-$*$-suitable.
\end{dfn}

\begin{rem}
 The next two definitions will only be sufficient in the case that the desired Q-structure or branch model does not involve short extenders which overlap $\delta$. The full generalization will involve inverse $*$-translation.
\end{rem}

Adapting \cite[Definition 5.4***]{vm2_v2}:
\begin{dfn}\label{dfn:P-con_for_gamma-stable-P-suitable}
 Let $\Tt\in M[g]$, where $g$ is set-generic over $M$, with $\Tt$ on $\vV_{n+1}$, be $\vec{\gamma}$-stable-$*$-suitable,
 and adopt the notation of Definition \ref{dfn:gamma-stable-P-suitable}.
 The \emph{P-construction}
 $\mathscr{P}^U(M(\Tt))$ of $U$ over $M(\Tt)$ is the structure $P$
 extending $M(\Tt)$,
 where for each $\nu\in[\delta,\OR^P]$:
 \begin{enumerate}
  \item \label{item:P|nu_active_iff_U|nu_active}
$P|\nu$ is active
 iff $U|\nu$ is active, and
 \item\label{item:P-con_F^P_from_F^M}
if $U|\nu$ is active,
then letting $G=F^{U|\nu}$:
\begin{enumerate}
\item if $\delta\leq\crit(G)$
 (hence $\delta<\crit(G)$)
 then $F^{P|\nu}\sub G$,
 \item\label{item:P-con_long_extender_commutings}if $\crit(G)<\delta$,
 hence $\crit(G)=\kappa_i^M=\kappa_i^U$
 for some $i\leq n$,
then $F^{P|\nu}$ is $i$-long,
and:
\begin{enumerate}
\item if $i=n$ then 
$F^{P|\nu}\circ E\rest(\vV_{n+1}|\Delta^{\vV_{n+1}})=G\rest(\vV_{n+1}|\Delta^{\vV_{n+1}})$, and
\item if $i<n$ then
$F^{P|\nu}\circ E\circ j_{i+1,n+1}\rest(\vV_{i+1}|\Delta^{\vV_{i+1}})=G\rest(\vV_{i+1}|\Delta^{\vV_{i+1}})$,
\end{enumerate}
\end{enumerate}
 \item if $\nu<\OR^P$ then $P|\nu$ is fully sound,
\item $P|\nu\sats$ ``$\delta$ is Woodin'', and
\item\label{item:OR^P_large_as_possible} $\OR^P$ is as large as possible under these conditions.
\end{enumerate}

Working instead in $\vV_{i+1}[g]$, where $i\leq n$ and $g$ is set-generic over $\vV_{i+1}$,
let $\Tt$ on $\vV_{n+1}$ be $\vec{\gamma}$-stable-$*$-suitable,
and adopt the notation as before.
The \emph{P-construction $\mathscr{P}^U(M(\Tt))$
of $U$ over $M(\Tt)$} is defined as above,
except that clause \ref{item:P-con_F^P_from_F^M}(\ref{item:P-con_long_extender_commutings}) is replaced with the following:
 \begin{enumerate}
  \setcounter{enumi}{1}
 \item\label{item:P-con_F^P_from_F^M_i+1}
\begin{enumerate}
 \setcounter{enumii}{1}
 \item \label{item:P-con_long_extender_commutingsi+1}if $\crit(G)<\delta$,
 hence either $G$ is $i'$-long for some $i'\leq i$, or $G$ is short and
$\crit(G)=\kappa_{i'}^M$
 for some $i'\in[i+1,n]$,
then $F^{P|\nu}$ is $i'$-long,
and:
\begin{enumerate}
\item if $\max(i,i')=n$ then 
\[ F^{P|\nu}\circ E\rest(\vV_{n+1\down i'+1}|\Delta^{\vV_{n+1\down i'+1}})=G\rest(\vV_{n+1\down i'+1}|\Delta^{\vV_{n+1\down i'+1}}),\]
\item if $i<i'<n$ then
\[ F^{P|\nu}\circ E\circ j_{i'+1,n+1}\rest(\vV_{i'+1}|\Delta^{\vV_{i'+1}})=G\rest(\vV_{i'+1}|\Delta^{\vV_{i'+1}}),\]
\item if $i'\leq i<n$ then
\[ F^{P|\nu}\circ E\circ j_{i+1,n+1}\rest(\vV_{i+1\down i'+1}|\Delta^{\vV_{i+1\down i'+1
}})=G\rest(\vV_{i+1\down i'+1}|\Delta^{\vV_{i+1\down i'+1}}).\qedhere\]
\end{enumerate}
\end{enumerate}
\end{enumerate}
\end{dfn}

\begin{dfn}\label{dfn:P-con_for_gamma-unstable-P-suitable}
 Let $\Tt\in M[g]$, where $g$ is set-generic over $M$, $\Tt$ on $\vV_{n+1}$, be $\gamma_n$-unstable-$*$-suitable,
 and adopt the notation of Definition \ref{dfn:gamma-unstable-P-suitable}.
 The \emph{P-construction}
 $P=\mathscr{P}^U(M(\Tt))$ of $U$ over $M(\Tt)$ is the structure
 extending $M(\Tt)$,
 where for each $\nu\in[\delta,\OR^P]$,
 conditions \ref{item:P|nu_active_iff_U|nu_active}--\ref{item:OR^P_large_as_possible} of Definition \ref{dfn:P-con_for_gamma-stable-P-suitable} hold, except that clause \ref{item:P-con_F^P_from_F^M}(\ref{item:P-con_long_extender_commutings}) is replaced with the following:
 \begin{enumerate}
 \setcounter{enumi}{1}
 \item
\begin{enumerate}
 \setcounter{enumii}{1}
 \item if $\crit(G)<\delta$,
 hence $\crit(G)=\kappa_i^M=\kappa_i^U$
 for some $i\leq n$,
then $F^{P|\nu}$ is $i$-long,
and:
\begin{enumerate}
\item if $i=n$ then
$F^{P|\nu}\circ j\rest(\vV_{n+1}|\Delta^{\vV_{n+1}})=G\rest(\vV_{n+1}|\Delta^{\vV_{n+1}})$
where $j:\vV_{n+1}\to M^{\Tt_0}_\infty$ 
is the iteration map,
\item if $i=n-1$ then 
$F^{P|\nu}\circ E\rest(\vV_{n}|\Delta^{\vV_{n}})=G\rest(\vV_{n}|\Delta^{\vV_{n}})$, and
\item if $i<n-1$ then
$F^{P|\nu}\circ E\circ j_{i+1,n}\rest(\vV_{i+1}|\Delta^{\vV_{i+1}})=G\rest(\vV_{i+1}|\Delta^{\vV_{i+1}})$.
\end{enumerate}
\end{enumerate}
\end{enumerate}

Working instead in $\vV_{i+1}[g]$, where $i\leq n$ and $g$ is set-generic over $\vV_{i+1}$, let $\Tt$ on $\vV_{n+1}$
be $\gamma_n$-unstable-$*$-suitable, and adopt the notation as before. The \emph{P-construction $\mathscr{P}^U(M(\Tt))$ of $U$ over $M(\Tt)$}
is defined as before,
except that clause \ref{item:P-con_F^P_from_F^M}(\ref{item:P-con_long_extender_commutings}) is replaced with the following:
 \begin{enumerate}
  \setcounter{enumi}{1}
 \item\label{item:P-con_F^P_from_F^M_i+1_2}
\begin{enumerate}
 \setcounter{enumii}{1}
 \item \label{item:P-con_long_extender_commutingsi+1_2}if $\crit(G)<\delta$,
 hence either $G$ is $i'$-long for some $i'\leq i$, or $G$ is short and
$\crit(G)=\kappa_{i'}^M$
 for some $i'\in[i+1,n]$,
then $F^{P|\nu}$ is $i'$-long,
and:
\begin{enumerate}
\item if $\max(i,i')=n$ then 
\[ F^{P|\nu}\circ j\rest(\vV_{n+1\down i'+1}|\Delta^{\vV_{n+1\down i'+1}})=G\rest(\vV_{n+1\down i'+1}|\Delta^{\vV_{n+1\down i'+1}})\]
where $j:\vV_{n+1}\to M^{\Tt_0}_\infty$ is the iteration map,
\item if $i<i'=n-1$ then
$F^{P|\nu}\circ E\rest(\vV_{n}|\Delta^{\vV_{n}})=G\rest(\vV_{n}|\Delta^{\vV_{n}})$, and
\item if $i<i'<n-1$ then
$F^{P|\nu}\circ E\circ j_{i'+1,n}\rest(\vV_{i'+1}|\Delta^{\vV_{i'+1}})=G\rest(\vV_{i'+1}|\Delta^{\vV_{i'+1}})$.
\item if $i'\leq i=n-1$ then
\[ F^{P|\nu}\circ E\rest(\vV_{n\down i'+1}|\Delta^{\vV_{n\down i'+1
}})=G\rest(\vV_{n\down i'+1}|\Delta^{\vV_{n\down i'+1}}).\]
\item if $i'\leq i<n-1$ then
\[ F^{P|\nu}\circ E\circ j_{i+1,n}\rest(\vV_{i+1\down i'+1}|\Delta^{\vV_{i+1\down i'+1}})=G\rest(\vV_{i+1\down i'+1}|\Delta^{\vV_{i+1\down i'+1}}).\qedhere\]
\end{enumerate}
\end{enumerate}
\end{enumerate}
\end{dfn}

\begin{dfn}
 Let $N$ be a non-dropping $\Sigma_M$-iterate of $M$,
 and let $\vV$ be a non-dropping $\Sigma_{\wW_{i+1}}$-iterate of $\wW_{i+1}$, for some $i\leq n$.
 Then we define \emph{$\gamma$-stable-$*$-suitable} and \emph{$\gamma$-unstable-$*$-suitable}
 and \emph{$*$-suitable} in $N[g]$
 and $\vV[g']$ (for trees on $\vV_{n+1}^N$ or $\vV_{n+1}^{\vV}$, and for set-generics $g,g'$) just as in $M[h]$
 and $\vV_{i+1}[h']$;
 likewise for the associated \emph{P-constructions}.
 (Note that these definitions are  first order over these models,
 so this makes sense. One important
 point is that, although at some points we may have referred to iteration maps on $\vV_{n+1}$, for example, we only actually needed to refer to restrictions of those maps which were in the models in question, and first order identifiable there. We use the same definitions for these strategy fragments in $N[g]$ or $\vV[g']$ as in $M[h]$ and $\vV_{i+1}[h']$.)
\end{dfn}
\begin{lem}\label{lem:strategy_agreement}
 Let $N$ be a non-dropping $\Sigma_M$-iterate.
Let $\vV=\vV_{n+1}^N$. Let $\Tt_{\wW_{n+1}\vV}$
be the tree via $\Sigma_{\wW_{n+1}}$ which leads from $\wW_{n+1}$ to $\vV$ (cf.~Corollary \ref{cor:Sigma_wW_n+1_down_0=Sigma_M}). 
 
 Let  $E,\Tt_0$ be as in
 the definition of \emph{$\vec{\gamma}$-stable-$*$-suitable} (\ref{dfn:gamma-stable-P-suitable}), but with $N$ replacing $M$ and $\vV$ replacing $\vV_{n+1}$.
 Let $\Tt$ be a putative short-normal tree on $\vV$ of the form $\Tt=\Tt_0\conc\Tt_1$, 
 and letting $\vV'=\Ult(\vV,E)$,
suppose $\Tt_1$ is above $\gamma_n^{\vV'}$.

Let $\Tt'$ be the translation of $\Tt$ to a tree on $N$,
so $\Tt'=\left<E\right>\conc\Tt'_1$
where $\Tt'_1$ has the same extender indices and tree structure as has $\Tt_1$. Then $\left<E\right>\conc\Tt'_1$ is via $\Sigma_N$ iff
 $\Tt$ is via $\Sigma_{\vV}$. 
\end{lem}
\begin{proof}
Recall that $\Sigma_N$ and $(\Sigma_{\vV})^{\sn}$
are defined via normalization from $\Sigma_M$ and $(\Sigma_{\wW_{n+1}})^{\sn}$ respectively.
But by Corollary \ref{cor:Sigma_wW_n+1_down_0=Sigma_M},
$\Sigma_M=\Sigma_{\wW_{n+1}}\down 0$,
and it follows that the two  normalization processes work in agreement, which yields the lemma.
\end{proof}

\begin{lem}\label{lem:strategy_agreement_2_n=0}
Let $N$ be a non-dropping $\Sigma_M$-iterate,
via the tree $\Tt_N$. Let $\vV=\vV_1^N$.
So 
$\vV$ is a  $\Sigma_{\wW_1}$-iterate.
Let $\PP\in N$ and $g$ be $(N,\PP)$-generic. 
 Let $E,\Tt_{00},\Tt_{01},\Tt_0$
 be as in the definition of $\gamma_0$-unstable-$*$-suitable (\ref{dfn:gamma-unstable-P-suitable}), but with $N$ replacing $M$ and $\vV_{1}^N$ replacing $\vV_{1}$.
 So recall that in particular,
 \begin{enumerate}[label=--]
 \item   $E=\emptyset$ (as we are considering $\gamma_0$-instability),
  \item $\Tt_0=\Tt_{01}$ is short-normal on $\vV$ and $\Tt_{00}$ is trivial,
  \item $\Tt_{0}$ is based on $\vV_1|\delta_0^{\vV}$,  has successor length,
 and $b^{\Tt_{0}}$ does not drop.
 \end{enumerate}
 Suppose further that
$\Tt_0$ is via $\Sigma_{\vV}$,
and let $\Tt_1$ be a putative short-normal tree based on $M^{\Tt_0}_\infty|\gamma_0^{M^{\Tt_0}_\infty}$, and above $\delta_0^{M^{\Tt_0}_\infty}$.
Let
 $\Tt_0'\conc\Tt_1'$ be the putative tree
 $(\Tt_0\conc\Tt_1)\down 0$.
Then 
$\Tt_0\conc\Tt_1$ is via $\Sigma_{\vV}$
iff  $\Tt_0'\conc\Tt_1'$ is via $\Sigma_{\vV\down 0}$.
\end{lem}
\begin{proof}
 $\Sigma_{\vV}$ is given via normalization from $\Sigma_{\wW_1}$, and $\Sigma_{\vV\down 0}$ via normalization from $\Sigma_M$.
 And $\Sigma_{\wW_1}$ is equivalent to $\Sigma_{M}$ for trees based on $\wW_1|\gamma_0^{\wW_1}$ (respectively, on $M|\kappa_0^{+M}$). The tree $\Uu$ leading from $\wW_1$ to $\vV$ is of form $\Uu=\Uu_0\conc\Uu_1$ where $\Uu_0$ is based on $\wW_1|\delta_0^{\wW_1}$
 and $\Uu_1$ is on $M^{\Uu_0}_\infty$ and is above $\gamma_0^{M^{\Uu_0}_\infty}$.
 But $\Uu_1$ has no impact on the normalization process for trees on $\vV$ which are based on $\vV|\gamma_0^{\vV}=M^{\Uu_0}_\infty|\gamma_0^{M^{\Uu_0}_\infty}$, and likewise for trees on $\vV\down 0$. The lemma's conclusion follows directly from these considerations.
\end{proof}

\begin{lem}\label{lem:strategy_agreement_2}
Suppose $n>0$.
Let $N$ be a non-dropping $\Sigma_M$-iterate,
via the tree $\Tt_N$. 
Let $\PP\in N$ and $g$ be $(N,\PP)$-generic. 
 Let $E,\Tt_{00},\Tt_{01},\Tt_0$
 be as in the definition of $\gamma_n$-unstable-$*$-suitable (\ref{dfn:gamma-unstable-P-suitable}), but with $N$ replacing $M$ and $\vV_{n+1}^N$ replacing $\vV_{n+1}$.
 So recall that in particular,
 \begin{enumerate}[label=--]
 \item   $E\neq\emptyset$ (as $n>0$),
  \item $\Tt_0=\Tt_{00}\conc\Tt_{01}$ is short-normal on $\vV$,
 \item $\Tt_{00}$ is on $\vV$, based on $\vV|\delta_{n-1}^{\vV}$, of successor length, $b^{\Tt_{00}}$ does not drop, and $\Tt_{00}$ is determined by $E$ (but $i^{\Tt_{00}}_{0\infty}$ is just  a factor of $i_E$),
  \item $\Tt_{01}$ is based on $M^{\Tt_{00}}_\infty|\delta_n^{M^{\Tt_{00}}_\infty}$, is above $\gamma_{n-1}^{M^{\Tt_{00}}_\infty}$, has successor length,
 and $b^{\Tt_{01}}$ does not drop.
 \end{enumerate}
 Suppose further that:
 \begin{enumerate}[label=--]
  \item 
$\Tt_0$ is via $\Sigma_{\vV}$ (we already know that
$\Tt_{00}$ is via $\Sigma_{\vV}$),
 \item $\Tt_1$ is based on $M^{\Tt_0}_\infty|\gamma_n^{M^{\Tt_0}_\infty}$, and above $\delta_n^{M^{\Tt_0}_\infty}$.
 \end{enumerate}
  Let $\bar{N}$ be the 
$\kappa_{n-1}^N$-core of $N$.
So $\vV_n^{\bar{N}}$ is the $\delta_{n-1}^{\vV_n^N}$-core of $\vV_n^N$. Let $\bar{\Tt}_n$ be the $\Sigma_{\wW_n}$-tree with last model $\vV_n^{\bar{N}}$, so $\bar{\Tt}_n$ is based on $\wW_n|\delta_{n-1}^{\wW_n}$.  Let $\bar{\Tt}_{n\uparrow n+1}$ be the corresponding tree on $\wW_{n+1}$ and let $\vV^*=M^{\bar{\Tt}_{n\uparrow(n+1)}}_\infty$
(so
$\vV_n^{\bar{N}}=\vV^*\down n$).
Then:
\begin{enumerate}
\item\label{item:vV_n+1^N-bar=Ult(vV^*,e^vV^*)}
$\vV_{n+1}^{\bar{N}}=\Ult(\vV^*,e_n^{\vV^*})$.
\item\label{item:M^Tt_0_infty_is_Sigma_vV^*-it} $M^{\Tt_0}_\infty$
is a (non-dropping) $\Sigma_{\vV^*}$-iterate,
via a short-normal tree $\Tt^*_0=\Tt_{00}^*\conc\Tt_{01}^*\conc\Uu^*$ where:
\begin{enumerate}
\item $\Tt_{00}^*$ is the successor length tree
determined by $E$,
\item $\Tt_{01}^*$ is based on $M^{\Tt_{00}^*}_\infty|\delta_n^{M^{\Tt_{00}^*}_\infty}$ and above $\gamma_{n-1}^{M^{\Tt_{00}^*}_\infty}$,
\item $\Uu^*$ is above  $\gamma_n^{M^{\Tt_{01}^*}_\infty}$,
\item $\Uu^*$ is trivial iff $\vV$ is $\delta_n^{\vV}$-sound iff $N$ is $\kappa_{n+1}^N$-sound,
\end{enumerate}
 so $M^{\Tt_0}_\infty=M^{\Tt_0^*}_\infty$ and $M^{\Tt_0}_\infty\down n$
is a  $\Sigma_{\vV^*\down n}$-iterate,
via $\Tt^*\down n$,
\item\label{item:get_tree_on_M_via_Sigma_M}letting $\eta$ be least such that $E\in\es^{M^{\Tt_N}_\eta}$,
then there is a $\Sigma_M$-tree
 $\Tt'_0$ extending $\Tt_N\rest(\eta+1)$,
 such that $E^{\Tt'_0}_\eta=E$
 and the length, extender indices and tree structure
 of $\Tt'_0\rest(\eta,\infty)$
 are just those of $(\Tt_{01}^*\conc\Uu^*)\down n$
 (which are just those of $\Tt_{01}^*\conc\Uu^*$), 
 and $\vV_n^{M^{\Tt'_0}_\infty}=M^{\Tt^*_0\down n}_\infty=M^{\Tt^*_0}_\infty\down n=M^{\Tt_0}_\infty\down n$, and
 \item \label{item:Tt'_conc_b_via_Sigma_N}
 letting:
 \begin{enumerate}[label=--]
 \item $\Tt^*_1=\Tt_1$
 (so $\Tt_1^*$ is based on $M^{\Tt^*_0}_\infty|\gamma_n^{M^{\Tt_0^*}_\infty}$, above $\delta_n^{M^{\Tt_0^*}_\infty}$;
  note that $\Tt_1^*\down n$
 is a putative tree on $M^{\Tt_0^*\down n}_\infty$,
 based on $M^{\Tt_0^*\down n}_\infty|\kappa_n^{+M^{\Tt_0^*\down n}_\infty}$, 
 above $\delta_n^{M^{\Tt_0^*\down n}_\infty}$), and
 \item $\Tt_1'$ be the  putative
 tree on $M^{\Tt_0'}_\infty$ with the same
 length, extender indices and tree structure as has $\Tt_1^*\down n$ (which are the same as those of $\Tt_1^*$),
 \end{enumerate}
then the following are equivalent:
\begin{enumerate}[label=(\roman*)]
\item\label{item:(i)_tree_equiv}
$\Tt_1$ is via $\Sigma_{M^{\Tt_0}_\infty}$,
\item\label{item:(ii)_tree_equiv} $\Tt_1^*$ is via $\Sigma_{M^{\Tt_0^*}_\infty}$,
\item\label{item:(iii)_tree_equiv}  $\Tt_1^*\down n$ is via $\Sigma_{M^{\Tt_0^*\down n}_\infty}$,
\item\label{item:(iv)_tree_equiv}  $\Tt_1'$ is via $\Sigma_{M^{\Tt_0'}_\infty}$.\end{enumerate}
\end{enumerate}
\end{lem}
\begin{proof}
 Part \ref{item:vV_n+1^N-bar=Ult(vV^*,e^vV^*)}: 
We have $\vV_{n+1}^{\vV^*\down n}=\Ult(\vV^*,e_n^{\vV^*})$.
 But  $\vV^*\down n=\vV_n^{\bar{N}}$, so $\vV_{n+1}^{\vV^*\down n}=\vV_{n+1}^{\vV_n^{\bar{N}}}=\vV_{n+1}^{\bar{N}}$.
 
 Part \ref{item:M^Tt_0_infty_is_Sigma_vV^*-it}: 
  $\vV$ is a non-dropping
 $\Sigma_{\wW_{n+1}}$-iterate,
 and $\vV_n^N$ a non-dropping $\Sigma_{\wW_n}$-iterate, and the $\Sigma_{\wW_n}$-tree
 with last model $\vV^*\down n$
 is just the initial segment of the  $\Sigma_{\wW_n}$-tree from $\wW_n$ to $\vV_n^N$
which is based on $\wW_n|\delta_{n-1}^{\wW_n}$,
and this can be considered a tree on $\wW_{n+1}$,
which factors into the tree from $\wW_{n+1}$ to $\vV$.
Thus, $\vV$ is a $\Sigma_{\vV^*}$-iterate.
But $M^{\Tt_0}_\infty$ is itself a $\Sigma_{\vV}$-iterate, and since the segment of $\Tt_0$
based on $\vV|\delta_{n-1}^{\vV}$ just iterates
out to $i^N_E(\vV_n^N|\delta_{n-1}^{\vV_n^N})$,
the rest is clear.

Part \ref{item:get_tree_on_M_via_Sigma_M}:
By Corollary \ref{cor:Sigma_wW_n+1_down_0=Sigma_M}.

 Part \ref{item:Tt'_conc_b_via_Sigma_N}: We have $M^{\Tt_0}_\infty=M^{\Tt_0^*}_\infty$
 and $\Tt_1=\Tt_1^*$, so the equivalence of \ref{item:(i)_tree_equiv} and \ref{item:(ii)_tree_equiv} is just because $\Sigma_{M^{\Tt_0}_\infty}$
 depends only on the short-normal tree leading to $M^{\Tt_0}_\infty$. The equivalence of \ref{item:(ii)_tree_equiv} and \ref{item:(iii)_tree_equiv} is by conservativity,
 and that of \ref{item:(iii)_tree_equiv} and \ref{item:(iv)_tree_equiv} by Corollary  \ref{cor:Sigma_wW_n+1_down_0=Sigma_M}
 and that this yields corresponding normalization processes.\end{proof}

\begin{lem}\label{lem:*-suitable_P-con_correctness_no_short_overlaps}
Let $N$ be a non-dropping $\Sigma_M$-iterate
and $\vV=\vV_{n+1}^N$.
Let  $i\leq n+1$,
$g$ be set-generic over $\vV_i^N[g]$,
$\Tt\in\vV_{i}^N[g]$
be on $\vV_{n+1}^N$ and such that $\vV_i^N[g]\sats$``$\Tt$ is $*$-suitable'', 
and adopt the notation of \ref{dfn:P-con_for_gamma-stable-P-suitable} and \ref{dfn:P-con_for_gamma-unstable-P-suitable}.  Suppose that $\Tt=\Tt_0\conc\Tt_1$ is via $\Sigma_{\vV}$, and let $b=\Sigma_{\vV}(\Tt)$.
If $\Tt$ is $\delta_{n+1}$-short,
let $Q=Q(\Tt,b)$, and otherwise let $Q=M^\Tt_b$.
Then:
\begin{enumerate}
\item\label{item:delta(Tt)_overlapped_iff_reach_corresp_stage_in_P-con} The following are equivalent:
\begin{enumerate}
\item \label{item:short_overlapping_E_exists}there is a short extender $E\in\es_+^Q$ with $\crit(E)\leq\delta\leq\lh(E)$,
\item\label{item:projecting_stage_of_P-con_exists} 
$P=\mathscr{P}^U(M(\Tt))$ is unsound and there is $n<\om$
such that $P$ is $n$-sound and
 $\rho_{n+1}^P<\delta\leq\rho_n^P$
 and there is $\mu<\delta$
 such that for all $\gamma\in[\mu,\delta)$, $P$ does not have the $(n+1,\pvec_{n+1}^P)$-hull property at $\gamma$.
\end{enumerate}
\item\label{item:Q=P-con_if_delta_not_shortly_overlapped}  Suppose there is no short extender $E\in\es_+^Q$ with $\crit(E)\leq\delta\leq\lh(E)$.
Let $P=\mathscr{P}^U(M(\Tt))$.
Then:
\begin{enumerate}
\item\label{item:Q=P_when_set-sized} If $Q$ is set-sized then $Q=P$.
\item\label{item:Q-bar=P-bar_when_proper_class} Suppose $Q$ is proper class.
Then so is $P$,
and $Q,P$ have the same $\delta(\Tt)$-core.
\end{enumerate}
\end{enumerate}
\end{lem}
\begin{proof}
We will assume throughout for simplicity  of notation that $g=\emptyset$; the general case is essentially the same.

\begin{case}$i=0$ (so $\Tt\in N[g]\sats$``$\Tt$ is $*$-suitable'').

Part \ref{item:Q=P-con_if_delta_not_shortly_overlapped}: We first establish
parts \ref{item:Q=P_when_set-sized} and \ref{item:Q-bar=P-bar_when_proper_class},
via comparison arguments analogous to those in \cite{vm2_v2}.
 
 \begin{scase} $\Tt$ is $\vec{\gamma}~$-stable.
  
  Let $\Tt'$ be the translation of $\Tt$ to a tree on $N$, as in Lemma \ref{lem:strategy_agreement}.
  So $\Tt'=\Tt_0'\conc\Tt_1'$ where $\Tt_0'=\left<E\right>$ and $\Tt_1'$ is above $\kappa_n^{+\Ult(N,E)}$.
 By Lemma \ref{lem:strategy_agreement},
it suffices to work with $\Tt'$ instead of $\Tt$.
 Let $b'=\Sigma_N(\Tt')$.
 Let $Q'=Q(\Tt',b')$,
 if $\Tt'$ (equivalently $\Tt$) is $\delta_{n+1}$-short,
 and $Q'=M^{\Tt'}_{b'}$ otherwise. Let $\Phi(\Tt',Q')$ be the phalanx obtained from the phalanx $\Phi(\Tt')$
 by replacing the last model $M^{\Tt'}_{b'}$ with $Q'$. If $F=\emptyset$ then let $\Uu$ be the trivial tree on $N$, and otherwise let $\Uu$ be the tree on $N$ which uses only $F$.  We compare $\Phi(\Tt',Q')$ with the phalanx $\Phi(\Uu)$,
 modulo the generic at $\delta$
 and translating overlapping extenders appropriately.
That is, the comparison is the pair $(\Vv,\Ww)$ of trees, with $\Vv$ on $\Phi(\Tt',Q')$ and $\Ww$ on $\Phi(\Uu)$,
where given $(\Vv,\Ww)\rest(\alpha+1)$,
then letting $\gamma$ be the least ordinal
such that letting $G=F^{M^\Vv_\alpha|\gamma}$
and $H=F^{M^\Ww_\alpha|\gamma}$,
then $G\neq\emptyset$ or $H\neq\emptyset$,
but neither of the following two options holds:
\begin{enumerate}[label=--]
 \item 
  $G\rest\OR=H\rest\OR$,
  and either $\delta<\crit(G)$
  or $\crit(G)=\kappa_i^N$ for some $i<n$,\footnote{Recall here that $M^{\Vv}_\alpha$
  is an iterate of $N$, not of $\vV_{n+1}$,
  so we can indeed expect that $\crit(G)=\kappa_i^N$ (and $G$ being short), as opposed to being an $i$-long extender.} 
  \item $\crit(G)=\kappa_n^{\Ult(N,E)}$ and $\crit(H)=\kappa_n^N$ and $G\circ i^N_E\rest\OR=H\rest\OR$,
\end{enumerate}
then $E^\Vv_\alpha=G$ and $E^\Ww_\alpha=H$
(one of these extenders might be empty);
if there is no such $\gamma$ then the comparison terminates at this stage.

The same kinds of arguments used in \cite{vm2_v2} show that if $Q$ is set-sized then the comparison is trivial (i.e.~no extenders are used on either side)
and \ref{item:Q=P_when_set-sized} holds,
and that if $Q$ is proper class then no extenders with critical points $<\delta$ are used, there are no drops along $b^\Vv,b^\Ww$, and
part  \ref{item:Q-bar=P-bar_when_proper_class}
holds.
\end{scase}
\begin{scase}$\Tt$ is $\gamma_n$-unstable.

If $n=0$, things are easier, and as in \cite{vm2_v2}
(and Lemma \ref{lem:strategy_agreement_2_n=0} plays
the role that Lemma \ref{lem:strategy_agreement}
played in the previous case).

So suppose $n>0$, so Lemma \ref{lem:strategy_agreement_2} applies. Fix notation as there. By the lemma,
and since $\Tt=\Tt_0\conc\Tt_1$ is via $\Sigma_{\vV}$,
we know that ($\Tt_0'$ is via $\Sigma_M$ and) $\Tt_1'$ is via $\Sigma_{M^{\Tt_0'}_\infty}$, and letting $b=\Sigma_{\vV}(\Tt)$ and $b'=\Sigma_M(\Tt')$,
then $b$ is essentially the same as $b'$.
 Note that since $\Tt$ is $\gamma_n$-unstable
(on $\vV=\vV_{n+1}^N$),
certainly $Q$ is set-sized.
In fact $\gamma_n^Q$ exists then $\OR^Q=\gamma_n^Q$; moreover,
$\gamma_n^Q$ exists iff $b'$
is non-dropping.
 
We compare the ``phalanx'' \[\Phi(\Tt_1',(M^{\Tt_1'}_{b'},Q)),\]
with $\Phi(\Uu)$, where $\Uu$ is given by  $F$ as in the previous case. Here we ``iterate'' the ``phalanx'' $\Phi(\Tt_1',(M^{\Tt_1'}_{b'},Q))$ like usual, except that we have a bicephalus $(M^{\Tt_1'}_{b'},Q)$ as the last structure,
and we need to specify how this is handled (but this is just like in the analogous situation in \cite{vm2_v2}). We treat the bicephalus as though it is only $Q$, unless $b'$ does not drop
(so $\gamma_n^Q=\OR^Q$)
and until, if ever, 
we reach a stage $(\Vv,\Ww)\rest(\alpha+1)$ of the comparison at which
the rules of the comparison would lead us to use a long extender $G$. Say this happens at stage $\alpha$,
and $(N',Q')=M^{\Vv}_\alpha$, and said long extender has index $\xi$. Then at this stage,
we discard $Q'$  from the process and continue as though we had been iterating $\Phi(\Tt_1',b')$. In particular,
 if $G$ is $i$-long where $i<n$,
 then we next use the corresponding overlapping extender $G'$ with $\crit(G')=\kappa_i^{M^{\Tt_0'}_\infty}$ (here $G'=F^{N'|\xi}$);
if $G$ is  $n$-long (hence $\gamma_n^Q$ exists
and $b'$ is non-dropping and $\xi=\kappa_n^{+N'}$)
then we pad at stage $\alpha$ in $\Vv$,
and set $M^{\Vv}_{\alpha+1}=N'$, and then $Q$ and its images are irrelevant for the remainder of the comparison. Other than this, the rules for determining the least disagreement are much like those in the previous case,
determined by the commutativity of
the restrictions of overlapping extenders to the ordinals.

Like  in \cite{vm2_v2}, it follows that $Q$ is set-sized
and part
  \ref{item:Q=P_when_set-sized} holds.

\end{scase}

Part \ref{item:delta(Tt)_overlapped_iff_reach_corresp_stage_in_P-con}:
If there is no $E$ as in \ref{item:delta(Tt)_overlapped_iff_reach_corresp_stage_in_P-con}(\ref{item:short_overlapping_E_exists}),
then by part \ref{item:Q=P-con_if_delta_not_shortly_overlapped}, it is straightforward to see that \ref{item:delta(Tt)_overlapped_iff_reach_corresp_stage_in_P-con}(\ref{item:projecting_stage_of_P-con_exists}) fails. Conversely,
suppose that there is such an extender $E$. Let $\Tt^+=\Tt\conc\left<E\right>$ where
$E$ is the least such extender; note that $\crit(E)<\delta<\lh(E)$. Let $Q'=M^{\Tt^+}_\infty$. Note that
$b^{\Tt^+}$ drops in model,
and if $\Tt$ is $\gamma_n$-unstable-$*$-suitable, then $b^{\Tt^+}$
drops below the image of $\vV_{n+1}|\gamma^{\vV_{n+1}}$.
Arguing like in the proof of part \ref{item:Q=P-con_if_delta_not_shortly_overlapped},
$Q'=\mathscr{P}^U(M(\Tt))$, and 
therefore
\ref{item:delta(Tt)_overlapped_iff_reach_corresp_stage_in_P-con}(\ref{item:projecting_stage_of_P-con_exists}) holds.
\end{case}

\begin{case}\label{case:i>0_working_in_vV_i}
 $i>0$.
 
This case almost follows from the previous one, because $\vV_i^N\sub N$,
and because $\vV_i^N\sats$``$\Tt$ is $*$-suitable'', we almost have that $N\sats$``$\Tt$ is $*$-suitable''.
The only slight difference is that
we only demanded that $\vV_i^N|\delta$
be extender algebra generic over $M(\Tt_1)$, whereas we want that $N|\delta$ is extender algebra generic.
But since $N|\delta$ is given by a P-construction over $\vV_i^N|\delta$,
above the generic for a forcing in $\vV_i^N|\delta$, the same arguments still go through as before. This completes this case.\qedhere
\end{case}
\end{proof}

\begin{rem}
We now want to extend the lemma above to handle the case that $\vV$ is an arbitrary non-dropping $\Sigma_{\wW_{n+1}}$-iterate, and $\Tt\in\vV[g]\sats$``$\Tt$ is $*$-suitable (on me)'',
without assuming that $\vV=\vV_{n+1}^N$ for some $\Sigma_M$-iterate $M$.
 The comparison
arguments  in \cite{vm2_v2} used to verify correctness of the P-construction in such cases
make use of
generic premice playing 
a role analogous to that of $N$ in the previous lemma.
The author had trouble generalizing the use of generic premice
to the current context. Here we use another method, and we also allow iterates $\vV$ of $\wW_{n+1}$,
not just of $\vV_{n+1}$.
We will use the fact that we can lift from $\vV$ into $\vV_{n+1}^{\vV\down 0}$, and $\vV\down 0$ is a non-dropping $\Sigma_M$-iterate, to which the lemma above applies. But for this to work, we will need to assume $g=\emptyset$. Later we will deduce the $g\neq\emptyset$ case from the $g=\emptyset$ one.
\end{rem}

\begin{lem}\label{lem:*-suitable_P-con_correctness_no_short_overlaps_vV}
Let $\vV$ be a non-dropping $\Sigma_{\wW_{n+1}}$-iterate.
Let $\Tt\in\vV$
be on $\vV$ and such that $\vV\sats$``$\Tt$ is $*$-suitable'', 
and adopt the notation of \ref{dfn:P-con_for_gamma-stable-P-suitable} and \ref{dfn:P-con_for_gamma-unstable-P-suitable}. Suppose that $\Tt=\Tt_0\conc\Tt_1$ is via $\Sigma_{\vV}$, and let $b=\Sigma_{\vV}(\Tt)$.
If $\Tt$ is $\delta_{n+1}$-short,
let $Q=Q(\Tt,b)$, and otherwise let $Q=M^\Tt_b$.
Then:
\begin{enumerate}
\item Part \ref{item:delta(Tt)_overlapped_iff_reach_corresp_stage_in_P-con}
of \ref{lem:*-suitable_P-con_correctness_no_short_overlaps} holds.
\item\label{item:Q=P-con_if_delta_not_shortly_overlapped_no_N} Part \ref{item:Q=P-con_if_delta_not_shortly_overlapped} of \ref{lem:*-suitable_P-con_correctness_no_short_overlaps} holds.
\end{enumerate}
\end{lem}

\begin{proof}
Part \ref{item:Q=P-con_if_delta_not_shortly_overlapped_no_N}: We consider two cases.
 
 \begin{case}
$P$ is set-sized.

Let $b$ be the $\Tt$-cofinal branch that this determines (which it does by elementarity between $\vV_{n+1}$ and $\vV$). 
 
 Note that Lemma \ref{lem:*-suitable_P-con_correctness_no_short_overlaps}
 applies to $\vV'=\vV_{n+1}^{\vV\down 0}=\Ult(\vV,\vec{e}^{\vV})$.
 Let $j:\vV\to\vV'$ be the ultrapower map. Then by \ref{lem:*-suitable_P-con_correctness_no_short_overlaps}, $j(\Tt,b)$ is a correct tree on $\vV'$.
 
 Let $(\Uu,c)$ be the minimal $\vec{e}$-inflation of $(\Tt,b)$. Let $(\Uu',c)$ be the translation of this to a (padded) tree on $\vV\down 0$.
 
 Note that $(\Uu,c)\in\vV$,
 and $j(\Uu,c)$ is the minimal $\vec{e}$-inflation of $j(\Tt,b)$, and $j(\Uu',c)$ its translation to $\vV'\down 0$. Since $j(\Tt,b)$ is correct,
so is $j(\Uu',c)$, by Lemma \ref{lem:Sigma_vV_is_e-vec_min_pullback}.

Now using $j$,
we can lift iterations of the phalanx $\ph=\Phi(\Uu',c)$
to iterations of $\Phi(j(\Uu',c))$. Therefore $\ph$ is iterable. But the $\Sigma_M$-tree leading to $\vV\down 0$ is $(\delta_n,\kappa_n)$-stable, and $\Uu$ is of
the form $\Uu_0\conc\Uu_1$,
where $\Uu_0$ is $(\delta_n,\kappa_n)$-stable and $\Uu_1$ is either based on $(\delta_n,\kappa_n^+)$,
or based on $(\kappa_n^+,\delta_{n+1})$.  Therefore $\Sigma_{\vV\down 0}(\Uu)$
is the unique $c'$ such that $\Phi(\Uu,c')$ is iterable.
So $c=c'$, so $b$ is correct.
\end{case}

\begin{case} $P$ is proper class.

The proof in this case is essentially the same;
however, the branch $b$ might
not be in $\vV$. But one just uses the usual methods with indiscernibles to handle this; all of the proper initial segments of $b$ are in $\vV$,
and the various iteration maps are continuous at $\delta_{n+1}$, etc. Also the conclusion that $Q,P$ have the same $\delta(\Tt)$-core then follows from these considerations and computations with indiscernibles. We leave the details to the reader.
\end{case}

Part  \ref{item:delta(Tt)_overlapped_iff_reach_corresp_stage_in_P-con}: We leave this to the reader.
\end{proof}

\begin{dfn}
 Let $\vV$ be a non-dropping $\Sigma_{\wW_{n+1}}$-iterate.
 Then $\Sigma^{\sn}_{\vV,\delta_{n+1}\sss}$ denotes
 the restriction of $\Sigma_{\vV}^{\sn}$
 to $\delta_{n+1}$-short trees.
\end{dfn}

\begin{dfn}
 Let $\vV$ be a non-dropping $\Sigma_{\wW_{n+1}}$-iterate. A short-normal tree $\Tt$ on $\vV$ according to $\Sigma_\vV$ is called \emph{tame} if for every limit $\eta\leq\lh(\Tt)$, letting $b=\Sigma_{\vV}(\Tt\rest\eta)$, if $Q(\Tt,b)$ exists then there is no short $E\in\es_+^{Q(\Tt,b)}$ which overlaps $\delta(\Tt\rest\eta)$, i.e.~with $\crit(E)<\delta(\Tt)<\lh(E)$
 (hence there is no such $E$ with $\crit(E)\leq\delta(\Tt)\leq\lh(E)$).
\end{dfn}

\begin{rem}\label{rem:every_delta_n+1-max_tame_stable}
Note that if $\Tt$ is on $\vV$, based on $\vV|\delta_{n+1}^{\vV}$,  via $\Sigma_{\vV}$,
and either $\delta_{n+1}$-maximal
or of successor length with $b^\Tt$ non-dropping,
then $\Tt$ is $\vec{\gamma}$-stable and tame.
This includes all trees which are relevant to the next direct limit system.
\end{rem}

\begin{lem}\label{lem:delta_n+1-sts_def}  We have:
\begin{enumerate}
	\item\label{item:delta_n+1-sts_def_over_M-iterates_N}
	Let $N$ be a non-dropping $\Sigma_M$-iterate
	and $\vV=\vV_{n+1}^N$. Let $\Psi=\Sigma^{\sn}_{\vV,\delta_{n+1}\sss}$  and $\Psi_{\mathrm{tm}}$
	be the restriction of   $\Psi$ to   tame trees. Then:
\begin{enumerate}[label=\tu{(}\alph*\tu{)},ref=(\alph*)]	\item\label{item:delta_n+1-N_internal_sts} $N$ is closed under $\Psi_{\mathrm{tm}}$ and $\Psi_{\mathrm{tm}}\rest N$ is lightface definable over $N$ (hence $\dom(\Psi_{\mathrm{tm}}\rest N)$ is also lightface definable over $N$),
	\item\label{item:all_delta_n+1-max_absorbed} for each $\delta_{n+1}$-maximal tree $\Tt\in N$ via $\Psi$
	(by \ref{rem:every_delta_n+1-max_tame_stable},
	$\Tt$ is $\vec{\gamma}$-stable and tame) 
	there is a ($\vec{\gamma}$-stable)-$*$-suitable $\delta_{n+1}$-maximal  tree  $\Uu\in N$ via $\Psi$ such that, letting $b=\Sigma_{\vV}(\Tt)$ and $c=\Sigma_{\vV}(\Uu)$, then $M^\Uu_c$ is a $\Sigma_{M^\Tt_b}$-iterate of $M^\Tt_b$,\item\label{item:delta_n+1_N_computes_maximal_P-suitable_branch_models_mod_soundness} for each $\delta_{n+1}$-maximal ($\vec{\gamma}$-stable)-$*$-suitable tree $\Tt\in N$
	via $\Psi$, with $*$-suitability witnessed by $U,\delta$
	(so $\mathscr{P}^U(M(\Tt))$ is proper class),
	the $\delta$-core of $\mathscr{P}^{U}(M(\Tt))$ is the $\delta$-core
	of $M^\Tt_b$ where $b=\Sigma_{\vV}(\Tt)$, and
	\item\label{item:delta_n+1_N_sts_uniformity} moreover, the definitions of $\Psi_{\mathrm{tm}}\rest N$ and $\Tt\mapsto(U,\mathscr{P}^U(M(\Tt)))$ (for $\Tt$
	as in part \ref{item:delta_n+1_N_computes_maximal_P-suitable_branch_models_mod_soundness}) are uniform in $N$.
	\end{enumerate}
\item\label{item:delta_1-sts_def_over_vV_1-iterates}  
Let $\vV$ be a non-dropping $\Sigma_{\wW_{n+1}}$-iterate.
Let $\Psi=\Sigma_{\vV,\delta_{n+1}\sss}$
and $\Psi_{\mathrm{tm}}$ be the restriction of $\Psi$ to tame trees.
Then the conclusions of part \ref{item:delta_n+1-sts_def_over_M-iterates_N}
hold after replacing $N$ throughout with $\vV$
(but $\vV$ remains $\vV$).
(Note that we do not assume that $\vV=\vV_{n+1}^N$ for some $\Sigma_M$-iterate $N$.)
\end{enumerate}
	\end{lem}
\begin{proof}[Proof] Part \ref{item:delta_n+1-sts_def_over_M-iterates_N}: 	
Part \ref{item:delta_n+1-N_internal_sts}:
For trees based on $\vV|\Delta^{\vV}$,
\ref{item:hyp_8} suffices (and we don't need the restriction to tame trees).
For trees not based on $\vV|\Delta^{\vV}$,
use Lemma \ref{lem:*-suitable_P-con_correctness_no_short_overlaps} as in \cite{vm2_v2}
(that is, using minimal genericity inflation to lift arbitrary trees
to $*$-suitable ones).
Part \ref{item:all_delta_n+1-max_absorbed} is likewise.
Part \ref{item:delta_n+1_N_computes_maximal_P-suitable_branch_models_mod_soundness}
is directly by Lemma \ref{lem:*-suitable_P-con_correctness_no_short_overlaps} (and just included here for reiteration).
Part \ref{item:delta_n+1_N_sts_uniformity} is clear.

Part \ref{item:delta_1-sts_def_over_vV_1-iterates}:  The proof is like that for part \ref{item:delta_n+1-sts_def_over_M-iterates_N},
but using  \ref{lem:*-suitable_P-con_correctness_no_short_overlaps_vV} instead of \ref{lem:*-suitable_P-con_correctness_no_short_overlaps}.
	\end{proof}

	We can now extend Lemma \ref{lem:*-suitable_P-con_correctness_no_short_overlaps_vV} to handle trees in $\vV[g]$:
	
	\begin{lem}\label{lem:*-suitable_P-con_correctness_no_short_overlaps_vV[g]}
	Adopt the hypotheses and notation of Lemma \ref{lem:*-suitable_P-con_correctness_no_short_overlaps_vV}, except for the assumption that $\Tt\in\vV$
	and $\vV\sats$``$\Tt$ is $*$-suitable''.
	Suppose instead that $g$ is set-generic over $\vV$ and  $\Tt\in\vV[g]\sats$``$\Tt$ is $*$-suitable''. Then:
\begin{enumerate}
\item\label{item:delta(Tt)_overlapped_iff_reach_corresp_stage_in_P-con[g]} Part \ref{item:delta(Tt)_overlapped_iff_reach_corresp_stage_in_P-con}
of \ref{lem:*-suitable_P-con_correctness_no_short_overlaps} holds.
\item\label{item:Q=P-con_if_delta_not_shortly_overlapped_no_N[g]} Part \ref{item:Q=P-con_if_delta_not_shortly_overlapped} of \ref{lem:*-suitable_P-con_correctness_no_short_overlaps} holds.
\end{enumerate}
\end{lem}
\begin{proof}
Part \ref{item:Q=P-con_if_delta_not_shortly_overlapped_no_N[g]}:
Suppose for the moment there is a $\Sigma_M$-iterate $N$ 
such that $\vV=\vV_{n+1}^N$.
(Actually then the desired conclusions are already established by Lemma  \ref{lem:*-suitable_P-con_correctness_no_short_overlaps}, but we want to make some further observations in this case.)  Let $\PP\in\vV$
be such that $g\sub\PP$
is $(\vV,\PP)$-generic.

Suppose first that $\Tt$ is $\delta_{n+1}$-short,
so $b=\Sigma_{\vV}(\Tt)\in\vV[g]$. Let $(\dot{\Tt},\dot{b})\in\vV$ be names for $(\Tt,b)$.
Working in $\vV$,
using the method of proof of \cite[Theorem 9.1]{fullnorm_v3}
(also cf.~\cite[\S7.1]{iter_for_stacks}),
we can define a sequence $\vec{\Uu}=\left<\Uu_\alpha\right>_{\alpha<\xi}$ such that each $\Uu_\alpha$ is on $\vV$,
 according to $\Psi=\Sigma_{\vV,\delta_{n+1}\sss}$, is $\delta_{n+1}$-short, and $\PP$ forces
that there is $\alpha<\xi$
and a minimal tree embedding from $(\Tt,b)$ into $\Uu_\alpha$. The definition is $\vec{\Uu}$ is moreover uniform in $\PP$, $(\dot{\Tt},\dot{b})$ and $\vV$. 
Note that by Lemma \ref{lem:delta_n+1-sts_def},
the properties in question are also uniformly  first order over non-dropping iterates of $\vV_{n+1}$. But now drop the assumption that we have such an $N$. Our observations so far apply to the pair $(\vV\down 0,\vV_{n+1}^{\vV\down 0})$, and so we can pull  the first order consequences back to $\vV$.
So we get the same situation there (since Lemma \ref{lem:delta_n+1-sts_def} applies to $\vV$, we know the trees $\Uu_\alpha\in\vV$  are correct).

Now suppose instead that $\Tt$ is $\delta_{n+1}$-maximal.
Suppose agan that we have $N$ as above.
Using a process similar to that above, but applied to a name $\dot{\Tt}$ for $\Tt$,
where $\PP$ forces that $\dot{\Tt}$ is $*$-suitable and $\delta_{n+1}$-maximal (as determined by P-construction
in the generic extension).
This yields a $\delta_{n+1}$-maximal tree $\Uu\in\vV$, with $\Uu$ via $\Psi$,  such that $\PP$ forces that $M(\dot{\Tt})$ iterates out to $M(\Uu)$. Still working in $\vV$, we can iterate $M(\Uu)$ further to obtain some $\Uu'$ which is $\delta_{n+1}$-suitable $*$-suitable, via $\Psi$. Note then the various facts forced about theories of indiscernibles and parameters $<\delta(\dot{\Tt})$ in $P=\mathscr{P}^{\dot{U}}(M(\dot{\Tt}))$ (where $\dot{U}$ is forced to witnesses $*$-suitability), and their relationship to the corresponding theories for $P'=\mathscr{P}^{U'}(M(\Uu'))$
(where $U'$ witnesses the $*$-suitability of $\Uu'$). 
Now work instead in an arbitrary $\Sigma_{\wW_{n+1}}$-iterate $\vV$. The foregoing discussion applies to $(\vV\down 0,\vV_{n+1}^{\vV\down 0})$,
and since we have $j:\vV\to\vV_{n+1}^{\vV\down 0}$,
the relevant (forcing) facts pull back,
including regarding indiscernibles.
 The desired conclusions follow.
 
 Part \ref{item:delta(Tt)_overlapped_iff_reach_corresp_stage_in_P-con[g]}: We leave this to the reader.
\end{proof}

	\begin{lem}\label{lem:wW_n+1_def_from_seg_in_univ}
	 $\wW_{n+1}$ is definable over its universe from the parameter $\wW_{n+1}|\kappa_n^{\wW_{n+1}}$.
	\end{lem}
\begin{proof}
 Like the proof  of \cite[***Lemma 5.17(2)]{vm2_v2}
 (but forming $\Ult(U,\vec{e}^{\wW_{n+1}})$, where $U$ is the universe of $\wW_{n+1}$, and recovering $\vV\down 0$
 from this model).
\end{proof}

	\begin{lem}\label{lem:vV_def_over_vV[g]}
	Let $\vV$ be a non-dropping $\Sigma_{\wW_{n+1}}$-iterate.
	Let $\PP\in\vV$ where $\PP\sub\lambda$ and $\lambda\geq\delta_n^{\vV}$.
	Let $g$ be $(\vV,\PP)$-generic.
	Then $\vV$ is definable over the universe of $\vV[g]$
	from the parameter $x=\vV|\lambda^{+\vV}$, uniformly in $\vV,\PP,g$.
	\end{lem}
\begin{proof}
Like the proof of \cite[***Lemma 5.18]{vm2_v2}.
	\end{proof}
\begin{lem}\label{lem:delta_n+1-sts_def_g}\  We have:
\begin{enumerate}
	\item\label{item:delta_n+1-sts_def_over_M-iterates_N_g}
	Let $N$ be a   non-dropping $\Sigma_M$-iterate
	and $\vV=\vV_{n+1}^N$. Let $\lambda\in\OR$ and $\PP\sub\lambda$ and $x=N|\lambda^{+N}$ and $g\sub\PP$ be $(N,\PP)$-generic. Let $\Psi=\Sigma^{\sn}_{\vV,\delta_{n+1}\sss}$ and
	$\Psi_{\mathrm{tm}}$ be its restriction to tame trees.  Then:
\begin{enumerate}[label=\tu{(}\alph*\tu{)},ref=(\alph*)]
\item\label{item:delta_n+1-N_internal_sts_g} $N[g]$ is closed under $\Psi_{\mathrm{tm}}$ and $\Psi_{\mathrm{tm}}\rest N[g]$ is  definable over $N[g]$ from the parameter $x$ (hence $\dom(\Psi_{\mathrm{tm}})\rest N[g]$ is likewise definable;  by Lemma \ref{lem:delta_n+1-sts_def},
if $g=\emptyset$ then we actually get lightface definability here),
	\item\label{item:all_delta_n+1-max_absorbed_g} for each $\delta_{n+1}$-maximal tree $\Tt\in N[g]$ via $\Psi$
	there is a $*$-suitable $\delta_{n+1}$-maximal  tree  $\Uu\in N$ \tu{(}not just $N[g]$\tu{)} via $\Psi$ such that, letting $b=\Sigma_{\vV}(\Tt)$ and $c=\Sigma_{\vV}(\Uu)$, then $M^\Uu_c$ is a $\Sigma_{M^\Tt_b}$-iterate of $M^\Tt_b$,\item\label{item:delta_n+1_N_computes_maximal_P-suitable_branch_models_mod_soundness_g} for each $\delta_{n+1}$-maximal $*$-suitable tree $\Tt\in N[g]$
	via $\Psi$, with $*$-suitability witnessed by $U,\delta$
	(so $\mathscr{P}^U(M(\Tt))$ is proper class),
	the $\delta$-core of $\mathscr{P}^{U}(M(\Tt))$ is the $\delta$-core
	of $M^\Tt_b$ where $b=\Sigma_{\vV}(\Tt)$, and
	\item\label{item:delta_n+1_N_sts_uniformity_g} moreover, the definitions of $\Psi_{\mathrm{tm}}\rest N[g]$ and $\Tt\mapsto(U,\mathscr{P}^U(M(\Tt)))$ (for $\Tt$
	as in part \ref{item:delta_n+1_N_computes_maximal_P-suitable_branch_models_mod_soundness_g}) are uniform in $N,g,x$
	(and by Lemma \ref{lem:delta_n+1-sts_def},
	uniform in $N$ if $g=\emptyset$).
	\end{enumerate}
\item\label{item:vV[g]_computes_sts}  
Let $\vV$ be a non-dropping $\Sigma_{\wW_{n+1}}$-iterate.
Let $\lambda\in\OR$ with $\delta_n^{\vV}\leq\lambda$,
let $\PP\sub\lambda$, and let $g$ be $(\vV,\PP)$-generic.
Let $x=\vV|\lambda^{+\vV}$.
Let $\Psi=\Sigma^{\mathrm{sn}}_{\vV,\delta_{n+1}\sss}$
and $\Psi_{\mathrm{tm}}$ its restriction to tame trees.
Then the conclusions of part \ref{item:delta_n+1-sts_def_over_M-iterates_N_g}
hold after replacing $N$ throughout with $\vV$
(but $\vV$ remains $\vV$).
\end{enumerate}
	\end{lem}
\begin{proof}
 Part \ref{item:vV[g]_computes_sts}: By Lemma \ref{lem:vV_def_over_vV[g]},
 we can define $\vV$ over $\vV[g]$ from the parameter $x$.
 But given $\vV$, the rest is like the proof of part \ref{item:delta_1-sts_def_over_vV_1-iterates} of Lemma \ref{lem:delta_n+1-sts_def}, 
but using \ref{lem:*-suitable_P-con_correctness_no_short_overlaps_vV[g]} instead of \ref{lem:*-suitable_P-con_correctness_no_short_overlaps_vV}.

Part \ref{item:delta_n+1-sts_def_over_M-iterates_N_g}: This can be proved either like  Lemma \ref{lem:delta_n+1-sts_def}, handling
the $g\neq\emptyset$ case
like for $\vV[g]$ above,
or else can be deduced from the proof of part \ref{item:vV[g]_computes_sts}, since $\vV_{n+1}^N$ is a set ground of $N$, and over $N[g]$, we can compute $N$, and hence $\vV_{n+1}^N$, from the parameter $x$.
\end{proof}

Finally we state the full version of computation of the $\delta_{n+1}$-short tree strategy (which we will need for the proof of Theorem \ref{tm:vV_computes_Sigma_vV} and consequences such as Theorem \ref{tm:mantle}),
but defer its proof:
	\begin{lem}\label{lem:sts}
The strengthenings of Lemmas \ref{lem:delta_n+1-sts_def} and \ref{lem:delta_n+1-sts_def_g}
given by replacing each instance of $\Psi_{\mathrm{tm}}$ with $\Psi$,\footnote{Note that $\Psi_{\mathrm{tm}}$ is not referred to, for example,
in parts \ref{item:delta_n+1-sts_def_over_M-iterates_N_g}\ref{item:all_delta_n+1-max_absorbed_g} and \ref{item:delta_n+1-sts_def_over_M-iterates_N_g}\ref{item:delta_n+1_N_computes_maximal_P-suitable_branch_models_mod_soundness_g} of Lemma \ref{lem:delta_n+1-sts_def_g}.}
both hold.
	\end{lem}

	\begin{proof}[Proof deferral]
	  What remains of the proof requires  $*$-translation, to handle trees $\Tt$ having Q-structures with short extenders overlapping $\delta(\Tt)$. This is deferred to \cite{*-trans_add}. 
\end{proof}

\subsection{Uniform grounds and $\mM_{\infty,n+1}[*_{n+1}]$}

\cite[\S5.2]{vm2_v2} is now adapted readily, including to non-dropping iterates of $\wW_{n+1}$, not just of $\vV_{n+1}$.
The details are like in \cite{vm2_v2}.
Let us give an argument for condition (ug19)
of \cite[\S2]{vm2_v2}, however (instead of \cite[Lemma 5.26***]{vm2_v2}; we write $\mathscr{F}_{n+1}^M$ for the external system of iterates of $\vV_{n+1}$ associated to $M$):

\begin{lem}
\cite[(ug19)]{vm2_v2} holds for the system $\mathscr{F}_{n+1}^M$.
\end{lem}
\begin{proof}Let $P,R\in\mathscr{F}^{M}_{n+1}$. 
We want to find some $S\in\mathscr{F}^M_{n+1}\cap\mathscr{F}^P_{n+1}$
with $S$ an iterate of both $P$ and $R$. First let $R_1\in\mathscr{F}^M_{n+1}$ be a common iterate of $P$ and $R$. Let $\dot{R}_1\in P|\kappa_{n+1}^P$ be a $\BB_{\delta_{n+1}^P}^P$-name for $R_1$. Let $\gamma\in(\delta_{n+1}^P,\kappa_{n+1}^P)$ be a cardinal strong $\{\kappa_0^M,\ldots,\kappa_n^M\}$-cutpoint of $M$ such that $\dot{R}_1\in P|\gamma$. Letting $g$ be $(P,\Coll(\om,\gamma))$-generic,
in $P[g]$, we can compare all $\delta_{n+1}$-maximal iterates of $P|\delta_{n+1}^P$
produced by the short tree strategy,
which are in $(P|\gamma)[g]$.
This results in a  $\delta_{n+1}$-maximal iterate $P'$ of $P$.
By homogeneity, we have $P'|\delta_{n+1}^{P'}\in P$. And $P'|\delta_{n+1}^{P'}$ is an iterate of $R_1|\delta_{n+1}^{R_1}$. But then in $P$, we can find some $S\in\mathscr{F}^P_{n+1}\cap\mathscr{F}^M_{n+1}$ such that $S|\delta_{n+1}^S$ is an iterate of $P'$; for this, we form the  genericity iteration for simultaneously making $P|\delta$ generic (where $\delta$ is the eventual image of $\delta_{n+1}^{P'}$), and such that $\BB_{\delta_{n+1}^P}^P$ forces that $\tau|\delta$ is also generic, where $\tau$ is the canonical name for $M$
given by adjoining $M|\delta_{n+1}^P$
and then using P-construction. Note that this genericity iteration is appropriately definable from parameteres both over $P|\delta$ and over $M|\delta$ (for appropriate $\delta$)
(which is one of the requirements
for having $S\in\mathscr{F}^P_{n+1}\cap\mathscr{F}^M_{n+1}$).
\end{proof}

We have now established the abstract hypotheses for uniform grounds
as in \cite[\S2]{vm2_v2}. So we get the results there, producing $\mM_{\infty,n+1}$ and $*_{n+1}$
and $\mM_{\infty,n+1}[*_{n+1}]$
(which we might just refer to as $\mM_{\infty,n+1}[*]$), etc. 

\subsection{Action of iteration maps on $\mM_{\infty,n+1}$}

We now adapt \cite[\S5.4]{vm2_v2}. The adaptation is very routine, but it is not quite literally as there,
because in \cite{vm2_v2},
every non-dropping iterate of $M_{\swsw}$ is $\kappa_1^{M_{\swsw}}$-sound, but there is no such convenience in the present context. This is a small complication. We also want to consider $\Sigma_{\wW_{n+1}}$-iterates, not just $\Sigma_{\vV_{n+1}}$-iterates. So we state the appropriate lemmas.

\begin{dfn}
 Let $\vV$ be a non-dropping $\Sigma_{\wW_{n+1}}$-iterate.
 For $P\in\mathscr{F}_{n+1}^{\vV}$ let $H^P=\Hull^P_1(\mathscr{I}^{\vV}\cup\delta_{n+1}^P)$,
 $\bar{P}$ be the transitive collapse of $H^P$,
 and $\pi_{\bar{P}P}:\bar{P}\to P$ the uncollapse map.
 As in \cite[Lemma 5.16(1)]{vm2_v2}, $\delta_{n+1}^{\bar{P}}=\delta_{n+1}^P$ and
 $\bar{P}$ is a $\delta_{n+1}^{\bar{P}}$-sound $\Sigma_{\wW_{n+1}}$-iterate. Define $(\mM_{\infty,n+1}^{\overline{\mathrm{ext}}})_{\vV}$ as the direct limit of the (iterates) $\bar{P}$ for $P\in\mathscr{F}_{n+1}^{\vV}$ (both under the iteration maps induced by $\Sigma_{\wW_{n+1}}$). (Note that if $\vV$ is $\delta_{n+1}^{\vV}$-sound then $\bar{P}=P$ and $\pi_{\bar{P}P}=\id$
 for all $P\in\mathscr{F}^{\vV}_{n+1}$,
 so in this case we can define the external direct limit directly as the direct limit of the models $P\in\mathscr{F}^{\vV}_{n+1}$ under the iteration maps.)\footnote{We could follow \cite{vm2_v2}
 by defining
 $(\mM_{\infty,n+1}^{\mathrm{ext}})_{\vV}$ in general,
 by iterating the base model $\vV$ via the trees (based on $\delta_{n+1}^{\vV}$) in the system. But in the case
 that $\vV$ is not $(\delta_{n+1},\kappa_{n+1})$-stable,
 this model is not relevant.}
 
 Let $\alpha\in\OR$ and $P\in\mathscr{F}_{n+1}^{\vV}$. We say that $\alpha$ is \emph{$(P,\mathscr{F}_{n+1}^{\vV})$-stable} iff for all $Q\in\mathscr{F}_{n+1}^{\vV}$ with $P\preccurlyeq Q$, we have $\alpha\in H^Q$ and $\pi_{\bar{Q}Q}\com i_{\bar{P}\bar{Q}}\com\pi_{\bar{P}P}^{-1}(\alpha)=\alpha$.
\end{dfn}
\begin{lem}\label{lem:action_iter_maps_on_M_infty}
 Let $\vV$ be a $\kappa_{n+1}^{\vV}$-sound non-dropping $\Sigma_{\wW_{n+1}}$-iterate.
 Let $\bar{\vV}$ be the $\delta_{n+1}^{\vV}$-core of $\vV$.
  Let $\mathscr{N}=(\mM_{\infty,n+1})^{\vV}$. Then:
  \begin{enumerate}
   \item For all $P,Q\in\mathscr{F}^{\vV}_{n+1}$,
   if $P\preccurlyeq Q$ then $H^P\sub H^Q$,
   \item For all $\alpha\in\OR$,
   there is $P\in\mathscr{F}^{\vV}_{n+1}$
   such that $\alpha$ is $(P,\mathscr{F}^{\vV}_{n+1})$-stable,
   \item $\mathscr{I}^{\mM_{\infty,n+1}}=\mathscr{I}^{\vV_{n+1}}=\mathscr{I}^M$ and $i_{\vV_{n+1}\mM_{\infty,n+1}}\rest\mathscr{I}^{\vV_{n+1}}=\id=*_{n+1}\rest\mathscr{I}^{\vV_{n+1}}$,
   \item $\mathscr{I}^{\mathscr{N}}=\mathscr{I}^{\vV}$
   and $(*_{n+1})^{\vV}\rest\mathscr{I}^{\vV}=\id$,
   \item $\mathscr{N}=(\mM_{\infty,n+1}^{\overline{\mathrm{ext}}})_{\vV}$ is a $\delta_{n+1}^{\N}$-sound
   $\Sigma_{\bar{\vV}}$-iterate.
      \item $(\mM_{\infty,n+1})^{\mathscr{N}}$ is a $\delta_{n+1}^{(\mM_{\infty,n+1})^{\mathscr{N}}}$-sound
   $\Sigma_{\mathscr{N}}$-iterate, \[*_{n+1}^{\vV}\sub\pi^{\vV}_{n+1,\infty}:\mathscr{N}\to(\mM_{\infty,n+1})^{\mathscr{N}},\]
   and $\pi^{\vV}_{n+1,\infty}$ is the $\Sigma_{\mathscr{N}}$-iteration map
   (note that $\pi^{\vV}_{n+1,\infty}$
   is defined analogously to \cite[Definition 5.33***]{vm2_v2}).
   \item \label{item:iteration_action_on_M_infty_between_iterates}Let $\vV'$ be a $\kappa_{n+1}^{\vV'}$-sound non-dropping $\Sigma_{\vV}$-iterate
   with $\vV|\delta_{n+1}^{\vV}\pins\vV'$.
   Let $\mathscr{N}'=(\mM_{\infty,n+1})^{\vV'}=(\mM_{\infty,n+1}^{\overline{\mathrm{ext}}})_{\vV'}$.
   Then:
   \begin{enumerate}
\item\label{item:rest_iter_maps_commute} $i_{\vV\vV'}\com i_{\bar{\vV}\mathscr{N}}=i_{\bar{\vV}\mathscr{N}'}$,
\item \label{item:when_iterate_above_kappa_n+1}
If $\vV|\kappa_{n+1}^{\vV}\pins\vV'$ then:
\begin{enumerate}
\item  $\mathscr{N}'$ is a $\Sigma_{\mathscr{N}}$-iterate, and
   \item  $i_{\vV\vV'}\rest\mathscr{N}:\mathscr{N}\to\mathscr{N}'$ is just the $\Sigma_{\mathscr{N}}$-iteration map.
   \end{enumerate}
   \end{enumerate}
  \end{enumerate}
\end{lem}
\begin{proof}
For the most part see the proof of \cite[Lemma 5.45]{vm2_v2}.
Part \ref{item:rest_iter_maps_commute} (the analogue of which was not stated within \cite[Lemma 5.45]{vm2_v2})
follows from standard calculations using that initial segments of the correct cofinal branches are determined by the action on enough indiscernibles,
which $i_{\vV\vV'}$ is correct about, and by the definability of the relevant limit length iteration trees within $\vV$ and $\vV'$.
\end{proof}

\begin{cor}\label{cor:action_iter_maps_on_M_infty_unsound}
 Let $\wW$ be a  non-dropping $\Sigma_{\wW_{n+1}}$-iterate \tu{(}we do not assume that $\wW$ is $\kappa_{n+1}^{\wW}$-sound\tu{)}. Let $\delta=\delta_{n+1}^{\wW}$
 and $\kappa=\kappa_{n+1}^{\wW}$.
 Let $\nN=\mM_{\infty,n+1}^{\wW}$.
 Let $\bar{\wW}$ be the $\kappa$-core of $\wW$.
 Let $\bar{\nN}=\mM_{\infty,n+1}^{\bar{\wW}}$.
 Let $\bar{\bar{\wW}}$ be the $\delta$-core of $\wW$.
 Then:
 \begin{enumerate}
 \item\label{item:delta_and_kappa_agreement} $\delta_{n+1}^{\bar{\bar{\wW}}}=\delta_{n+1}^{\bar{\wW}}=\delta=\delta_{n+1}^{\wW}$
 and $\kappa_{n+1}^{\bar{\wW}}=\kappa=\kappa_{n+1}^{\wW}$,
 \item $\bar{\bar{\wW}}$ is a $\delta$-sound $\Sigma_{\wW_{n+1}}$-iterate,
 \item $\bar{\wW}$ is a $\kappa$-sound $\Sigma_{\bar{\bar{\wW}}}$-iterate, via an above-$\delta$ tree,
 \item\label{item:wW_is_Sigma_wW-bar_it} $\wW$ is a $\Sigma_{\bar{\wW}}$-iterate,
 via an above-$\kappa^{+\bar{\wW}}$ tree,
   \item\label{item:nN-bar_is_iterate} $\bar{\nN}$ is a $\delta_{n+1}^{\bar{\nN}}$-sound $\Sigma_{\bar{\bar{\wW}}}$-iterate,
  via a tree based on $\bar{\bar{\wW}}|\delta=\bar{\wW}|\delta=\wW|\delta$.									`
  \item\label{item:nN-bar_is_hull_of_indiscs} $\bar{\nN}=\cHull^{\nN}(\delta_{n+1}^{\nN}\cup\mathscr{I}^{\wW})$

  \item\label{item:nN|kappa^+=nN-bar|kappa^+}  $\nN|\kappa_{n+1}^{+\nN}=\bar{\nN}|\kappa_{n+1}^{+\bar{\nN}}$
  and $(\mM_{\infty,n+1})^{\nN}|\delta_{n+1}^{(\mM_{\infty,n+1})^{\nN}}=(\mM_{\infty,n+1})^{\bar{\nN}}|\delta_{n+1}^{(\mM_{\infty,n+1})^{\bar{\nN}}}$
  and $\kappa_{n+1}^{+\nN}=\delta_{n+1}^{(\mM_{\infty,n+1})^{\nN}}=\delta_{n+1}^{(\mM_{\infty,n+1})^{\bar{\nN}}}=\kappa_{n+1}^{+\bar{\nN}}$,
     \item\label{item:mM_infty,n+1^nN-bar_is_Sigma-it} $(\mM_{\infty,n+1})^{\bar{\nN}}$ is a $\delta_{n+1}^{(\mM_{\infty,n+1})^{\bar{\nN}}}$-sound $\Sigma_{\bar{\nN}}$-iterate, and \[*_{n+1}^{\bar{\wW}}\rest\delta_{n+1}^{\bar{\nN}}=*_{n+1}^{\wW}\rest\delta_{n+1}^{\nN}\sub\pi^{\wW}_{n+1,\infty}:\mathscr{N}\to(\mM_{\infty,n+1})^{\mathscr{N}},\]
   and $\pi^{\wW}_{n+1,\infty}\rest(\nN|\delta_{n+1}^{\nN})$ agrees with the $\Sigma_{\bar{\nN}}$-iteration map
   (note that $\pi^{\wW}_{n+1,\infty}$
   is defined analogously to \cite[Definition 5.33***]{vm2_v2}).
 \end{enumerate}

 Moreover, let $\wW'$ be a non-dropping $\Sigma_{\wW}$-iterate with $\wW|\delta\pins\wW'$.
 Let $\delta'=\delta_{n+1}^{\wW'}$ and $\kappa'=\kappa_{n+1}^{\wW'}$.
 Let $\nN'=(\mM_{\infty,n+1})^{\wW'}$.
 Let $\bar{\wW'}$ be the $\kappa'$-core of $\wW'$ and $\bar{\nN}'=(\mM_{\infty,n+1})^{\bar{\wW'}}$.
 \tu{(}So $\delta_{n+1}^{\wW'}=\delta$ and note that parts \ref{item:delta_and_kappa_agreement}--\ref{item:nN-bar_is_hull_of_indiscs} also apply to $(\wW',\nN',\bar{\wW}',\bar{\nN}',\bar{\bar{\wW}})$.\tu{)} Then:
 \begin{enumerate}[resume*]
  \item\label{item:wW-bar'_is_it_of_wW-bar} $\bar{\wW}'$ is a $\Sigma_{\bar{\wW}}$-iterate with $\bar{\wW}|\delta\pins\bar{\wW}'$,
  and $i_{\wW\wW'}\rest\kappa^{+\wW}=i_{\bar{\wW}\bar{\wW}'}\rest\kappa^{+\bar{\wW}}$,
   \item\label{item:iteration_action_on_M_infty_between_iterates_non_kappa-sound}
   We have:
   \begin{enumerate}
\item\label{item:rest_iter_maps_commute_non_kappa-sound} $i_{\wW\wW'}\com i_{\bar{\bar{\wW}}\bar{\nN}}\rest(\bar{\bar{\wW}}|\delta)=i_{\bar{\wW}\bar{\wW'}}\com i_{\bar{\bar{\wW}}\bar{\nN}}\rest(\bar{\bar{\wW}}|\delta)=i_{\bar{\bar{\wW}}\bar{\nN'}}\rest(\bar{\bar{\wW}}|\delta)$,

\item\label{item:when_iterate_above_kappa_n+1_non_kappa-sound}
If $\wW|\kappa\pins\wW'$ then:
\begin{enumerate}
\item $\bar{\wW}|\kappa\pins\bar{\wW}'$,
\item  $\bar{\nN}'$ is a $\Sigma_{\bar{\nN}}$-iterate, and
   \item  $i_{\wW\wW'}\rest(\bar{\nN}|\delta_{n+1}^{\bar{\nN}}):(\bar{\nN}|\delta_{n+1}^{\bar{\nN}})\to(\bar{N}'|\delta_{n+1}^{\bar{\nN}'})$ is just the restriction of the $\Sigma_{\bar{\mathscr{N}}}$-iteration map.
   \end{enumerate}
  \end{enumerate}
  \end{enumerate}
\end{cor}
\begin{proof}
Parts \ref{item:delta_and_kappa_agreement}--\ref{item:wW_is_Sigma_wW-bar_it} are clear,
and it follows that Lemma \ref{lem:action_iter_maps_on_M_infty} applies 
to $\vV=\bar{\wW}$ and $\bar{\vV}=\bar{\bar{\wW}}$.
 Parts \ref{item:nN-bar_is_iterate} and \ref{item:nN-bar_is_hull_of_indiscs} 
  are an easy consequence.
  Parts \ref{item:nN|kappa^+=nN-bar|kappa^+}
  and \ref{item:mM_infty,n+1^nN-bar_is_Sigma-it}
  then follow from Lemma \ref{lem:action_iter_maps_on_M_infty}
  and the local definability
  of (what will be defined as) $\vV_{n+2}|\gamma_{n+1}^{\vV_{n+2}}$ over $\wW|\kappa_{n+1}^{+\wW}$, as in \cite[***Lemma 4.41]{vm2_v2}.
Part \ref{item:wW-bar'_is_it_of_wW-bar} is straightforward, and together with Lemma \ref{lem:action_iter_maps_on_M_infty},
yields part \ref{item:iteration_action_on_M_infty_between_iterates_non_kappa-sound}.
\end{proof}

\begin{lem}\label{lem:N_kappa_n+1-sound_implies_vV^N_n+1}
 Let $N$ be a $\kappa_{n+1}^N$-sound non-dropping $\Sigma_M$-iterate.  Then  $\vV_{n+1}^N$ is a $\kappa_{n+1}^{\vV_{n+1}^N}$-sound
 non-dropping $\Sigma_{\wW_{n+1}}$-iterate.
\end{lem}
\begin{proof}
Since $\kappa_{n+1}^{\vV_{n+1}}=\kappa_{n+1}^M$ by definition, this follows from Corollary \ref{cor:Sigma_wW_n+1_down_0=Sigma_M}
and its proof.
\end{proof}

\begin{lem}\label{lem:}\  We have:
\begin{enumerate}
	\item\label{item:Sigma_vV_n+2,Delta_def_over_N}
	Let $P$ be a $\kappa_{n+1}^P$-sound non-dropping $\Sigma_M$-iterate. Let $\lambda\in\OR$ and $\PP\sub\lambda$ and $x=P|\lambda^{+P}$ and $g\sub\PP$ be $(P,\PP)$-generic. Let $\mathscr{N}=\mM_{\infty,n+1}^P$.  Then:
\begin{enumerate}[label=\tu{(}\alph*\tu{)},ref=(\alph*)]
\item\label{item:nN_is_delta_n+1-sound} $\mathscr{N}$ is a  $\delta_{n+1}^{\mathscr{N}}$-sound non-dropping $\Sigma_{\wW_{n+1}}$-iterate,
\item\label{item:Sigma_def_over_N} $P$
is closed under $\Sigma_{\mathscr{N},\delta_{n+1}^{\mathscr{N}}}$  and $\Sigma_{\nN,\delta_{n+1}^{\nN}}\rest P$ is lightface definable over $P$,
 \item\label{item:Sigma_def_over_N[g]} $P[g]$ is closed under $\Sigma_{\nN,\delta_{n+1}^{\nN}}$ and $\Sigma_{\nN,\delta_{n+1}^{\nN}}\rest P[g]$
 is definable over $P[g]$ from the parameter $x$,
	\item moreover, the definitions  are uniform in $P,g,\lambda,x$.
	\end{enumerate}
\item\label{item:Sigma_vV_n+2,Delta_def_over_vV}  
Let $\wW$ be a $\kappa_{n+1}^{\wW}$-sound non-dropping $\Sigma_{\wW_{n+1}}$-iterate.
Let $\lambda\in\OR$ with $\delta_n^{\wW}\leq\lambda$,
let $\PP\sub\lambda$, and let $g$ be $(\wW,\PP)$-generic.
Let $x=\wW|\lambda^{+\wW}$.
Let $\nN=\mM_{\infty,n+1}^{\wW}$.
Then the conclusions of part \ref{item:Sigma_vV_n+2,Delta_def_over_N}
hold after replacing $P$ throughout with $\wW$.
\end{enumerate}
	\end{lem}
\begin{proof}
 Part \ref{item:Sigma_vV_n+2,Delta_def_over_N}:
 Let $\wW=\vV_{n+1}^N$. Then $\kappa_{n+1}^{\wW}=\kappa_{n+1}^P$ and by Lemma \ref{lem:N_kappa_n+1-sound_implies_vV^N_n+1},
 $\wW$ is a $\kappa_{n+1}^{\wW}$-sound
$\Sigma_{\wW_{n+1}}$-iterate.
But also $\mM_{\infty,n+1}^{P}=\nN=\mM_{\infty,n+1}^{\wW}$.  So part \ref{item:nN_is_delta_n+1-sound} follows from Lemma \ref{lem:action_iter_maps_on_M_infty}.
 Now by Lemma \ref{lem:sts}
 (and normalization),
$P$ and $P[g]$ are both closed under $\Sigma_{\nN,\delta_{n+1}\sss}$, and the restriction of this strategy to $P$ and $P[g]$ is appropriately definable
(note that this includes non-tame trees).
For $\delta_{n+1}$-maximal trees,
argue as in \cite{vm2_v2},
 using the argument in \cite[Footnote 27]{vm2_v2}.
 
 Part \ref{item:Sigma_vV_n+2,Delta_def_over_vV}:
 The proof here is analogous, but for the version of part \ref{item:Sigma_def_over_N[g]} for $\wW$, first use Lemma \ref{lem:wW_n+1_def_from_seg_in_univ} to definably recover
 $\wW$ from $x$ in the universe of $\wW[g]$.
\end{proof}

\subsection{$\vV_{n+2}$ and $\wW_{n+2}$}

From this point, \cite[\S5.3]{vm2_v2} proceeds essentially as there, and in particular:
\begin{dfn}$\vV_{n+2}=\vV_{n+2}^M$ is defined
 over $M$ essentially as in \cite[\S5.3]{vm2_v2},
 but as a (putative) $(n+2)$-pVl.
\end{dfn}

The details discussed so far,
in the proof of the lemma below and
in \cite{vm2_v2} can be used to write precise formulas $\varphi_1,\varphi_2,\varphi_3$ used for $\vec{\varphi}$ in the specification
of the kind of pVls we use, and from now on, they are all of that kind; we  also now  retroactively declare that these were the formulas we have used all along.

\begin{lem}\label{lem:vV_n+2_is_pVl}$\vV_{n+2}$
is an $(n+2)$-pVl.\end{lem}
\begin{proof}
One first verifies by induction through the segments of $\vV_{n+2}$ that its fine structure
is well behaved, and above $\gamma_{n+1}^{\vV_{n+2}}$,
matches that of the corresponding segment of $M$.
This proceeds basically like in the proof of \cite[Lemmas 4.41, 4.42]{vm2_v2}, but the following claim
is the key point verifying the relevant
definability and formulating the formula
$\varphi_1$ (within the formulas $\vec{\varphi}$
of Definition \ref{dfn:pVl}).
 Let $\nN=\mM_{\infty,n+1}$.

\begin{clm}\label{clm:vV_n+2_self-it_tame} 
Let $\Psi=\Sigma_{\nN,\delta_{n+1}}$,
and $\Psi_{\mathrm{tm}}$ be its restriction to tame trees. Let $\xi>\gamma_{n+1}^{\vV_{n+2}}$
be such that $M|\xi$ is active with an extender $F_0$ with $\crit(F_0)=\kappa_i^M$ for some $i\leq n+1$.
Let $\mu$ be the largest cardinal of $M||\xi$
(which by the induction mentioned above
is also the largest cardinal of $\vV_{n+2}||\xi$).
Let $U=\Ult(M,E)$.
Then:
\begin{enumerate}
 \item \label{item:vV_n+2|mu_closed_under_sts}$\vV_{n+2}|\mu$ is closed under $\Psi_{\mathrm{tm},\delta_i\sss}$, and $\Psi_{\mathrm{tm},\delta_i\sss}\rest(\vV_{n+2}|\mu)$ is lightface definable over $\vV_{n+2}|\mu$,
 \item\label{item:P-con_in_vV_n+2|mu} working in $\vV_{n+2}|\mu$, we can define a natural analogue of $*$-suitability and P-construction over $M(\Tt)$
 for $*$-suitable $\Tt$, with properties analogous to those for $M$ (such as in Lemma \ref{lem:delta_n+1-sts_def} part \ref{item:delta_n+1-sts_def_over_M-iterates_N}), \item \label{item:mM_infty_locally_def}$\mM_{\infty,i}^{U}|\xi$
 is lightface definable over $\vV_{n+2}||\xi$
 (uniformly in $\xi$).
 \item\label{item:F^vV|xi_is_correct}  $\mM_{\infty,i}^{U}$ is a
$\delta_i^{\mM_{\infty,i}^{U}}$-sound non-dropping $\Sigma_{\mM_{\infty,i}^M}$-iterate, 
and $F^{\vV_{n+2}|\xi}$ is the $(\delta_i^{\mM_{\infty,i}^M},\delta_i^{\mM_{\infty,i}^U})$-extender derived from the $\Sigma_{\mM_{\infty,i}^M}$-iteration map.
 \end{enumerate}
\end{clm}
\begin{proof}
The proof to follow is similar to  that of \cite[Lemma 5.47]{vm2_v2}.

Part \ref{item:vV_n+2|mu_closed_under_sts}:
We first observe that $\vV_{n+2}|\mu$ can define the $\delta_{i}$-short tree strategy (for trees in $\vV_{n+2}|\mu$); even if $i\leq n$ this does not seem entirely obvious, and in particular does not seem to follow directly from our inductive hypotheses. So we argue directly, observing that:
\begin{enumerate}[label=(\roman*)]
 \item\label{item:simulate_P-con} for each $j\leq i$, the P-constructions used in $M$ to compute tame $\delta_j$-short tree strategy 
 can be simulated appropriately in $\vV_{n+2}|\mu$, and
 \item for each $j<i$, $\delta_j$-maximal branches can be computed using $j$-long extenders in $\es^{\vV_{n+2}|\mu}$.
\end{enumerate}

Toward this, let $\gamma=\gamma_{n+1}^{\vV_{n+2}}$. So above $\gamma$, $\es^{\vV_{n+2}}$ is induced by $\es^M$ via P-construction.

First consider $\Sigma_{\vV_{n+2},\delta_0}$. For trees $\Tt$ based on $\vV_{n+2}|\delta_0^{\vV_{n+2}}$,
let $F'\in\es_+^{\vV_{n+2}|\xi}$ be $0$-long with $\gamma<\lh(F')$. Let $\mu'$ be the largest cardinal
of $\vV_{n+2}|\lh(F')$.
Then $\vV_{n+2}|\mu'$ can compute P-constructions (and $*$-translations) over ``$*$-suitable'' trees $\Tt$ on $\vV_{n+2}$, based on $\vV_{n+2}|\delta_0^{\vV_{n+2}}$, with $\lh(\Tt)<\mu'$. Here ``$*$-suitable'' is defined much like in $M$ for trees on $M|\delta_0^M$; we should have $\delta=\delta(\Tt)<\mu$
and $\vV_{n+2}|\delta$ generic over $M(\Tt)$, work above a cutpoint $\eta$ of $\vV_{n+2}||\xi$, etc. 
Since $F'$ is induced by some $F''\in\es^M$
with $\crit(F'')=\kappa_0$,
and given the equivalence between $M$ and $\vV_{n+2}$ above $\gamma$, the $*$-suitability and P-construction defined in $\vV_{n+2}$ works correctly for essentially the same reasons it does in $M$.\footnote{Like in Case \ref{case:i>0_working_in_vV_i} of the proof of Lemma \ref{lem:*-suitable_P-con_correctness_no_short_overlaps}, the only reason that it doesn't follow directly \emph{from} the correctness of the P-constructions of $M$, is that
the genericity requirements are slightly (or at least superficially) different.} Now suppose $0<i$. Then for $\delta_0$-maximal trees, factor into $0$-long extenders in $\es^{\vV_{n+2}|\mu}$ as usual,
i.e.~again like in  \cite[Footnote 27]{vm2_v2}. For $\gamma_0$-unstable trees,  simulate the P-constructions of $M$
using $\es^{\vV_{n+2}|\mu}$ above $\gamma$, much as above.
For $\vec{\gamma}$-stable trees $\Tt$ on $\vV_{n+2}$
of the form  $\Tt=\Tt_0\conc\Tt_1$ with $\Tt_0$ induced by some $0$-long extender $E'\in\es^{\vV_{n+2}|\mu}$ with $\gamma<\lh(E')$, and $\Tt_1$ on $M^{\Tt_0}_\infty$, based on $M^{\Tt_0}_\infty|\delta_1^{M^{\Tt_0}_\infty}$ and  above $\gamma_0^{M^{\Tt_0}_\infty}$,
 and for $1$-long extenders $F'\in\es_+^{\vV_{n+2}|\xi}$
such that $\lh(E')<\lh(F')$,
and for $\mu'$ the largest cardinal of $\vV_{n+2}|\lh(F')$,
 define ``$*$-suitable'' and ``P-construction''  with respect to such $\Tt$ with  $\lh(E')<\delta(\Tt)<\mu'$ and the usual further properties. Continue in this manner through $\delta_{n+1}$. 

 Part \ref{item:P-con_in_vV_n+2|mu}:
 Like part \ref{item:vV_n+2|mu_closed_under_sts}.
 
 Part \ref{item:mM_infty_locally_def}:
 This follows from the previous two parts,
as in the proof of \cite[Lemma 4.41(b)***]{vm2_v2}.

Part \ref{item:F^vV|xi_is_correct}: This follows from the previous parts and Lemma \ref{lem:action_iter_maps_on_M_infty}.
\end{proof}

The formula $\varphi_1$, within the tuple $\vec{\varphi}$ of formulas of Definition \ref{dfn:pVl}, should be extracted from the proof of the claim above.
The formulas $\varphi_2,\varphi_3$ can be extracted from an examination of details from \cite{vm2_v2}.

For $\Delta$-conservativity
(condition \ref{item:Delta-conservativity} of Definition \ref{dfn:pVl}), one can argue like in
 in \cite{Theta_Woodin_in_HOD}, \cite{vm1} and/or \cite{vm2_v2}, but here is a slightly simplified argument.
 As in those references, letting \[j:\vV_{n+1}\to\vV_{(n+2)\down(n+1)}=\mM_{\infty,n+1} \]
 be the iteration map, we have
 \[ \Hull^{\vV_{n+2}}(\rg(j))\cap\mM_{\infty,n+1}=\rg(j).\]
 So defining
 \[ \vV'_{n+2}=\cHull^{\vV_{n+2}}(\rg(j)) \]
 and $j^+:\vV'_{n+2}\to \vV_{n+2}$ as the uncollapse map,
 then $j\sub j^+$, and letting $\Tt$ be the tree leading from $\vV_{n+1}$ to $\mM_{\infty,n+1}$,
 and $\Tt^+$ be the corresponding tree on $\vV_{n+2}'$ (this makes sense as usual),
 then $j^+$ is in fact the iteration map of $\Tt^+$.
 But because all $\vV_{n+1}$-total measures $E\in\es^{\vV_{n+1}}|\delta_{n+1}^{\vV_{n+1}}$ are absorbed into $j$ (i.e.~for each such $E$,
 $j$ factors as $k\com i^{\vV_{n+1}}_E$, where $k:\Ult(\vV_{n+1},E)\to\mM_{\infty,n+1}$ is the iteration map), it easily follows that $\vV'_{n+2}$ is $\Delta$-conservative, and hence so is $\vV_{n+2}$.
 
 The first order condensation properties relevant for (short-normal) mhc (contained in condition \ref{item:normalization_cond} of Definition \ref{dfn:pVl}) are inherited from $M$ because of the fine structural correspondence between segments of $M$ and segments of $\vV_{n+2}$, just like in \cite{vm2_v2}.
 
 This essentially completes the proof.
\end{proof}

\begin{rem}\label{rem:wW_n+2}
Like in the proof of $\Delta$-conservativity
within the proof of Lemma \ref{lem:vV_n+2_is_pVl},
and as usual,
$\vV_{(n+2)\down(n+1)}$ is also a non-dropping $\Sigma_{\wW_{n+1}}$-iterate,
and letting \[i:\wW_{n+1}\to\vV_{(n+2)\down(n+1)} \]
be the iteration map, then $\mathscr{I}^M=\mathscr{I}^{\vV_{n+2}}=j``\mathscr{I}^{\wW_{n+1}}=\mathscr{I}^{\wW_{n+1}}$ and
\[ \Hull^{\vV_{n+2}}(\mathscr{I}^M)\cap(\vV_{(n+2)\down(n+1)})=\rg(i),\]
so defining
\[ \wW=\cHull^{\vV_{n+2}}(\mathscr{I}^M) \]
and letting $i^+:\wW\to\vV_{n+2}$
be the iteration map, then $\wW$ is an $(n+2)$-pVl
and $i\sub i^+$, and $\wW_{n+1}=\wW\down(n+1)$.
\end{rem}

\begin{dfn}\label{dfn:wW_n+2}
Define $\wW_{n+2}=\wW$, introduced in Remark \ref{rem:wW_n+2} above.

Let $\Sigma_{\wW_{n+2}}^{\sn}$ denote the $\vec{e}$-minimal pullback of $\Sigma_M$ (Definition \ref{dfn:e-vec_minimal_pullback}).\end{dfn}

\begin{lem}$\Sigma^{\sn}_{\wW_{n+2}}$ is a short-normal $(0,\OR)$-strategy for $\wW_{n+2}$, with short-normal mhc,
which
is conservative
over $\Sigma_{\wW_{i}}$ for $i\leq n+1$.
\end{lem}
\begin{proof}
See Lemmas \ref{lem:e-vec_minimal_pullback} and \ref{lem:mhc_inheritance}.\end{proof}
 But before we can extend
$\Sigma^{\sn}_{\wW_{n+2}}$ to the full $(0,\OR)$-strategy
$\Sigma_{\wW_{n+2}}$, we need to know it is self-coherent.

\subsection{Self-coherence of $\Sigma_{\wW_{n+2}}$}

Related to \cite[Lemma 5.47]{vm2_v2},
we have the following, which will be generalized in Lemmas \ref{lem:vV_n+2_self-it} and \ref{lem:wW_n+2_self-it}:
\begin{lem}\label{lem:vV_n+2_self-it_tame}Let $N$ be a  non-dropping $\Sigma_M$-iterate. Let $\wW=\vV_{n+2}^N$
(so $\wW\down(n+1)=\mM_{\infty,n+1}^N$).
Let $\Psi=\Sigma_{\wW_{n+2},\delta_{n+1}}$,
and $\Psi_{\mathrm{tm}}$ be its restriction to tame trees. Then:
 \begin{enumerate}
  \item\label{item:vV_n+2_self-it} $\wW$ is closed under $\Psi_{\mathrm{tm}}$, and $\Psi_{\mathrm{tm}}\rest\wW$ is lightface definable over $\wW$.
  \item\label{item:vV_n+2[g]_it_vV_n+2} Let $\lambda\in\OR$ with $\delta_{n+1}^{\wW}\leq\lambda$,
  let $\PP\in\wW$ with $\PP\sub\lambda$, and let $g$ be $(\wW,\PP)$-generic (with $g$ appearing in some generic extension of $V$). Then $\wW[g]$ is closed under $\Psi_{\mathrm{tm}}$
  and $\Psi_{\mathrm{tm}}\rest\wW$ is  definable over the universe of $\wW[g]$ from the parameter $x=\wW|\lambda^{+\wW}$, uniformly in $\lambda$.
 \end{enumerate}
\end{lem}
\begin{proof}
Part \ref{item:vV_n+2_self-it}: Like the proof
of Claim \ref{clm:vV_n+2_self-it_tame} 
in the proof of Lemma \ref{lem:vV_n+2_is_pVl},
but 
appealing to Corollary \ref{cor:action_iter_maps_on_M_infty_unsound} instead of Lemma \ref{lem:action_iter_maps_on_M_infty}.

Part \ref{item:vV_n+2[g]_it_vV_n+2}: As for part \ref{item:vV_n+2_self-it}, but  using Lemma \ref{lem:wW_n+1_def_from_seg_in_univ} for the definability of $\wW$.
\end{proof}

\begin{lem}\label{lem:vV_n+2_self-it}Let $N$ be a  non-dropping $\Sigma_M$-iterate. Let $\wW=\vV_{n+2}^N$.
Let $\Psi=\Sigma_{\wW_{n+2},\delta_{n+1}}$.
Then the conclusions of Lemma \ref{lem:vV_n+2_self-it_tame}
still hold, after replacing $\Psi_{\mathrm{tm}}$
throughout by $\Psi$.
\end{lem}
\begin{proof}[Proof deferral]
The proof requires $*$-translation and is deferred to \cite{*-trans_add}.
\end{proof}

\begin{lem}\label{lem:M_infty_to_M_infty^M_infty_iter_map}
Let $\wW=\wW_{n+1}$ and $\mathscr{N}=\mM_{\infty,n+1}^{\wW}$. Then:
\begin{enumerate}
\item\label{item:N_delta_n+1-sound}  $\mathscr{N}$ is a $\delta_{n+1}$-sound
 $\Sigma_{\wW}$-iterate,
 \item\label{item:M_inf^N_delta_n+1-sound} $\mM_{\infty,n+1}^{\mathscr{N}}$ is a $\delta_{n+1}$-sound
 $\Sigma_{\mathscr{N}}$-iterate,
 \item\label{item:e_n+1^vV_n+2^wW_correct}
 $e_{n+1}^{\vV_{n+2}^{\wW}}$ 
 is  the $(\delta_{n+1}^{\mathscr{N}},\delta_{n+1}^{\mM_{\infty,n+1}^{\mathscr{N}}})$-extender
 derived from the $\Sigma_{\mathscr{N}}$-iteration map
  $k:\mathscr{N}\to\mM_{\infty,n+1}^{\mathscr{N}}$,
  
\item\label{item:e_n+1^wW_n+2_correct} $e_{n+1}^{\wW_{n+2}}$ is 
the $(\delta_{n+1}^{\wW},\delta_{n+1}^{\mathscr{N}})$-extender derived from the $\Sigma_{\wW}$-iteration map
$k:\wW\to\mathscr{N}$,
  \end{enumerate}
\end{lem}
\begin{proof}
 Parts \ref{item:N_delta_n+1-sound}--\ref{item:e_n+1^vV_n+2^wW_correct} are by previous lemmas.
 
 Part \ref{item:e_n+1^wW_n+2_correct}:
 $\wW_{n+2}=\cHull^{\vV_{n+2}^{\wW}}(\mathscr{I}^{\wW})$,
 and letting $k:\wW_{n+2}\to\vV_{n+2}^{\wW}$
 be the uncollapse map, we know that $\wW=\wW_{n+2\down n+1}\sub\wW_{n+2}$ and $j(\wW)=\mathscr{N}$ and $j\rest\wW$ is just the $\Sigma_{\wW}$-iteration map.
 But then since the various iteration maps fix $\mathscr{I}^{\wW}$ pointwise and $j(e_{n+1}^{\wW_{n+2}})=e_{n+1}^{\vV_{n+2}^{\wW}}$,
 it follows that $e_{n+1}^{\wW_{n+2}}$ agrees with the iteration map.
\end{proof}

\begin{lem}\label{lem:Sigma_wW_n+2}
 $\Sigma^{\sn}_{\wW_{n+2}}$ is self-coherent.
\end{lem}
\begin{proof}

 Let $\Tt$ on $\wW_{n+2}$ be according to $\Sigma^{\sn}_{\wW_{n+2}}$, of successor length. Let $\wW=M^\Tt_\infty$. Let $E\in\es_+^{\wW}$ be a long extender.
 We may assume that $\gamma=\gamma_{n+1}^{\wW}$
 exists and $\gamma\leq\lh(E)$, since otherwise
 $E$ is induced by $\Sigma^{\sn}_{\wW_{n+1}}$,
 hence also by $\Sigma_{\wW_{n+2}}^{\sn}$.
 
 \begin{case}
 $\Tt$ is based on $\wW_{n+2}|\gamma_{n+1}^{\wW_{n+2}}$.
 
 Let $\Tt'$ be the corresponding tree on $\wW_{n+1}$, so $\Tt'$ is via $\Sigma_{\wW_{n+1}}$, and $\vV=M^{\Tt'}_\infty$
 is $\kappa_{n+1}^{\vV}$-sound.
 Let $\bar{\vV}$ be the $\delta_{n+1}^{\vV}$-core of $\vV$,
 so $\bar{\vV}=M^{\Tt'}_\alpha$
 where $\alpha\in b^{\Tt'}$ and $\alpha$ is least such that $[0,\alpha]^{\Tt'}\cap\dropset^{\Tt'}=\emptyset$
 and either $\alpha+1=\lh(\Tt')$ or $\delta_{n+1}^{M^{\Tt'}_\alpha}<\lh(E^{\Tt'}_\alpha)$.
 We have $\bar{\vV}|\delta_{n+1}^{\bar{\vV}}=\vV|\delta_{n+1}^{\vV}$.
 By Lemma \ref{lem:action_iter_maps_on_M_infty},
 $\mathscr{N}=\mM_{\infty,n+1}^{\vV}$ is
 a $\delta_{n+1}$-sound
 a $\Sigma_{\bar{\vV}}$-iterate
 of $\bar{\vV}$ (but if $\bar{\vV}\neq\vV$ then it is
 not an iterate of $\vV$!). Likewise for $\bar{\mathscr{N}}=\mM_{\infty,n+1}^{\bar{\vV}}$. And $[0,\alpha]^{\Tt}\cap\dropset^{\Tt}=\emptyset$. Note that facts from Lemma \ref{lem:M_infty_to_M_infty^M_infty_iter_map}
 are preserved appropriately by the iteration maps,
 so that  $e_{n+1}^{M^{\Tt}_\alpha}$ is the
 $(\delta_{n+1}^{\bar{\vV}},\delta_{n+1}^{\bar{\mathscr{N}}})$-extender
 derived from the $\Sigma_{\bar{\vV}}$-iteration map $\bar{k}:\bar{\vV}\to\bar{\mathscr{N}}$.
 But then if $\alpha+1<\lh(\Tt')$,
 then since $\gamma_{n+1}^{M^\Tt_\infty}$ exists,
 $b^{\Tt'}$ does not drop,
 so by Lemma \ref{lem:action_iter_maps_on_M_infty},
 $i^{\Tt'}_{\alpha\infty}\com\bar{k}=k$
 where $k:\bar{\vV}\to \mathscr{N}$ is the $\Sigma_{\bar{\vV}}$-iteration map.
 But then since $i^\Tt_{\alpha\infty}\sub i^{\Tt'}_{\alpha\infty}$, we have that $i^{\Tt'}_{\alpha\infty}\com e_{n+1}^{M^\Tt_\alpha}=e_{n+1}^{M^\Tt_\infty}$
 is derived from $k$,
 as desired.
 \end{case}

 \begin{case} Otherwise.

By the previous case, we may assume that $\gamma_{n+1}^{M^\Tt_\infty}<\lh(E)$, and that $\Tt=\Tt_0\conc\Tt_1$ where $\Tt_0$ is based on $\wW_{n+2}|\delta_{n+1}^{\wW_{n+2}}$ and has successor length and $b^{\Tt_0}$ is non-dropping, and $\Tt_1$ is above $\gamma_{n+1}^{M^{\Tt_0}_\infty}$. Let $\alpha$ be as before. So $\gamma_{n+1}^{M^\Tt_\infty}=\gamma_{n+1}^{M^\Tt_\alpha}$ and $\gamma_{n+1}^{M^\Tt_\alpha}<\lh(E^\Tt_\alpha)$, if $\alpha+1<\lh(\Tt)$.
And $e=e_{n+1}^{M^\Tt_\alpha}$ is via $\Sigma_{M^\Tt_\alpha}$ and via $\Sigma_{M^\Tt_\infty}$.

 Let $\Tt'=\Tt\down(n+1)$.
 Let $\xi=\lh(E)$. We have $U=\Ult_0(M^{\Tt}_\infty|\xi,e)\ins \vV_{n+2}^{M^{\Tt'}_\infty}$ and $\kappa_{n+1}^{+M^{\Tt'}_\infty}=\gamma_{n+1}^{\vV_{n+2}^{M^{\Tt'}_\infty}}<\OR^U$.
So $U$ is given by the P-construction of $M^{\Tt'}_\infty$
over $U|\gamma_{n+1}^U$. In particular,
$F^U$ is induced by an extender $F\in\es^{M^{\Tt'}_\infty}$
with $\crit(F)=\kappa_{n+1}^{M^{\Tt'}_\infty}$,
and we have $\kappa_{n+1}^{+M^{\Tt'}_\infty}=\kappa_{n+1}^{+M^{\Tt'}_\alpha}$ (and $[0,\alpha]^{\Tt'}\cap\dropset^{\Tt'}=\emptyset$). Assuming that we took $\Tt$ of minimal length such that $E\in\es_+^{M^\Tt_\infty}$, we can extend $\Tt'$ to $\Uu=\Tt'\conc\left<F\right>$,
as a $0$-maximal, and hence short-normal tree. Note that $b^\Uu$ is via $\Sigma_{\wW_{n+1}}$ and $b^\Uu$ is non-dropping,
and $M^\Uu_\infty$ is $\kappa_{n+1}^{M^\Uu_\infty}$-sound.
So Lemma \ref{lem:action_iter_maps_on_M_infty} 
(in particular part \ref{item:iteration_action_on_M_infty_between_iterates}) applies to the iteration map (and ultrapower map) $i^\Uu_{\alpha\infty}:M^{\Tt'}_\alpha\to\Ult(M^{\Tt'}_\alpha,F)$.
We have $M^{\Tt'}_\alpha|\kappa_{n+1}^{M^{\Tt'}_\alpha}\pins M^\Uu_\infty$,
so part \ref{item:when_iterate_above_kappa_n+1} applies.
It follows that $F$ is the extender of a correct iteration map on $\mM_{\infty,n+1}^{M^{\Tt'}_\alpha}$, and so $F\com e$ is likewise on $M^{\Tt'}_\alpha=M^\Tt_\alpha\down(n+1)$. But letting $j:M^\Tt_\infty|\xi\to U$ be the (degree $0$) ultrapower map, then $j\com i_E=i_{F\com e}$.
This will give that $E$ is also the extender of a correct iteration map on $M^{\Tt'}_\alpha$,
as long as the tree $\Xx$ on $M^{\Tt}_\alpha|\delta_{n+1}^{M^\Tt_\alpha}$ leading to $\mM_{\infty,n+1}^{M^\Tt_\infty||\xi}$ is itself correct.
But $j$ embeds $\Xx$ into $j(\Xx)$, which is correct,
and using an argument like that for \ref{lem:*-suitable_P-con_correctness_no_short_overlaps_vV} (applied to proper segments of $\Xx$, and using the current $j$, which might be set-sized) we get that $\Xx$ is also correct.\qedhere
 \end{case}
\end{proof}

\begin{dfn}
 $\Sigma_{\wW_{n+2}}$ is the unique self-consistent $(0,\OR)$-strategy for $\wW_{n+2}$ extending $\Sigma^{\sn}_{\wW_{n+2}}$
 (see Lemma \ref{lem:unique_ext_to_(0,OR)-strat_n+1}).
\end{dfn}

Generalizing Lemma \ref{lem:vV_n+2_self-it_tame}, we now get the following lemma.
Note that by conservativity of $\Sigma_{\wW_{n+2}}$,
$\wW\down i$ is a non-dropping $\Sigma_{\wW_i}$-iterate
for each $i\leq n+1$. 
\begin{lem}[Tame self-iterability for iterates of $\wW_{n+2}$]\label{lem:wW_n+2_self-it_tame}
Let $\wW$ be any non-dropping $\Sigma_{\wW_{n+2}}$-iterate. Let $\Psi_{\mathrm{tm}}$
be the restriction of $\Sigma_{\wW,\delta_{n+1}}$
to tame trees. Then:
\begin{enumerate}
  \item\label{item:wW_n+2_closed_under_strat}$\wW$ is closed under $\Psi_{\mathrm{tm}}$, and $\Psi_{\mathrm{tm}}\rest\wW$ is lightface definable over $\wW$.
  \item\label{item:wW_n+2[g]} Let $\lambda\geq\delta_{n+1}^{\wW}$,
  let $\PP\in\wW|\lambda$, and $g$ be $(\wW,\PP)$-generic (with $g$ appearing in some generic extension of $V$). Then $\wW[g]$ is closed under $\Psi_{\mathrm{tm}}$
  and $\Psi_{\mathrm{tm}}\rest\wW[g]$ is  definable over the universe of $\wW[g]$ from the parameter $x=\wW|\lambda^{+\wW}$, uniformly in $\lambda$.
 \end{enumerate}
\end{lem}
\begin{proof}
 Letting $\wW'=\Ult(\wW,\vec{e}^{\wW})=\vV_{n+2}^{\wW\down 0}$, the lemma holds for $\wW'$ in place of $\wW$, by Lemma \ref{lem:vV_n+2_self-it_tame}
 (here we use that by Lemma \ref{lem:Sigma_wW_n+2},
 $\wW'$ is a $\Sigma_{\wW}$-iterate,
 and that by Corollary \ref{cor:Sigma_wW_n+1_down_0=Sigma_M}
 (note that this is already know to hold for $\Sigma_{\wW_{n+2}}$),
$\Sigma_M=\Sigma_{\wW_{n+2}}\down 0$
and $\wW\down 0$ is a non-dropping $\Sigma_M$-iterate).

 Now deduce part \ref{item:wW_n+2_closed_under_strat}
 from the previous paragraph by arguing as in the proof of Lemma \ref{lem:*-suitable_P-con_correctness_no_short_overlaps_vV}.
 (That is, let $j:\wW\to\wW'$ be the $\vec{e}^{\wW}$-ultrapower map. Let $\Tt$ on $\wW\down(n+1)$
 be such that $j(\Tt)$ is via $\Sigma_{\wW'\down(n+1),\delta_{n+1}}$. By considering the minimal $\vec{e}$-inflation $\Uu$ of $\Tt$ (but note this only involves $\vec{e}^{\wW\down(n+1)}$, not $\vec{e}^{\wW}$)
 and $j(\Uu)$, observe that $\Tt$ is correct.)
Now deduce part \ref{item:wW_n+2[g]} from part \ref{item:wW_n+2_closed_under_strat} by absorbing generic trees by trees in the ground model much like in the proof of Lemma \ref{lem:*-suitable_P-con_correctness_no_short_overlaps_vV[g]}. 
\end{proof}

\begin{lem}[Self-iterability for iterates of $\wW_{n+2}$]\label{lem:wW_n+2_self-it}
Let $\wW$ be any non-dropping $\Sigma_{\wW_{n+2}}$-iterate. Let $\Psi=\Sigma_{\wW,\delta_{n+1}}$.
Then the conclusions of Lemma \ref{lem:wW_n+2_self-it_tame}
still hold, after replacing $\Psi_{\mathrm{tm}}$
with $\Psi$ throughout.
\end{lem}
\begin{proof}[Proof deferral]
 The proof requires $*$-translation,
 and is deferred to \cite{*-trans_add}.
\end{proof}

\begin{lem}\label{lem:N_can_iterate_vV_n+2^N_tame}
 Let $N$ be a non-dropping $\Sigma_M$-iterate.
 Let $\wW=\vV_{n+2}^N$.
 Let $\Psi_{\mathrm{tm}}$
be the restriction of $\Sigma_{\wW,\delta_{n+1}}$
to tame trees. Then:
\begin{enumerate}
  \item\label{item:N_closed_under_vV_n+2^N_strat}$N$ is closed under $\Psi_{\mathrm{tm}}$, and $\Psi_{\mathrm{tm}}\rest N$ is lightface definable over $N$, uniformly in $N$.
  \item\label{item:N[g]_closed_under_vV_n+2^N_strat} Let $\lambda\in\OR$,
  let $\PP\in N$ with $\PP\sub\lambda$, and $g$ be $(N,\PP)$-generic (with $g$ appearing in some generic extension of $V$). Then $N[g]$ is closed under $\Psi_{\mathrm{tm}}$
  and $\Psi_{\mathrm{tm}}\rest N[g]$ is  definable over the universe of $N[g]$ from the parameter $x=\pow(\lambda)^N$, uniformly in $\lambda,N$.
 \end{enumerate}
\end{lem}
\begin{proof}
 Since $N$ is a set-generic extension of $\wW$,
 this follows from Lemma \ref{lem:vV_n+2_self-it_tame} and its proof. (The only reason that it doesn't follow directly from the lemma itself is that the lemma doesn't quite give the stated definability of the restrictions of $\Psi_{\mathrm{tm}}$. But note that it suffices to see we can define $\wW$ in the manner mentioned, i.e.~lightface over $N$ if $g=\emptyset$,
 and from the parameter $x$ over $N[g]$ otherwise.
  But we can define $N$ in this manner,
 and hence also define $\wW=\vV_{n+2}^N$.)
\end{proof}

\begin{lem}
 Let $N$ be a non-dropping $\Sigma_M$-iterate.
 Let $\wW=\vV_{n+2}^N$.
 Let $\Psi=\Sigma_{\wW,\delta_{n+1}}$. Then the conclusions of Lemma \ref{lem:N_can_iterate_vV_n+2^N_tame}
 still hold, after replacing $\Psi_{\mathrm{tm}}$
 with $\Psi$ throughout.
\end{lem}
\begin{proof}[Proof deferral]
 The proof requires $*$-translation,
 and is deferred to \cite{*-trans_add}.
\end{proof}

\subsection{Conclusion}

This completes the propagation of the inductive hypotheses \ref{item:hyp_1}--\ref{item:hyp_8}, and hence the construction of all $\vV_{n},\wW_{n}$ and the iteration strategies $\Sigma_{\vV_{n}},\Sigma_{\wW_n}$.

\section{The $\om$th Varsovian models $\vV_\om$ and $\wW_\om$}

\subsection{Construction and properties of $\mM_{\infty\om}$}

 \begin{dfn}
 Let $\Tt_{n,n+1}$ be the limit length tree leading from $\vV_n$ to $\mM_{\infty n}$, $\Tt_{n,n+1}^+=\Tt_{n,n+1}\conc\Sigma_{\vV_n}(\Tt_{n,n+1})$, 
 and $j_{n,n+1}$ the iteration map
 of $\Tt_{n,n+1}^+$.
 Write $\Tt_{n+1,n+2\down n}=\Tt_{n+1,n+2}\down n$,
 which is on $\mM_{\infty n}$. Etc.
 Then
 \[ \Uu_n=\Tt_{n,n+1}^+\conc\Tt_{n+1,n+2\down n}^+\conc\Tt_{n+2,n+3\down n}^+\conc\ldots \]
 is an $\om$-stack on $\vV_n$ (which normalizes to a tree) according to $\Sigma_{\vV_n}$, of successor length, with $b^{\Uu_n}$ non-dropping.

 Let $\vV_{n\om}^+=M^{\Uu_n}_\infty$.
 Let $\vec{\vV}_\om^+=\left<\vV^+_{n\om}\right>_{n<\om}$.
 Let $\mM_{\infty\om}^+=(L[\vec{\vV}_{\om}^+],\vec{\vV}_{\om}^+)$.
\end{dfn} 

We will show that $\mM_{\infty\om}^+$
is definable inside $M$,
and relabel it as $\mM_{\infty\om}$.
 We will rearrange $\mM_{\infty\om}$ as a strategy mouse $\vV_{\om}$
 with $\om$ Woodin cardinals, fully iterable in $V$ and in $M$
 and closed under its strategy.
  We will show that the universe of $\vV_\om$ is
 the mantle of $M$, and is $\HOD^{M[G]}$, for $G$ a large enough collapse generic.

 We first show that $\mM_{\infty\om}^+$ is a class of $\vV_n$
 for each $n$ (hence in particular of $M$).
 \begin{dfn}
Work in $\vV_{n_0}$ (where $n_0<\om$).
So the sequence $\left<\vV_n\right>_{n\geq n_0}$ is a lightface class.
Given $s\in[\OR]^{<\om}$ with $\lambda<\max(s)$ and $n\in[n_0,\om)$,
say that $s$ is \emph{$n$-good}
iff for all $m\geq n$, $\Coll(\om,\max(s))$ forces
that there is a $\Tt_{m,m+1}$-cofinal branch $b$ such that
$\max(s)\in\mathrm{wfp}(M^{\Tt_{m,m+1}}_b)$
and $i^{\Tt_{m,m+1}}_b(s)=s$.
Let $\mathscr{D}$ be the class of all pairs $(n,s)$ such that $s$ is $n$-good. Given $(m,r),(n,s)\in\mathscr{D}$,
say $(m,r)\leq(n,s)$ iff $m\leq n$ and $r\sub s$.
Given $(n,s)\in\mathscr{D}$ and $\ell\leq n$,
define
\[ \gamma_{ns}=\delta_n^{\vV_n}\cap\Hull^{\vV_n|\max(s)}(s^-), \]
\[ H_{\ell;ns}=\Hull^{\vV_n|\max(s)}(s^-\cup\gamma_{ns})\cap((\vV_n\down \ell)|\max(s)) \]
and given also $(m,r)\in\mathscr{D}$ with $(m,r)\leq(n,s)$
and $\ell\leq m$, define
\[ \pi_{\ell;mr,ns}:H_{\ell;mr}\to H_{\ell;ns} \]
by
\[ \pi_{\ell;mr,ns}(x)=i^{\Tt_{m,m+1}\conc\Tt_{m+1,m+2\down m}\conc\ldots\conc\Tt_{n-1,n\down m}}_{\vec{b}}(x) \]
where $\vec{b}=(b_m,\ldots,b_{n-1})$
is any tuple of branches through $(\Tt_{m,m+1},\ldots,\Tt_{n-1,n})$
(equivalently, through $(\Tt_{m,m+1},\Tt_{m+1,m+2\down m},\ldots,\Tt_{n-1,n\down m})$)
witnessing that $r$ is $m$-good (note that $\pi_{\ell;mr,ns}$
is independent of $s$, and $s=r$ is possible).

Working in $V$,
say that $(n,s)$ is \emph{stable}
iff $j_{m,m+1}(s)=s$ for all $m\geq n$.
\end{dfn}

\begin{lem}We have:
\begin{enumerate}
 \item  For every $s\in[\OR]^{<\om}$ with $\lambda<\max(s)$, 
 there is $n<\om$ such that $(n,s)$ is stable.
 \item If $(k,s)$ is stable and $\ell\leq k\leq m\leq n$
 then $\pi_{\ell;ms,ns}\sub j_{n-1,n}\com \ldots\com j_{m,m+1}$.
 \end{enumerate}
 Therefore, working in $\vV_{n_0}$:
\begin{enumerate}[resume*]
\item for every $s\in[\OR]^{<\om}$ with $\lambda<\max(s)$,
 there is $n<\om$ such that $(n,s)\in\mathscr{D}$, 
 \item  $\leq$ is a directed partial order on $\mathscr{D}$, and
 \item if $(\ell,q)\leq(m,r)\leq(n,s)$ where $(\ell,q),(m,r),(n,s)\in\mathscr{D}$ and $k\leq\ell$, then $\pi_{k;\ell q,ns}=\pi_{k;mr,ns}\com\pi_{k;\ell q,mr}$.
 \end{enumerate}
\end{lem}

\begin{dfn}\label{dfn:covering_dl}
Work in $\vV_{n_0}$. Define
\[ \vV_{\ell\om}=\dirlim\left<H_{\ell;mr},H_{\ell;ns},\pi_{\ell;mr,ns}\right>_{(m,r)\leq(n,s)\in\mathscr{D}\text{ and }\ell\leq m}, \]
and let
\[ \pi_{\ell;ns,\infty}:H_{\ell;ns}\to \vV_{\ell\om} \]
be the direct limit map.

Given $\alpha\in\OR$, define $\alpha^*=\pi_{\ell;ns,\infty}(\alpha)$
 for any/all $\ell,n,s$ such that $\ell\leq n$, $(n,s)$ is good
 and $\alpha\in s^-$.

Work now in $V$.
Define $\chi_\ell:\vV_{\ell\om}\to \vV_{\ell\om}^+$
by setting
\[ \chi_\ell(\pi_{\ell;ns,\infty}(x))=i^{\Uu_n}_{0\infty}(x) \]
where $(n,s)$ is stable (and hence $(n,s)\in\mathscr{D}$).
\end{dfn}

\begin{rem}
 Note that $\vV_{\ell\om}$, $\pi_{\ell;ns,\infty}$,
 and $\alpha^*$ are independent of $n_0$ (determining
 the model $\vV_{n_0}$ in which we interpret the definitions).
\end{rem}

The usual considerations with Silver indiscernibles give:
\begin{lem}
 $\vV_{\ell\om}=\vV_{\ell\om}^+$
 and $\chi_\ell=\id$.
\end{lem}

\begin{dfn}
 Define $\mM_{\infty\om}$ to be the structure $(L[\left<\vV_{\ell\om}\right>_{\ell<\om}],\left<\vV_{\ell\om}\right>_{\ell<\om})$.

 Define $\delta_{n\infty}=i^{\Uu_0}_{0\infty}(\delta_n)$,
 and $\lambda_\infty=i^{\Uu_0}_{0\infty}(\lambda)=\sup_{n<\om}\delta_{n\infty}$.
\end{dfn}

\begin{lem}\label{lem:M_infty,om_not_project_across_lambda_inf}
$V_{\delta_{n\infty}}^{\mM_{\infty\om}}=V_{\delta_{n\infty}}^{\vV_{n\om}}$
for each $n<\om$. Therefore $\mM_{\infty\om}\sats$``$\delta_{n\infty}$ is Woodin'', for each $n<\om$.
\end{lem}
\begin{proof}
 Let $A\in\pow(\delta_{n\infty})\cap \mM_{\infty\om}$,
 and let $\xi$ be such that $A$ is the $\xi$th set in the $\mM_{\infty\om}$-constructibility order.
 Work in $\vV_{n+m+1}$
 where $m$ is large enough that $\xi$ is $(n+m+1)$-stable.
 By relabelling if necessary, we may assume $m=0$.
 Let
 \[k_{n+1,\om}:(\vV_{n+1}|\delta_n^{\vV_{n+1}})\to \vV_{n+1,\om}|\delta_n^{\vV_{n+1\om}} \]
 be the correct iteration map (recall $k_{n+1,\om}\in\vV_{n+1}$).
 Define
 \[ \bar{A}_{n+1}=(k_{n+1,\om}^{-1})``A.\]
 
 We claim that $i^{\Uu_{n+1}}_{0\infty}(\bar{A}_{n+1})=A$, which suffices.
 For this, it is enough to see that $j_{n+1,n+2}(\bar{A}_{n+1})=\bar{A}_{n+2}$. So let $\alpha<\delta_n^{\vV_{n+2}}=\delta_n^{\vV_{n+2\down n+1}}=\delta_n^{\mM_{\infty,n+1}}$.
 Then the following are equivalent, where
 where $e=F^{\vV_{n+2}|\gamma^{\vV_{n+2}}}$:
 \begin{enumerate}[label=(\roman*)]
  \item 
$\alpha\in\bar{A}_{n+2}$ 
\item $k_{n+2,\om}(\alpha)\in A$
\item $i_e^{\vV_{n+2}}(k_{n+2,\om}(\alpha))\in i_e^{\vV_{n+2}}(A)$ 
 \item $i_e^{\vV_{n+2}}(k_{n+2,\om})(i_e^{\vV_{n+2}}(\alpha))\in i_e^{\vV_{n+2}}(A)$
 \item\label{item:no_project_5} $(k_{n+2,\om})^{\vV^{\vV_{n+2\down n+1}}}(i_e^{\vV_{n+2}}(\alpha))\in i_e^{\vV_{n+2}}(A)$
 \item\label{item:no_project_6} $k_{n+1,\om}^{\vV_{n+2\down n+1}}(\alpha)\in j_{n+1,n+2}(A)$
 \item $\alpha\in j_{n+1,n+2}(\bar{A}_{n+1})$.
 \end{enumerate}
Here the equivalence between \ref{item:no_project_5}
and \ref{item:no_project_6}
is because
\[ (k_{n+2,\om})^{\vV^{\vV_{n+2\down n+1}}}(i_e^{\vV_{n+2}}(\alpha))=k_{n+1,\om}^{\vV_{n+2\down n+1}}(\alpha),\]
since the maps here are (restrictions of) correct iteration maps, which commute and appropriately agree,
and
\[ i_e^{\vV_{n+2}}(A)=j_{n+1,n+2}(A),\]
since:
\begin{enumerate}[label=--]
 \item  $i_e^{\vV_{n+2}}(A)$ is the $\xi$th set in the $(\mM_{\infty\om})^{\Ult(\vV_{n+2},e)}$-constructibility order
(as $\xi=i_e^{\vV_{n+2}}(\xi)$ by stability),
\item  $j_{n+1,n+2}(A)$
is the $\xi$th set in the $(\mM_{\infty\om})^{\vV_{n+2\down n+1}}$-constructibility order (as $\xi=j_{n+1,n+2}(\xi)$ by stability), and
\item  $(\mM_{\infty\om})^{\Ult(\vV_{n+2},e)}=(\mM_{\infty\om})^{\vV_{n+2\down n+1}}$ since $\Ult(\vV_{n+2},e)=\vV^{\vV_{n+2\down n+1}}$.\qedhere
\end{enumerate}
\end{proof}

\begin{dfn}
For $n<\om$, define
 $\mM_{\infty}^{\vV_{n\om}}$ over $\vV_{n\om}$ just as $\mM_{\infty,n+1}$ is defined over $\vV_n$ (using trees based
 on $\vV_{n\om}|\delta_{n+1}^{\vV_{n\om}}$, and the corresponding internal  covering system). Likewise define $\vV^{\vV_{n\om}}$,
 so $\vV^{\vV_{n\om}}$ has universe $\mM_{\infty,n+1}[*]$ for the appropriate $*$-map. Similarly,
 for $m,n<\om$, let $(\vV_{m})^{\vV_{n\om}}$
be defined over $\vV_{n\om}$ just as $\vV_{m}$ is defined over $\vV_n$.
Here we can have $m<n$, $m=n$ or $m>n$.
When $m<n$, we mean that $(\vV_m)^{\vV_{n\om}}$ is just $\vV_{n\om}\down m$, $(\vV_n)^{\vV_{n\om}}=\vV_{n\om}$, $(\vV_{n+1})^{\vV_{n\om}}=\vV^{\vV_{n\om}}$, $(\vV_{n+2})^{\vV_{n\om}}=\vV^{(\vV^{(\vV_{n\om})})}$, etc.
And likewise define $(\vV_{m\om})^{\vV_{n\om}}$;
here if $m<n$ then $(\vV_{m\om})^{\vV_{n\om}}=(\vV_{n\om})^{\vV_{n\om}}\down m$, etc.\end{dfn}

Note that because of the relationship between $\vV_{n\om}$ and $\vV_{n+1,\om}$, in particular that $\vV_{n+1,\om}$ is not a generic
extension of $\vV_{n\om}$,
there doesn't seem to be any reason to expect that $\vV_{m\om}^{\vV_{n\om}}=\vV_{m\om}^{\vV_{n',\om}}$ when $n'\neq n$. 

The usual arguments do not seem to suffice to show that, e.g.,
 $\mM_\infty^{\vV_{n\om}}$ is a correct iterate of $\vV_{n\om}$,
 because of a lack of soundness (that is,
 $\vV_{n\om}$ is not the hull of $\mathscr{I}^{\vV_{n\om}}\cup\delta_{n+1}^{\vV_{n\om}}$).

 \begin{dfn}\label{dfn:e-vec}
  Let $\vV$ be an $n$-pVl
  Let $k\leq\ell\leq n$.
  Then $\vec{e}^{\vV}_{k\ell}$
  is the extender of the iteration
  \[ (\left<M_i\right>_{k\leq i\leq\ell},\left<e_i\right>_{k\leq i<\ell}), \]
  where $M_k=\vV$, $M_{i+1}=\Ult(M_i,e_i)$ for $i\in[k,\ell)$,
  and $e_i=e_i^{M_i}$.
 \end{dfn}

 Like before we have:
 \begin{lem}\label{lem:vV_n^vV_k,om_as_Ult_of_vV_n,om}
  Let $k\leq n<\om$.
  Then $\Ult(\vV_{n\om},\vec{e}^{\vV_{n\om}}_{kn})=(\vV_n)^{\vV_{k\om}}$, $(\vV_n)^{\vV_{k\om}}$ is a $\Sigma_{\vV_{n\om}}$-iterate of $\vV_{n\om}$, and each of the (finitely many) extenders in the tuple $\vec{e}^{\vV_{n\om}}_{kn}$
  is the extender of a correct iteration map.
 \end{lem}

\begin{lem}\label{lem:Vv_m,om^V_n,om_is_correct_iterate}
For $n\leq m<\om$,
$(\vV_{m\om})^{\vV_{n\om}}$ is a correct iterate of $\vV_{m\om}$.
\end{lem}
\begin{proof}
This will follow  from Lemma \ref{lem:vV_n^vV_k,om_as_Ult_of_vV_n,om} and that the $\vV_{n\om}$ are $\lambda_\infty$-sound.

Working in $\vV_{n\om}$,
 $\left<\vV_{m\down k}^{\vV_{n\om}}\right>_{m\geq\max(n,k)}$
 is used to define $\vV_{k\om}^{\vV_{n\om}}$
 via the direct limit of the covering system defined as
above. But by Lemma \ref{lem:vV_n^vV_k,om_as_Ult_of_vV_n,om}
and by conservativity of the iteration strategies,
$\left<\vV_{m\down k}^{\vV_{n\om}}\right>_{m\geq\max(n,k)}$ is a stack induced by $\Sigma_{\vV_{k\om}}$ via normalization. But since $\vV_{n\om}$
and each $\vV_{m\down k}^{\vV_{n\om}}$
as above is $\lambda^{\vV_{n\om}}$-sound,
the direct limit $\vV_{k\om}^{\vV_{n\om}}$
of the covering system is just the same as the direct limit under the iteration maps. This completes the proof.
\end{proof}

\begin{dfn}
 Let $\ell_n:\vV_{n\om}\to(\vV_{n\om})^{\vV_{n\om}}$ be the correct iteration map.
\end{dfn}

Recall  the $*$-map $\alpha\mapsto\alpha^*$ introduced in Definition \ref{dfn:covering_dl}.

\begin{lem}\label{lem:alpha^*=lim_n_n(alpha)}
 For all $\alpha\in\OR$, there is $m<\om$ such that $\alpha^*=\ell_n(\alpha)$
 for all $n\geq m$.
\end{lem}
\begin{proof}
 Fix $(n,s)$ such that $s\in[\mathscr{I}^M]^{<\om}$
 and $\alpha\in\rg(\pi_{n;ns,\infty})$ and $(n,\{\alpha\})$ is stable.
Let $\bar{\alpha}\in H_{n;ns}$ be such that $\pi_{n;ns,\infty}(\bar{\alpha})=\alpha$.
 Then $\vV_n\sats$ ``$\alpha=\pi_{n;ns,\infty}(\bar{\alpha})$'',
 so $\vV_{n\om}\sats$``$\alpha^*=\pi_{n;n s^*,\infty}(\alpha)$'',
 and $s^*=s\in[\mathscr{I}^{\vV_{n\om}}]^{<\om}$,
but then the usual considerations give that $\ell_{n}(\alpha)=\alpha^*$.
\end{proof}

\begin{lem}\label{lem:alpha^*_as_ev_it_map} We have:
\begin{enumerate}
 \item\label{item:j_n_M_infty,om-def} $\ell_n$ is a lightface class of $\mM_{\infty\om}$,  uniformly in $n$. \item\label{item:star_is_M_infty,om-def} $\alpha\mapsto\alpha^*$ is a lightface class of $\mM_{\infty\om}$.
 \end{enumerate}
\end{lem}
\begin{proof}Part \ref{item:j_n_M_infty,om-def} is by Lemmas \ref{lem:vV_n^vV_k,om_as_Ult_of_vV_n,om}
and \ref{lem:Vv_m,om^V_n,om_is_correct_iterate} and their proofs. (Note they show 
that $\mM_{\infty\om}$
can define the iteration map $\vV_k^{\vV_{n\om}}\to\mM_{\infty k}^{\vV_{n\om}}$, uniformly in $n<\om$ and $k\in[n,\om)$.)
 Part \ref{item:star_is_M_infty,om-def} follows from this and Lemma \ref{lem:alpha^*=lim_n_n(alpha)}.
\end{proof}

We now adapt \cite[4.25--4.28***]{vm2_v2}, leading to Lemma \ref{lem:HOD_P} below:

\begin{rem}
Work in $M[G]$ where $G$ is $(M,\Coll(\om,\lambda))$-generic.
Then $M$ is definable over $M[G]$ from the parameter $M|\lambda^{+M}$
(and note $\lambda^{+M}=\om_1^{M[G]}$);
in fact,
 $M$ is the direct condensation stack above $M|\lambda^{+M}$(see \cite[Definition 4.2]{V=HODX_pub}), as computed in $M[G]$.
\end{rem}

\begin{dfn}\label{dfn:pP}Fix the natural $\Coll(\om,{<\lambda})$-name $\dot{\mathbb{R}}^*\in M$ for the set of symmetric reals in the $M[\dot{G}]$ where $\dot{G}$ is the standard name for the $\Coll(\om,{<\lambda})$-generic.
(Here by \emph{symmetric reals},
we just mean the reals in $\bigcup_{\alpha<\lambda}M[G\rest\alpha]$.) Now work in $M[G]$ where $G$ is $(M,\Coll(\om,{<\lambda}))$-generic.
Let $\mathscr{P}$ denote the set of premice $P$
such that $\OR^P=\om_1=\lambda^{+M}$, $P$ has similar first-order properties as does
 $M|\lambda^{+M}$, the direct condensation stack $C$ above $P$
is proper class and has standard condensation properties,
there is a $(C,\Coll(\om,{<\lambda}))$-generic filter $H$
such that $C[H]$ is the entire universe, 
there is $n<\om$ and $x\in\dot{\mathbb{R}}^*_G$ such that
$\vV_{n+1}^C=\vV_{n+1}$, $\kappa_n^{+P}=\delta_n^{\vV_{n+1}^C}$ and $x$ codes $P|\kappa_n^P$.
\end{dfn}

\begin{lem}\label{lem:HOD_P}
Let $G$ be $(M,\Coll(\om,\lambda))$-generic. Then
 $\mM_{\infty\om}$ and $\HOD^{M[G]}_{\{\mathscr{P}\}}$ have the same universe.
\end{lem}
\begin{proof}
We write $\HOD$ for $\HOD^{M[G]}_{\{\mathscr{P}\}}$.

We have $\mM_{\infty\om}\sub\HOD$ because $M|\lambda^{+M}\in\mathscr{P}$
and from each $P\in\mathscr{P}$
we compute $\mM_{\infty\om}$ in the same manner as from $M|\lambda^{+M}$:
first compute the direct condensation stack above $P$, then compute
the sequence $\left<\vV_n^P\right>_{n<\om}$,
and compute $\mM_{\infty\om}^P$ from this;
the tail agreement with $\left<\vV_n\right>_{n<\om}$
ensures that $\mM_{\infty\om}^P=\mM_{\infty\om}$.

Now let $A\sub\alpha\in\OR$ with $A\in\HOD$,
and let $\varphi$ be a formula and $\eta\in\OR$
be such that
\[ \beta\in A\iff M[G]\sats\varphi(\mathscr{P},\eta,\beta) \]
for all $\beta\in\OR$.
Given any model $N$ of the form of a $\vV_n$, we define $\mathscr{P}^N$
over $N[H]$, where $H$ is $N$-generic
for $\Coll(\om,\lambda^N)$,
just like we defined $\mathscr{P}$ over $M[G]$
(now using the sequence
$\left<\vV_m^N\right>_{m<\om}$).
Note that if $N|(\lambda^N)^{+N}\in\mathscr{P}$ then $\mathscr{P}^N=\mathscr{P}$. Working in $N$, we write $\dot{\mathscr{P}}$ to refer to $\mathscr{P}^N$.

Now the following are equivalent:
\begin{enumerate}[label=(\roman*)]
 \item 
$\beta\in A$
\item $\exists n<\om\ \vV_n\sats\text{``}\Coll(\om,\lambda)\forces\varphi(\dot{\mathscr{P}},\eta,\beta)\text{''}$
\item $\all n<\om\ \vV_n\sats\text{``}\Coll(\om,\lambda)\forces\varphi(\dot{\mathscr{P}},\eta,\beta)\text{''}$
\item\label{item:def_A_via_*} $\exists m<\om\ \all n\geq m\ \vV_{n\om}\sats\text{``}\Coll(\om,\lambda^*)\forces
\varphi(\dot{\mathscr{P}},\eta^*,\beta^*)\text{''}$.
\end{enumerate}
But working in $\mM_{\infty\om}$, we can define $\vec{\vV}_{\om}$
and $*:\OR\to\OR$ (using Lemma \ref{lem:alpha^*_as_ev_it_map}),
so we can define $A$ using \ref{item:def_A_via_*}.
\end{proof}

\begin{dfn}
Define
the structure $\mM_{\infty\om}|\lambda_\infty$
with ordinal height $\lambda_\infty$
and extender sequence $\es$
such that $\es\rest\Delta_n^{\vV_{n\om}}=\es^{\vV_{n\om}}|\Delta_n^{\vV_{n\om}}$.
\end{dfn}

By Lemma \ref{lem:M_infty,om_not_project_across_lambda_inf}, this structure has universe $V_{\lambda_\infty}^{\mM_{\infty\om}}$.

Like at the finite stages, we have the following crucial fact:

\begin{lem}\label{lem:hull_conservativity_vV_om}
We have:
\begin{enumerate}
\item\label{item:i_m_hull_conservativity} Let $i_m:\wW_m\to\vV_{m\om}$ be the $\Sigma_{\wW_m}$-iteration map.
 Let $H=\Hull^{\mM_{\infty\om}}(\mathscr{I}^M)$.
 Then $H\cap\vV_{m\om}=\rg(i_m)$. 
 
\item\label{item:j_m_hull_conservativity} Let $j_m:\vV_m\to\vV_{m\om}$
 be the $\Sigma_{\vV_m}$-iteration map.
 Let $H=\Hull^{\mM_{\infty\om}}(\rg(j_m))$.
 Then $H\cap\vV_{m\om}=\rg(j_m)$.
 \end{enumerate}
\end{lem}
\begin{proof}
Part \ref{item:i_m_hull_conservativity}:
The proof is basically like usual.
That is, let $s\in[\mathscr{I}^M]^{<\om}$ and $t$ be a term and $\alpha\in\OR$  be such that
\[ \mM_{\infty\om}\sats\alpha=t(s).\]
Let $j_n:\vV_n\to\vV_{n\om}$ and $j_{nk}:\vV_n\to\vV_{k\down n}$ be the $\Sigma_{\vV_n}$-iteration maps.
Let $n\in[m,\om)$ be large enough that $j_{nk}(\alpha)=\alpha$ for all $k\in[n,\om)$ and $\alpha\in\rg(j_n)$. We have that $\vV_{n}=\Hull^{\vV_{n}}(\mathscr{I}^M\cup\delta_n^{\vV_{n}})$,
so let $s'\in[\mathscr{I}^M]^{<\om}$ and $\gamma<\delta_n^{\vV_{n\om}}$ and $t'$ be a term $\bar{\alpha}\in\OR$ be such that
\[\vV_{n}\sats\bar{\alpha}=t'(s',\gamma) \]
and
\[ j_{n}(\bar{\alpha})=\alpha.\]
We may assume that $s=s'$.
Let $s^+=s\cup\{\xi\}$ where $\xi\in\mathscr{I}^{M}$
with $\max(s)<\xi$.
Let $H_{s^+}=\Hull^{\vV_{n}|\xi}(s\cup\delta_n^{\vV_{n}})$. Let $X=j_n``\delta_n^{\vV_n}$, so $X\in\vV_n$.
Let $H'_{s^+}=\Hull^{\vV_{n\om}|\xi}(s\cup X)$
and let
\[ \pi:H_{s^+}\to H'_{s^+} \]
be the isomorphism. So $\pi\in \vV_n$ and $\pi\sub j_n$
and $\bar{\alpha}\in H_{s^+}$, so $\pi(\bar{\alpha})=\alpha$. Note that $\pi\in\Hull^{\vV_n}(\mathscr{I}^M)$. But $j_n(\pi)(\alpha)=\alpha^*$,
and $\alpha^*\in\Hull^{\vV_{n\om}}(\mathscr{I}^M)$
as usual, but $j_n(\pi)\in\Hull^{\vV_{n\om}}(\mathscr{I}^M)$, and therefore $\alpha$ is also in that hull.

But finally by the conservativity of $\Sigma_{\vV_n}$
over $\Sigma_{\vV_{n\down m}}$,
we have $\Hull^{\vV_{n\om}}(\mathscr{I}^M)\cap\OR=\Hull^{\vV_{m\om}}(\mathscr{I}^M)\cap\OR$,
and therefore $\alpha\in\Hull^{\vV_{m\om}}(\mathscr{I}^M)$, as desired.

Part \ref{item:j_m_hull_conservativity} is similar.
\end{proof}

\begin{lem}\label{lem:M_infty,om_functions_repd_mod_measure_one}
 Let $E\in\es^{\mM_{\infty\om}|\lambda_\infty}$ be
 a short $\mM_{\infty\om}$-total extender and $a\in[\nu(E)]^{<\om}$. Let $\kappa=\crit(E)$
 and  $f\in \mM_{\infty\om}$ be such that $f:[\kappa]^{|a|}\to\OR$.
 Let $n<\om$ be such that $\lh(E)<\delta_n^{\mM_{\infty\om}}$.
 Then there is $g\in\vV_{0\om}$ such that $\{u\bigm|f(u)=g(u)\}\in E_a$.
\end{lem}
\begin{proof}
By related properties regarding
the relationship between $\vV_{0n}$
and $\vV_n$, it suffices to see there is such a $g\in\vV_{n\om}$.
 Suppose not and let $(f,E,a,n)$ be the $<_{\mM_{\infty\om}}$-least counterexample.
 Let $H=\Hull^{\mM_{\infty\om}}(\rg(j_n))$
 and $C$ its transitive collapse.
 Then by Lemma \ref{lem:hull_conservativity_vV_om}(\ref{item:j_m_hull_conservativity}), $\vV_n\sub C$ and letting $\pi:C\to \mM_{\infty\om}$ be the uncollapse map,
 we have $j_n\sub\pi$.
 By minimality,
 we have $(f,E,a)\in\rg(\pi)$, so let $\pi(\bar{f},\bar{E},\bar{a})=(f,E,a)$.
 By elementarity, $(\bar{f},\bar{E},\bar{a})$
 is also a counterexample with respect to $C,\vV_n$.
 Let $U=\Ult(\vV_n,\bar{E})$.
 Let $k:U\to\vV_{n\om}$ be the correct iteration map. We have $\pi(\bar{f})=f$.
 Let $g\in\vV_n$ and $c\in[\delta_n^{\vV_{n\om}}]^{<\om}$  be such
 that $g:[\delta_n^{\vV_n}]^{|c|}\to\OR$
 and $k(\bar{a})\sub c$ and \[\pi(g)(c)=j_n(g)(c)=f(k(\bar{a}))=\pi(\bar{f})(k(\bar{a})).\]Let $\bar{f}':[\delta_n^{\vV_n}]^{|c|}\to\OR$ be  defined $\bar{f}'(u)=\bar{f}(u^{c,k(\bar{a})})$.
 Letting $F$ be the $(\delta_n^{\vV_n},\delta_n^{\vV_{n\om}})$-extender
 derived from $j_n$ (equivalently, from $\pi$),
 then \[A=\{u\in[\delta_n^{\vV_n}]^{|c|}\bigm|\bar{f}'(u)=g(u)\}\in F_c.\] Also $\bar{E}_{\bar{a}}=F_{k(\bar{a})}$, and $\bar{E}_{\bar{a}}$
 is the projection of $F_{c}$ on $k(\bar{a})$.
 And $\bar{f}'(u)=\bar{f}'(u_1)$ whenever
 $u\in A$ and $u^{c,k(\bar{a})}=(u_1)^{c,k(\bar{a})}$. Let $B=\{v^{c,k(\bar{a})}\bigm|v\in A\}$ and let
  $g':B\to\OR$ be defined $g'(u)=g(v)$ for any $v\in A$ such that $u=v^{c,k(\bar{a})}$.
  Then  $g,A,B,g'\in\vV_n$,
and $\pi(g')(k(\bar{a}))=\pi(g)(c)=\pi(\bar{f})(k(\bar{a}))$, so the existence of $g'$
contradicts that $(\bar{f},\bar{E},\bar{a})$
is a counterexample.
\end{proof}

\begin{dfn}
 A \emph{\tu{(}putative\tu{)} $\mM_{\infty\om}|\lambda_\infty$-based short-normal tree on $\mM_{\infty\om}$}
 is a $0$-maximal iteration tree on $\mM_{\infty\om}$
 which uses only short extenders from
 $\mM_{\infty\om}|\lambda_\infty$ and its images. 
 
 The \emph{short-normal $\mM_{\infty\om}|\lambda_\infty$-based $\vec{e}$-minimal pullback strategy}
 $\Psi_{\mM_{\infty\om},\lambda_\infty}^{\mathrm{sn}}$
is the short-normal  iteration strategy for $\mM_{\infty\om}$ which is just the union of the short-normal $\vec{e}$-minimal pullback strategies
for $\vV_{n\om}$, for trees based on $\vV_{n\om}|\Delta^{\vV_{n\om}}$, over all $n<\om$.
\end{dfn}
\begin{lem}\label{lem:Psi_M_infty,om,sn_conservative_ext}
 $\Psi_{\mM_{\infty\om},\lambda_\infty}^{\mathrm{sn}}$
 is an iteration strategy,
 and it is conservative over $\Sigma_{\vV_{0\om}}$.
\end{lem}
\begin{proof}
 This follows directly from Lemma \ref{lem:M_infty,om_functions_repd_mod_measure_one} and properties of the $\Sigma_{\wW_n}$.
\end{proof}
Recall that $\ell_n:\vV_{n\om}\to(\vV_{n\om})^{\vV_{n\om}}$
is the iteration map. 
Using the previous lemma, we can define:
\begin{dfn}
Let $\ell_n^+:\mM_{\infty\om}\to (\mM_{\infty\om})^{\vV_{n\om}}$
be the elementary extension of $\ell_n$;
equivalently, if $\Tt$ is the $\Sigma_{\vV_{n\om}}$-tree
leading to $(\vV_{n\om})^{\vV_{n\om}}$ (so with iteration map $\ell_n$) then note there is a tree $\Tt'$ on $\mM_{\infty\om}$, via $\Psi_{\mM_{\infty\om},\lambda_\infty}^{\sn}$, such that $\Tt'\down 0=\Tt$,
and we define $\ell_n^+=i^{\Tt'}_{0\infty}$.
Also equivalently, $\ell_n^+$
is the iteration map given by the iteration of $\mM_{\infty\om}$ with models $\left<M_i\right>_{i\leq\om}$
and extenders $\left<E_i\right>_{i<\om}$
where $M_0=\mM_{\infty\om}$ and $E_i=e_{n_i}^{M_i}$.
(Then note that $M_0\down n=\vV_{n\om}$,
$M_1\down(n+1)=\vV_{n+1}^{\vV_{n\om}}$, etc.)
\end{dfn}

\begin{lem}\label{lem:k_n^+_def} We have:
\begin{enumerate}
\item\label{item:k_n^+_unif_def_over_M_infty,om}  $\left<\ell_n^+\right>_{n<\om}$ is a lightface class
 of $\mM_{\infty\om}$,
\item\label{item:k_n^+_rest_lambda_infty} $\left<\ell_n^+\rest(\mM_{\infty\om}|\lambda_\infty)\right>_{n<\om}$
is definable in the codes over $\mM_{\infty\om}|\lambda_\infty$,
\item\label{item:k_n^+_rest_P|lambda_infty^+} if $S\sub \mM_{\infty\om}$ is a fine structural strategy premouse such that $\mM_{\infty\om}|\lambda_\infty\pins S$ and $S\sats$``$(\lambda_\infty)^{+}$ exists'',
then $S|(\lambda_\infty^{+S})$ is closed under each $\ell_n^+$, and $\left<\ell_n^+\rest (S|(\lambda_\infty)^{+S})\right>_{n<\om}$ is a lightface class of $S|(\lambda_\infty)^{+S}$.
\end{enumerate}
\end{lem}
\begin{proof}
Parts \ref{item:k_n^+_unif_def_over_M_infty,om}
and \ref{item:k_n^+_rest_lambda_infty}: By Lemma \ref{lem:Psi_M_infty,om,sn_conservative_ext} and the proof of Lemma \ref{lem:Vv_m,om^V_n,om_is_correct_iterate}.

Part \ref{item:k_n^+_rest_P|lambda_infty^+}:
This follows from part \ref{item:k_n^+_rest_lambda_infty},
using that $\mM_{\infty\om}|\lambda_\infty\pins S$ and $V_{\lambda_\infty}^S=V_{\lambda_\infty}^{\mM_{\infty\om}}$.
\end{proof}

\subsection{Construction of $\vV_\om$ and equivalence with $\mM_{\infty\om}$}\label{sec:con_vV_om}

The fine structure of $\mM_{\infty\om}$ above
$\lambda^\infty$ is, so far, unclear. We now rearrange $\mM_{\infty\om}$
as a fine structural strategy mouse $\vV_\om$ extending $\mM_{\infty\om}|\lambda^\infty$.
Analogously to the successor case,
$\es^{\vV_\om}\rest(\lambda^\infty,\infty)$ will  just be the restriction of $\es^M$, and $\vV_\om$ will be a set-ground of $M$. The forcing for this will be a certain instance of Vopenka.

However, (noting that $\lambda<\lambda_\infty<\lambda^{+M}=(\lambda_\infty)^{+M}=(\lambda_\infty)^{+\mM_{\infty\om}}$)
the interval $(\lambda_\infty,\lambda_\infty^{+\mM_{\infty\om}})$ is more subtle, and here we will need to use various local versions of Vopenka and the P-construction
producing $\vV_\om$ there.

The argument we use is motivated by that of \cite[\S7]{odle_v2} (also see \cite[\S9]{odle_v2}),
to which it is also very similar.
The author's original setup involved a variant $\mathscr{P}'$ of $\mathscr{P}$ from
\ref{dfn:pP}, in which there is no requirement
 that some $x\in\dot{\mathbb{R}}^*_G$ codes $P|\kappa_n^P$ (but of course, it is simply coded by some $x\in\mathbb{R}^{M[G]}$) (thus $\dot{\mathbb{R}}^*$ was irrelevant). One can still prove
the version of Theorem \ref{lem:HOD_P} given by replacing $\mathscr{P}$ by $\mathscr{P}'$. However, it is not clear that the Vopenka forcing associated to $\HOD_{\mathscr{P}'}$
is a subset of $\lambda^{+M}$; it is (or can be taken as) a subset of $\lambda^{++M}$. This made the kind of analysis to follow (where the Vopenka forcing associated to $\mathscr{P}$ is indeed a subset of $\lambda^{+M}$) more complicated. Before the author got to writing the original version up in detail, Ralf Schindler
suggested employing another forcing,
which has size $\lambda^{+M}$, which significantly simplified the analysis. Later, the author noticed that by using $\mathscr{P}$ instead of $\mathscr{P}'$,
the resulting Vopenka has size $\lambda^{+M}$,
which then results in a similarly simplified analysis. We will use the latter Vopenka forcing (associated to $\mathscr{P}$).

For this section, fix $G$ which is $(M,\Coll(\om,\lambda))$-generic.
 Recall that $\mathscr{P}\in M[G]$ was
introduced in Definition \ref{dfn:pP}.
\begin{dfn}
Work in $M[G]$.
For $P\in\mathscr{P}$ let $\ds(P)$ denote the direct condensation stack above $P$ (cf.~\cite[Definition 4.2]{V=HODX_pub} and Definition \ref{dfn:pP}).

Let $\varphi$ be a formula and $\lambda^{\mM_{\infty\om}}<\alpha<\beta<\lambda^{+M})$. Then $A_{\beta\alpha\varphi}$ denotes the set of all $P\in\mathscr{P}$ such that $\ds(P)|\beta\sats\varphi(\alpha)$.

We say that $A\sub\mathscr{P}$ is \emph{nice}
if $A=A_{\beta\alpha\varphi}$ for some such $\beta,\alpha,\varphi$.
\end{dfn}

\begin{lem}
Work in $M[G]$ and let $A\sub\mathscr{P}$.
Then $A$ is $\OD_{\{\mathscr{P}\}}$ iff $A$ is nice.
\end{lem}
\begin{proof}
Clearly every nice $A$ is $\OD_{\{\mathscr{P}\}}$
(recall that for $P\in\mathscr{P}$,
we have $\OR^P=\lambda^{+M}$, so the definability of $\ds(P)$ from $P$ is irrelevant).
The other direction is much like the proof of \cite[Lemma 7.3, Claim 1]{odle_v2}.
\end{proof}

\begin{dfn}[$\mathscr{P}$-Vopenka]
Let $S$ be a strategy premouse such that $S\sub \mM_{\infty\om}$, $S$ extends
$\mM_{\infty\om}|\lambda_\infty$ and  $(\lambda_\infty)^{+S}<\OR^S$.
Then $\Vop^S_{\mathscr{P}}$ denotes the forcing such that:
\begin{enumerate}[label=--]
 \item 
the conditions of $\Vop_{\mathscr{P}}^S$ are the
tuples $(\beta,\alpha,\varphi)$
such that $\varphi$ is a formula
and $\alpha<\beta<\lambda^{+S}$
and $A_{\beta\alpha\varphi}\neq\emptyset$, and
\item given $p,p'\in\Vop^S_{\mathscr{P}}$,
letting $p=(\beta,\alpha,\varphi)$
and $p'=(\beta',\alpha',\varphi')$, then
\[ p\leq p'\Leftrightarrow A_{\beta\alpha\varphi}\sub A_{\beta'\alpha'\varphi'}.\]
\end{enumerate}
Given $p=(\beta,\alpha,\varphi)\in\Vop^S_{\mathscr{P}}$, let $A_p=A_{\beta\alpha\varphi}$.
\end{dfn}

\begin{lem}
$\Vop^S_{\mathscr{P}}$ is definable over $S|(\lambda_\infty)^{+S}$. In particular, $\Vop^S_{\mathscr{P}}\in S$.\end{lem}
\begin{proof}
 by Lemma \ref{lem:k_n^+_def} and since
the references to truth in $\vV_{n\om}$
only really depend on $\vV_{n\om}|\lambda_\infty$. 

For $n<\om$, let $\dot{\mathscr{P}}^{\vV_n}\in\vV_n$
be the natural $\Coll(\om,{<\lambda})$-name
for $\mathscr{P}$.

Let $p=(\beta,\alpha,\varphi)$ be such that $\alpha<\beta<(\lambda_\infty)^{+S}$
and $\varphi$ is a formula.

Then the following are equivalent:
\begin{enumerate}[label=--]
 \item 
$p\in\Vop_{\mathscr{P}}^S$, \item $A_p=A_{\beta\alpha\varphi}\neq\emptyset$
\item  for any/all $n<\om$,  we have
$\vV_n\sats\ \sststile{\Coll(\om,{<\lambda})}{}\exists P\in\dot{\mathscr{P}}^{\vV_n}\ \Big[(P|\beta)\sats\varphi(\alpha)\Big]$,
\item $\exists m<\om\ \all n\geq m\ \Big(\vV_{n\om}\sats\ \sststile{\Coll(\om,{<\lambda_\infty})}{}\exists P\in\dot{\mathscr{P}}^{\vV_{n\om}}\ \Big[(P|\beta^*)\sats\varphi(\alpha^*)\Big]\Big)$.
\end{enumerate}
So by Lemmas \ref{lem:alpha^*=lim_n_n(alpha)} and \ref{lem:k_n^+_def}, $S|(\lambda_\infty)^{+S}$
can define the set of conditions of $\Vop_{\mathscr{P}}^S$. 
It can similarly define the ordering of $\Vop_{\mathscr{P}}^S$.
\end{proof}

\begin{lem}\label{lem:lambda^+-cc} Let $S$ be a strategy premouse such that $S\sub\mM_{\infty\om}$ and there is some $P\ins M$ with $\lambda<\OR^P$,
 $P\sats$``$\lambda^+$ exists'',
$\lambda^{+P}=(\lambda_\infty)^{+S}$
and $S\sub P$.
Then $\Vop_{\mathcal{P}}^S$ is $((\lambda_\infty)^+)^S$-cc in $S$.\end{lem}
\begin{proof}Note that  $p,q\in\Vop^S_{\mathcal{P}}$ are compatible iff $A_p\cap A_q\neq\emptyset$.
Let $X\in\pow(\Vop^S_{\mathscr{P}})\cap S$ have cardinality $(\lambda_\infty)^{+S}$ in $S$.
Then working in $P[G]$, where we have both $X$ and
$\mathscr{P}$, and $\mathscr{P}$  has cardinality $\lambda_\infty$,
there is $P\in\mathscr{P}$ such that $P\in A_p$ for $\lambda^{+P}$-many $p\in\Vop_{\mathscr{P}}^S$. This suffices.\end{proof}

\begin{lem}\label{lem:M|lambda_is_P-generic}
Let $S$ be as in Lemma \ref{lem:lambda^+-cc}.
Suppose further that $S|(\lambda_\infty)^{+S}$
is level-by-level definable over each $P\in\mathscr{P}$, uniformly in $P$;
that is, there is a fixed formula $\psi$
such that for all $P\in\mathscr{P}$ and all $\alpha,\beta<(\lambda_\infty)^{+S}$
with $\lambda_\infty<\alpha<\beta$, and all formulas $\varphi$,
we have
\[ S|\beta\sats\varphi(\alpha)\Leftrightarrow P\sats\psi(\beta,\alpha,\varphi).\]

Let $G$ be the set of all conditions
$p\in\Vop_{\mathscr{P}}^S$ such that $M|\lambda^{+M}\in A_p$ \tu{(}equivalently,
$M|(\lambda_\infty)^{+S}\in A_p$\tu{)}.
Then $G$ is an $(S,\Vop^S_{\mathscr{P}})$-generic filter. Moreover, $S[G]$ is equivalent to $S[M|(\lambda_\infty)^{+S}]$.\end{lem}
\begin{proof} Clearly $G$ is a filter.
So let $X\in\pow(\Vop^S_{\mathscr{P}})\cap S$ be a maximal antichain. We need to see that there is some $p\in X$ such that $M|\lambda^{+M}\in A_p$. 

By Lemma \ref{lem:lambda^+-cc},
$X$ has cardinality $\lambda_\infty$ in $S$.
So $X\in S|(\lambda_\infty)^{+S}$.
Let $\alpha<(\lambda_\infty)^{+S}$
be such that $X$ is the $\alpha$th
set in the $S$-constructibility order,
and let $\beta\in(\alpha,(\lambda_\infty)^{+S})$.

Now if $M|\lambda^{+M}\notin A_p$
for all $p\in X$, 
then we get a condition $q=(\beta,\alpha,\varrho)\in\Vop_{\mathscr{P}}^S$
by taking $\varrho(v)$ to be the natural formula in the language of passive premice, whose only free variable is $v$ (to be interpreted as $\alpha$), which means
``Let $Y$ be the $v$th set in the constructibility order of $S|\OR$.
Then there is no $(\beta',\alpha',\varphi')\in Y$
such that $L[\es]|\beta'\sats\varphi(\alpha')$'',
when interpreted by any $P\in\mathscr{P}$
and with $v$ set to $\alpha$.
Here the reference to $S$ in $\varrho$
is achieved through use of the hypothesized formula $\psi$. (Note that by the contradictory hypothesis, $M|\lambda^{+M}\in A_q$.)
But $A_q\cap A_p=\emptyset$
for all $p\in X$, contradicting the maximality of $X$. This completes the proof.
\end{proof}

\begin{dfn}
We now define $\vV_\om$, setting $\mM_{\infty\om}|\lambda_\infty\ins\vV_\om$,
and then restricting $M$'s extender sequence above $\lambda_\infty$
to form $\es^{\vV_\om}$ above $\lambda_\infty$ (modulo shifting the domain of extenders overlapping $\lambda_\infty$ in the usual manner). That is, $\es^{\vV_\om}$ is defined as follows (recalling that $\mM_{\infty\om}|\lambda_\infty$ is definable in the codes over $M|\lambda$, so $\lambda_\infty<\gamma$, where $\gamma$
is least such that $\gamma>\lambda$ and $M|\lambda$ is admissible):
\begin{enumerate}
 \item $\vV_\om|\lambda_\infty=\mM_{\infty\om}|\lambda_\infty$, and
 \item for $\nu>\lambda_\infty$,
 \begin{enumerate}
\item $\vV_\om|\nu$ is active iff $M|\nu$ is active
\item if $M|\nu$ is active
and $\lambda<\crit(F^{M|\nu})$
then $F^{\vV_\om|\nu}=F^{M|\nu}\rest(\vV_\om|\nu)$ (which is a short extender),
\item if $M|\nu$ is active 
and $\crit(F^{M|\nu})<\lambda$,
then letting $n<\om$ be such that $\crit(F^{M|\nu})=\kappa_n^M$, $F^{\vV_\om|\nu}$
is the unique $(\delta_n^{\vV_\om},\nu)$-extender $F$
over $\vV_\om|\delta_n^{\vV_\om}$ such that
\[ F\com j_{n+1,\om}\rest(\vV_{n+1}|\Delta^{\vV_{n+1}})=F^{M|\nu}\rest(\vV_{n+1}|\Delta^{\vV_{n+1}})\]
(which is a long extender).\qedhere
\end{enumerate}
\end{enumerate}
\end{dfn}

\begin{lem}We have:
\begin{enumerate}
 \item \label{item:vV_om_sub_M_infty,om}
 $\vV_{\om}\sub \mM_{\infty\om}$. 
 \item\label{item:vV_om_not_project_across_lambda_inf} No segment of $\vV_\om$ projects across $\lambda_\infty$.
 \end{enumerate}
\end{lem}
\begin{proof}
 Part \ref{item:vV_om_sub_M_infty,om}: This
 follows immediately from the following
 characterizations:
 \begin{enumerate}
  \item $\vV_\om|\nu$ is active iff for all sufficiently large $n$, $\vV_{n\om}|\nu^*$ is active, and if active,
  $F^{\vV_\om|\nu}$ is short iff for all sufficiently large $n$, $F^{\vV_{n\om}|\nu^*}$ is short,
  \item if $\vV_\om|\nu$ is active with a short extender then for all $\alpha,\beta<\nu$, \[ F^{\vV_\om|\nu}(\alpha)=\beta\Leftrightarrow
\all^*n\ F^{\vV_{n\om}|\nu^*}(\alpha^*)=\beta^*, \]
 \item\label{item:vV_om|alpha_active_with_long} if $\vV_\om|\nu$ is active with a long extender then for all $\alpha,\beta<\nu$,
 \[ F^{\vV_\om|\nu}(\alpha)=\beta
\Leftrightarrow \all^*n<\om\ F^{\vV_{n\om}|\nu^*}(\alpha)=\beta^* \]
(in part \ref{item:vV_om|alpha_active_with_long},
note the input
to $F^{\vV_{n\om}|\nu^*}$ is $\alpha$,
not $\alpha^*$). 
 \end{enumerate}
 
 Part \ref{item:vV_om_not_project_across_lambda_inf}
is an immediate consequence of part \ref{item:vV_om_sub_M_infty,om}
and Lemma \ref{lem:M_infty,om_not_project_across_lambda_inf}.
\end{proof}

\begin{lem}
 Let $\alpha\in(\lambda_\infty,\lambda^{+M})$.
 Then letting $P=\vV_\om|\alpha$, we have:
 \begin{enumerate}
  \item\label{item:P_is_Delta_1^M} $P$ is $\rDelta_1^{M|\alpha}(\{\lambda\})$,
  \item $P$ is fully sound, 
  \item $M|\alpha\sats$``$\lambda^+$ exists''
  iff $P\sats$``$(\lambda_\infty)^{+}$ exists'',
  \item if $M|\alpha\sats$``$\lambda^+$ exists'' then  we have:
  \begin{enumerate}
\item  $\lambda^{+(M|\alpha)}=(\lambda_\infty)^{+P}$,
\item  $M|\lambda^{+(M|\alpha)}=M|(\lambda_\infty)^{+P}$ is $(P,\Vop^P_{\mathscr{P}})$-generic,
\item  $M|\alpha$ has universe that of $P[M|\lambda^{+(M|\alpha)}]$, which is the same as that of $P[M|\lambda]$,
\item  $M|\alpha$ is $\rDelta_1^{P[M|\lambda]}(\{P,M|\lambda\})$,
\item if $k<\om$ and $\lambda^{+M|\alpha}\leq\rho_k^{M|\alpha}$
then
 $\rho_k^P=\rho_k^{M|\alpha}$ and $p_k^{P}=p_k^{M|\alpha}$,
  \end{enumerate}
  \item if $\rho_\om^{M|\alpha}=\lambda$ then
   $\rho_\om^P=\lambda_\infty$.
 \end{enumerate}
\end{lem}

\begin{proof}
Part \ref{item:P_is_Delta_1^M}: This is an easy consequence of the fact that
 $\mM_{\infty\om}$ is definable in the codes over $M|\lambda$.
 
 We prove the rest of the lemma by induction on $\alpha$. Suppose it holds for all $\alpha<\eta$, and $M|\eta\sats$``$\lambda$ is the largest cardinal''.
 By induction then, $\vV_\om|\eta\sats$``$\lambda_\infty$ is the largest cardinal''.

 Suppose $\vV_\om|\eta$ projects $\leq\lambda_\infty$; so in fact $\rho_\om^{\vV_\om|\eta}=\lambda_\infty$. Then $M$ projects to $\lambda$, by condensation for $M$
 and since $\vV_\om|\eta$ is  definable over $M|\eta$ from the parameter $\lambda$.
And since $\vV_\om|\eta\sats$``$\lambda_\infty$ is the largest cardinal'',
it is easy to see that $\vV_\om|\eta$ is sound.

Now suppose that $\rho_\om^{\vV_\om|\eta}=\eta$. So $P\sats$``$\eta=(\lambda_\infty)^+$'' where $P=\J(\vV_\om|\eta)$.
Note that  $\vV_\om|\eta$ is level-by-level
definable over each $P\in\mathscr{P}$, 
uniformly in $P$, in the sense described
in the hypotheses to Lemma \ref{lem:M|lambda_is_P-generic}. So by Lemmas \ref{lem:lambda^+-cc}
and \ref{lem:M|lambda_is_P-generic},
$M|(\lambda_\infty)^{+P}$ is $(P,\Vop^P_{\mathscr{P}})$-generic
and $\eta$ is regular in $P[M|(\lambda_\infty)^{+P}]$,
but noting that $M|\eta\in P[M|(\lambda_\infty)^{+P}]$,
therefore $\rho_\om^{M|\eta}=\eta$.
 
 Now the usual arguments work to maintain
 the inductive conditions through to the least $\alpha$ such that $\eta<\alpha$
 and $\rho_\om^{M|\alpha}=\lambda$,
 so consider this, and let $k$ be least such that $\rho_{k+1}^{M|\alpha}=\lambda$.
 Let $P=\vV_\om|\alpha$.
 Then we get that $P$ is $i$-sound and $\rho_i^{P}=\rho_i^{M|\alpha}$
 and $p_i^P=p_i^{M|\alpha}$ for $i\leq k$ as usual. We have $\rho_{k+1}^{M|\alpha}=\lambda$, which gives that $\rho_{k+1}^{P}\leq(\lambda_\infty)^{+P}$ and $P$ is $(\lambda_\infty)^{+P}$-sound. If $\rho_\om^{P}=(\lambda_\infty)^{+P}$
 then as before, $(\lambda_\infty)^{+P}=\lambda^{+(M|\alpha)}$ remains regular in $\J(P)[M|\lambda]$, contradicting
 that $M|\alpha\in\J(P)[M|(\lambda_\infty)^{+P}]$
 and $M|\alpha$ projects to $\lambda$.
 So $\rho_\om^P=\lambda_\om$,
 and it remains to see that $P$ is fully sound.
 
 Suppose there is $n<\om$ such that $\rho_{n+1}^P=\lambda_\infty<(\lambda_\infty)^{+P}=\rho_n^P$. Then $P$ is $n$-sound.
 And 
 \[ \Hull_{n+1}^P(\lambda_\infty\cup\{\pvec_{n+1}^P\})=P \]
 since, as usual, it
 cofinal in $\rho_n^P$. And $P$ is $(n+1)$-solid  because if $p_{n+1}^P\neq\emptyset$ then there is $\xi\in(\lambda_\infty,(\lambda_\infty)^{+P})$
 such that $p_{n+1}^P=\{\xi\}$,
 and \[(\lambda_\infty)^{+P}\cap\Hull_{n+1}^P(\xi\cup\{\pvec_n^P\})=\xi,\]
 which implies that the theory of this hull is in $P$ (irrespective of whether condensation holds here). This gives $(n+1)$-soundness,
 and hence $\om$-soundness, for $P$.
 
 Now suppose otherwise, so there is $n<\om$
 such that $\rho_{n+1}^P=\lambda_\infty<(\lambda_\infty)^{+P}<\rho_n^P$. Then the witnessed $\rSigma_{n+1}$ forcing relation for $(P,\Vop^P_{\mathscr{P}})$
 is $\rSigma_{n+1}^P(\{\lambda_\infty\})$,
 and $\rho_n^P=\rho_n^{M|\alpha}$
 and $\pvec_n^P=\pvec_n^{M|\alpha}$
 and $p_{n+1}^P\cut(\lambda_\infty)^{+P}=p_{n+1}^{M|\alpha}\cut\lambda^{+(M|\alpha)}$. And $\rho_{n+1}^{M|\alpha}=\lambda$,
 since otherwise we easily get a contradiction by condensation applied to $M|\alpha$.
Let
\[H=\Hull_{n+1}^P(\lambda_\infty\cup\{\pvec_{n+1}^P\}).\]

We claim $p_{n+1}^{M|\alpha}\cap\lambda^{+(M|\alpha)}\leq p_{n+1}^{P}\cap(\lambda_\infty)^{+P}$. For suppose $p_{n+1}^{M|\alpha}\cap\lambda^{+(M|\alpha)}>p_{n+1}^{P}\cap(\lambda_\infty)^{+P}$. Then there is $\iota$ such that $p_{n+1}^{M|\alpha}\cap\lambda^{+(M|\alpha)}=\{\iota\}$, and
\[ H'=\Hull_{n+1}^{M|\alpha}(\iota\cup\pvec_{n+1}^{M|\alpha}\cut\lambda^{+(M|\alpha)}) \]
is such that $H'\cap\lambda^{+(M|\alpha)}=\iota$, and $H\sub H'$
and $H'$ computes $H$, 
and letting $C'$ be the transitive collapse
of $H'$, then by condensation for $M|\alpha$,
$C'\pins M|\alpha$ (noting that $\lambda$ is not measurable in segments of $M$,
and $M|\iota$ is not active, since it has largest cardinal $\lambda$, and we are using Jensen indexing), but then $\vV_\om|\OR^{C'}\in P$ and this structure computes the transitive collapse $C$ of $H$,
so $C\in P$, which is a contradiction.

We now claim that $H=P$.
For $H\cap\rho_n^P$
is cofinal in $\rho_n^P=\rho_n^{M|\alpha}$,
and by the previous paragraph,
$\pvec_{n+1}^{M|\alpha}\in H$
and $\lambda\sub\lambda_\infty\sub H$.
Let $\beta<\alpha$ and $\gamma<\lambda$
and $u$ be an $\rSigma_{n+1}$ term
such that $\beta=u^{M|\alpha}(\gamma,\pvec_{n+1}^{M|\alpha})$. There is $\eta<\rho_n^{P}$ and $p\in\Vop^P_{\mathscr{P}}$
such that $P\sats p\forces$``$\widetilde{t}_\eta$
codes a witness to $u(\gamma,\pvec)=\beta$'',
where $\widetilde{t}_\eta$ is the canonical
name for $\Th_{\rSigma_n}(\eta\cup\{\pvec_n^{M|\alpha}\})$ and $\pvec=\pvec_{n+1}^{M|\alpha}$. Given $\xi<\rho_n^{P}$,
let $A_\xi$ be the antichain $\left<p^\xi_\delta\right>_{\delta<\gamma_\xi}$ of conditions of $\Vop^P_{\mathscr{P}}$ which is defined by recursion on $\xi$, with $A_\xi\ins A_{\xi'}$
when $\xi<\xi'$, and for each $p\in A_\xi$,
there is $\beta_p$ such that
$P\sats p\forces$``$\widetilde{t}_\xi$ codes
a witness to $u(\gamma,\pvec)=\beta_p$'',
and if $p,q\in A_\xi$ with $p\neq q$ then $\beta_p\neq\beta_q$,
and if $\beta_1<\alpha$ is such that there is $p\in\Vop^P_{\mathscr{P}}$ such that $P\sats p\forces$``$\widetilde{t}_\xi$ codes a witness to $u(\gamma,\pvec)=\beta_1$'', then there is some $p\in A_\xi$ such that $\beta_p=\beta_1$,
and using the $P$-constructibility order
to rank the conditions $p\in A_\xi\cut\bigcup_{\xi'<\xi}A_{\xi'}$. Note here
that $A_\xi\in P$ and $\gamma_\xi<(\lambda_\infty)^{+P}$. 
It follows that we can effectively assign to $A_\xi$
an enumeration $\left<p_\mu\right>_{\mu<\lambda_\infty}\in P$
of $A_\xi$. It follows that $\beta\in H$, and therefore $P=H$, as desired.

It remains to see that $P$ is $(n+1)$-solid.
For $p_{n+1}^P\cut(\lambda_\infty)^{+P}$
this is as usual, so suppose $p_{n+1}^P\cap(\lambda_\infty)^{+P}=\{\iota^P\}\neq\emptyset$.
So $p_{n+1}^{M|\alpha}\cap\lambda^{+(M|\alpha)}=\{\iota^{M|\alpha}\}\neq\emptyset$,
and $\iota^{M|\alpha}\geq\iota^P$.

Suppose that $\iota^{M|\alpha}<\iota^P$.
Then $\Hull_{n+1}^P(\iota^P\cup\{\pvec_{n+1}^P\cut(\lambda_\infty)^{+P}\})$
is bounded in $\rho_n^P$, by the preceding argument (otherwise we get $P=H$ again).
But then the corresponding theory is in $P$, giving the desired solidity witness.

Finally supppose that $\iota=\iota^{M|\alpha}=\iota^P$.
Let  $H=\Hull_{n+1}^M(\iota\cup\{\pvec_{n+1}^{M|\alpha}\cut\lambda^{+(M|\alpha)}\})$
and $C$ be the transitive collapse of $H$
and $\pi:C\to H$ the uncollapse.
Then $H\cap\lambda^{+(M|\alpha)}=\iota=\crit(\pi)$
and $\iota=\lambda^{+C}$. For every $\rSigma_{k+1}$ term $u$
and $\gamma<\iota$ and $\beta<\alpha$
such that $\beta=u^{M|\alpha}(\pvec,\gamma)$,
where $\pvec=\pvec_{n+1}^{M|\alpha}\cut\lambda^{+(M|\alpha)}$,
there is $p\in G_{M|(\lambda_\infty)^{+P}}$
and $\xi<\rho_n^{M|\alpha}$
such that $P\sats p\forces$``$\widetilde{t}_\xi$
codes a witness to $u(\vec{p},\gamma)=\beta$'',
and so in fact there is some such $p\in P|\iota$,
by $\rSigma_{n+1}$ elementarity. 
It follows that $\Hull_{n+1}^P(\iota\cup\pvec)\cap\OR=H\cap\OR$. But then $C\pins M$
and $P|\OR^C$ encodes the desired solidity witness in $P$, giving $(n+1)$-solidity. This completes the proof.
\end{proof}

So $\vV_\om$ is a fine structural strategy premouse,
$M|\lambda^{+M}$ is $(\vV_\om,\Vop^{\vV_\om}_{\mathscr{P}})$-generic,
and $\vV_\om[M|\lambda^{+M}]$ (or equivalently,
$\vV_\om[M|\lambda]$) has universe that of $M$.
We also saw that $\vV_\om\sub\mM_{\infty\om}$.
However, it doesn't quite seem we can immediately deduce
that $\vV_\om$ and $\mM_{\infty\om}$ have the same universe,
since we haven't shown that $G_{M|\lambda^{+M}}$
(the $\Vop_{\mathscr{P}}^{\vV_\om}$-generic filter) is also generic over $\mM_{\infty\om}$. (Recall that if it is, then since $\vV_\om[M|\lambda]$ and $\mM_{\infty\om}[M|\lambda]$ have the same universe (that of $M$), it would follow that $\vV_\om$ and $\mM_{\infty\om}$
also have the same universe.) So instead,
we adapt the process that works at the finite stages,
via which $\vV_{n+1\down n}$ is a lightface class of $\vV_{n+1}$,
to show:
\begin{lem}\label{lem:vV_om_mM_infty,om_equivalence} $\vV_{n\om}$ is a lightface class of $\vV_\om$,
uniformly in $n<\om$. Therefore $\vV_\om$ and $\mM_{\infty\om}$ have the same universe and are lightface classes in each other.\end{lem}
\begin{proof}
We already have $\vV_\om|\lambda_\infty$ lightface defines $\left<\vV_{n\om}|\lambda_\infty\right>_{n<\om}$.
Now working in $\vV_\om$, we can define $\vV_\om^{\vV_{0\om}}$
 as the direct limit of the (non-normal)
iteration on $\vV_\om$  given by the extenders $E_0=e_0^{\vV_\om}$,
$E_1=i_{E_0}^{\vV_\om}(e_1^{\vV_\om})$,
etc. This yields $\vV_\om^{\vV_{0\om}}$.
But since $\vV_\om^{\vV_{0\om}}$ is given by P-construction
in $\vV_{0\om}$ above $(\lambda_\infty)^{\vV_{0\om}}$
(and there are no extenders in $\es^{\vV_{0\om}}$
indexed in the interval $[\lambda_\infty,(\lambda_\infty)^{\vV_{0\om}}]$)
it follows that we can recover $\es^{\vV_{0\om}}\rest[\lambda_\infty,\OR)$ from $\vV_{\om}^{\vV_{0\om}}$.
This gives $\vV_{0\om}$. The recovery of $\vV_{n\om}$ is analogous.
\end{proof}

So by the lemma, we have successfully reorganized
$\mM_{\infty\om}$ as a fine structural structure $\vV_\om$,
and $\mM_{\infty\om}$ is
a ground of $M$ via $\Vop^{\vV_\om}_{\mathscr{P}}$.

\begin{dfn}\label{dfn:wW_om}
 Let $\wW_\om=\cHull^{\vV_\om}(\mathscr{I}^M)$
 and $i_\om:\wW_\om\to\vV_\om$ be the uncollapse map.
\end{dfn}

\begin{rem}
By Lemmas \ref{lem:vV_om_mM_infty,om_equivalence}
and \ref{lem:hull_conservativity_vV_om},
we get that $\wW_n\sub\wW_\om$ for each $n<\om$,
and $i_n\sub i$ for each $n<\om$ (recall $i_n:\wW_n\to\vV_{n\om}$ is the iteration map).
\end{rem}

\subsection{The iteration strategy $\Sigma_{\wW_\om}$}

\begin{dfn}
 \emph{Short-normal} trees on $\wW_\om$ and related structures are defined via the obvious adaptation of the definition for trees on $\wW_n$.
\end{dfn}

\begin{dfn}
 We define $\Sigma^{\sn}_{\wW_\om}$, the \emph{$\vec{e}$-minimal pullback
 of $\Sigma_M$}, as the following short-normal strategy for $\wW_\om$; this is just a straightforward generalization of the version for the models $\wW_n$. 
 
 For trees based on $\wW_\om|\delta_n^{\wW_\om}$, we follow $\Sigma_{\wW_n}$.
 Now let $\Tt$ be a successor length short-normal tree based on $\wW_\om|\lambda$ such that $b^\Tt$ does not drop. We define a strategy for above-$\lambda^{M^\Tt_\infty}$ short-normal trees $\Uu$
on $M^\Tt_\infty$.
 Note that since $\wW_\om$ has no total measures with critical point $\geq\lambda$, and the only extenders in $\es^{\wW_\om}$ which overlap $\lambda$ are long,
there is some $M^\Tt_\infty$-cardinal $\eta\geq\lambda^{M^\Tt_\infty}$ such that $\Uu$ is above $\eta$ and based on $P$ for some $P\pins M^\Tt_\infty|\eta^{+M^\Tt_\infty}$ such that $\rho_\om^P=\eta$.
Let $\Xx$ be the $\vec{e}$-iteration of $M^\Tt_\infty$; i.e., $\Xx$ is the length $\om+1$ tree
such that $E^\Xx_k=e_k^{M^\Xx_k}$ for all $k<\om$.
Then $M^\Xx_\infty=\vV_\om^{(M^\Tt_\infty)\down 0}$.
We now declare that $\Tt\conc\Uu$ is via $\Sigma_{\wW_\om}^{\sn}$ iff $i^\Xx_{0\om}``\Uu$ translates to a tree via $\Sigma_{M^\Tt_\infty\down 0}$.
\end{dfn}

\begin{lem}
 $\Sigma^{\sn}_{\wW_\om}$ is a self-coherent short-normal $(0,\OR)$-strategy for $\wW_\om$ with mhc.
\end{lem}
\begin{proof}
 It is clearly a short-normal $(0,\OR)$-strategy.
 The proof of mhc is like for $\Sigma_{\wW_n}$.
 The proof of self-coherence is like that of Lemma \ref{lem:Sigma_wW_n+2}
 (note we do not need the preparatory lemmas here
 which handle $\delta_{n+1}$ there).
\end{proof}

\begin{dfn}
 $\Sigma_{\wW_\om}$ is the unique self-consistent $(0,\OR)$-strategy for $\wW_\om$ which extends $\Sigma^{\sn}_{\wW_\om}$. 
\end{dfn}

\begin{rem}
As usual, for a non-dropping $\Sigma_{\wW_\om}$-iterate $\vV$,
 $\Sigma^{\sn}_{\vV}$ denotes the short-normal strategy for $\vV$ induced by $\Sigma^{\sn}_{\wW_\om}$ via normalization, and $\Sigma_{\vV}$
 is its extension to a self-consistent $(0,\OR)$-strategy. \end{rem}
 
\begin{lem}\label{lem:vV_om^N_is_it}
 Let $N$ be a non-dropping $\Sigma_M$-iterate.
 Then:
 \begin{enumerate}
  \item 
$\vV_\om^N$ is a  $\Sigma_{\wW_\om}$-iterate.
\item $\vV_{n\om}^N$ is a $\Sigma_{\wW_n}$-iterate.
\item\label{item:vV_n,om_is_Sigma_vV-it} Let $n<\om$ and let $\vV$ be the $\delta_n^{\vV_{n+1}^N}$-core
of $\vV_{n+1}^N$. Then $\vV_{n\om}^N$
is a $\Sigma_{\vV}$-iterate.
\end{enumerate}
\end{lem}
\begin{proof}
Let $\vV$ be as in part \ref{item:vV_n,om_is_Sigma_vV-it}.
It is straightforward to see that there is a successor length $\Sigma_{\vV}$-tree $\Uu_n$
such that $[0,\infty]^{\Uu_n}$ is non-dropping,
$\Uu_n$ is based on $\vV|\delta_n^{\vV}$, and
$\vV_\om^N|\delta_n^{\vV_\om^N}=M^{\Uu_n}_\infty|\delta_n^{M^\Uu_\infty}$.
  But there is also a $\Sigma_{\wW_{n+1}}$-tree $\Tt_n$, based on $\wW_{n+1}|\delta_n^{\wW_{n+1}}$, of successor length and with $[0,\infty]^{\Tt_n}$ non-dropping,
such that $\vV=M^{\Tt_n}_\infty$.
Given the agreement between $\Sigma_{\wW_m},\Sigma_{\wW_n},\Sigma_{\wW_\om}$ for $m,n<\om$,
it follows that there is a $\Sigma_{\wW_\om}$-tree
$\Ww$ of successor length, such that $[0,\infty]^{\Ww}$ is non-dropping,
and $\vV_\om^N|\lambda^{\vV_\om^N}=M^{\Ww}_\infty|\lambda^{M^\Ww_\infty}$. But it is straightforward
to see that for each $n<\om$,
$\vV_{n\om}^N=\Hull^{\vV_{n\om}^N}(\lambda^{\vV_\om^N}\cup\mathscr{I}^N)$
and therefore (using the preceding remarks,  ranging over all $m<\om$) $\vV_{n\om}^N$ is  a $\Sigma_{\wW_n}$-iterate. But $\vV_\om^N$
is computed from $\left<\vV_{n\om}^N\right>_{n<\om}$
in the same manner as $\wW_\om$ is computed
 from $\left<\wW_n\right>_{n<\om}$,
 and it follows that $\vV_\om^N=M^{\Ww}_\infty$.
 The remaining details now follow readily.
\end{proof}
\begin{dfn}
Given a non-dropping $\Sigma_{\wW_\om}$-iterate $\vV$,
 $\Sigma_{\vV}\down 0$ is the $(0,\OR)$-strategy for $\vV\down 0$ computed from $\Sigma_{\vV}$
 by analogy with how $\Sigma_{\wW_n}\down 0$
 is computed from $\Sigma_{\wW_n}$.
\end{dfn}

\begin{lem}
 Let $\vV$ be any non-dropping $\Sigma_{\wW_\om}$ iterate. Then $\Sigma_{\vV}\down 0=\Sigma_{\vV\down 0}$.
\end{lem}
\begin{proof}
 As in the proof of Corollary \ref{cor:Sigma_wW_n+1_down_0=Sigma_M}.
\end{proof}
`
\begin{lem}\label{lem:N_computes_Sigma_vV,lambda^vV}
 Let $N$ be a non-dropping $\Sigma_M$-iterate.
 Let $\vV=\vV_\om^N$.
 Let $\xi\in\OR$  and $x=N|\xi^{+N}$. Let $\PP\in\pow(\xi)\cap N$
 and $g$ be $(N,\PP)$-generic. 
 Then:
 \begin{enumerate}
  \item $N$ is closed under $\Sigma_{\vV,\lambda^{\vV}}$ and $\Sigma_{\vV,\lambda^{\vV}}\rest N$ is lightface definable over $N$, uniformly in $N$.
  \item $N[g]$ is closed under $\Sigma_{\vV,\lambda^{\vV}}$ and $\Sigma_{\vV,\lambda^{\vV}}\rest N[g]$ is definable over $N[g]$ from the parameter $x$,
  uniformly in $N,g,x$.
  \end{enumerate}
\end{lem}
\begin{proof}[Proof (deferral)]
 This follows directly from Lemma \ref{lem:vV_om^N_is_it}
 and, for each $n<\om$, the closure of $N$ (respectively, $N[g]$) under $\Sigma_{\vV_{n+1}^N,\delta_n^{\vV_{n+1}^N}}$ and its (uniform) definability. For non-tame trees, the proof relies on $*$-translation, due to its reliance on earlier lemmas which use this, and hence is only fully proven with \cite{*-trans_add}.
\end{proof}

\begin{lem}\label{lem:vV_computes_Sigma_vV,lambda^vV}
Let $\vV$ be a non-dropping $\Sigma_{\wW_\om}$-iterate.
 Let $\xi\in\OR$ with $\lambda^{\vV}\leq\xi$ and $x=\vV|\xi^{+\vV}$. Let $\PP\in\pow(\xi)\cap \vV$
 and $g$ be $(\vV,\PP)$-generic. 
 Then:
 \begin{enumerate}
  \item $\vV$ is closed under $\Sigma_{\vV,\lambda^{\vV}}$ and $\Sigma_{\vV,\lambda^{\vV}}\rest \vV$ is lightface definable over $\vV$, uniformly in $\vV$,
  \item $\vV[g]$ is closed under $\Sigma_{\vV,\lambda^{\vV}}$ and $\Sigma_{\vV,\lambda^{\vV}}\rest \vV[g]$ is definable over $\vV[g]$ from the parameter $x$,
  uniformly in $\vV,g,x$.
  \end{enumerate}
\end{lem}
\begin{proof}[Proof (deferral)]
Suppose first that $\vV=\vV_\om^N$
for some non-dropping $\Sigma_M$-iterate $N$.
Then the conclusions of the lemma with respect to tame trees are established via a slight variant of the proof of Claim \ref{clm:vV_n+2_self-it_tame} of the proof of Lemma \ref{lem:vV_n+2_is_pVl}. For non-tame trees, one combines this with $*$-translation,
so the proof is deferred to \cite{*-trans_add}.

The general case can be reduced to the one of the previous paragraph like in the proof of Lemma \ref{lem:wW_n+2_self-it_tame}, using that $N=\vV\down 0$
is a non-dropping $\Sigma_M$-iterate and $\vV'=\Ult(\vV,\vec{e}^{\vV})=\vV_\om^N$.
\end{proof}

We now want to generalize the previous two lemmas to arbitrary trees, not just those based on $\vV|\lambda^{\vV}$. Somewhat modifying the approach used for the $\vV_n$s, we will first show that $M$ is closed under $\Sigma_{\vV_\om\down 0}$ and can define $\Sigma_{\vV_\om\down 0}\rest M$,
 and deduce the same for $\Sigma_{\vV_\om}$ from this. For this, we will consider P-constructions
 for trees on $\vV_\om\down 0$, instead of trees on $\vV_\om$.

\begin{dfn}\label{dfn:star-suitable_vV_om}
Let $N$ be a non-dropping $\Sigma_M$-iterate.
Let $\PP\in N$ and $g$ be $(N,\PP)$-generic. Working in $N[g]$, we say
 that a pair of trees $(\Tt,\Tt')$
is \emph{$*$-suitable} if there
  are $\Tt_0,\Tt_1,\Tt_0',\Tt_1',\eta,\delta$ such that:
 \begin{enumerate}
  \item $\Tt'=\Tt_0'\conc\Tt_1'$ is short normal on $\vV_\om^N$,
  \item $\Tt=\Tt'\down 0$ and $\Tt_0=\Tt_0'\down 0$
   (so $\Tt$ and $\Tt_0$ are on $K=\vV_\om\down 0$),
 \item $\Tt=\Tt_0\conc\Tt_1$,
  \item $\Tt_0'$ is based on $\vV_\om^N|\lambda^{\vV_\om^N}$,
  is via $\Sigma_{\vV_\om^N}$,
  has successor length, and $b^{\Tt_0'}$ does not drop (so $\Tt_0$ is based on $K|\lambda^{K}$, etc),
   \item $\Tt_1'$ is above $\lambda^{M^{\Tt'_0}_\infty}$   and has limit length
   (so $\Tt_1$ is above $\lambda^{M^{\Tt_0}_\infty}$ and has limit length),
   \item\label{item:eta,delta,lambda_vV_om} $\eta$ is a cardinal of $N$ with $\max(\lambda^N,\lh(\Tt_0))<\eta<\delta=\delta(\Tt_1)$, $\eta$ is the largest cardinal of $N|\delta$,
  and $\PP\in N|\eta$, 
  \item\label{item:U|delta_is_generic_vV_om}$(N|\delta,g)$ is extender algebra generic over $M(\Tt_1)$
  for the extender algebra at $\delta$, and $\Tt_1$ is definable from parameters over $(N|\delta)[g]$.
 \end{enumerate}

Working instead in $\vV_\om^N[g]$ where
 $\PP\in\vV_\om^N$
and $g$
is $(\vV_\om^N,\PP)$-generic,
we say that a pair of trees $(\Tt,\Tt')$ 
is \emph{$*$-suitable}
if there are $\Tt_0,\ldots,\delta$ as before, except that we   replace ``$N$'' throughout
with ``$\vV_\om^N$''.
\end{dfn}

\begin{dfn}\label{dfn:P-con_vV_om}
 Let $N$ be a non-dropping $\Sigma_M$-iterate.
 Let $\PP\in N$ and $g$ be $(N,\PP)$-generic.
 Work in $N[g]$.
 Let $(\Tt,\Tt')$ be $*$-suitable.
 Then the \emph{P-construction} $\mathscr{P}^{N}(M(\Tt))$ is the (putative) premouse $P$ such that (we will be hoping to build the Q-structure for $M(\Tt)$,
 not $M(\Tt')$, and recall that $\Tt$ is on the premouse $K$):
 \begin{enumerate}
  \item $M(\Tt)\ins P$,
  \item for each $\nu>\delta$, we have:
  \begin{enumerate}
   \item $P|\nu$ is active iff $N|\nu$ is active, and
   \item if $N|\nu$ is active then letting $F=F^{N|\nu}$ and $G=F^{P|\nu}$, either:
   \begin{enumerate}
   \item $\crit(F)>\lambda^N$ (it follows that $\crit(F)>\delta$: $\crit(F)>\eta$ since $\eta$ is an $N$-cardinal and by the minimality of $M$, but $\eta$ is the largest cardinal of $N|\delta$) and $G\rest\OR=F\rest\OR$, or
   \item\label{item:commuting_ext_P-con_for_vV_om_in_N} for some $n<\om$, we have $\crit(F)=\kappa_n^N$, and \[G\com E_n\rest(\vV_{n+1}^{N}|\delta_n^{\vV_{n+1}^{N}})=F\rest(\vV_{n+1}^N|\delta_n^{\vV_{n+1}^N}),\]
   where $E_n$ is the $(\delta_n^{\vV_{n+1}^N},\delta_n^{\vV_{n+1}^{M^{\Tt_0}_\infty}})$-extender
   derived from the $\Sigma_{\vV_{n+1}^N}$-iteration map $j:\vV_{n+1}^N\to\vV$
   where $\vV=M^\Uu_\infty$ where $\Uu$ is the short-normal $\Sigma_{\vV_{n+1}^N}$-tree
   which is based on $\vV_{n+1}^N|\delta_n^{\vV_{n+1}^N}$
   and which iterates $\vV_{n+1}^N|\delta_n^{\vV_{n+1}^N}$
   out to $\vV_{n+1}^{M^{\Tt_0}_\infty}|\delta_n^{\vV_{n+1}^{M^{\Tt_0}_\infty}}$;
   note that $E_n\in (N|\eta)[g]$.
   \end{enumerate}
  \end{enumerate}
 \end{enumerate}

 Now work instead in $\vV_\om^N[g]$,
 where $\PP\in\vV_\om^N$ and $g$ is $(\vV_\om^N,\PP)$-generic. Let $(\Tt,\Tt')$ be $*$-suitable. 
 Then the \emph{P-construction} $\mathscr{P}^{\vV_\om^N}(M(\Tt))$ is the (putative) premouse $P$ defined as above, except that:
 \begin{enumerate}[label=--]
  \item ``$N$'' is replaced by ``$\vV_\om^N$'' throughout,
  \item condition \ref{item:commuting_ext_P-con_for_vV_om_in_N} is replaced by:
    \begin{enumerate}[label=\arabic*.]
\setcounter{enumii}{1}
  \item 
  \begin{enumerate}[label=(\alph*)]
\setcounter{enumiii}{1}
  \item 
  \begin{enumerate}[label=\roman*'.]
  \setcounter{enumiv}{1}
  \item for some $n<\om$, $F$ is $n$-long, and
\[G\com E_n\rest(\vV_\om^{N}|\delta_n^{\vV_\om^N})=F\rest(\vV_\om^N|\delta_n^{\vV_\om^N}),\]
   where $E_n$ is the $(\delta_n^{\vV_\om^N},\delta_n^{\vV_{n+1}^{M^{\Tt_0}_\infty}})$-extender
   derived from the $\Sigma_{\vV_\om^N}$-iteration map $j:\vV_\om^N\to\vV$
   where $\vV=M^\Uu_\infty$ where $\Uu$ is the short-normal $\Sigma_{\vV_\om^N}$-tree
   which is based on $\vV_\om^N|\delta_n^{\vV_\om^N}$
   and which iterates $\vV_\om^N|\delta_n^{\vV_\om^N}$
   out to $\vV_{n+1}^{M^{\Tt_0}_\infty}|\delta_n^{\vV_{n+1}^{M^{\Tt_0}_\infty}}$;
   note that $E_n\in (\vV_\om^N|\eta)[g]$.\qedhere
     \end{enumerate}
  \end{enumerate}
 \end{enumerate}
 \end{enumerate}

\end{dfn}

\begin{lem}\label{lem:*-suitable_P-con_correctness_no_short_overlaps_vV_om}
Let $N$ be a non-dropping $\Sigma_M$-iterate
and $\vV=\vV_\om^N$.
Let $i=0$ or $i=\om$.
Let $\PP\in\vV_i^N$ and $g$ be $(\vV_i^N,\PP)$-generic.
Work in $\vV_i^N[g]$.
Let $(\Tt,\Tt')$ be $*$-suitable,
and adopt the notation of Definition \ref{dfn:P-con_vV_om}. Let $K=\vV_\om^N\down 0$.
Suppose that $\Tt$ is via $\Sigma_{K}$, and let $b=\Sigma_{K}(\Tt)$.
Let $Q=Q(\Tt,b)$.
Then:
\begin{enumerate}
\item\label{item:delta(Tt)_overlapped_iff_reach_corresp_stage_in_P-con_vV_om} The following are equivalent:
\begin{enumerate}
\item\label{item:short_overlapping_E_exists_vV_om}there is a short extender $E\in\es_+^Q$ with $\crit(E)\leq\delta\leq\lh(E)$,
\item\label{item:projecting_stage_of_P-con_exists_vV_om} 
$P=\mathscr{P}^U(M(\Tt))$ is unsound and there is $n<\om$
such that $P$ is $n$-sound and
 $\rho_{n+1}^P<\delta\leq\rho_n^P$
 and there is $\mu<\delta$
 such that for all $\gamma\in[\mu,\delta)$, $P$ does not have the $(n+1,\pvec_{n+1}^P)$-hull property at $\gamma$.
\end{enumerate}
\item\label{item:Q=P-con_if_delta_not_shortly_overlapped_vV_om}  Suppose there is no short extender $E\in\es_+^Q$ with $\crit(E)\leq\delta\leq\lh(E)$.
Then $Q$ is set-sized and $Q=\mathscr{P}^{\vV_i^N}(M(\Tt))$.
\end{enumerate}
\end{lem}
\begin{proof}
 This is much like the proof of Lemma \ref{lem:*-suitable_P-con_correctness_no_short_overlaps}. The ($0$-maximal) $\Sigma_M$-tree leading from $M$ to $M^\Tt_b$
 is used as the phalanx on the $Q$ side,
 and $N$ is used on the P-construction side.
 Otherwise, the details are as usual.
\end{proof}

Modulo our accumulated debts to $*$-translation,
we can now establish that $M$ and its (correct) iterates can correctly iterate both $\vV_\om^N$ and $\vV_\om^N\down 0$, and in particular, that $M$ can correctly iterate $\vV_\om$ and $\vV_\om\down 0$,
and recall that the latter is a (correct) iterate of $M$:

\begin{tm}\label{tm:N_computes_Sigma_vV}
 Let $N$ be a non-dropping $\Sigma_M$-iterate.
 Let $\vV=\vV_\om^N$.
 Let $\xi\in\OR$  and $x=N|\xi^{+N}$. Let $\PP\in\pow(\xi)\cap N$
 and $g$ be $(N,\PP)$-generic. 
 Then:
 \begin{enumerate}
  \item\label{item:N_closed_under_Sigma_vV} $N$ is closed under $\Sigma_{\vV}$ and $\Sigma_{\vV}\rest N$ is lightface definable over $N$, uniformly in $N$.
    \item\label{item:N_closed_under_Sigma_vV_down_0} Like part \ref{item:N_closed_under_Sigma_vV},
    but with $\Sigma_{\vV\down 0}$ replacing $\Sigma_{\vV}$.
  \item\label{item:N[g]_closed_under_Sigma_vV} $N[g]$ is closed under $\Sigma_{\vV}$ and $\Sigma_{\vV}\rest N[g]$ is definable over $N[g]$ from the parameter $x$,
  uniformly in $N,g,x$.
    \item\label{item:N[g]_closed_under_Sigma_vV_down_0} Like part \ref{item:N[g]_closed_under_Sigma_vV}, but with $\Sigma_{\vV\down 0}$ replacing $\Sigma_{\vV}$.
  \end{enumerate}
\end{tm}
\begin{proof}[Proof (deferral)]
Parts \ref{item:N_closed_under_Sigma_vV}
and \ref{item:N[g]_closed_under_Sigma_vV}:
It suffices to prove the same things for  just $\Sigma^{\sn}_{\vV}$ instead of $\Sigma_{\vV}$. But short-normal (putative) trees
on $\vV$ are of the form $\Uu'=\Tt'_0\conc\Uu_1'$,
where $\Tt'_0$ is based on $\vV|\lambda^{\vV}$,
and if $\Tt'_0$ has successor length and $[0,\infty]^{\Tt'_0}$ does not drop, then $\Uu'_1$ is above $\lambda^{M^{\Tt'_0}_\infty}$.
Tame such trees  $\Uu'$ via $\Sigma_{\vV}$ can be reduced to trees $\Tt'$ such that there is some $*$-suitable $(\Tt,\Tt')$, via minimal genericity inflation of $\Uu_1'$ to some minimal inflation $\Tt_1'$, producing $\Tt'=\Tt_0'\conc\Tt_1'$,
in the usual manner. So Lemmas \ref{lem:N_computes_Sigma_vV,lambda^vV} and \ref{lem:*-suitable_P-con_correctness_no_short_overlaps_vV_om}
suffice to handle such tame trees $\Uu'$.
For non-tame trees, as usual, it is similar,
but the proof involves $*$-translation,
so is deferred to \cite{*-trans_add}.

Parts \ref{item:N_closed_under_Sigma_vV_down_0}
and \ref{item:N[g]_closed_under_Sigma_vV_down_0}:
By parts \ref{item:N_closed_under_Sigma_vV}
and \ref{item:N[g]_closed_under_Sigma_vV} respectively, since $\Sigma_{\vV\down 0}=\Sigma_{\vV}\down 0$. 
\end{proof}

And such $\vV$ similarly compute their own strategy and that for $\vV\down 0$:

\begin{tm}\label{tm:vV_computes_Sigma_vV}
Let $\vV$ be a non-dropping $\Sigma_{\wW_\om}$-iterate.
 Let $\xi\in\OR$ with $\lambda^{\vV}\leq\xi$ and $x=\vV|\xi^{+\vV}$. Let $\PP\in\pow(\xi)\cap \vV$
 and $g$ be $(\vV,\PP)$-generic. 
 Then:
 \begin{enumerate}
  \item\label{item:vV_closed_under_Sigma_vV} $\vV$ is closed under $\Sigma_{\vV}$ and $\Sigma_{\vV}\rest \vV$ is lightface definable over $\vV$, uniformly in $\vV$,
      \item\label{item:vV_closed_under_Sigma_vV_down_0} Like part \ref{item:vV_closed_under_Sigma_vV},
    but with $\Sigma_{\vV\down 0}$ replacing $\Sigma_{\vV}$.
 
  \item\label{item:vV[g]_closed_under_Sigma_vV} $\vV[g]$ is closed under $\Sigma_{\vV}$ and $\Sigma_{\vV}\rest \vV[g]$ is definable over $\vV[g]$ from the parameter $x$,
  uniformly in $\vV,g,x$.
      \item\label{item:vV[g]_closed_under_Sigma_vV_down_0} Like part \ref{item:vV[g]_closed_under_Sigma_vV}, but with $\Sigma_{\vV\down 0}$ replacing $\Sigma_{\vV}$.
  \end{enumerate}
\end{tm}
\begin{proof}[Proof (deferral)]
Parts \ref{item:vV_closed_under_Sigma_vV}
and \ref{item:vV[g]_closed_under_Sigma_vV}:
The proof follows the same structure as earlier analgoues: If $\vV=\vV_\om^N$
for some non-dropping $\Sigma_M$-iterate $N$,
then  it is much like the proof of Theorem \ref{tm:N_computes_Sigma_vV};
and the general case can be reduced to the just mentioned  one  like in the proof of Lemma \ref{lem:wW_n+2_self-it_tame},
or more recently in that of Lemma \ref{lem:vV_computes_Sigma_vV,lambda^vV}.
Again the full proof is deferred to \cite{*-trans_add}.

Parts \ref{item:N_closed_under_Sigma_vV_down_0}
and \ref{item:N[g]_closed_under_Sigma_vV_down_0}:
Since $\Sigma_{\vV\down 0}=\Sigma_{\vV}\down 0$
and by parts  \ref{item:vV_closed_under_Sigma_vV}
and \ref{item:vV[g]_closed_under_Sigma_vV}
(so the full proof is deferred to \cite{*-trans_add}).
\end{proof}

\section{The core model and proof of   main theorem}\label{sec:main_proof}

Before we prove the main theorem,
we need to make clear what is meant by the \emph{core model} of $M$ and discuss this somewhat. Similarly to in \cite{vm1} and adapting the definition in \cite{core_model_in_transcendent_mice},
we (want to) define the \emph{core model} $K^M$ of $M$ as follows, and the \emph{core model} $K^{\vV_\om}$ of $\vV_\om$, as follows. We first state the definitions quantifying over classes
(which we take to be definable from parameters over the respective model), and will later observe that this can be eliminated. So, working either in $M$ or in $\vV_\om$, a \emph{weasel} $W$
is a proper class premouse, \emph{iterability}
for such $W$ is $(0,\OR)$-iterability,
and \emph{universality} says that there is no weasel $X$ that out-iterates $W$.
A class of ordinals $A$ is \emph{thick}
if it is a proper class and is $\eta$-closed for all sufficiently large regular cardinals $\eta$
(this notion suffices in the present context, because $M$ only has boundedly many measurable cardinals).
Let $W$ be a universal iterable weasel.
A \emph{thick hull} of $W$ is a class of the form $\Hull_{\Sigma_1}^{W}(A)$ for some thick class $A$ of ordinals.
Let $\Def(W)$ be the intersection of all 
thick hulls of $W$, and let $K(W)$ be the transitive collapse of $\Def(W)$.
We will observe that $K(W)$ is independent of $W$,
and define the core model $K$ to be $K(W)$ for any such $W$.

Admitting the above quantifiers over classes, the following lemma is the key point, and identifies $K$ as $\vV_\om\down 0$:
\begin{lem}\label{lem:vV_om_down_0=K}
 $\vV_\om\down 0$ is a universal iterable weasel,
and
\[ K(\vV_\om\down 0)=\Def(\vV_\om\down 0)=\vV_\om\down 0.\]
\end{lem}
\begin{proof}
 We have already verified the iterability (modulo \cite{*-trans_add}).

The universality of $\vV_\om\down 0$ in $\vV_\om$ and in $M$ is because in both of those models, for all singular strong limit cardinals $\gamma$
of cofinality $>\lambda^{+M}$,
we have $\gamma^{+\vV_\om\down 0}=\gamma^{+\vV_\om}=\gamma^{+M}$; this holds since $\Ult(\vV_\om,\vec{e}^{\vV_\om})=(\vV_\om)^{\vV_\om\down 0}\sub\vV_\om\down 0$ 
and $M$ is a generic extension of $\vV_\om$ via a forcing of size $\lambda^{+M}$.

For the last part of the lemma,
since $M$ is a set-generic extension of $\vV_\om$, it suffices to see  $\mathrm{Def}^{\vV_\om}(\vV_\om\down 0)=\vV_\om\down 0$. So suppose otherwise. Then working in  $\vV_\om\down 0$, there is a class $A$ definable from parameters over  $\vV_\om$, such that $\vV_\om\sats$ ``$A$ is a thick class of ordinals''
and $\OR\not\sub\Hull_{\Sigma_1}^{\vV_\om\down 0}(A)$. By fixing the complexity of the definition of $A$ and minimizing the parameter, we may assume that $A$ is outright definable over $\vV_\om$, without parameters. But then $\wW_\om$ satisfies the same statement. So fix such a class $A$ of $\wW_\om$. Note (by thickness) that $\mathscr{I}^M=\mathscr{I}^{\wW_\om}\sub A\sub \Hull_{\Sigma_1}^{\wW_\om\down 0}(A)$, but since $\wW_\om\down 0=M$,
therefore $\OR\sub\Hull^{\wW_\om\down 0}(A)$,
a contradiction.
\end{proof}

\begin{rem}
 We now eliminate various class quantifiers used earlier. Let $W$ be any proper class premouse. Note first that $W$ is iterable and universal iff there are set-sized iteration trees $\Tt,\Uu$
 on $\vV_\om\down 0$ and $W$ respectively,
 having common last model $Q$, with no drops on $[0,\infty]^\Tt$ and $[0,\infty]^\Uu$.
 Now fix such $W,\Tt,\Uu$. Let $A$ be any proper class containing eventually all inaccessible cardinals and with $A\cap\lambda^Q=\emptyset$. Then we claim that $W=\bar{W}$ where
 \[ \bar{W}=\cHull^{Q}_{\Sigma_1}((i^\Uu``\lambda^W)\cup A).\]
For note that $A$ is fixed pointwise by $i^\Uu$ (and by $i^\Tt$), and therefore $W|\lambda^W=\bar{W}|\lambda^{\bar{W}}$. But because $\bar{W}$ is embedded into $Q$, $\bar{W}$ is also iterable. So we can compare $\bar{W}$ with $W$.
This produces trees $\bar{\Vv}$ and $\Vv$ respectively, with common last model $R$,
and such that $[0,\infty]^{\bar{\Vv}}$ and $[0,\infty]^{\Vv}$ do not drop (because $W$ and $\bar{W}$ are both $M$-like). Now suppose $W\neq\bar{W}$. Since they have no measurables $\geq\lambda^W=\lambda^{\bar{W}}$ but $W|\lambda^W=\bar{W}|\lambda^{\bar{W}}$,
there is $n<\om$ such that $\crit(i^{\Vv})=\crit(i^{\bar{\Vv}})=\kappa_n^W=\kappa_n^{\bar{W}}$. But because $M$ has the hull property (defined in a reasonable manner)
at $\kappa_n$, so does $\vV_\om\down 0$
and so do $W$ and $\bar{W}$. This produces the usual contradiction. So $\bar{W}=W$, as desired.

Similar considerations now yield that
the universal iterable weasels are exactly
those structures $W$ of form $\cHull^{M^{\Tt}_\infty}(X\cup A)$
for some successor length $\Sigma_{\vV\down 0}$-tree
$\Tt$ such that $[0,\infty]^{\Tt}$ is non-dropping, some $X\sub\lambda^{Q}$,
and $A$ being the class of inaccessibles $>\lambda^Q$.  (Moreover, we get a canonical
$(0,\OR)$-strategy for such $W$,
given by the pullback of $\Sigma_Q\rest M$ or $\Sigma_Q\rest\vV_\om$ under the uncollapse map $\pi:W\to Q$.)

Finally, note that $\mathrm{Def}(W)$ can equivalently be defined as
$\bigcap_{\xi\in\OR\cut\lambda^W}\Hull_{\Sigma_1}^W(A_\xi)$
where  $A_\xi$ is the class of all strong limit cardinals $\eta$ of cofinality $>\xi$.
(Compare $W$ with $\vV_\om$,
producing common last model $Q$.
With $\xi$ large enough, $A_\xi$ is fixed pointwise
by both iteration maps.
By Lemma \ref{lem:vV_om_down_0=K}, $\Hull_{\Sigma_1}^Q(A_\xi)=\rg(j)$ where $j:\vV_\om\down 0\to Q$
is the iteration map. But also by that lemma,
$\Hull_{\Sigma_1}^Q(A_\xi)\sub\Hull_{\Sigma_1}^Q(A)$ for all thick $A\sub A_\xi$,
and this transfers to $W$.)

\end{rem}

\begin{proof}[Proof of Theorem \ref{tm:main} (modulo \cite{*-trans_add})]

Set $K=\vV_\om\down 0$.
Modulo \cite{*-trans_add}, we have already proven
most of the theorem (with this $K$). What remains are parts \ref{item:main_3b}, \ref{item:main_3c} and \ref{item:main_6}.

\begin{clm}
Part \ref{item:main_6} holds; that is, $\vV_\om\down 0$ is the core model of $M$
and of $\vV_\om$.
\end{clm}
\begin{proof}
By Lemma \ref{lem:vV_om_down_0=K}.
\end{proof}

\begin{clm}
Part \ref{item:main_3b} holds; that is,
the universe of $\vV_\om$ is 
the eventual generic HOD of $M$ (i.e.~$\HOD^{M[G]}$ whenever $\eta$ is a large enough ordinal and $G$ is $(M,\Coll(\om,\eta))$-generic).\end{clm}
\begin{proof}
Since it doesn't take too long,
we give a direct proof which is independent of the considerations above on the core model.
But modulo the core model there is a simpler argument.

Let $G$ be $(M,\Coll(\om,\eta))$-generic where $\eta\geq\lambda^{+M}$. We show that
$\HOD^{M[G]}$ is the universe of $\vV_\om$.

Certainly $\HOD^{M[G]}\sub\vV_\om$,
since the Vopenka forcing in $\vV_\om$,
for which $M$ is generic over $\vV_\om$,
has size $\lambda^{+M}=(\lambda^{\vV_\om})^{+\vV_\om}$.

So it suffices to see that $\vV_\om$
is definable without parameters over $M[G]$.
So work in $M[G]$.
Say a premouse $P$ is \emph{good} iff
$\OR^P=\om_1=\eta^{+M}$
and the direct condensation stack $P^+$ (see \cite[Definition 4.2]{V=HODX_pub})
over $P$ exists and is proper class
and is $M$-like, and there is some $\eta'<\omega_1$ which is a $P^+$-cardinal and some $H$ which is $(P,\Coll(\om,\eta'))$-generic and such that $P[H]$ has universe $V$ (i.e.~universe $M[G]$).
Let $\mathscr{G}$ be the set of all good premice.
Still working in $M[G]$,
note that by $M$-likeness and Lemma \ref{tm:N_computes_Sigma_vV}, we can
successfully compare all  premice  of form $\vV_\om^{P^+}\down 0$ (with $P$ ranging over $\mathscr{G}$), using the  strategy for $\vV_\om^{P^+}\down 0$ defined in the manner
given by the proof of Lemma \ref{tm:N_computes_Sigma_vV}
part \ref{item:N[g]_closed_under_Sigma_vV_down_0}.
Letting $\Tt_P$ be the tree on $\vV_\om^{P^+}\down 0$, clearly $[0,\infty]^{\Tt_P}$ is non-dropping,
and $M^{\Tt_P}_\infty=M^{\Tt_{P'}}_\infty$
is the same (proper class) model $Q$ for all $P,P'\in\mathscr{G}$.
By properties of $\Sigma_{\wW_\om}$
and its relationship to $\Sigma_{\wW_\om}\down 0$,
for each $P\in\mathscr{G}$ we get a canonical embedding $j_P:\vV_\om^{P^+}\to\vV_\om^Q$,
and $\vV_\om^Q$ is also a $\Sigma_{\vV_\om^{P^+}}$-iterate, where $\Sigma_{\vV_\om^{P^+}}$ is the canonical strategy for $\vV_\om^{P^+}$ defined using $P^+$, and $j_P$ is the iteration map.
Let $I$ be the class of all ordinals
fixed by $j_P$ for all $P\in\mathscr{G}$. Then $I$ is proper class.
Let $H=\Hull_{\Sigma_1}^{\vV_\om^Q}(I)$ and $C$ its transitive collapse and $\pi:C\to H$ the uncollapse. Now $Q,H,C,\pi$ are classes of $\vV_\om$ (since they are ordinal definable in $M[G]$ and since $\vV_\om$ is a ground of $M$
via a forcing of size $(\lambda^{\vV_\om})^{+\vV_\om}$ and $\eta^{+M}\geq(\lambda^{\vV_\om})^{+\vV_\om}$).
But then, since $M|\eta^{+M}\in\mathscr{G}$ and $(M|\eta^{+M})^+=M$,  working in $\vV_\om$,
we can define $\sigma:C\to \vV_\om$ by setting $\sigma(x)=j_{P}^{-1}(\pi(x))$, where $P=M|\eta^{+M}$, and $\sigma$ is elementary.

Suppose for some such $\eta$, we have $\sigma\neq\id$. Then working in $\vV_\om$, we can define the least such $\eta$, and hence define $\crit(\sigma)$ outright. But then $\crit(\sigma)\in\rg(\sigma)$, contradiction.

So for all such $\eta$, continuing with notation as above, we have $\sigma=\id$ and $C=\vV_\om$.
It follows that $\vV_\om\sub\HOD^{M[G]}$.

(If one makes use of what we have proved regarding the core model of $M$ and of $\vV_\om$,
there is a shorter argument. It is straightforward to see that the core model is generically invariant. Therefore $M[G]\sats$``$K=\vV_\om\down 0$'', and in particular, $\vV_\om\down 0$
is outright definable. But there can be only one structure $\vV$ with properties like $\vV_\om$
and such that $\vV_\om\down 0=\vV\down 0$:
we have
\[ \Ult(\vV,\vec{e}^{\vV})=\vV_\om^{\vV\down 0}=\Ult(\vV_\om,\vec{e}^{\vV_\om}),\]
and we have the two ultrapower maps $i^{\vV}:\vV\to\vV_\om^{\vV\down 0}$ and $i^{\vV_\om}:\vV_\om\to\vV_\om^{\vV\down 0}$. 
But letting $A$ be the class of common fixed points of $i^{\vV}$ and $i^{\vV_\om}$,
then as before, we get $\vV_\om=\Hull^{\vV_\om}(A)$, but we also have 
\[\vV\down 0=\vV_\om\down 0=K=\Hull^{\vV_\om\down 0}(A)=\Hull^{\vV\down 0}(A), \]
and so $\Hull^{\vV}(A)=\vV$,
so $\vV=\vV_\om$.)
 \end{proof}
 
 \begin{clm}
  Part \ref{item:main_3c} holds; that is, the universe of $\vV_\om$ is the mantle of $M$ and of all set-generic extensions of $M$.\end{clm}
  
 \begin{proof}
Note that for each ground $W$ of $M$, for some $\eta$ as above, defining $C$ as there, we also get $C\sub W$, but we saw that $C=\vV_\om$, so $\vV_\om\sub W$, as desired.
We similarly get that the universe of $\vV_\om$
is the mantle of every set-generic extension of $M$. 
\end{proof}

This completes the proof.
\end{proof}

\bibliographystyle{plain}
\bibliography{../bibliography/bibliography}

\end{document}